\documentclass[11pt,letterpaper]{amsart}

\allowdisplaybreaks

\usepackage{amsmath,amssymb,amsthm,amsfonts}
\usepackage{hyperref}
\usepackage{mathrsfs}
\usepackage{appendix}
\usepackage{graphicx}
\usepackage{setspace}
 \usepackage{multirow}
 \usepackage{float}

\usepackage{color}

\newtheorem{lemma}{Lemma}[section]
\newtheorem{theorem}{Theorem}[section]

\newtheorem{proposition}{Proposition}[section]
\newtheorem{remark}{Remark}[section]
\newtheorem{corollary}{Corollary}[section]

\numberwithin{equation}{section}

\newcommand{\R}{\mathbb{R}}

\newcommand{\na}{\nabla}

\newcommand{\al}{\alpha}

\newcommand{\om}{\omega}

\newcommand{\pa}{\partial}

\newcommand{\f}{\frac}

\newcommand{\beq}{\begin{equation}}
\newcommand{\eeq}{\end{equation}}
\newcommand{\ben}{\begin{eqnarray}}
\newcommand{\een}{\end{eqnarray}}
\newcommand{\beno}{\begin{eqnarray*}}
\newcommand{\eeno}{\end{eqnarray*}}

\begin{document}
\title[Nonlinear stability for 3-D Poiseuille flow]{Nonlinear stability for 3-D plane Poiseuille flow in a finite channel}


\author[Qi Chen]{Qi Chen}
\address[Q. Chen]{School of Mathematical Sciences, Zhejiang University, Hangzhou, 310058, China}
\email{chenqi123@zju.edu.cn}

\author[Shijin Ding]{Shijin Ding}
\address[S. Ding]{School of Mathematical Sciences, South China Normal University,
Guangzhou, 510631, China}
\email{dingsj@scnu.edu.cn}

\author[Zhilin Lin]{Zhilin Lin}
\address[Z. Lin]{School of Mathematical Sciences, South China Normal University,
Guangzhou, 510631, China}
\email{zllin@m.scnu.edu.cn}

\author[Zhifei Zhang]{Zhifei Zhang}
\address[Z. Zhang]{School of Mathematical Sciences, Peking University,
Beijing, 100871, China}
\email{zfzhang@math.pku.edu.cn}

\date{\today}

\begin{abstract}
In this paper, we study the nonlinear stability for the 3-D plane Poiseuille flow $(1-y^2,0,0)$ at high Reynolds number $Re$ in a finite channel $\mathbb{T}\times [-1,1 ]\times \mathbb{T}$ with non-slip boundary condition. We prove that if the initial velocity $v_0$ satisfies
$\|v_0-(1-y^2,0,0)\|_{H^{4}}\leq c_0 Re^{-\frac{7}{4}}$ for some $c_0>0$ independent of $Re$, then the solution of 3-D Naiver-Stokes equations is global in time and does not transit away from the plane Poiseuille flow. To our knowledge, this is the first nonlinear stability result for the 3-D plane Poiseuille flow and the transition threshold is accordant with the numerical result by Lundbladh et al. \cite{LHR}.
\end{abstract}

\maketitle
\tableofcontents

\vspace{-5mm}

\section{Introduction}

The hydrodynamics stability has been a main theme in the fluid mechanics since Reynolds's famous experiment in 1883 \cite{Rey}. This field is mainly  concerned with the transition of fluid motion from laminar to turbulent flow. Despite the efforts of many great scientists like Reynolds, Kelvin, Orr, Sommerfeld etc, up to now,  the hydrodynamic stability theory was incompletely understood \cite{Eck}. It is well known that some laminar flows such as plane Couette flow and pipe Poiseuille flow are linearly stable for any Reynolds number \cite{Rom, CWZ-cpam, DR, Yaglom}. However, these flows could become unstable and transit to the turbulence when Reynolds number exceeds some number, which is usually much smaller than the critical Reynolds number predicted by the eigenvalue analysis. This is so called the subcritical transition \cite{DHB} or Sommerfeld paradox \cite{Li}.

\subsection{Transition threshold problem}

Many attempts from different points of view are devoted to resolving the Sommerfeld paradox, see \cite{Tre, Chapman, Li} and  references therein. In particular, the resolution introduced by Kelvin \cite{Kelvin} is that the basin of attraction of laminar flows may shrink to zero as the Reynolds number tends to infinity. With this point of view, the transition threshold problem, which was firstly proposed by Trefethen et al. \cite{Tre} and formulated as a mathematical version in \cite{BGM-bams}, is stated as follows \smallskip

\emph{Given a norm $\Vert \cdot\Vert_X$, determine a $\gamma=\gamma(X)$ such that
$$\Vert u_0\Vert_X \ll Re^{-\gamma} \Longrightarrow \ \ stability,$$
$$\Vert u_0\Vert_X \gg Re^{-\gamma} \Longrightarrow\ \ instability.$$}
\noindent  Such exponent $\gamma$ is referred to as the transition threshold. Later on, a lot of works are devoted to determining $\gamma$ for some important laminar flows such as Couette flow and Poiseuille flow via the numerical or asymptotic analysis, see  \cite{LHR, Reddy, Chapman} and references therein.

Recently, there are many important progress on the transition threshold problem for  the 3-D plane Couette flow or Kolmogorov flow via rigorous mathematical analysis \cite{BGM1, BGM2, BGM3,  LWZ, WZ-cpam, CWZ-mem}. In particular, Chen, Wei and Zhang proved that the transition threshold is less than 1 for the 3-D plane Coutte flow in a finite channel for $H^2$ perturbations \cite{CWZ-mem}, which confirms the result of Chapman via the asymptotic analysis \cite{Chapman}.
For the 3-D Kolmogorov flow, Li, Wei and Zhang proved that the transition threshold is less than $\f 74$ for $H^2$ perturbations \cite{LWZ}.
The following table summarizes  some known results.
\begin{table}[H]
\begin{tabular}{|c|c|c|c|c|}
\hline
\textbf{Shear flow}                                                               & \textbf{Domain}                                       & \textbf{\begin{tabular}[c]{@{}c@{}}Boundary \\ condition\end{tabular}} & \textbf{\begin{tabular}[c]{@{}c@{}}Transition\\  threshold\\ (class \\ of  $X$)\end{tabular}} & \textbf{Reference} \\ \hline
\multirow{4}{*}{\begin{tabular}[c]{@{}c@{}}2-D \\ Couette \\ flow\end{tabular}}   & $\mathbb{T}\times \mathbb{R}$                         & None                                                                   & \begin{tabular}[c]{@{}c@{}}$\gamma=0$\\ (Gevrey)\end{tabular}                                      &\cite{BMV}                    \\ \cline{2-5}
                                                                                  & $\mathbb{T}\times \mathbb{R}$                         & None                                                                   & \begin{tabular}[c]{@{}c@{}}$\gamma\leq \frac{1}{2}$\\ (Sobolev)\end{tabular}                       &       \cite{BVW}             \\ \cline{2-5}
                                                                                  & $\mathbb{T}\times [-1,1]$                             & Non-slip                                                                & \begin{tabular}[c]{@{}c@{}}$\gamma\leq \frac{1}{2}$\\ (Sobolev)\end{tabular}                       &             \cite{CLWZ}       \\ \cline{2-5}
                                                                                  & $\mathbb{T}\times \mathbb{R}$                          & None                                                                   & \begin{tabular}[c]{@{}c@{}}$\gamma\leq \frac{1}{3}$\\ (Sobolev)\end{tabular}                       &      \cite{MZ, WZ-tmj}              \\ \hline
\multirow{4}{*}{\begin{tabular}[c]{@{}c@{}}3-D\\ Couette \\ flow\end{tabular}}    & $\mathbb{T}\times \mathbb{R}\times \mathbb{T}$        & None                                                                   & \begin{tabular}[c]{@{}c@{}}$\gamma=1$\\ (Gevrey)\end{tabular}                                      &\cite{BGM2}                    \\ \cline{2-5}
                                                                                  & $\mathbb{T}\times \mathbb{R}\times \mathbb{T}$        & None                                                                   & \begin{tabular}[c]{@{}c@{}}$\gamma\leq \frac{3}{2}$\\ (Sobolev)\end{tabular}                       &       \cite{BGM1}             \\ \cline{2-5}
                                                                                  & $\mathbb{T}\times \mathbb{R}\times \mathbb{T}$        & None                                                                   & \begin{tabular}[c]{@{}c@{}}$\gamma\leq 1$\\ (Sobolev)\end{tabular}                                 &    \cite{WZ-cpam}                \\ \cline{2-5}
                                                                                  & $\mathbb{T}\times [-1,1] \times \mathbb{T}$           & Non-slip                                                                & \begin{tabular}[c]{@{}c@{}}$\gamma\leq 1$\\ (Sobolev)\end{tabular}                                 &      \cite{CWZ-mem}              \\ \hline
\multirow{2}{*}{\begin{tabular}[c]{@{}c@{}}Kolmogorov \\ flow\\ $(\delta<1)$\end{tabular}}       & $\mathbb{T}_{2\pi \delta}\times \mathbb{T}^2$                                        & None                                                                   & \begin{tabular}[c]{@{}c@{}}$\gamma\leq \frac{7}{4}$\\ (Sobolev)\end{tabular}                       &          \cite{LWZ}          \\ \cline{2-5}
                                                                                  & $\mathbb{T}_{2\pi \delta}\times \mathbb{T}$ & None                                                                   & \begin{tabular}[c]{@{}c@{}}$\gamma\leq \frac{2}{3}+$\\ (Sobolev)\end{tabular}                      &      \cite{WZZ-aim}              \\ \hline
\multirow{3}{*}{\begin{tabular}[c]{@{}c@{}}2-D\\ Poiseuille\\  flow\end{tabular}} & $\mathbb{T}\times \mathbb{R}$                         & None                                                                   & \begin{tabular}[c]{@{}c@{}}$\gamma \leq \frac{3}{4}+$\\ (Sobolev)\end{tabular}                     &     \cite{CEK}               \\ \cline{2-5}
                                                                                  & $\mathbb{T}\times [-1,1]$                             & Navier-slip                                                            & \begin{tabular}[c]{@{}c@{}}$\gamma \leq \frac{3}{4}$\\ (Sobolev)\end{tabular}                      &      \cite{DL}              \\ \cline{2-5}
                                                                                  & $\mathbb{T}\times \mathbb{R}$                         & None                                                                   & \begin{tabular}[c]{@{}c@{}}$\gamma \leq \frac{2}{3}+$\\ (Sobolev)\end{tabular}                     & \cite{De-Zo}                   \\ \hline
\end{tabular}
\label{table1}
\end{table}

In the fundamental fluid dynamics, one of the oldest yet unsolved problems is to understand the hydrodynamic stability of pipe flow \cite{Ker}, which is related to Reynolds experiment. In this setting, a classical example of laminar flow is the pipe Poiseuille flow $(0,0, 1-r^2)$. For the transition threshold of this flow, both experimental and numerical results seem to indicate $\beta=1$ \cite{HJM, MM}. However, it is challenging to solve this conjecture. In a recent work, Chen, Wei and Zhang proved the linear stability of pipe Poiseuille flow at high Reynolds number regime \cite{CWZ-cpam}, which lays a foundation for the theoretical  analysis of hydrodynamic stability of pipe flow. In this paper, we consider the nonlinear stability of the 3-D plane Poiseuille flow $u_s=(1-y^2,0,0)$ in a finite channel $\Omega=\mathbb{T}\times [-1,1]\times \mathbb{T}$.  In addition to its own physical importance, we expect that the study of this flow can bring new insights into the stability of pipe Poiseuille flow due to their similarity. The plane Poiseuille flow is in fact linearly unstable when the streamwise direction is infinitely long \cite{GGN}. To study the subcritical problem,  we choose the streamwise direction to be periodic such that the flow is linearly stable at high Reynolds number regime, which will be proved in section \ref{sec:resolvent-slip}.

\subsection{Main result}

We consider the 3-D incompressible Navier-Stokes(NS) equations with high Reynolds number $Re\gg 1$:
\begin{equation}\label{ins}
\left \{
\begin{aligned}
&\partial_t v-\nu\Delta v+v\cdot\nabla v+\nabla q=0,\\
&\mathrm{div}\ v=0,\\
&v|_{t=0}=v_0(x,y,z),\ (x,y,z)\in \mathbb{T}\times [-1,1]\times \mathbb{T},
\end{aligned}\right.
\end{equation}
where $v(t,x,y,z)=(v_1,v_2,v_3)\in \mathbb{R}^3$ is the velocity field and $q\in\mathbb{R}$ is the pressure, and $\nu=Re^{-1}$ is the viscosity. It is obvious that 3-D plane Poiseuille flow $u_s=(1-y^2,0,0)$ is a steady solution to \eqref{ins} with constant pressure gradient $\nabla Q=(Q_0,0,0)$ for some constant $Q_0$.
Let  $u=v-u_s,  p=q-Q$. The perturbed Navier-Stokes system takes as follows
\begin{equation}\label{pertur-ins}
\left \{
\begin{aligned}
&\partial_t u-\nu\Delta u+(1-y^2)\partial_x u+u\cdot\nabla u+\left(
\begin{array}{c} -2y u_2 \\ 0 \\ 0\end{array}\right)+\nabla p=0,\\
&\mathrm{div}\ u=0,\\
&u|_{t=0}=u_0(x,y,z),
\end{aligned}\right.
\end{equation}
with non-slip boundary condition
\begin{equation}\label{no-slip}
u(t,x,\pm 1,z)=0.
\end{equation}

Before stating our main result, let us define
\begin{align*}
 f(t,x,y,z)=\sum_{k\in\mathbb{Z}^2}f_k(t,k_1,y,k_3)e^{ik_1x+ik_3z},
  \end{align*}
where $f_k=f_{(k_1,k_3)}=\frac{1}{(2\pi)^2}\int_{\mathbb{T}^2}f(t,x,y,z)e^{-i(k_1x+k_3z)}\mathrm{d}x\mathrm{d}z$. We define the zero mode and non-zero mode as follows
\begin{align*}
   & \overline{f}(t,y,z)=(P_0f)(t,y,z)=\dfrac{1}{2\pi}\int_{\mathbb{T}} f(t,x,y,z)\mathrm{d}x,\quad  f_{\not=}=f-P_0f.
\end{align*}

Our main result is stated as follows.

\begin{theorem}\label{main-result}
Assume that $u_0\in H^1_0\cap H^{4}(\Omega)$ with $\mathrm{div}u_0=0$. There exist constants $\nu_0,c_0,c,C>0$ independent of $\nu$ such that if $\|u_0\|_{H^{4}}\leq c_0 \nu^{\frac{7}{4}}, 0<\nu \leq \nu_0$, then the solution $u$ of the problem \eqref{pertur-ins}-\eqref{no-slip} is global in time with the following stability estimates:
\begin{itemize}

\item[(1)] Uniform bounds and decay of the background streak:
\begin{align}
&\| \overline{u}_1(t)\|_{H^2}+\|\overline{u}_1(t)\|_{L^\infty}\leq C\nu^{-1} e^{-\nu t}\|u_0\|_{H^{4}}, \label{sta-est-1}\\
&\|\overline{ u}_2(t)\|_{H^2}+\|\overline{ u}_3(t)\|_{H^1}+\|( \overline{u}_2,\overline{ u}_3)(t)\|_{L^\infty}\leq Ce^{-\nu t}\|u_0\|_{H^{4}}.\label{sta-est-2}
\end{align}

\item[(2)] Rapid convergence to a streak:
\begin{align}
&\label{streak-1} \nu^{\frac{1}{4}}\Big(\|(\partial_x,\partial_z)\partial_x u_{\not=}(t)\|_{L^2}+\|(\partial_x^2+\partial_z^2)u_{3,\not=}(t)\|_{L^2}\\
&\qquad +\|(\partial_x,\partial_z)\nabla u_{2,\not=}(t)\|_{L^2}+\nu^{\f1{24}}\|u_{2,\not=}(t)\|_{H^2}\Big)\nonumber\\
&\qquad +\nu^{s}\|u_{2,\not=}(t)\|_{L^\infty}  +\nu^{\frac{3}{8}}\Big(\|(u_{1,\not=},u_{3,\not=})(t)\|_{L^\infty}\nonumber\\
&\qquad +\|(u_{1,\not=},u_{3,\not=})(t)\|_{H^1}\Big) \leq Ce^{-c\nu^{\frac{1}{2}}t}\|u_0\|_{H^{4}},\nonumber\\
&\nu^{\frac 1 8}\Big(\|u_{\not=}\|_{L^\infty L^2}+\nu^{\frac{3}{4}}\|t(u_{1,\not=},u_{3,\not=})\|_{L^2 L^2}\Big)\label{streak-2}  \\
&\qquad +\|\nabla u_{2,\not=}\|_{L^\infty L^2}+\|\nabla u_{2,\not=}\|_{L^2 L^2}\leq C\|u_0\|_{H^{4}},\nonumber
\end{align}
where $s>0$ is {any constant independent of $\nu$}.
\end{itemize}

\end{theorem}

Let us give some remarks about our result.

\begin{itemize}

\item[1.]To our knowledge, Theorem \ref{main-result} is the first result rigorously analyzing the transition threshold of 3-D plane Poiseuille flow with non-slip boundary condition. The transition threshold $\gamma\leq \frac{7}{4}$ is accordant with the numerical result by Lundbladh et al. \cite{LHR}. We believe that this result should not be optimal. Based on the asymptotic result by Chapman \cite{Chapman}, it seems reasonable to conjecture that the transition threshold $\gamma=\frac{3}{2}$. To solve this conjecture, we have to study the secondary instability of oblique(non-zero) modes induced by  3-D lift-up effect. This is left to the future work.

\item[2.] The global estimates yield that
$$\| u(t)\|_{L^\infty} \leq c_0 Ce^{-\nu t} \to 0 \quad \mathrm{as}\ \ t\to +\infty,$$
which implies that 3-D plane Poiseuille flow is nonlinearly stable in $L^\infty$ sense provided that the perturbation is $o(\nu^{\frac{7}{4}})$ in $H^{4}$.

\item[3.] In \eqref{streak-2}, we obtain the following estimate independent of $\nu$
$$\|\nabla u_{2,\not=}\|_{L^\infty L^2}+\|\nabla u_{2,\not=}\|_{L^2 L^2}\leq C\|u_0\|_{H^{4}}.$$
This estimate is due to the inviscid damping effect  induced by the Poiseuille flow.

\item[4.] We believe that the method and framework developed in this paper could be applied to the stability threshold problem of the 2-D plane Poiseuille flow with non-slip boundary condition. In this case, it is very  interesting  to prove the stability threshold $\gamma \leq \frac{1}{2}$ as in the 2-D Couette flow \cite{CLWZ}.

\end{itemize}

\subsection{Ideas and sketch of the proof}\label{sec:analysis}

\subsubsection{Reformulation of nonlinear system}
The zero mode $P_0u$ enjoys
\begin{equation}\label{zero}
\left \{
\begin{aligned}
&(\partial_t-\nu\Delta)P_0u_1-2yP_0u_2+P_0(u\cdot\nabla u_1)=0,\\
&(\partial_t-\nu\Delta )P_0u_j+\partial_j P_0 p+(P_0 u_2\partial_y+P_0u_3\partial_z)P_0u_j \\
&\quad+P_0(P_{\not=}u\cdot\nabla (u_j)_{\not=})=0, \ j=2,3.
\end{aligned}
\right.
\end{equation}
For nonzero modes, it is more convenient to use the system in terms of $(\Delta u_2, \om_2)$ with $\omega_2=\partial_zu_1-\partial_xu_3$:
\begin{equation}\label{vor-ins}
\left \{
\begin{aligned}
&\big(\partial_t-\nu \Delta+(1-y^2)\partial_x\big)\Delta u_2+2\partial_x u_2=-\Delta_{x,z}  (u\cdot\nabla u_2)  \\
&\qquad\qquad\qquad\qquad \qquad+\partial_y[\partial_x (u\cdot \nabla u_1)+\partial_z (u\cdot \nabla u_3)],\\
&(\partial_t-\nu \Delta+(1-y^2)\partial_x)\omega_2-2y\partial_z u_2=-\partial_z (u\cdot \nabla u_1)+\partial_x (u\cdot\nabla u_3),\\
&(u_2,\partial_y u_2,\omega_2)|_{y=\pm 1}=(0,0,0).
\end{aligned}\right.
\end{equation}
The above formulation firstly introduced by Kelvin \cite{Kelvin}  plays a crucial role in this work and also in \cite{BGM1, WZ-cpam, CWZ-mem, Chapman, LWZ}.


Taking the Fourier transform in $(x,z)$, the estimates for non-zero modes could be reduced to studying the following linearized NS system
\begin{equation}\label{eq:LNS-non}
\left \{
\begin{array}{lll}
(\partial_t+\mathscr{H}_k)\omega=F,\\
(\partial_y^2-|k|^2)\varphi=\omega,\quad \varphi(\pm 1)=0,\\
\omega(\pm 1)=0,\ \
\omega|_{t=0}=\omega_0(k_1,y,k_3),
\end{array}
\right.
\end{equation}
and
\begin{equation}\label{eq:LNS-full}
\left \{
\begin{array}{lll}
(\partial_t+\mathscr{L}_k)\omega=F ,\\
  (\partial_y^2-|k|^2)\varphi=\omega,\\ \omega|_{t=0}=\omega_0, \ \varphi(\pm 1)=\varphi'(\pm 1)=0,
\end{array}\right.\end{equation}
where $|k|^2=k_1^2+k_3^2,\ k_1\neq0,$ and
\beno
&&\mathscr{H}_k=-\nu(\partial_y^2-|k|^2)+ik_1(1-y^2),\\
 &&\mathscr{L}_k=-\nu (\partial_y^2-|k|^2)+ik_1(1-y^2)+2ik_1 (\partial_y^2-|k|^2)^{-1}.
 \eeno

 As in \cite{WZ-cpam, CWZ-mem}, one of  the key ingredients is to establish the space-time estimates for the linearized NS system.

\subsubsection{Resolvent estimates}

To establish the space-time estimates of the linearized NS system \eqref{eq:LNS-non} and \eqref{eq:LNS-full}, we will use the resolvent estimate method developed in \cite{CLWZ, LWZ, CWZ-mem}, which will be conducted in section 2 and section 3.

The key point for the resolvent estimates is to solve the  Orr-Sommerfeld (OS) equation as follows
\begin{equation}\label{OS-operator}
\left \{
\begin{aligned}\mathbf{OS}[w]:=-\nu(\partial_y^2-|k|^2)w+ik_1[(1-y^2-\lambda)w+2\varphi]=F,\\
\varphi(\pm 1)=\varphi'(\pm 1)=0, \ \ (\partial_y^2-|k|^2)\varphi=w,\ \ |k|^2=k_1^2+k_3^2.
\end{aligned}
\right.
\end{equation}
The main difficulty is to deal with three kinds of singularities from the critical point $y=0$, critical layer and boundary layer respectively:
  \begin{figure}[H]
  \label{fig1}
		\centering
		\includegraphics[width=0.76\textwidth]{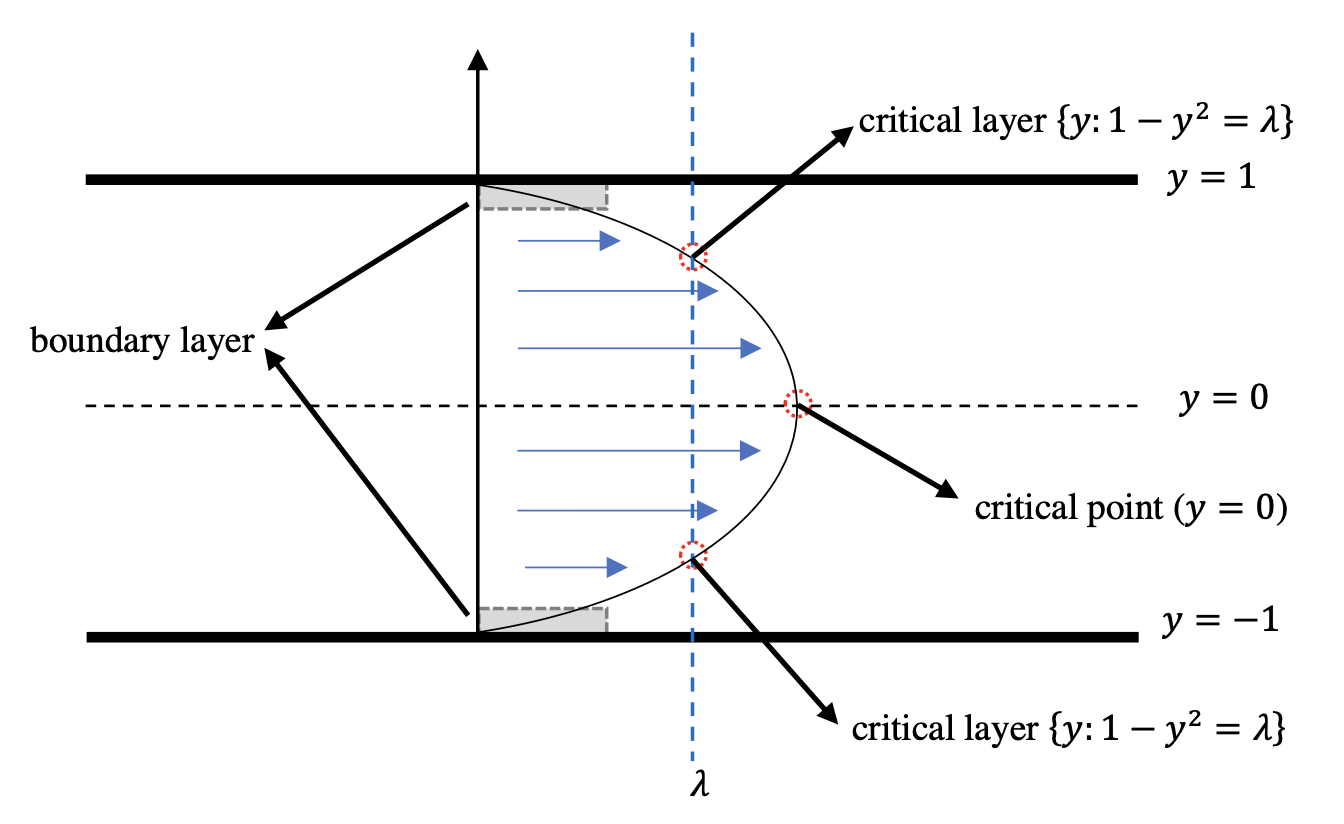}
		\caption{Plane Poiseuille flow $(1-y^2,0,0)$}		
\end{figure}

Inspired by \cite{CLWZ,CWZ-mem}, we decompose $w=w_{Na}+w_I$, where
\begin{equation*}
\left \{
\begin{aligned}
&\mathbf{OS}[w_{Na}]=F,\ (\partial_y^2-|k|^2)\varphi_{Na}=w_{Na}, (w_{Na},\varphi_{Na})|_{y=\pm 1}=(0,0),\\
&\mathbf{OS}[w_{I}]=0,\ (\partial_y^2-|k|^2)\varphi_{I}=w_{I},\, \varphi_{I}|_{y=\pm 1}=0.
\end{aligned}
\right.
\end{equation*}
The advantage of this decomposition is to separate the singularity of  the boundary layer from the singularity of the critical point and critical layer.
Indeed, the favorable boundary condition for $w_{Na}$ ensures that we can apply the energy method with a good choice of the weight to establish the resolvent estimates for $F\in L^2,H^{-1}_k, H^1_0$($H^{-1}_k$ is defined in \eqref{def:H-1k-00}). See section \ref{sec:resolvent-slip} for the details.

{\bf Compared with the Couette flow, the new difficulty is that the critical layer and the critical point will collide when $\lambda$ is  close to 1.} Thus, for $\lambda\in (0,1)$ so that $1-y_1^2=1-y_2^2=\lambda$, the proof is much difficult and involved. We need to make subtle treatments near  or away from the colliding interval $(y_1,y_2)$, where the special structure of the Poiseuille flow plays a crucial role in dealing with the colliding singularity. For example,  there holds
\begin{equation}\nonumber
\begin{aligned}
\int_{-1}^1&\rho_2^2\left|(1-y^2-\lambda)^{\frac{1}{2}}w+\frac{2 \varphi}{(1-y^2-\lambda)^{\frac{1}{2}}}\right|^2\mathrm{d}y\\
&+2\int_{-1}^1\left|(\rho_2\varphi)'+\frac{2y}{1-y^2-\lambda}\rho_2\varphi\right|^2 \mathrm{d}y+2|k|^2\int_{-1}^1\rho^2_2|\varphi|^2\mathrm{d}y\\
=&2\int_{-1}^1\partial_y(\rho_2\rho_2')|\varphi|^2\mathrm{d}y+\mathrm{Re}\langle (1-y^2-\lambda)w+2\varphi,\rho^2_2w\rangle,
\end{aligned}
\end{equation}
where $\rho_2$ is a cut-off function, see Lemma \ref{H-1-wL2inside}. Here {\bf a very interesting question} is {\it whether the stable non-monotone shear flow under the Euler equations is also linearly stable under the Navier-Stokes equations.} In this case, the main trouble is the absence of good structures as the above formula for general flows. For the stable monotone flow, Chen, Wei and Zhang overcome this difficulty by using  the compactness method \cite{CWZ-cmp}.\smallskip

{\bf Another key ingredient is to establish the weak-type resolvent estimate}(see Proposition \ref{weaktype} for example), which is crucial to obtain a sharp estimate for the velocity. This relies on the resolvent estimate for the Rayleigh equation of the Poiseuille flow
\begin{equation}\label{Rayleigh}
(1-y^2-\lambda+i\delta)W+2\Phi=f,\
(\partial_y^2-|k|^2)\Phi=W,\ \Phi(\pm 1)=0.
\end{equation}
More precisely, there holds
  \begin{align*}
     & \|(\partial_y,|k|)\Phi\|_{L^2}+\delta^{\f12}\|W\|_{L^2}+ \delta^{\f32}(|1-\lambda|+\delta)^{-\f12}\|W'\|_{L^2}\le  C\|(\partial_y,|k|)f\|_{L^2}.
  \end{align*}
This estimate has been proved in \cite{WZZ-apde} for any fixed $k$ via the limiting absorption principle. For the nonlinear stability, we have to give a precise bound for the constant $C$. For this, we need to refine the proof in \cite{WZZ-apde}. This will be conducted in the appendix \ref{sec:lap}. \smallskip

Finally, the boundary corrector $w_I$ is used to match non-slip boundary condition. Near the boundary  $\{y=\pm 1\}$, it holds that $1-y^2\sim (1-|y|)$, which implies that the main part of $w_I$ could be approximated by the Airy type function $W_{Airy}$ as in the Coutte flow. More precisely, we write $w_I=W_{Airy}+W_{err}$, where
\begin{equation*}
\left \{
\begin{array}{lll}
-\nu (\partial_y^2-|k|^2)W_{Airy}+ik_1 (2(1\pm y)-\lambda)W_{Airy}-\epsilon \nu^{\frac{1}{2}} |k_1|^{\frac{1}{2}} W_{Airy}=0,\\
 (\partial_y^2-|k|^2)\Phi_{Airy}=W_{Airy},\ \Phi_{Airy}(\pm 1)=0,
 \end{array}
\right.
\end{equation*}
and
\begin{equation*}
\left \{
\begin{array}{lll}
\mathbf{OS}[W_{err}]=-ik_1(1\pm y)^2W_{Airy}+2ik_1\Phi_{Airy},\\
 (\partial_y^2-|k|^2)\Phi_{err}=W_{err},\ \Phi_{err}(\pm 1)=W_{err}(\pm 1)=0.
\end{array}
\right.
\end{equation*}
The key point is that the source terms $(1\pm y)^2W_{Airy}\in L^2$ and $\Phi_{Airy} \in H^1_0$ could be viewed as a perturbation. Thus, the error part $W_{err}$ can be controlled via the resolvent estimates for the OS equation with Navier-slip boundary condition. See section \ref{sec-no-slip} for the details.

 \subsubsection{Space-time estimates}
 In section 4, we will establish the space-time estimates for the linearized NS system.  As in the Couette flow, the space-time estimates will incorporate some important physical effects such as the enhanced dissipation, inviscid damping and boundary layer, which play crucial roles in the nonlinear stability. Let us mention some relevant results \cite{AH, AH2, BH, BM, GNR, IJ1, IJ2, IJ3, IMM, LX, MZ2, WZ-SCM, WZZ-CPAM, Z1, Z2}.\smallskip

For the linearized system \eqref{eq:LNS-non} with $F=ik_1g_1-\partial_yg_2-ik_3g_3-g_4$, it holds that
\begin{equation*}
\begin{aligned}
&\nu \|e^{c\nu^{\frac{1}{2}}t} (\partial_y,|k|)\omega\|_{L^2 L^2}^2+(\nu |k_1|)^{\frac{1}{2}}\|e^{c\nu^{\frac{1}{2}}t} \omega\|_{L^2 L^2}^2\\
&\lesssim   \|\omega_0\|_{L^2}^2 + |k_1|^{-1}|k|^{-2}\| e^{c\nu^{\frac{1}{2}}t} (\partial_y,|k|)g_4\|_{L^2 L^2}^2+\nu^{-1}\|e^{c \nu^{\frac{1}{2}}t} g_2\|_{L^2 L^2}^2\\
&\quad+\min\{(\nu |k|^2)^{-1},(\nu |k_1|)^{-1/2} \}\|e^{c \nu^{\frac{1}{2}}t} (k_1g_1+k_3g_3)\|_{L^2 L^2}^2.
\end{aligned}
\end{equation*}
In this case, the proof is relatively simple due to good boundary condition and the absence of nonlocal term.
See Theorem \ref{thm:sp-without-nonlocal} for the details. Let us emphasize that here we do not obtain a sharp
inviscid damping estimate due to the absence of nonlocal term, which plays a crucial role near the critical point.

For the linearized system \eqref{eq:LNS-full} with $F=-ik_1f_1-\partial_y f_2-ik_3f_3$, it holds that
\begin{equation*}
\begin{aligned}
&(|k|+\nu^{\f12} |k|^2) \|e^{c\nu^{\frac{1}{2}}t} u\|_{L^\infty L^2}^2+(\nu|k|+ \nu^{\f32} |k|^2) \|e^{c\nu^{\frac{1}{2}}t} \omega\|_{ L^2 L^2 }^2\\
 &\quad+(|k_1|+\nu^{\f12} |k_1| |k|) \|e^{c\nu^{\frac{1}{2}}t} u\|_{L^2 L^2}^2+{\nu^{\f34} |k|^{\f34} |k_1|^{\f14} \|e^{c\nu^{\frac{1}{2}}t} \omega\|_{ L^2 L^2 }^2}\\
 &\quad{+\nu^{\f7{12}} (\|e^{c\nu^{\frac{1}{2}}t} \omega\|_{L^\infty L^2}^2+\nu \|e^{c\nu^{\frac{1}{2}}t} \partial_y \omega\|_{L^2 L^2}^2)}\\
 &\lesssim \|(\partial_y^2-|k|^2) \omega_0\|_{L^2}^2 +\nu^{-1} \|e^{c\nu^{\frac{1}{2}}t} (f_1,f_2,f_3)\|_{L^2 L^2}^2 .
\end{aligned}
\end{equation*}
For this system, the proof is much difficult due to non-slip boundary condition and the presence of nonlocal term.  We first introduce a decomposition $\omega=\omega_I+\omega_{H}$, where $\omega_I$ is the solution of the inhomogeneous problem with zero initial data, and $\omega_{H}$ is the solution of the homogeneous problem. The estimates of $\omega_I$ is based on the resolvent estimates established in section \ref{sec-no-slip}. The estimates of  $\omega_H$ is much involved and difficult. Motivated by \cite{CLWZ}, we need to make a further decomposition
\begin{align}\nonumber
   & \omega_H=\omega_H^{(1)}+\omega_H^{(2)}+\omega_H^{(3)},
\end{align}
where $\omega_H^{(1)}(t,k_1,y,k_3)=g(t,y)\omega_H^{(0)}$ with $g(t,y)=e^{-\nu |k|^2t-4\nu k_1^2y^2t^3/3-|\nu k_1|^{1/2}t}$ and $\omega_H^{(0)}$ solving
\begin{equation*}
\left\{
\begin{aligned}
&\partial_t\omega_H^{(0)}+ik_1(1-y^2)\omega_H^{(0)}+2ik_1\varphi_H^{(0)}=0,\\
&(\partial_y^2-|k|^2) \varphi_H^{(0)}=\omega_H^{(0)},\ \omega_H^{(0)}|_{t=0}=\omega_0,
\end{aligned}\right.
\end{equation*}
and $\omega_H^{(2)}$ satisfies
\begin{equation*}
\left \{
\begin{array}{lll}
(\partial_t+\mathcal{L}_k)\omega_H^{(2)}=\nu\partial_y^2\omega_{H}^{(1)}+ 4\nu k_1^2y^2t^2\omega_{H}^{(1)} +|\nu k_1|^{1/2}\omega_{H}^{(1)}\\
    \qquad\qquad\quad+2ik_1(g(t,y)\varphi_{H}^{(0)}-\varphi_{H}^{(1)}),\\
 (\partial_y^2-|k|^2) \varphi_H^{(2)}=\omega_H^{(2)},\
\omega_H^{(2)}|_{t=0}=0, \  \langle \omega_H^{(2)},e^{\pm |k|y}\rangle=0,
\end{array}\right.
\end{equation*}
and $\omega_H^{(3)}$ satisfies
\begin{equation*}
\left \{
\begin{array}{lll}
(\partial_t+\mathcal{L}_k)\omega_H^{(3)}=0
,\\ (\partial_y^2-|k|^2) \varphi_H^{(3)}=\omega_H^{(3)},\
\omega_H^{(3)}|_{t=0}=0, \  \langle \omega_H^{(1)}+\omega_H^{(3)},e^{\pm |k|y}\rangle=0.
\end{array}\right.
\end{equation*}
Then the estimate of $\omega^{(1)}_H$ is based on the linear inviscid damping estimates of the linearized Euler equation around the Poiseuille flow, which is highly nontrivial. The estimate of $\omega^{(2)}_H$ is based on the estimates of $\omega_I$, and the estimate of $\omega_H^{(3)}$ is based on
the estimate of boundary correctors and the associated coefficients. See Theorem \ref{thm:sp-no-slip} for the details.

\subsubsection{Energy functional and nonlinear stability}

To control zero mode $\overline{u}_1$, we introduce the following energy functional
\begin{align}\label{eq:E1-define-00}
  & E_{1}=\|\overline{u}_{1}\|_{L^{\infty}H^2}+\nu^{\f12}\|\nabla \overline{u}_{1}\|_{L^{2}H^2},
\end{align}
and to control zero modes $(\overline{u}_2,\overline{u}_3)$, we introduce
\begin{align*}
E_2=&\|\Delta\overline{u}_2\|_{L^{\infty}L^2}
+\nu^{\f12}\|\nabla\Delta\overline{u}_2\|_{L^2L^2} +\nu^{\f12}\|\Delta\overline{u}_2\|_{L^2L^2}\\
&+\|\nabla\overline{u}_3\|_{L^{\infty}L^2}+\nu^{\f12}\|\Delta\overline{u}_3\|_{L^2L^2}
+\nu^{\f12}\|\nabla\overline{u}_3\|_{L^2L^2}.
\end{align*}
The energy of zero modes is consistent with the heat diffusion mechanism.

For non-zero modes, the construction of energy functional is based on the space-time estimates for the linearized NS system in terms of  $(\Delta u_2,\omega_2)$.
Let $\Lambda_{x,z}:=(-\partial_x^2-\partial_z^2)^{1/2}$. The energy functional $E_3$ is defined by
\begin{equation}\label{eq:E3-define-00}
\begin{aligned}
      &E_{3}= E_{3,0}+E_{3,1}+E_{3,2}+E_{3,3},
\end{aligned}
\end{equation}
where $E_{3,i}(i=0,1,2,3)$ are defined by
 \begin{equation}\label{eq:E30-define-00}
\begin{aligned}
E_{3,0}=
&\nu^{\f12}\|e^{c\nu^{\frac{1}{2}}t}\Lambda_{x,z}^{1/2}\Delta u_{2,\neq }\|_{L^2 L^2}+ \nu^{\f14}\|e^{c\nu^{\frac{1}{2}}t} |\partial_{x}|^{1/2} \Lambda_{x,z}^{1/2}\nabla u_{2,\neq }\|_{L^2 L^2}\\
&+\nu^{\f7{24}}(\|e^{c\nu^{\f12}t} \Delta u_{2,\neq}\|_{L^\infty L^2}+\nu^{\f12}\|e^{c\nu^{\f12}t} \partial_y \Delta u_{2,\neq}\|_{L^2 L^2})\\
&+ \nu^{\f34}\|e^{c\nu^{\frac{1}{2}}t}(\partial_x,\partial_z) \Delta u_{2,\neq}\|_{L^2L^2} +\|e^{c\nu^{\frac{1}{2}}t}\Lambda_{x,z}^{1/2}\nabla u_{2,\neq }\|_{L^\infty L^2}\\
&+ \|e^{c\nu^{\frac{1}{2}}t}|\partial_{x}|^{1/2}\nabla u_{2,\neq }\|_{L^2 L^2}+\nu^{\f14}\|e^{c\nu^{\frac{1}{2}}t}(\partial_x,\partial_z)\nabla u_{2,\neq }\|_{L^\infty L^2}\\
&+\nu^{\f18} \|e^{c\nu^{\f12}t} |\partial_x|^{\f18}\Lambda_{x,z}^{3/8} \Delta u_{2,\neq}\|_{L^2 L^2},
\end{aligned}
\end{equation}
\begin{equation}\label{eq:E31-define-00}
\begin{aligned}
E_{3,1}= &\nu^{\f38} \big(\|e^{c\nu^{\frac{1}{2}}t}\Lambda_{x,z}^{-1/2}\nabla \omega_{2,\neq }\|_{L^\infty L^2}+ \nu^{\f12}\|e^{c\nu^{\frac{1}{2}}t}\Lambda_{x,z}^{-1/2} \Delta \omega_{2,\neq }\|_{L^2 L^2}\\
      &+{\nu^{\f34}\|e^{c\nu^{\frac{1}{2}}t}\Delta \omega_{2,\neq }\|_{L^2 L^2}}+\nu^{\f18}\|e^{c\nu^{\frac{1}{2}}t} \Lambda_{x,z}^{-1/2} \nabla    \omega_{2,\neq }\|_{L^2 L^2}\big),
\end{aligned}
\end{equation}
\begin{equation}\label{eq:E32-define-00}
\begin{aligned}
      E_{3,2}=&  \nu^{\f18} \|e^{c\nu^{\frac{1}{2}}t}\Lambda_{x,z}^{3/2}u_{3,\neq }\|_{L^\infty L^2}+ \nu^{\f12}\|e^{c\nu^{\frac{1}{2}}t}\Lambda_{x,z}^{3/2} \nabla u_{3,\neq }\|_{L^2 L^2}\\
      &+\nu^{\f14} \|e^{c\nu^{\frac{1}{2}}t}|\partial_x|^{1/2} \Lambda_{x,z}^{3/2} u_{3,\neq }\|_{L^2 L^2},
\end{aligned}
\end{equation}
and
\begin{equation}\label{eq:E33-define-00}
E_{3,3}= \nu^{\f14}\|e^{c\nu^{\frac{1}{2}}t}(\partial_x^2+\partial_z^2) u_{3,\neq }\|_{L^\infty L^2}+ \nu^{\f34}\|e^{c\nu^{\frac{1}{2}}t}(\partial_x^2+\partial_z^2) \nabla u_{3,\neq }\|_{L^2 L^2}.
\end{equation}

In section 5, we establish nonlinear interaction estimates among different modes such as $0\cdot 0\to 0, \, 0\cdot \neq \to \neq,\, \neq\cdot\neq\to \neq$ or 0. Due to different behaviors of  different modes and different velocity components,
we need to use the anisotropic product laws in the Sobolev spaces.\smallskip

In section 6,  by using the energy method and nonlinear estimates, we can show that
\begin{align*}
&E_{1}\lesssim \|{u}(0)\|_{H^2}+\nu^{-1}E_2+\nu^{-1}E_2E_{1}+{\nu^{-\f{15}{8}}}E_3^2,\\
&E_2\lesssim \big(1+\nu^{-1}E_2\big)^2\big(\|u(0)\|_{H^2}+\nu^{-\f32}E_3^2\big).
\end{align*}
See Proposition \ref{prop:E1}-\ref{prop:E2} for the details. \smallskip

In section 7, by using the space-time estimates and nonlinear estimates, we can show that
\begin{align*}
&E^2_{3,0}+E_{3,2}^2+E_{3,3}^2  \lesssim \|u(0)\|^2_{H^{3}}+{\nu^{-3}E_3^4+\nu^{-3}E_2^2E_3^2}+\nu^{-\f32} E_1^2E_3^2,\\
& E^2_{3,1} \lesssim {\|u(0)\|^2_{H^4}}+{\nu^{-\f{13}4}E_3^4+\nu^{-3}E_2^2E_3^2+\nu^{-\f32} E_1^2E_3^2}.
\end{align*}
See Proposition \ref{prop:E30}-\ref{prop:E31} for the details.\smallskip

In section 8, we prove the nonlinear stability, which is based on a bootstrap argument. Under the bootstrap assumption
 \begin{align}\nonumber
E_1\leq \varepsilon_1\nu^{\f34}, \quad E_2\leq \varepsilon_1\nu^{\f74},\quad  E_3\leq \varepsilon_1\nu^{\f74},
\end{align}
if $\|u_0\|_{H^{4}}\le c\nu^{\f74}$, the above estimates of  $E_1, E_2, E_3$ imply the global nonlinear stability. Notice that due to the lift-up effect, the bound of $E_1$ is much bigger than those of $E_2, E_3$.\medskip

Finally, we introduce some notations.  We denote by  $\|\cdot\|_{L^p L^q}:=\|\cdot \|_{L^p(0,t;L^q(D))}$ the space-time norm,
   where $D=\Omega$ or $D=[-1,1]$, and $t=T$ or $t=+\infty$ in the arguments.
 Through this paper, we assume that $|k_1 |\geq1$, and denote by $C$ a positive constant independent of $\nu,k,\lambda$. In addition, we use the notation $a \sim b$ for $C^{-1}b \leq a \leq Cb$ and $a \lesssim b$ for $a \leq Cb$, where $C>0$ is an absolute constant.

\section{Resolvent estimates with Navier-slip boundary condition}\label{sec:resolvent-slip}

In this section, we establish the resolvent estimates for the Orr-Sommerfeld equation with Navier-slip boundary condition
\begin{equation}\label{OS-nopertur}
\left \{
\begin{array}{lll}
-\nu(\partial_y^2-|k|^2)w+ik_1[(1-y^2-\lambda)w+2\varphi]=F,\\
(\partial_y^2-|k|^2)\varphi=w,\quad \varphi(\pm 1)=\varphi''(\pm 1)=0,
\end{array}
\right.
\end{equation}
where  $|k|^2=k_1^2+k_3^2$ and $k_1\neq 0$.

In the sequel, we denote $u=(\partial_y\varphi,-i|k|\varphi)$ and
$$\|u\|_{L^2}^2=\|(\partial_y,|k|)\varphi\|_{L^2}^2=\langle w,-\varphi\rangle.$$

\subsection{Resolvent estimate for $F\in L^2$}
First of all, we consider the case of $\lambda\in\mathbb{R}$.

\begin{proposition}\label{resolvent-L2}
Assume that $\lambda\in\mathbb{R}$ and $F\in L^2(I)$. Let $w\in H^2(I)$ be a solution to \eqref{OS-nopertur}. Then we have
\begin{equation}\nonumber
\nu^{\frac{3}{8}}|k_1|^{\frac{5}{8}}(\|w\|_{L^1}+|k|^{\frac{1}{2}}\|u\|_{L^2})
+\nu^{\frac{3}{4}}|k_1|^{\frac{1}{4}}\|w'\|_{L^2}+\nu^{\frac{1}{2}}|k_1|^{\frac{1}{2}}\|w\|_{L^2}\leq C\|F\|_{L^2}.
\end{equation}
\end{proposition}

\begin{proof}
The estimate $(\nu|k_1|)^{\frac{1}{2}}\|w\|_{L^2} \leq C \|F\|_{L^2}$ can be obtained by a similar argument as in Proposition 3.1 in \cite{DL}.
It is easy to see that
\begin{equation}\label{inte-1}
\begin{aligned}
&\nu\|(\partial_y,|k|)w\|_{L^2}^2+ik_1\int_{-1}^1\left[(1-y^2-\lambda)|w|^2- 2(|\varphi'|^2+|k|^2|\varphi|^2)\right]\mathrm{d}y\\
&=\langle F,w\rangle,
\end{aligned}
\end{equation}
which gives
$$\nu\|(\partial_y,|k|)w\|_{L^2}^2=\left|\mathrm{Re}\langle F,w\rangle\right|\leq \|F\|_{L^2}\|w\|_{L^2},$$
and then
$$\|w'\|_{L^2}\lesssim\nu^{-\frac{1}{2}}\|F\|_{L^2}^{\frac{1}{2}}\|w\|_{L^2}^{\frac{1}{2}}\lesssim \nu^{-\frac{3}{4}}|k_1|^{-\frac{1}{4}}\|F\|_{L^2}.$$

Now we turn to establish other resolvent estimates. \smallskip

\noindent\textbf{Case 1}. $\lambda\geq  1$.

In this case, by taking the imaginary part of \eqref{inte-1}, we get
$$\int_{-1}^1(\lambda-1+y^2)|w|^2\mathrm{d}y +2\|u\|_{L^2}^2\leq |k_1|^{-1}\Vert F\Vert_{L^2}\Vert w \Vert_{L^2}.$$

For $\delta=(\nu |k_1|^{-1})^{\frac{1}{4}}\ll 1$,  we have
$\lambda-1+y^2\geq \delta^2$ for $y\in (-1,1)\setminus(-\delta,\delta)$ and
\begin{align*}
\Vert w \Vert_{L^1}=&\int_{-\delta}^\delta|w|\mathrm{d}y+\int_{(-1,1)\setminus (-\delta,\delta)}|w|\mathrm{d}y\\
\leq &(2\delta)^{\frac{1}{2}}\Vert w \Vert_{L^2}+\Vert \sqrt{\lambda-1+y^2}w\Vert_{L^2} \|(\lambda-1+y^2)^{-\frac{1}{2}}\|_{L^2\left((-1,1)\setminus (-\delta,\delta)\right)}\\
\lesssim&\delta^{\frac{1}{2}}(\nu|k_1|)^{-\frac{1}{2}}\|F\|_{L^2}+\delta^{-\frac{1}{2}}\nu^{-\frac{1}{4}}|k_1|^{-\frac{3}{4}}\|F\|_{L^2}
\lesssim \nu^{-\frac{3}{8}}|k_1|^{-\frac{5}{8}}\|F\|_{L^2},
\end{align*}
from which and Lemma \ref{lap-w-phi}, we infer that
\begin{align*}
\Vert u \Vert_{L^2}^2\lesssim |k|^{-1}\Vert w\Vert_{L^1}^2 \lesssim \nu^{-\frac{3}{4}}|k_1|^{-\frac{5}{4}}|k|^{-1}\|F\|_{L^2}^2.
\end{align*}

\noindent\textbf{Case 2.} $\lambda \leq 0$.

Taking the imaginary part of \eqref{inte-1} and using Lemma \ref{positive} to obtain
\begin{align*}
  &(-\lambda)\|w\|_{L^2}^2+\big(\|(\partial_y,|k|)\varphi\|_{L^2}^2+|k|^2\|\varphi\|_{L^2}^2\big)/9 \\
  &\leq  (-\lambda)\|w\|_{L^2}^2+ \langle (1-y^2)w+2\varphi,w\rangle= k_1^{-1}\mathrm{Im}\langle F,w\rangle,
\end{align*}
which gives
\begin{equation}\nonumber
\begin{aligned}
(-\lambda)\|w\|_{L^2}^2+\|u\|_{L^2}^2 \lesssim  |k_1|^{-1}\|F\|_{L^2}\|w\|_{L^2}.
\end{aligned}
\end{equation}
Thanks to $(-\lambda)\|w\|_{L^2}^2+ \|(1-y^2)^{\f12}w\|_{L^2}^2-2\|u\|_{L^2}^2=k_1^{-1} \mathrm{Im}\langle F,w\rangle$, we have
\begin{equation}\nonumber
\begin{aligned}
\|(1-y^2)^{\f12}w\|_{L^2}^2\leq |k_1|^{-1}\|F\|_{L^2}\|w\|_{L^2}+2\|u\|_{L^2}^2\lesssim |k_1|^{-1}\|F\|_{L^2}\|w\|_{L^2}.
\end{aligned}
\end{equation}
The above inequality, Lemma \ref{hardy-type} and the fact
$ \nu\|\partial_yw\|_{L^2}^2\leq \|F\|_{L^2}\|w\|_{L^2},
$
yield that
\begin{align*}
    & (\nu |k_1|^2)^{\f13}\|w\|_{L^2}\leq C\|F\|_{L^2}.
 \end{align*}
Thus, we obtain
 \begin{equation}\label{improvedL2}
   \begin{aligned}
  &   \nu^{\f23}|k_1|^{\f13}\|w'\|_{L^2}+ \nu^{\f16}|k_1|^{\f56}\big(\|u\|_{L^2} +\|(1-y^2)^{\f12}w\|_{L^2}+(-\lambda)^{\f12}\|w\|_{L^2} \big)\\
    &\quad+(\nu|k_1|^2)^{\f13}\|w\|_{L^2} \lesssim \|F\|_{L^2}.
 \end{aligned}
 \end{equation}
Moreover, we have
\beno
&&\|w\|_{L^1}\lesssim \|w\|_{L^2}\lesssim \nu^{-\frac{1}{3}}|k_1|^{-\frac{2}{3}}\|F\|_{L^2},\\
 &&\|u\|_{L^2}^2 \lesssim|k|^{-1}\|w\|_{L^1}^2\lesssim\nu^{-\frac{2}{3}}|k_1|^{-\frac{4}{3}}|k|^{-1}\|F\|_{L^2}^2.
 \eeno
This implies the desired estimates.\smallskip

\noindent\textbf{Case 3.} $\lambda\in (0,1)$.

 Let $1-y_i^2=\lambda \in (0,1)$ with $-1\leq y_1\leq 0\leq y_2\leq 1, \ y_1=-y_2$, and introduce the decomposition $\varphi=\varphi_1+\varphi_2$ as follows
\begin{equation}\label{streamdeco1}
\left \{
\begin{array}{lll}
(\partial_y^2-|k|^2)\varphi_1=w,\\
\varphi_1(\pm 1)=\varphi_1(y_i)=0,\ i=1,2,\end{array}
\right.
\end{equation}
and
\begin{equation}\label{streamdeco2}
\left \{
\begin{array}{lll}
(\partial_y^2-|k|^2)\varphi_2=0,\\
\varphi_2(\pm 1)=0,\ \varphi_2(y_i)=\varphi(y_i),\ i=1,2,\end{array}
\right.
\end{equation}
It is easy to see that
\begin{equation}\nonumber
\varphi_2(y)=\left \{
\begin{aligned}
&\frac{\sinh \big(|k|(1+y)\big)}{\sinh\big(|k|(1+y_1)\big)}\varphi(y_1),&y\in [-1,y_1],\\
&\frac{\sinh\big(|k|(y-y_1)\big)\varphi(y_2)+\sinh\big(|k|(y_2-y)\big)\varphi(y_1)}{\sinh \big(|k|(y_2-y_1)\big)},&y\in [y_1,y_2],\\
&\frac{\sinh \big(|k|(1-y)\big)}{\sinh\big(|k|(1-y_2)\big)}\varphi(y_2),&y\in [y_2,1].
\end{aligned}
\right.
\end{equation}

If $\frac{y_2-y_1}{4}\leq \nu^{\frac{1}{4}}|k_1|^{-\frac{1}{4}}$, let $\delta=\nu^{\frac{1}{4}}|k_1|^{-\frac{1}{4}}\ll 1$. We know from  (3.15) in \cite{DL} that
\begin{align}\label{est:wL1-E1(w)-00}
\Vert w \Vert_{L^1((-1,1)\setminus (y_1-\delta,y_2+\delta))}^2
\lesssim \mathscr{E}_1(w)(y_2-y_1+\delta),
\end{align}
where  $\mathscr{E}_1$ is defined in Lemma \ref{lem:outside}.  From the proof of Proposition 3.1 in \cite{DL}, we can see that
\begin{align*}
 \|w\|_{L^1}\lesssim \delta^{\frac{1}{2}}\|w\|_{L^2}+\sqrt{\mathscr{E}_1(w)(y_2-y_1+\delta)}\lesssim \nu^{-\frac{3}{8}}|k_1|^{-\frac{5}{8}}\|F\|_{L^2},
\end{align*}
which along with Lemma \ref{lap-w-phi} shows
\begin{align*}
\|u\|_{L^2}^2\leq &|k|^{-1}\|w\|_{L^1}^2\lesssim \nu^{-\frac{3}{4}}|k_1|^{-\frac{5}{4}}|k|^{-1}\|F\|_{L^2}^2.
\end{align*}

If $\frac{y_2-y_1}{4}\geq \nu^{\frac{1}{4}}|k_1|^{-\frac{1}{4}}$, we take $\delta^3(y_2-y_1)|k_1|=\nu$ and then  $0<\delta\leq \nu^{\frac{1}{4}}|k_1|^{-\frac{1}{4}}\leq \frac{y_2-y_1}{4}$.
In this case, by Lemma \ref{lem:outside}, Lemma \ref{lem:inside}, and  the fact that $\mathscr{E}_1(w)\lesssim \mathscr{E}_2(w)$, we have
\begin{align}\label{est:w-E2(w)-00}
  \|w\|_{L^2}^2\leq \mathscr{E}_2(w).
\end{align}
 Therefore, we get by \eqref{est:wL1-E1(w)-00} and \eqref{est:w-E2(w)-00} that
\begin{align}\label{est:wL1-E1(w)-01}
 \|w\|_{L^1}^2&\leq (y_2-y_1+\delta)\mathscr{E}_1+\delta\|w\|_{L^2}^2+(y_2-y_1)\mathscr{E}_2\nonumber\\
 &\lesssim (y_2-y_1+\delta)(\mathscr{E}_2+\|w\|_{L^2}^2).
\end{align}

It remains to estimate each term of $\mathscr{E}_2(w)$.  Due to $|\nu/k_1|=\delta^3(y_2-y_1)$, we have
  \begin{align*}
&\frac{\nu}{|k_1|\delta (y_2-y_1)}\|w'\overline{w}\|_{L^\infty(B(y_1,\delta)\cup B(y_2,\delta))} \lesssim \delta\|w\|_{L^\infty}^2+ \delta^3\|w'\|_{L^\infty(B(y_1,\delta)\cup B(y_2,\delta))}^2,\\
&\frac{\nu^2\|w\|_{L^\infty}^2}{|k_1|^2(y_2-y_1)^3\delta^4}= \dfrac{\delta^6(y_2-y_1)^2}{(y_2-y_1)^3\delta^4}\|w\|^2_{L^\infty}\lesssim \delta\|w\|_{L^\infty}^2 ,\\
&\frac{\nu\|w'\|_{L^2}\|w\|_{L^\infty}}{|k_1|\delta^{\frac{3}{2}}(y_2-y_1)}
=\delta^{\f32}\|w'\|_{L^2}\|w\|_{L^\infty}\lesssim\delta^2\|w'\|_{L^2}^2 +\delta\|w\|_{L^\infty}^2 ,\\
&\dfrac{\|F\|_{L^2}^2}{|k_1|^2(y_2-y_1)^3\delta}\lesssim \dfrac{\|F\|_{L^2}^2}{|k_1|^2(y_2-y_1)^2\delta^2},\\
& \dfrac{\nu^2\|w'\|_{L^2}^2}{|k_1|^2(y_2-y_1)^3\delta^3}= \dfrac{\delta^6(y_2-y_1)^2}{(y_2-y_1)^3\delta^3}\|w'\|^2_{L^2}\lesssim \delta^2\|w'\|_{L^2}^2.
\end{align*}
Therefore, we conclude that
\begin{align}\label{est:e2(w)-00}
\mathscr{E}_2(w) \lesssim& |k_1\delta(y_2-y_1)|^{-1}\|F\|_{L^2}\|w\|_{L^2} +|k_1\delta(y_2-y_1)|^{-2}\|F\|_{L^2}^2+\delta\|w\|_{L^\infty}^2\nonumber\\
   &+\delta^2\|w'\|_{L^2}^2+\delta^3\|w'\|_{L^\infty(B(y_1,\delta)\cup B(y_2,\delta))}^2+\frac{\|\varphi_2\|_{L^\infty}^2}{(y_2-y_1)^2\delta}.
\end{align}
By Lemma \ref{lem:varphi-Linfty} and Lemma \ref{lem:dw-Linfty}, we get
\begin{align*}
\frac{\|\varphi_2\|_{L^\infty}^2}{(y_2-y_1)^2\delta}\lesssim& \mathscr{F}_1(w)\leq  |k_1\delta(y_2-y_1)|^{-1}\|F\|_{L^2}\|w\|_{L^2} \\
&+|k_1\delta(y_2-y_1)|^{-2}\|F\|_{L^2}^2+\delta\|w\|_{L^\infty}^2 \end{align*}
and
\begin{align*}
\delta^3\|w'&\|_{L^\infty(B(y_1,\delta)\cup B(y_2,\delta))}^2\lesssim \frac{\delta^6(y_2-y_1)^2|k|^2}{\nu^2}\mathscr{F}_1(w)+\frac{\delta^4}{\nu^2}\|F\|_{L^2}^2+\delta\|w\|_{L^\infty}^2\\
\lesssim &|k_1\delta(y_2-y_1)|^{-1}\|F\|_{L^2}\|w\|_{L^2} +|k_1\delta(y_2-y_1)|^{-2}\|F\|_{L^2}^2+\delta\|w\|_{L^\infty}^2.\end{align*}
Plugging them into \eqref{est:e2(w)-00}, we obtain
\begin{align*}
     \mathscr{E}_2(w) \lesssim& |k_1\delta(y_2-y_1)|^{-1}\|F\|_{L^2}\|w\|_{L^2} \\
    &+|k_1\delta(y_2-y_1)|^{-2}\|F\|_{L^2}^2+\delta\|w\|_{L^\infty}^2+
    \delta^2\|w'\|_{L^2}^2,
 \end{align*}
 which along with $\delta\|w\|_{L^\infty}^2\lesssim\delta\|w\|_{L^2}\|w'\|_{L^2}$, $\|w'\|_{L^2}^2\leq\nu^{-1}\|F\|_{L^2}\|w\|_{L^2}$ and \eqref{est:w-E2(w)-00} gives
\begin{align}\label{eq:wL2-1prove}
\|w\|_{L^2}+\mathscr{E}_2(w)^{\f12}&\lesssim |k_1\delta(y_2-y_1)|^{-1}\|F\|_{L^2}\nonumber\\
&= \nu^{-\frac{1}{3}}|k_1|^{-\frac{2}{3}}(y_2-y_1)^{-\frac{2}{3}}\|F\|_{L^2}.
\end{align}
Therefore, we infer from \eqref{est:wL1-E1(w)-01} and \eqref{eq:wL2-1prove}  that
\begin{align}\label{est:u-FL2-01}
|k|\|u\|_{L^2}^2+\|w\|_{L^1}^2\lesssim &\|w\|_{L^1}^2 \lesssim (y_2-y_1+\delta)(\mathscr{E}_2+\|w\|_{L^2}^2)\nonumber\\
\lesssim& \nu^{-\frac{2}{3}}|k_1|^{-\frac{4}{3}}(y_2-y_1)^{-\frac{1}{3}}\|F\|_{L^2}^2.
\end{align}
Thanks to $\nu^{\frac{1}{4}}|k_1|^{-\frac{1}{4}}\leq \frac{y_2-y_1}{4}$, we arrive at
\begin{align*}
  |k|\|u\|_{L^2}^2+\|w\|_{L^1}^2
\lesssim& \nu^{-\frac{2}{3}}|k_1|^{-\frac{4}{3}}(\nu^{-\frac{1}{12}}|k_1|^{\frac{1}{12}})\|F\|_{L^2}^2
\lesssim \nu^{-\frac{3}{4}}|k_1|^{-\frac{5}{4}}\|F\|_{L^2}^2.
\end{align*}
\end{proof}

As a corollary, we can establish the resolvent estimate when $\lambda\in \Omega_\epsilon=\{-i\epsilon\nu^{\frac{1}{2}}|k_1|^{-\frac{1}{2}}+\lambda_r:\lambda_r\in\mathbb{R}\}$.

\begin{corollary}\label{resolvent-L2-complex}
Let $w\in H^2(I)$ be a solution to \eqref{OS-nopertur} with $\lambda\in\Omega_\epsilon$ and $F\in L^2(I)$. Then there exists  $\epsilon$ small enough such that
\begin{equation}\nonumber
\nu^{\frac{3}{8}}|k_1|^{\frac{5}{8}}(\|w\|_{L^1}+|k|^{\frac{1}{2}}\|u\|_{L^2})
+\nu^{\frac{3}{4}}|k_1|^{\frac{1}{4}}\|w'\|_{L^2}+\nu^{\frac{1}{2}}|k_1|^{\frac{1}{2}}\|w\|_{L^2}\leq C\|F\|_{L^2}.
\end{equation}
\end{corollary}

\begin{proof}
Let $\tilde{F}=F+\epsilon \nu^{\frac{1}{2}}|k_1|^{\frac{1}{2}}w$. Then by taking $\epsilon$ small enough, we get
\begin{equation}\nonumber
\begin{aligned}
\|F\|_{L^2}\geq& \|\tilde{F}\|_{L^2}-\epsilon(\nu|k_1|)^{\frac{1}{2}}\|w\|_{L^2}\geq(1-C\epsilon)\|\tilde{F}\|_{L^2}\\
\gtrsim&
\nu^{\frac{3}{8}}|k_1|^{\frac{5}{8}}(\|w\|_{L^1}+|k|^{\frac{1}{2}}\|u\|_{L^2})
+\nu^{\frac{3}{4}}|k_1|^{\frac{1}{4}}\|w'\|_{L^2}+\nu^{\frac{1}{2}}|k_1|^{\frac{1}{2}}\|w\|_{L^2}.
\end{aligned}
\end{equation}
\end{proof}

\begin{corollary}\label{resol-lambda}
Under the same assumptions of Corollary \ref{resolvent-L2-complex}, it holds that
 \begin{equation}\nonumber
 \begin{aligned}
\nu^{\frac{2}{3}}|k_1|^{\frac{1}{3}}&\big( |\lambda-1|^{\frac{1}{2}}+|\nu/k_1|^{\frac{1}{4}}\big)^{\frac{1}{3}}\|(\partial_y,|k|)w\|_{L^2}\\
& +\nu^{\frac{1}{3}}|k_1|^{\frac{2}{3}}\big( |\lambda-1|^{\frac{1}{2}}+|\nu/k_1|^{\frac{1}{4}}\big)^{\frac{2}{3}}\|w\|_{L^2}\lesssim \|F\|_{L^2}.
 \end{aligned}
\end{equation}
\end{corollary}
\begin{remark}\label{rk-resolvent-1}
This result implies that away from the critical point $y=0$,  we can obtain a resolvent estimate such as
 \beno
 (\nu |k_1|^2)^{\frac{1}{3}}\|w\|_{L^2}\lesssim \|F\|_{L^2},
  \eeno
 which is the same as one for the monotone flow considered in \cite{CWZ-cmp}. Near the critical point, we just have
  \beno
 (\nu |k_1|)^{\frac{1}{2}}\|w\|_{L^2}\lesssim \|F\|_{L^2}.
  \eeno
 \end{remark}

\begin{proof}
We only consider the case of $\lambda\in \R$. The case of $\lambda\in \Omega_\epsilon$ can be proved by the perturbation argument as in Corollary \ref{resolvent-L2-complex}.\smallskip

\noindent{\bf Case 1.} {$\lambda\geq 1$.}

First of all, we know that
$$\nu^{\frac{1}{3}}|k_1|^{\frac{2}{3}} (\nu^{\frac{1}{4}} |k_1|^{-\frac{1}{4}})^{\frac{2}{3}}\|w\|_{L^2} =\nu^{\frac{1}{2}} |k_1|^{\frac{1}{2}}\|w\|_{L^2}\lesssim \|F\|_{L^2},$$
It remains to show $ (\nu |k_1|^2)^{\frac{1}{3}}|\lambda-1|^{\frac{1}{3}}\|w\|_{L^2}\lesssim \|F\|_{L^2}$. For $\lambda\geq 1$,  we have
 \begin{equation*}
|\lambda-1|\|w\|_{L^2}^2\leq |k_1|^{-1}\|F\|_{L^2}\|w\|_{L^2},
\end{equation*}
which gives
 \begin{equation*}
|k_1| |\lambda-1| \|w\|_{L^2}\lesssim \|F\|_{L^2}.
\end{equation*}
Therefore, we obtain
 \begin{equation*}
 \begin{aligned}
(\nu |k_1|^2)^{\frac{1}{3}}|\lambda-1|^{\frac{1}{3}}\|w\|_{L^2}=(|k_1| |\lambda-1| \|w\|_{L^2})^{\frac{1}{3}} (\nu^{\frac{1}{2}} |k_1|^{\frac{1}{2}}\|w\|_{L^2})^{\frac{2}{3}}\lesssim \|F\|_{L^2}.
\end{aligned}
\end{equation*}

\noindent{\bf Case 2.} {$\lambda\leq 0$.}

In this case, we have $(\nu |k_1|^2)^{\frac{1}{3}}\|w\|_{L^2}\lesssim \|F\|_{L^2}$ by \eqref{improvedL2}. Due to $(\nu |k_1|^{-1})^{\frac{1}{4}}\ll 1\leq |1-\lambda|^{\frac{1}{2}}$, it remains to show $ (\nu |k_1|^2)^{\frac{1}{3}}|\lambda-1|^{\frac{1}{3}}\|w\|_{L^2}\lesssim \|F\|_{L^2}$.
Indeed, we find that
 \begin{equation*}
  \begin{aligned}
|\lambda-1|\|w\|_{L^2}^2\leq & |k_1|^{-1}\|F\|_{L^2}\|w\|_{L^2}+\|yw \|_{L^2}^2+2\|(\partial_y,|k|)\varphi\|_{L^2}^2\\
\lesssim& |k_1|^{-1}\|F\|_{L^2}\|w\|_{L^2}+\|w \|_{L^2}^2\\
\lesssim& \nu^{-\frac{1}{3}}|k_1|^{-\frac{5}{3}}\|F\|_{L^2}^2+ \nu^{-\frac{2}{3}}|k_1|^{-\frac{4}{3}}\|F\|_{L^2}^2\lesssim \nu^{-\frac{2}{3}}|k_1|^{-\frac{4}{3}}\|F\|_{L^2}^2,
\end{aligned}
\end{equation*}
which yields that
\begin{equation*}
  \begin{aligned}
  \nu^{\frac{1}{3}}|k_1|^{\frac{2}{3}}|\lambda-1|^{\frac{1}{2}} \|w\|_{L^2}\lesssim \|F\|_{L^2}.
\end{aligned}
\end{equation*}
Thus, we have
 \begin{equation*}
 \begin{aligned}
(\nu |k_1|^2)^{\frac{1}{3}}|\lambda-1|^{\frac{1}{3}}\|w\|_{L^2}=( \nu^{\frac{1}{3}}|k_1|^{\frac{2}{3}}|\lambda-1|^{\frac{1}{2}} \|w\|_{L^2})^{\frac{2}{3}} ( \nu^{\frac{1}{3}}|k_1|^{\frac{2}{3}}\|w\|_{L^2})^{\frac{1}{3}}\lesssim \|F\|_{L^2}.
\end{aligned}
\end{equation*}

\noindent{\bf Case 3.} {$\lambda \in (0,1)$.}

Let $1-y_i^2=\lambda$ with $y_1=-y_2$ and $|1-\lambda|^{\frac{1}{2}}= (y_2-y_1)/2$. If $|1-\lambda|^{\frac{1}{2}}\leq 4|\nu /k_1|^{\frac{1}{4}}$, then we get by $(\nu |k_1|)^{\frac{1}{2}}\|w\|_{L^2}\lesssim \|F\|_{L^2}$ that
\begin{equation}\nonumber
 \begin{aligned}
(\nu |k_1|^2)^{\frac{1}{3}} (|\lambda-1|^{\frac{1}{2}}+\nu^{\frac{1}{4}} |k_1|^{-\frac{1}{4}})^{\frac{2}{3}} \|w\|_{L^2}
&\lesssim
(\nu |k_1|^2)^{\frac{1}{3}} (\nu^{\frac{1}{4}} |k_1|^{-\frac{1}{4}})^{\frac{2}{3}} \|w\|_{L^2}\\
&=(\nu |k_1|)^{\frac{1}{2}} \|w\|_{L^2}\leq C \|F\|_{L^2}.
\end{aligned}
\end{equation}
If $|1-\lambda|^{\frac{1}{2}}\geq 4 |\nu /k_1|^{\frac{1}{4}}$, then $|1-\lambda|^{\frac{1}{2}}= (y_2-y_1)/2\sim |1-\lambda|^{\frac{1}{2}}+|\nu /k_1|^{\frac{1}{4}}$, and we can deduce from \eqref{eq:wL2-1prove} that
\begin{align*}
\|w\|_{L^2}\lesssim& \nu^{-\frac{1}{3}}|k_1|^{-\frac{2}{3}}(y_2-y_1)^{-\frac{2}{3}}\|F\|_{L^2}\\
\lesssim&\nu^{-\frac{1}{3}}|k_1|^{-\frac{2}{3}}\big( |\lambda-1|^{\frac{1}{2}}+|\nu/k_1|^{\f14}\big)^{-\frac{2}{3}}\|F\|_{L^2}.
\end{align*}

Combining the above cases, we arrive at $\nu^{\frac{1}{3}}|k_1|^{\frac{2}{3}}\big( |\lambda-1|^{\frac{1}{2}}+|\nu/k_1|^{\frac{1}{4}}\big)^{\frac{2}{3}}\|w\|_{L^2}\lesssim \|F\|_{L^2}$, which along with $\nu\|(\partial_y,|k|)w\|_{L^2}^2\leq\|F\|_{L^2}\|w\|_{L^2}$ shows
 \begin{equation}\nonumber
\nu^{\frac{2}{3}}|k_1|^{\frac{1}{3}}\big( |\lambda-1|^{\frac{1}{2}}+|\nu/k_1|^{\frac{1}{4}}\big)^{\frac{1}{3}}\|(\partial_y,|k|)w\|_{L^2} \lesssim \|F\|_{L^2}.
\end{equation}
\end{proof}

\subsection{Weak-type resolvent estimate for $F\in L^2$}

To establish the weak-type resolvent estimate, we first need to introduce some estimates for the Rayleigh equation
\begin{equation}\label{eq:Ray-01}
(1-y^2-\lambda+i\delta)W+2\Phi=f,\quad
(\partial_y^2-|k|^2)\Phi=W,\quad \Phi(\pm 1)=0,
\end{equation}
here $\lambda\in \R$ and we denote by $W=\mathbf{Ray}_{\delta}^{-1}f$ the solution.

\begin{lemma}\label{lem:Ray-simple-1}
  Let $(\Phi,W)$ solve \eqref{eq:Ray-01} with $f(\pm1)=0$ and $\delta\ll 1$. Then it holds that
  \begin{align*}
     & \|(\partial_y,|k|)\Phi\|_{L^2}+\delta^{\f12}\|W\|_{L^2}+ \delta^{\f32}(|1-\lambda|+\delta)^{-\f12}\|W'\|_{L^2}\le  C\|(\partial_y,|k|)f\|_{L^2}.
  \end{align*}
\end{lemma}

\begin{proof}
Using Lemma \ref{est-1} for large $|k|$ and Proposition 6.1 of \cite{WZZ-apde} for others $|k|$, we have
\begin{align*}
   & \|(\partial_y,|k|)\Phi\|_{L^2}\lesssim\|(\partial_y,|k|)f\|_{L^2}.
\end{align*}
 Taking the inner product between \eqref{eq:Ray-01} and $W$, we obtain
  \begin{align*}
     & \int_{-1}^{1}(1-y^2-\lambda)|W|^2\mathrm{d}y-\|(\partial_y,|k|)\Phi\|^2_{L^2} +i\delta \|W\|_{L^2}^2 =\langle f,-W\rangle,
  \end{align*}
  which gives
  \begin{align*}
     &\delta\|W\|_{L^2}^2\lesssim \|(\partial_y,|k|)f\|_{L^2}\|(\partial_y,|k|)\Phi\|_{L^2}\lesssim \|(\partial_y,|k|)f\|_{L^2}^2.
  \end{align*}

Thanks to
  \begin{align*}
     & W'=\dfrac{f'-2\Phi'}{1-y^2-\lambda+i\delta}+\dfrac{2yW}{1-y^2-\lambda+i\delta},
  \end{align*}
we infer that
\begin{align*}
   \|W'\|_{L^2}&\leq \delta^{-1}\|f'-2\Phi'\|_{L^2}+ 2\|W\|_{L^2}\left\|\dfrac{y}{1-y^2-\lambda+i\delta}\right\|_{L^\infty}\\
   &\lesssim \delta^{-1}\|(\partial_y,|k|)f\|_{L^2}+ \delta^{-\f12}\|(\partial_y,|k|)f\|_{L^2}\cdot \delta^{-1}(|1-\lambda|+\delta)^{\f12}\\
   &\lesssim \delta^{-\f32}(|1-\lambda|+\delta)^{\f12}\|(\partial_y,|k|)f\|_{L^2}.
\end{align*}
\end{proof}

\begin{lemma}\label{weaktype-FL2}
Let $(\varphi,w)$ be a solution of \eqref{OS-nopertur}. If $\nu|k|^3|k_1|^{-1}\leq|\lambda-1|^{\f12}+|\nu/k_1|^{\f14}$ and ${ F}\in L^2(I)$, then we have
\begin{align*}
|\langle w,f\rangle|
\le &C|k_1|^{-1}\|F\|_{L^2}\big(
\|\mathbf{Ray}_{\delta_1}^{-1}f\|_{L^2}+|\nu/k_1|^{\f12}\delta_1^{-\f12} \|\partial_y(\mathbf{Ray}_{\delta_1}^{-1}f)\|_{L^2}\big).
\end{align*}
where $\delta_1 =|\nu /k_1|^{\frac{1}{3}}\big(|\lambda-1|^{\frac{1}{2}}+|\nu/k_1|^{\f14}\big)^{\frac{2}{3}}$.
\end{lemma}

\begin{proof}
For $\phi\in H^1_0(-1,1)$, we have
\begin{align*}
\|F\|_{L^2}\|\phi\|_{L^2}\geq& |\langle F,\phi\rangle|
=|-\nu(\partial_y^2-|k|^2)w+ik_1(1-y^2-\lambda)w
+2ik_1\varphi,\phi\rangle|\\
\geq &-\nu\|w'\|_{L^2}\|\phi'\|_{L^2}-\big|{\nu|k|^2}-k_1\delta_1\big|\|w\|_{L^2}\|\phi\|_{L^2}\\
&+|\langle k_1(1-y^2-\lambda-i\delta_1)w+2k_1(\partial_y^2-|k|^2)^{-1}w, \phi\rangle|.
\end{align*}
Thanks to $ \nu|k|^3|k_1|^{-1}\leq|\lambda-1|^{\f12}+|\nu/k_1|^{\f14}$, we have
$\left|\nu|k|^2-k_1\delta_1\right|\lesssim |k_1|\delta_1.$ Then we infer that
\begin{equation*}
\begin{aligned}
&|k_1||\langle w, (1-y^2-\lambda+i\delta_1)\phi+2(\partial_y^2-|k|^2)^{-1}\phi\rangle|\\
&\lesssim  \|F\|_{L^2}\|\phi\|_{L^2}+\nu\|w'\|_{L^2}\|\phi'\|_{L^2}+ |k_1|\delta_1\|w\|_{L^2}\|\phi\|_{L^2},
\end{aligned}
\end{equation*}
which along with Corollary \ref{resol-lambda} gives
\begin{equation}\label{basicphi-FL2}
  \begin{aligned}
   &|k_1||\langle w, (1-y^2-\lambda+i\delta_1)\phi+2(\partial_y^2-|k|^2)^{-1}\phi\rangle| \\ &\lesssim \|F\|_{L^2}\|\phi\|_{L^2}+|\nu/k_1|^{\f12}\delta_1^{-\f12}\|F\|_{L^2}\|\phi'\|_{L^2}.
\end{aligned}
\end{equation}
Now the lemma follows by taking $\phi=\mathbf{Ray}_{\delta_1}^{-1}f$.
\end{proof}

\begin{proposition}\label{resol-lambda-00}
Under the same assumptions of Corollary \ref{resolvent-L2-complex},  it holds that
 \begin{equation}\label{resol-est-L2-w-00}
\nu^{\f16}|k_1|^{\f56}\big( |\lambda-1|^{\frac{1}{2}}+|\nu/k_1|^{\frac{1}{4}}\big)^{\frac{1}{3}}\|u\|_{L^2}\le C\|F\|_{L^2}.
\end{equation}
\end{proposition}
\begin{proof}
We only consider the case of $\lambda\in \R$, and the case of $\lambda\in \Omega_\epsilon$ can be proved by a perturbation argument.
\smallskip

\noindent{\textbf{Case 1.} $\nu|k|^3|k_1|^{-1}\leq|\lambda-1|^{\f12}+|\nu/k_1|^{\f14}$}.

By Lemma \ref{lem:Ray-simple-1}, we have
\begin{align*}
    & \|\mathbf{Ray}_{\delta_1}^{-1}\varphi\|_{L^2}+|\nu/k_1|^{\f12}\delta_1^{-\f12} \|\partial_y(\mathbf{Ray}_{\delta_1}^{-1}\varphi)\|_{L^2} \\
    &\lesssim \delta_1^{-\f12}\big(1+|\nu/k_1|^{\f12}\delta_1^{-\f32}(|1-\lambda|+\delta_1)^{\f12}\big) \|u\|_{L^2}.
\end{align*}
Taking $\delta_1 =|\nu /k_1|^{\frac{1}{3}}\big(|\lambda-1|^{\frac{1}{2}}+|\nu/k_1|^{\f14}\big)^{\frac{2}{3}}$, we have $|1-\lambda|+\delta_1\lesssim \big(|\lambda-1|^{\frac{1}{2}}+|\nu/k_1|^{\f14}\big)^{2}$, and then $|\nu/k_1|^{\f12}\delta_1^{-\f32}(|1-\lambda|+\delta_1)^{\f12}\lesssim1$ and
\begin{align*}
   & \|\mathbf{Ray}_{\delta_1}^{-1}\varphi\|_{L^2}+|\nu/k_1|^{\f12}\delta_1^{-\f12} \|\partial_y(\mathbf{Ray}_{\delta_1}^{-1}\varphi)\|_{L^2} \lesssim \delta_1^{-\f12}\|u\|_{L^2}.
\end{align*}
Using Lemma \ref{weaktype-FL2} with $f=\varphi$, we infer that
\begin{align*}
   & \|u\|_{L^2}^2=|\langle w,f\rangle|\lesssim |k_1|^{-1}\delta_1^{-\f12}\|F\|_{L^2}\|u\|_{L^2}.
\end{align*}

\noindent{\textbf{Case 2.} $\nu|k|^3|k_1|^{-1}\geq|\lambda-1|^{\f12}+|\nu/k_1|^{\f14}$}.

 In this case, $\nu|k|^3|k_1|^{-1} \geq |\nu/k_1|^{\f14}$ and $\nu|k|^4|k_1|^{-1}\geq 1$. By Corollary \ref{resol-lambda}, we get
\begin{align*}
    \|u\|_{L^2}&\leq |k|^{-1}\|w\|_{L^2}\leq  |k|^{-2}\|(\partial_y,|k|)w\|_{L^2}\\
    &\lesssim \nu^{-\f23}|k_1|^{-\f13}|k|^{-2}\big(|\lambda-1|^{\f12}+|\nu/k_1|^{\f14}\big)^{-\f13} \|F\|_{L^2}\\
    &=\nu^{-\f16}|k_1|^{-\f56}\big( |\lambda-1|^{\frac{1}{2}}+|\nu/k_1|^{\frac{1}{4}}\big)^{-\frac{1}{3}}\|F\|_{L^2} \times \big(\nu|k|^4|k_1|^{-1}\big)^{-\f12}\\
    &\lesssim\nu^{-\f16}|k_1|^{-\f56}\big( |\lambda-1|^{\frac{1}{2}}+|\nu/k_1|^{\frac{1}{4}}\big)^{-\frac{1}{3}}\|F\|_{L^2},
\end{align*}
which gives the desired bound.
\end{proof}



\subsection{Resolvent estimate for $F\in {H_k^{-1}}$}
For $k=(k_1,k_3)\in\mathbb{Z}^2,\ k_1\neq 0$, we define the norms $\|\cdot\|_{H^{1}_k}$ and $\|\cdot\|_{H^{-1}_k}$ as
\begin{equation}\label{def:H-1k-00}
  \|f\|_{H^{-1}_k}:=\inf_{\{g\in H_0^1:\|g\|_{H^1_k}=1\}} \big\langle f,g\big\rangle, \ \|g\|_{H^1_k}:=\|(k_1g,\partial_yg,k_3g)\|_{L^2},
\end{equation}
This subsection is to establish the resolvent estimates for the Orr-Sommerfeld equation when $F\in H^{-1}_k$. Consider
\begin{equation}\label{OSH-1}
\left \{
\begin{array}{lll}
-\nu(\partial_y^2-|k|^2)w+ik_1[(1-y^2-\lambda)w+2\varphi]-\epsilon (\nu |k_1|)^{\frac{1}{2}}w=F,\\
 (\partial_y^2-|k|^2)\varphi=w,\quad \varphi(\pm 1)=\varphi''(\pm 1)=0,
 \end{array}
\right.
\end{equation}
where $0<\epsilon\ll 1,\lambda\in\mathbb{R}$.

\begin{proposition}\label{resolvent-H-1}
 Let $w\in H^1(I)$ be a solution of \eqref{OSH-1} with $F\in H^{-1}_k(I)$. Then it holds that
$$\nu\|w'\|_{L^2}+\nu^{\frac{3}{4}}|k_1|^{\frac{1}{4}}\|w\|_{L^2}\leq C\|F\|_{H^{-1}_k}.$$
\end{proposition}

To prove Proposition \ref{resolvent-H-1}, let us present some key lemmas, which will be applied in the case of $\lambda \in (0,1)$.

\begin{lemma}\label{H-1-wL2outside}
Assume that $-1\leq y_1\leq 0\leq y_2\leq 1$ with $1-y^2_i=\lambda\in (0,1), y_1=-y_2$. Then it holds that for $\delta\in (0,1]$,
\begin{equation}\label{H-1-wL2outside-estimate}
\|w\|_{L^2((-1,1)\setminus (y_1,y_2))}^2\lesssim\mathscr{E}_3,
\end{equation}
where
\begin{equation}\nonumber
\begin{aligned}
\mathscr{E}_3=&
\delta\|w\|_{L^\infty}^2+\frac{\|F\|_{H_k^{-1}}\|(\partial_y,|k|)w\|_{L^2}}{|k_1|(y_2-y_1+\delta)\delta}
+\frac{\|F\|_{H^{-1}_k}\|w\|_{L^\infty}}{|k_1|(y_2-y_1+\delta)\delta^{\frac{3}{2}}}
\\
&+\frac{\nu\|w'\|_{L^2}\|w\|_{L^\infty}}{|k_1|(y_2-y_1+\delta)\delta^{\frac{3}{2}}}+\frac{\|\varphi\|_{L^\infty(B(y_1,\delta)\cup B(y_2,\delta))}^2}{(y_2-y_1+\delta)^2\delta}.
\end{aligned}
\end{equation}
\end{lemma}

\begin{proof}
We follow the proof of Lemma 3.4 in \cite{DL}. For any $\delta\in (0,1]$, we introduce a smooth cut-off function $\rho_1(y)\in [0,1]$ as follows
\begin{equation*}
\rho_1(y)=\left \{
\begin{array}{lll}
0,&y\in (y_1-\frac{\delta}{2},y_2+\frac{\delta}{2}),\\
1,&y\in (-1,1)\setminus (y_1-\delta,y_2+\delta),\\
\mathrm{smooth},&\mathrm{others}.
\end{array}
\right.
\end{equation*}

Taking the inner product between \eqref{OSH-1} and $\rho_1^2w$ and taking the imaginary part to the resulting equation, we get
\begin{equation}\label{outside1}
\begin{aligned}
&\int_{-1}^1\rho_1^2\bigg((\lambda-1+y^2)|w|^2+2(|\varphi'|^2+|k|^2|\varphi|^2)\bigg)\mathrm{d}y\\
&\lesssim \frac{\|F\|_{H^{-1}_k}\|(\partial_y,|k|)w\|_{L^2}}{|k_1|}
+\frac{\|F\|_{H^{-1}_k}\|w\|_{L^\infty}}{|k_1|\delta^{\frac{1}{2}}}
+\frac{\nu\|w'\|_{L^2}\|w\|_{L^\infty}}{|k_1|\delta^{\frac{1}{2}}}\\
&\quad+\delta^{-1}\|\varphi\|_{L^\infty(B(y_1,\delta)\cup B(y_2,\delta))}^2,
\end{aligned}
\end{equation}
where we  have used the fact
\begin{equation*}
\begin{aligned}
&\left|2\mathrm{Re}\int_{-1}^1\varphi\rho_1\rho'_1\overline{\varphi}'\mathrm{d}y\right|= \left|\int_{-1}^1\rho_1\rho_1'(|\varphi|^2)'\mathrm{d}y\right|\\
&= \left|\int_{-1}^1(\rho_1\rho_1')'|\varphi|^2\mathrm{d}y\right|\lesssim \delta^{-1}\|\varphi\|_{L^\infty(B(y_1,\delta)\cup B(y_2,\delta))}^2.
\end{aligned}
\end{equation*}

Taking the inner product between \eqref{OSH-1} and $\frac{\rho_1^2w}{1-y^2-\lambda}$ and taking the imaginary part, we can obtain(see Lemma 3.4 in \cite{DL} for the detail)
\begin{equation}\label{outside2}
\begin{aligned}
\int_{-1}^1\rho_1^2|w|^2\mathrm{d}y
\lesssim&\frac{\|F\|_{H^{-1}_k}\|(\partial_y,|k|)w\|_{L^2}}{|k_1|(y_2-y_1+\delta)\delta}
+\frac{\|F\|_{H_k^{-1}}\|w\|_{L^\infty}}{|k_1|(y_2-y_1+\delta)\delta^{\frac{3}{2}}}\\
&+\frac{\nu\|w'\|_{L^2}\|w\|_{L^\infty}}{|k_1|(y_2-y_1+\delta)\delta^{\frac{3}{2}}}+\int_{-1}^1\frac{|\rho_1\varphi|^2}{(1-y^2-\lambda)^2}\mathrm{d}y.
\end{aligned}
\end{equation}
Therefore, we deduce that
\begin{equation}\label{outside3}
\begin{aligned}
\|w\|_{L^2((-1,1)\setminus (y_1,y_2))}^2
\lesssim&\delta\|w\|_{L^\infty}^2+\frac{\|F\|_{H_k^{-1}}\|(\partial_y,|k|)w\|_{L^2}}{|k_1|(y_2-y_1+\delta)\delta}
\\&+\frac{\|F\|_{H_k^{-1}}\|w\|_{L^\infty}}{|k_1|(y_2-y_1+\delta)\delta^{\frac{3}{2}}}\\
&+\frac{\nu\|w'\|_{L^2}\|w\|_{L^\infty}}{|k_1|(y_2-y_1+\delta)\delta^{\frac{3}{2}}}+\int_{-1}^1\frac{|\rho_1\varphi|^2}{(1-y^2-\lambda)^2}\mathrm{d}y.
\end{aligned}
\end{equation}

It remains to bound the last term on the right hand side of \eqref{outside3}. Due to $|1-y^2-\lambda|=|(y-y_1)(y-y_2)|$, we get by Hardy's inequality that
\begin{equation*}
\begin{aligned}
&\int_{(-1,1)\setminus (y_1-\frac{\delta}{2},y_2+\frac{\delta}{2})}\frac{|\rho_1\varphi|^2}{(1-y^2-\lambda)^2}\mathrm{d}y\\
&\lesssim \frac{1}{(y_2-y_1+\delta)^2}\left(\int_{-1}^{y_1-\frac{\delta}{2}}\frac{|\rho_1\varphi|^2}{(y-y_1)^2}\mathrm{d}y+\int_{y_2+\frac{\delta}{2}}^1
\frac{|\rho_1\varphi|^2}{(y-y_2)^2}\mathrm{d}y\right)\\
&\lesssim \frac{1}{(y_2-y_1+\delta)^2}\left(\int_{-1}^1\rho_1^2|\varphi'|^2\mathrm{d}y+\delta^{-1}\|\varphi\|_{L^\infty(B(y_1,\delta)\cup B(y_2,\delta))}^2\right)\\
&\lesssim \frac{1}{(y_2-y_1+\delta)^2}\bigg(\frac{\|F\|_{H_k^{-1}}\|(\partial_y,|k|)w\|_{L^2}}{|k_1|}+
\frac{\|F\|_{H_k^{-1}}\|w\|_{L^\infty}}{|k_1|\delta^{\frac{1}{2}}}\\
&\qquad+\frac{\nu\|w'\|_{L^2}\|w\|_{L^\infty}}{|k_1|\delta^{\frac{1}{2}}}
+\delta^{-1}\|\varphi\|_{L^\infty(B(y_1,\delta)\cup B(y_2,\delta))}^2\bigg).
\end{aligned}
\end{equation*}

Then the proof is completed.
\end{proof}

Now we give the estimate on the interval $(y_1,y_2)$.

\begin{lemma}\label{H-1-wL2inside}
Assume that $-1\leq y_1\leq 0\leq y_2\leq 1$ with $1-y^2_i=\lambda\in (0,1), y_1=-y_2$. Then it holds that for any $\delta\in (0,\frac{y_2-y_1}{4}],\ \delta\in (0,2|k|^{-1}]$,
\begin{equation}\label{H-1-wL2inside-estimate}
\|w\|_{L^2(y_1,y_2)}^2\lesssim \mathscr{E}_4,
\end{equation}
where
\begin{equation}\nonumber
\begin{aligned}
\mathscr{E}_4=&\frac{\|F\|_{H_k^{-1}}\|(\partial_y,|k|)w\|_{L^2}}{|k_1|(y_2-y_1)\delta}
+\frac{\|F\|_{H_k^{-1}}\|w\|_{L^\infty}}{|k_1|(y_2-y_1)\delta^{\frac{3}{2}}}+\frac{\nu\|w'\|_{L^2}\|w\|_{L^\infty}}{|k_1|(y_2-y_1)\delta^{\frac{3}{2}}}
+\delta\|w\|_{L^\infty}^2\\&+\frac{\|F\|_{H^{-1}_k}^2}{|k_1|^2(y_2-y_1)^3\delta^3}
+\frac{\|\varphi\|_{L^\infty(B(y_1,\delta)\cup B(y_2,\delta))}^2}{(y_2-y_1)^2\delta}+
\epsilon^2\frac{\nu\|w\|_{L^2}^2}{|k_1|(y_2-y_1)^3\delta}\\
&+\frac{\nu^2\|w'\|_{L^2}^2}{|k_1|^2(y_2-y_1)^3\delta^3}+\frac{\nu^2\|w\|_{L^\infty}^2}{|k_1|^2(y_2-y_1)^3\delta^4}.
\end{aligned}
\end{equation}
\end{lemma}

\begin{proof}
We follow the proof of Lemma 3.5 in \cite{DL}. For any $\delta\in (0,\frac{y_2-y_1}{4}]$, $\delta\in(0,2|k|^{-1}]$, we introduce a smooth cut-off function $\rho_2(y)\in [0,1]$ as follows
\begin{equation*}
\rho_2(y)=\left \{
\begin{array}{lll}
1,&y\in [y_1+\delta,y_2-\delta],\\
0,&y\in (-1,1)\setminus (y_1+\frac{\delta}{2},y_2-\frac{\delta}{2}),\\
\mathrm{smooth},&\mathrm{others}.
\end{array}
\right.
\end{equation*}

By integration by parts, it is easy to see that
\begin{align*}
   & \int_{-1}^1\left|(\rho_2\varphi)'+\frac{2y}{1-y^2-\lambda}\rho_2\varphi\right|^2 \mathrm{d}y= \int_{-1}^1\left(|(\rho_2\varphi)'|^2-\frac{2\rho_2^2|\varphi|^2}{1-y^2-\lambda}\right) \mathrm{d}y,
\end{align*}
which gives
\begin{equation}\label{eq:rho22-w-00}
\begin{aligned}
\int_{-1}^1&\rho_2^2\left|(1-y^2-\lambda)^{\frac{1}{2}}w+\frac{2 \varphi}{(1-y^2-\lambda)^{\frac{1}{2}}}\right|^2\mathrm{d}y\\
&+2\int_{-1}^1\left|(\rho_2\varphi)'+\frac{2y}{1-y^2-\lambda}\rho_2\varphi\right|^2 \mathrm{d}y+2|k|^2\int_{-1}^1\rho^2_2|\varphi|^2\mathrm{d}y\\
=&2\int_{-1}^1\partial_y(\rho_2\rho_2')|\varphi|^2\mathrm{d}y+\mathrm{Re}\langle (1-y^2-\lambda)w+2\varphi,\rho^2_2w\rangle.
\end{aligned}
\end{equation}

Taking the inner product between \eqref{OSH-1} and $\rho_2^2w$ and taking the imaginary part, we obtain
\begin{equation}\nonumber
\begin{aligned}
&k_1\bigg(2\int_{-1}^1\partial_y(\rho_2\rho_2')|\varphi|^2\mathrm{d}y +\mathrm{Re}\langle (1-y^2-\lambda)w+2\varphi,\rho^2_2w\rangle\bigg)\\
&=\mathrm{Im}\left\{\langle F,\rho^2_2w\rangle+2\nu\int_{-1}^1\rho_2\rho_2'\partial_yw\overline{w}\mathrm{d}y\right\} +2k_1\int_{-1}^1\partial_y(\rho_2\rho_2')|\varphi|^2\mathrm{d}y,
\end{aligned}
\end{equation}
which along with \eqref{eq:rho22-w-00} gives
\begin{equation}\label{inside1}
\begin{aligned}
&\frac{1}{(y_2-y_1)^2}\bigg(\int_{-1}^1\rho_2^2\left|(1-y^2-\lambda)^{\frac{1}{2}}w+\frac{2 \varphi}{(1-y^2-\lambda)^{\frac{1}{2}}}\right|^2\mathrm{d}y\\
&\quad +2\int_{-1}^1\left|(\rho_2\varphi)'+\frac{2y}{1-y^2-\lambda}\rho_2\varphi\right|^2 \mathrm{d}y+2|k|^2\int_{-1}^1\rho^2_2|\varphi|^2\mathrm{d}y\bigg)\\
\lesssim&\frac{1}{(y_2-y_1)^2}\bigg(\frac{\|F\|_{H_k^{-1}}\|(\partial_y,|k|)w\|_{L^2}}{|k_1|}+
\frac{\|F\|_{H_k^{-1}}\|w\|_{L^\infty}}{|k_1|\delta^{\frac{1}{2}}}\\
& \quad +\frac{\nu\|w'\|_{L^2}\|w\|_{L^\infty}}{|k_1|\delta^{\frac{1}{2}}}+\delta^{-1}\|\varphi\|_{L^\infty(B(y_1,\delta)\cup B(y_2,\delta))}^2\bigg)
\lesssim \mathscr{E}_4.
\end{aligned}
\end{equation}

Taking the inner product between \eqref{OSH-1} and $\frac{\rho_2^2w}{1-y^2-\lambda}$ and taking the imaginary part, and using the fact that $y_2-y_1\geq 4\delta$,
we can obtain (see the proof of (3.20) in \cite{DL})
\begin{equation}\label{inside2}
\begin{aligned}
\int_{-1}^1\rho_2^2|w|^2\mathrm{d}y
\lesssim&\frac{\|F\|_{H_k^{-1}}\|(\partial_y,|k|)w\|_{L^2}}{|k_1|(y_2-y_1)\delta}+ \frac{\|F\|_{H_k^{-1}}\|w\|_{L^\infty}}{|k_1|(y_2-y_1)\delta^{\frac{3}{2}}}\\
&+ \frac{\nu\|w'\|_{L^2}\|w\|_{L^\infty}}{|k_1|(y_2-y_1)\delta^{\frac{3}{2}}}+\int_{-1}^1\frac{|\rho_2\varphi|^2}{(1-y^2-\lambda)^2}\mathrm{d}y.
\end{aligned}
\end{equation}
Thanks to
$$\|w\|_{L^2(y_1,y_2)}^2\lesssim \|\rho_2 w\|_{L^2(-1,1)}^2+\delta\|w\|_{L^\infty}^2,$$
 it remains to estimate the last term $\int_{-1}^1\frac{|\rho_2\varphi|^2}{(1-y^2-\lambda)^2}\mathrm{d}y$. Following the proof of Lemma 3.2 in \cite{DL}, we have
\begin{equation}\nonumber
\begin{aligned}
\int_{-1}^1\frac{|\rho_2\varphi|^2}{(1-y^2-\lambda)^2}\mathrm{d}y\lesssim& \frac{1}{(y_2-y_1)^2}\int_{-1}^1\left|(\rho_2\varphi)'+\frac{2y\rho_2\varphi}{1-y^2-\lambda}\right|^2\mathrm{d}y+\frac{|\rho_2(0)\varphi(0)|^2}{(y_2-y_1)^3}\\
\lesssim& \mathscr{E}_4+\frac{|\rho_2(0)\varphi(0)|^2}{(y_2-y_1)^3}.
\end{aligned}
\end{equation}
Now we bound $\frac{|\rho_2(0)\varphi(0)|^2}{(y_2-y_1)^3}$. To this end, note that
\begin{equation*}
\begin{aligned}
&\int_0^{y_2-\theta}\left(\frac{\rho_2\varphi}{1-y^2-\lambda}\right)'
\left(y_2-\theta-y\right)\mathrm{d}y-\int_{y_1+\theta}^{0}\left(\frac{\rho_2\varphi}{1-y^2-\lambda}\right)'\left(y-y_1-\theta\right)\mathrm{d}y\\
&=\int_{y_1+\theta}^{y_2-\theta}\frac{\rho_2\varphi}{1-y^2-\lambda}\mathrm{d}y- \frac{(\rho_2\varphi)(0)}{1-\lambda}(y_2-y_1-2\theta),
\end{aligned}
\end{equation*}
here $\frac{\delta}{2}\leq \theta\leq \frac{y_2-y_1}{4}$.
Following the proof of (3.24) and (3.25) in \cite{DL}, we can obtain
\begin{equation*}
\begin{aligned}
\frac{|\rho_2(0)\varphi(0)|^2}{(y_2-y_1)^3}
\lesssim&\frac{1}{y_2-y_1}\left|\int_{y_1+\theta}^{y_2-\theta}\frac{\rho_2\varphi}{1-y^2-\lambda}\mathrm{d}y\right|^2\\
&+\frac{1}{(y_2-y_1)^2}\int_{y_1+\theta}^{y_2-\theta}\left|
\left(\frac{\rho_2\varphi}{1-y^2-\lambda}\right)'\right|^2(1-y^2-\lambda)^2\mathrm{d}y\\
\lesssim&\frac{1}{y_2-y_1}\left|\int_{y_1+\theta}^{y_2-\theta} \frac{\rho_2\varphi}{1-y^2-\lambda}\mathrm{d}y\right|^2+\mathscr{E}_4.
\end{aligned}
\end{equation*}

To close the estimate, we consider the following two cases.\smallskip

{\bf Case 1.} $|k|^2(y_2-y_1)^2\geq 1$.

 Take $\theta=\frac{y_2-y_1}{4}$ and then
\begin{equation}\nonumber
\begin{aligned}
&\frac{1}{y_2-y_1}\bigg|\int_{y_1+\theta}^{y_2-\theta}\frac{\rho_2\varphi}{1-y^2-\lambda}
\mathrm{d}y\bigg|^2\\
&\lesssim \frac{\|\rho_2\varphi\|_{L^2(y_1+\theta,y_2-\theta)}^2 }{y_2-y_1}\|(1-y^2-\lambda)^{-1}\|_{L^2(y_1+\theta,y_2-\theta)}^2\\
&\lesssim \frac{ \|\rho_2\varphi\|_{L^2(y_1+\theta,y_2-\theta)}^2}{(y_2-y_1)^3\theta}
\lesssim\frac{|k|^2}{(y_2-y_1)^2}\|\rho_2\varphi\|_{L^2(y_1+\theta,y_2-\theta)}^2 \lesssim\mathscr{E}_4,
\end{aligned}
\end{equation}
where we have used \eqref{inside1} in the last step.\smallskip

{\bf Case 2. $|k|^2(y_2-y_1)^2\leq 1$}.

In this case, take $\theta=\frac{\delta}{2}$. Let us introduce
\begin{equation}\nonumber
\chi(y)=\tilde{\chi}\left(\frac{y}{y_2-y_1}\right) \quad \mathrm{with}\quad \tilde{\chi}(z)=\left \{
\begin{array}{lll}
1,&|z|\leq1,\\
0,&|z|\geq 2.
\end{array}
\right.
\end{equation}
It is clear that $\chi(y)=\rho_2(y)=1$ for $y\in [y_1+\delta,y_2-\delta]$ and
$$\int_{-1}^1(\varphi'\chi'+|k|^2\varphi\chi)\mathrm{d}y=-\int_{-1}^1w\chi\mathrm{d}y.$$
Therefore, we have
\begin{equation}\nonumber
\begin{aligned}
\left|\int_{y_1+\frac{\delta}{2}}^{y_2-\frac{\delta}{2}}\rho_2w\right|\leq &\left|\int_{y_1+\delta}^{y_2-\delta}w(y)\chi(y)\mathrm{d}y\right|+\left|\int_{(y_1+\frac{\delta}{2},y_2-\frac{\delta}{2})\setminus (y_1+\delta,y_2-\delta)}\rho_2w\mathrm{d}y\right|\\
\lesssim&\left|\int_{-1}^1(\varphi'\chi'+|k|^2\varphi\chi)\mathrm{d}y\right|+\delta\|w\|_{L^\infty}+\left|\int_{(-1,1)\setminus (y_1-\delta,y_2+\delta)}w\mathrm{d}y\right|.
\end{aligned}
\end{equation}
From \eqref{OSH-1},  we have
$$w=\frac{-\frac{1}{ik_1}\left(F+\nu(\partial_y^2-|k|^2)w+\epsilon(\nu|k_1|)^{\frac{1}{2}}w\right)-2\varphi}{1-y^2-\lambda},$$
and then
\begin{equation}\nonumber
\begin{aligned}
&\left|\int_{y_1+\frac{\delta}{2}}^{y_2-\frac{\delta}{2}}\frac{2\rho_2\varphi}{1-y^2-\lambda}\mathrm{d}y\right|-\left|\int_{y_1+\frac{\delta}{2}}^{y_2-\frac{\delta}{2}}
\frac{1}{ik_1}\frac{\rho_2\left(F+\nu(\partial_y^2-|k|^2)w
+\epsilon(\nu|k|)^{\frac{1}{2}}w\right)}{1-y^2-\lambda}\mathrm{d}y\right|\\
&\lesssim \left|\int_{-1}^1(\varphi'\chi'+|k|^2\varphi\chi)\mathrm{d}y\right|+\delta\|w\|_{L^\infty}+\int_{(-1,1)\setminus (y_1-\delta,y_2+\delta)}|w|\mathrm{d}y,
\end{aligned}
\end{equation}
which gives
\begin{equation}\nonumber
\begin{aligned}
&\frac{1}{y_2-y_1}\left|\int_{y_1+\frac{\delta}{2}}^{y_2-\frac{\delta}{2}}\frac{2\rho_2\varphi}{1-y^2-\lambda}\mathrm{d}y\right|^2\\
\lesssim&\frac{1}{y_2-y_1}\left(\left|\int_{-1}^1(\varphi'\chi'+|k|^2\varphi\chi)\mathrm{d}y\right|^2+\delta^2\|w\|_{L^\infty}^2+\left|\int_{(-1,1)\setminus (y_1-\delta,y_2+\delta)}|w|\mathrm{d}y\right|^2\right)\\
+&\frac{1}{|k_1|^2(y_2-y_1)}\left(\left|\int_{y_1+\frac{\delta}{2}}^{y_2-\frac{\delta}{2}}\frac{\rho_2F}{1-y^2-\lambda}\mathrm{d}y\right|^2+\nu^2
\left|\int_{y_1+\frac{\delta}{2}}^{y_2-\frac{\delta}{2}}
\frac{\rho_2\partial_y^2w}{1-y^2-\lambda}\mathrm{d}y\right|^2\right)\\
+&\frac{\nu^2|k|^4}{|k_1|^2(y_2-y_1)}\left|\int_{y_1+\frac{\delta}{2}}^{y_2-\frac{\delta}{2}}\frac{\rho_2w}{1-y^2-\lambda}
\mathrm{d}y\right|^2+\frac{\epsilon^2\nu}{|k_1|(y_2-y_1)}\left|\int_{y_1+\frac{\delta}{2}}^{y_2-\frac{\delta}{2}}
\frac{\rho_2w}{1-y^2-\lambda}\mathrm{d}y\right|^2\\
=:&I_1+\cdots+I_7.
\end{aligned}
\end{equation}

It remains to deal with $I_i,i=1,2,\cdots,7$. Due to $|k|^2(y_2-y_1)^2\leq 1$, by \eqref{outside1} and \eqref{inside1}, we have
\begin{equation}\nonumber
\begin{aligned}
I_1\lesssim&\frac{1}{y_2-y_1}\left(\|\varphi'\|_{L^2((-1,1)\setminus (y_1-\delta,y_2+\delta))}^2\|\chi'\|_{L^2}^2+|k|^4\|\varphi\|_{L^1(B(0,2(y_2-y_1))}^2\right)\\
\lesssim&\frac{\|\varphi'\|_{L^2((-1,1)\setminus (y_1-\delta,y_2+\delta))}^2}{(y_2-y_1)^2}+|k|^4\int_{\{y:\rho_1=1\}\cup\{y:\rho_2=1\}}|\varphi|^2\mathrm{d}y\\
&
+\frac{|k|^4\delta^2\|\varphi\|_{L^\infty(B(y_1,\delta)\cup B(y_2,\delta))}^2}{y_2-y_1}\\
\lesssim&\frac{\|\varphi'\|_{L^2((-1,1)\setminus (y_1-\delta,y_2+\delta))}^2}{(y_2-y_1)^2}+\frac{|k|^2}{(y_2-y_1)^2} \int_{\{y:\rho_1=1\}\cup\{y:\rho_2=1\}}|\varphi|^2\mathrm{d}y\\
&+\frac{\|\varphi\|_{L^\infty(B(y_1,\delta)\cup B(y_2,\delta))}^2}{(y_2-y_1)^2\delta}\lesssim\mathscr{E}_4.
\end{aligned}
\end{equation}
For $I_2$, we have
\begin{equation}\nonumber
\begin{aligned}
I_2\lesssim &\delta\|w\|_{L^\infty}^2\lesssim \mathscr{E}_4.
\end{aligned}
\end{equation}
Due to $y_2-y_1\geq 4\delta$, we get by \eqref{outside1} that
\begin{equation}\nonumber
\begin{aligned}
I_3\lesssim&\frac{1}{y_2-y_1}\|w\|_{L^1((-1,1)\setminus (y_1-\delta,y_2+\delta))}^2\\
\lesssim&\frac{ \|\sqrt{\lambda-1+y^2}w\|_{L^2((-1,1)\setminus (y_1-\delta,y_2+\delta))}^2}{y_2-y_1}\left\|\frac{1}{1-y^2-\lambda}\right\|_{L^1((-1,1)\setminus (y_1-\delta,y_2+\delta))}\\
\lesssim&\frac{1}{y_2-y_1}\frac{1+\ln\left(1+\frac{y_2-y_1}{\delta}\right)}{(y_2-y_1+\delta)}
\bigg(\frac{\|F\|_{H^{-1}}\|w\|_{H^1}}{|k_1|}+\frac{\|F\|_{H^{-1}}\|w\|_{L^\infty}}{|k_1|\delta^{\frac{1}{2}}}\\
&+\frac{\nu\|w'\|_{L^2}
\|w\|_{L^\infty}}{|k_1|\delta^{\frac{1}{2}}}+\delta^{-1}\|\varphi\|_{L^\infty(B(y_1,\delta)\cup B(y_2,\delta))}^2\bigg) \lesssim\mathscr{E}_4,
\end{aligned}
\end{equation}
and $I_4$ is bounded by
\begin{equation}\nonumber
\begin{aligned}
I_4\lesssim&\frac{1}{|k_1|^2(y_2-y_1)}\|F\|_{H^{-1}_k}^2 \left\|\frac{\rho_2}{1-y^2-\lambda}\right\|_{H^1_k}^2\lesssim \frac{\|F\|_{H^{-1}_k}^2}{|k_1|^2(y_2-y_1)^3\delta^3}\lesssim\mathscr{E}_4.
\end{aligned}
\end{equation}
For $I_5$, we have
\begin{equation}\nonumber
\begin{aligned}
I_5\lesssim&\frac{1}{|k_1|^2(y_2-y_1)}\left|\int_{y_1+\frac{\delta}{2}}^{y_2-\frac{\delta}{2}}\partial_yw\left(\frac{\rho_2'}{1-y^2-\lambda}
-\frac{\rho_2(-2y)}{(1-y^2-\lambda)^2}\right)\mathrm{d}y\right|
\\ \lesssim&\frac{\nu^2\|w'\|_{L^2}^2}{|k_1|^2(y_2-y_1)^3\delta^3}\lesssim\mathscr{E}_4,
\end{aligned}
\end{equation}
and for $I_6$, we obtain
\begin{equation}\nonumber
\begin{aligned}
I_6\lesssim&\frac{\nu^2|k|^4(y_2-y_1-\delta)}{|k_1|^2(y_2-y_1)}\int_{y_1+\frac{\delta}{2}}^{y_2-\frac{\delta}{2}}\frac{|\rho_2 w|^2}{(1-y^2-\lambda)^2}\mathrm{d}y\\
\lesssim&\frac{\nu^2|k|^4(y_2-y_1-\delta)}{|k_1|^2(y_2-y_1)}\frac{1}{(y_2-y_1+\delta)\delta}
\bigg(\int_{y_1+\frac{\delta}{2}}^{0}\frac{\left|\int_{y_1}^y w'\mathrm{d}z\right|^2+|w(y_1)|^2}{(y-y_1)^2}\mathrm{d}y\\
&+\int_{0}^{y_2-\frac{\delta}{2}}\frac{\left|\int_{y_2}^y w'\mathrm{d}z\right|^2+|w(y_2)|^2}{(y-y_2)^2}\mathrm{d}y\bigg)\\
\lesssim&\frac{\nu^2|k|^4(y_2-y_1-\delta)}{|k_1|^2(y_2-y_1)}\left(\frac{\|w'\|_{L^2}^2}{\delta^2}+\frac{\|w\|_{L^\infty}^2}{\delta^3}\right)\\
\lesssim&\frac{\nu^2\|w'\|_{L^2}^2}{|k_1|^2(y_2-y_1)^3\delta^3}+\frac{\nu^2\|w\|_{L^\infty}^2}{|k_1|^2(y_2-y_1)^3\delta^4}\lesssim \mathscr{E}_4.
\end{aligned}
\end{equation}
Finally, $I_7$ can be estimated as
\begin{equation}\nonumber
\begin{aligned}
I_7\lesssim&\epsilon^2\frac{\nu\|w\|_{L^2}^2}{|k_1|(y_2-y_1)^3\delta}\lesssim \mathscr{E}_4.
\end{aligned}
\end{equation}

The proof is completed.
\end{proof}

\begin{lemma}\label{streamL-infty-H-1}
For any $\delta\in (0,2|k|^{-1}], \ \lambda\in (0,1)$, it holds that
\begin{equation}\nonumber
\begin{aligned}
\frac{\|\varphi\|_{L^\infty(B(y_1,\delta)\cup B(y_2,\delta))}^2}{(y_2-y_1+\delta)^2\delta}\lesssim\mathscr{F}_3,
\end{aligned}
\end{equation}
where
\begin{equation}\nonumber
\begin{aligned}
\mathscr{F}_3=&\frac{\|F\|_{H^{-1}_k}^2}{|k_1|^2(y_2-y_1+\delta)^2\delta^4}+\frac{\nu^2\|w'\|_{L^2}^2}{|k_1|^2(y_2-y_1+\delta)^2\delta^4}
+\frac{\nu^2|k|^4\|w\|_{L^\infty}^2}{|k_1|^2(y_2-y_1+\delta)^2\delta}
\\
&+\frac{\epsilon\nu\|w\|_{L^2}^2}{|k_1|^3(y_2-y_1+\delta)^2\delta^2}+\delta\|w\|_{L^\infty}^2.
\end{aligned}
\end{equation}
\end{lemma}
\begin{proof}
We only give the estimate on $B(y_1,\delta)$ and the estimate on $B(y_2,\delta)$ is similar. For any $a \in B(y_1,\delta)$ and $b\in \left[\frac{\delta}{2},\delta\right]$, let us introduce a smooth cut-off function $\rho_3(y)$ with $0\leq \rho_3\leq 1, \ \rho_3(y+a)=\rho_3(a-y)$ and
\begin{equation}\nonumber
\rho_3(y)=\left \{
\begin{array}{lll}
1,&y\in \left[a-\frac{b}{2},a+\frac{b}{2}\right],\\
0,&y\not\in [a-b,a+b],\\
\mathrm{smooth},&\mathrm{others}.
\end{array}
\right.
\end{equation}
It is obvious that
$$\|\rho_3\|_{L^\infty}\leq 1, \quad \|\rho_3'\|_{L^1}\lesssim 1,\quad  \|\rho_3'\|_{L^2}\lesssim b^{-\frac{1}{2}},$$
and
$$\int_{a-b}^a\int_{a-b}^z\rho_3(z_1)\mathrm{d}z_1\mathrm{d}z=\int_{a}^{a+b}\int_{z}^{a+b}\rho_3(z_1)\mathrm{d}z_1\mathrm{d}z.$$
In addition, it is easy to verify that
\begin{equation}\nonumber
\begin{aligned}
\int_{a-b}^{a+b}\varphi(y)\rho_3(y)\mathrm{d}y=&\varphi(a)\int_{a-b}^{a+b}\rho_3(y)\mathrm{d}y\\
&+\int_a^{a+b}\varphi''(z)
\left(\int_z^{a+b}\int_{z_1}^{a+b}\rho_3(s)\mathrm{d}s\mathrm{d}z_1\right)\mathrm{d}z\\
&+\int_{a-b}^{a}\varphi''(z)\left(\int_{a-b}^{z}\int_{a-b}^{z_1}\rho_3(s)\mathrm{d}s\mathrm{d}z_1\right)\mathrm{d}z.
\end{aligned}
\end{equation}
Notice that
$$ b|\varphi(a)|=|\varphi(a)|\left|\int_{a-\frac{b}{2}}^{a+\frac{b}{2}}\rho_3(y)\mathrm{d}y\right|\leq \left|\varphi(a)\int_{a-b}^{a+b}\rho_3(y)\mathrm{d}y\right|,$$
therefore, we have
\begin{equation}\nonumber
\begin{aligned}
|\varphi(a)|\leq &\frac{1}{b}\left|
\int_{a-b}^{a+b}\varphi(y)\rho_3(y)\mathrm{d}y\right|
+b^2\|\varphi''\|_{L^\infty}\\
\lesssim&\frac{1}{b|k_1|}\bigg(\left|\int_{-1}^1F\rho_3\mathrm{d}y\right|+\left|\int_{a-b}^{a+b}\nu\partial_y^2w \rho_3\mathrm{d}y\right|+\left|\int_{a-b}^{a+b}\nu|k|^2w \rho_3\mathrm{d}y\right|\\
&+\int_{a-b}^{a+b}|k_1||1-y^2-\lambda|\mathrm{d}y\|w\|_{L^\infty}+\epsilon(\nu |k_1|^{-1})^{\frac{1}{2}}b^{\frac{1}{2}}\|w\|_{L^2}\bigg)+b^2\|w\|_{L^\infty}\\
\lesssim&\frac{\|F\|_{H_k^{-1}}}{|k_1|\delta^{\frac{3}{2}}}+\frac{\nu\|w'\|_{L^2}}{|k_1|\delta^{\frac{3}{2}}}
+\frac{\nu|k|^2\|w\|_{L^\infty}}{|k_1|}+(y_2-y_1+\delta)\delta\|w\|_{L^\infty}\\
&+\epsilon\frac{\nu^{\frac{1}{2}}\|w\|_{L^2}}{|k_1|^{\frac{3}{2}}\delta^{\frac{1}{2}}}+\delta^2\|w\|_{L^\infty},
\end{aligned}
\end{equation}
which gives our conclusion.
\end{proof}

Now we are in a position to prove Proposition \ref{resolvent-H-1}.

\begin{proof}[Proof of  Proposition \ref{resolvent-H-1}]
By integration by parts, it is easy to see that
\begin{equation}\label{2.0-basic}
\begin{aligned}
&\nu\|w'\|_{L^2}^2+\nu|k|^2\|w\|_{L^2}^2+ik_1\int_{-1}^1(1-y^2-\lambda)|w|^2\mathrm{d}y
\\&\qquad+2ik_1\langle \varphi,w\rangle
-\epsilon\nu^{\frac{1}{2}}|k_1|^{\frac{1}{2}}\|w\|_{L^2}^2=\langle F,w\rangle.
\end{aligned}
\end{equation}
Taking the real part of \eqref{2.0-basic}  to yield that
\begin{equation}\label{2.1}
\begin{aligned}
 \nu\|(\partial_y,|k|)w\|_{L^2}^2\leq& \|F\|_{H^{-1}_k}\|(\partial_y,|k|)w\|_{L^2}+\epsilon\nu^{\frac{1}{2}}|k_1|^{\frac{1}{2}}\|w\|_{L^2}^2,
\end{aligned}
\end{equation}
which gives by Young's inequality that
\begin{equation}\label{2.2}
\begin{aligned}\|(\partial_y,|k|)w\|_{L^2}\leq \nu^{-1}\|F\|_{H_k^{-1}}+\sqrt{2\epsilon}\nu^{-\frac{1}{4}}|k_1|^{\frac{1}{4}}\|w\|_{L^2}.
\end{aligned}
\end{equation}

\noindent\textbf{Case 1. $\lambda\geq 1$.}

In this case, for $\delta=|\nu/k_1|^{\frac{1}{4}}\ll 1$, we get by taking the imaginary part of \eqref{2.0-basic} that
\begin{equation}\nonumber
\begin{aligned}
\delta^2 \|w\|_{L^2\big((-1,1)\setminus (-\delta,\delta)\big)}^2\lesssim & \int_{-1}^1(\lambda-1+y^2)|w|^2\mathrm{d}y+2(\|\varphi'\|_{L^2}^2+|k|^2\|\varphi\|_{L^2}^2)\\
\leq& |k_1|^{-1}\|F\|_{H^{-1}_k}\|(\partial_y,|k|)w\|_{L^2},
\end{aligned}
\end{equation}
where we have used the fact that $\lambda-1+y^2\geq \delta^2$ for $y\in (-1,1)\setminus (-\delta,\delta)$. Then we have
\begin{equation}\nonumber
\begin{aligned}
\|w\|_{L^2}^2\lesssim &\delta\|w\|_{L^\infty}^2+\|w\|_{L^2((-1,1)\setminus (-\delta,\delta))}^2\\
\lesssim &\delta\|w\|_{L^2}\|w'\|_{L^2}+\delta^{-2}|k_1|^{-1}\|F\|_{H_k^{-1}}\|(\partial_y,|k|)w\|_{L^2}\\
\lesssim&\delta\|w\|_{L^2}(\nu^{-1}\|F\|_{H_k^{-1}}+\sqrt{2\epsilon}\nu^{-\frac{1}{4}}|k_1|^{\frac{1}{4}}\|w\|_{L^2})\\
&+\delta^{-2}|k_1|^{-1}\|F\|_{H_k^{-1}}(\nu^{-1}\|F\|_{H_k^{-1}}+\sqrt{2\epsilon}\nu^{-\frac{1}{4}}|k_1|^{\frac{1}{4}}\|w\|_{L^2}).
\end{aligned}
\end{equation}
Using Young's inequality and  $0<\epsilon\ll 1$, we get
\begin{equation}\nonumber
\begin{aligned}
\|w\|_{L^2}^2 \lesssim \nu^{-\frac{3}{2}}|k_1|^{-\frac{1}{2}}\|F\|_{H_k^{-1}}^2,
\end{aligned}
\end{equation}
which along with \eqref{2.2} gives our result. \smallskip

\noindent\textbf{Case 2. $\lambda\leq 0$.}

Taking the imaginary part of \eqref{2.0-basic}, we obtain
\begin{equation}\label{eq:lam0-im-1}
\begin{aligned}
\int_{-1}^1(1-y^2-\lambda)|w|^2\mathrm{d}y+2\langle \varphi,w\rangle \leq |k_1|^{-1}\|F\|_{H_k^{-1}}\|(\partial_y,|k|)w\|_{L^2}.
\end{aligned}
\end{equation}
It follows from Lemma \ref{hardy-type} that
\begin{equation}\label{wL2}
\begin{aligned}
\|w\|_{L^2}\lesssim \|(1-y^2)^{\f12}w\|_{L^2}^{\f23}\|w/(1-y^2)\|_{L^2}^{\f13}\lesssim \|(1-y^2)^{\f12}w\|_{L^2}^{\f23}\|w'\|_{L^2}^{\f13}.
\end{aligned}
\end{equation}
Taking the real part of \eqref{2.0-basic}, we get
\begin{equation}\label{2.01}
\begin{aligned}
\nu\|w'\|_{L^2}^2+\nu|k|^2\|w\|_{L^2}^2\leq & \|F\|_{H^{-1}_k}\|(\partial_y,|k|)w\|_{L^2}+\epsilon\nu^{\frac{1}{2}}|k_1|^{\frac{1}{2}}\|w\|_{L^2}^2\\
\leq &\|F\|_{H^{-1}_k}\|(\partial_y,|k|)w\|_{L^2}+\epsilon\nu^{\frac{1}{3}}|k_1|^{\frac{2}{3}}\|w\|_{L^2}^2,
\end{aligned}
\end{equation}
which yields that
\begin{equation}\label{2.02}
\begin{aligned}
\|(\partial_y,|k|)w\|_{L^2} \leq \nu^{-1} \|F\|_{H^{-1}_k}+\sqrt{2\epsilon}\nu^{-\frac{1}{3}}|k_1|^{\frac{1}{3}}\|w\|_{L^2}.
\end{aligned}
\end{equation}

On the other hand, by Lemma \ref{positive} and \eqref{eq:lam0-im-1}, we have
\begin{equation}\nonumber
\begin{aligned}
\big(\|u\|_{L^2}^2+|k|\|\varphi\|_{L^2}^2\big)/9 \leq&\int_{-1}^1(1-y^2-\lambda)|w|^2\mathrm{d}y+2\langle \varphi,w\rangle\\
 \leq&  |k_1|^{-1}\|F\|_{H^{-1}_k}\|(\partial_y,|k|)w\|_{L^2},
\end{aligned}
\end{equation}
which along with \eqref{eq:lam0-im-1}  gives
\begin{equation}\label{L2u}
\begin{aligned}
&\|u\|_{L^2}^2 +|k|^2\|\varphi\|_{L^2}^2+\|(1-y^2)^{\f12}w\|_{L^2}^2\\
&\lesssim |k_1|^{-1}\|F\|_{H^{-1}_k}\|(\partial_y,|k|)w\|_{L^2}.
\end{aligned}
\end{equation}
Collecting \eqref{wL2} and \eqref{L2u}, we obtain
\begin{align*}
   \|w\|_{L^2} &\lesssim |k_1|^{-\f13}\|F\|_{H^{-1}_k}^{\f13}\|(\partial_y,|k|)w\|_{L^2}^{\f23},
 \end{align*}
from  which and \eqref{2.02}, we infer that
\begin{equation}\nonumber
\begin{aligned}
\|w\|_{L^2}\lesssim \nu^{-\frac{2}{3}}|k_1|^{-\frac{1}{3}}\|F\|_{H_k^{-1}}+ \nu^{-\f29}|k_1|^{-\f19}\|F\|_{H^{-1}_k}^{\f13}\|w\|_{L^2}^{\f23}.
\end{aligned}
\end{equation}
Therefore, we obtain
$$\|w\|_{L^2}\lesssim \nu^{-\frac{2}{3}}|k_1|^{-\frac{1}{3}}\|F\|_{H_k^{-1}}.$$
This along with \eqref{2.02} proves our result.\smallskip

\noindent\textbf{Case 3. $\lambda\in (0,1).$}

Let $1-y_i^2=\lambda \in (0,1), i=1,2$ with $-1\leq y_1\leq 0\leq y_2\leq 1$ and $y_1=-y_2$.\smallskip

\textbf{Case 3.1} $\nu|k|^3|k_1|^{-1}\geq |1-\lambda|^{\f12}+|\nu/k_1|^{\f14}$.

In this case, it holds that $\nu|k|^3|k_1|^{-1}\geq2|\nu/k_1|^{\f14},$ and we have $\nu^{\frac{1}{2}} |k_1|^{\frac{1}{2}}\lesssim \nu |k|^2$ and by \eqref{2.2},
\begin{equation}\nonumber
\begin{aligned}
\nu|k|^2\|w\|_{L^2}^2\leq & \|F\|_{H^{-1}_k}\|(\partial_y,|k|)w\|_{L^2}+\epsilon\nu^{\frac{1}{2}}|k_1|^{\frac{1}{2}}\|w\|_{L^2}^2\\
\leq&\|F\|_{H_k^{-1}}(\nu^{-1}\|F\|_{H_k^{-1}}
+\sqrt{2\epsilon}\nu^{-\frac{1}{4}}|k_1|^{\frac{1}{4}}\|w\|_{L^2})+\epsilon\nu|k|^2\|w\|_{L^2}^2\\
\leq&2\epsilon\nu|k|^2\|w\|_{L^2}^2+ C\nu^{-1}\|F\|_{H^{-1}_k}^2.
\end{aligned}
\end{equation}
Due to $0<\epsilon\ll 1$ and $\nu|k|^3 |k_1|^{-1} \gtrsim |\nu/k_1|^{\f14}$, we obtain
\begin{equation}\nonumber
\begin{aligned}
\|w\|_{L^2}^2 \lesssim \nu^{-2}|k|^{-2}\|F\|_{H^{-1}_k}^2\lesssim \nu^{-\frac{3}{2}}|k_1|^{-\frac{1}{2}}\|F\|_{H^{-1}_k}^2.
\end{aligned}
\end{equation}

\textbf{Case 3.2.}  $y_2-y_1\leq 4|\nu/k_1|^{\frac{1}{4}}$.

Indeed, we get by \eqref{2.1} that
\begin{equation}\nonumber
\begin{aligned}
\nu\|(\partial_y,|k|)w\|_{L^2}^2\leq &\|F\|_{H_k^{-1}}\|(\partial_y,|k|)w\|_{L^2} +\epsilon\nu^{\frac{1}{2}}|k_1|^{\frac{1}{2}}\|w\|_{L^2}^2\\
\lesssim&\|F\|_{H^{-1}_k}\|(\partial_y,|k|)w\|_{L^2}+\epsilon\nu^{\frac{1}{2}}|k_1|^{\frac{1}{2}}\|(\partial_y,|k|)w\|_{L^2}\|u\|_{L^2},
\end{aligned}
\end{equation}
which gives
\begin{equation}\label{dw-H-1}
\begin{aligned}
\|(\partial_y,|k|)w\|_{L^2}
\lesssim\nu^{-1}\|F\|_{H_k^{-1}}+\sqrt{\epsilon}\nu^{-\frac{1}{2}}|k_1|^{\frac{1}{2}}\|u\|_{L^2}.
\end{aligned}
\end{equation}

On the other hand, we have
\begin{equation}\nonumber
\begin{aligned}
\|yw\|_{L^2}^2+\|u\|_{L^2}^2\lesssim &|k_1|^{-1}\|F\|_{H^{-1}_k}\|(\partial_y,|k|)w\|_{L^2}+(y_2-y_1)^2\|w\|_{L^2}^2\\
\lesssim&|k_1|^{-1}\|F\|_{H^{-1}_k}\|(\partial_y,|k|)w\|_{L^2}
+\nu^{\frac{1}{2}}|k_1|^{-\frac{1}{2}}\|(\partial_y,|k|)w\|_{L^2}\|u\|_{L^2}\\
\lesssim& |k_1|^{-1}\|F\|_{H^{-1}_k}\big(\nu^{-1}\|F\|_{H^{-1}_k}+\sqrt{\epsilon}\nu^{-\frac{1}{2}}|k_1|^{\frac{1}{2}}\|u\|_{L^2}\big)\\
&+\nu^{\frac{1}{2}}|k_1|^{-\frac{1}{2}}\|u\|_{L^2}\big(\nu^{-1}\|F\|_{H^{-1}_k}+\sqrt{\epsilon}\nu^{-\frac{1}{2}}|k_1|^{\frac{1}{2}}\|u\|_{L^2}\big),
\end{aligned}
\end{equation}
where we used \eqref{dw-H-1}. Then we get by Young's inequality and $0<\epsilon\ll 1$ that
\begin{equation}\nonumber
\begin{aligned}
\|u\|_{L^2}^2\lesssim \nu^{-1}|k_1|^{-1}\|F\|_{H^{-1}_k}^2,
\end{aligned}
\end{equation}
which along with \eqref{dw-H-1} gives
\begin{equation}\nonumber
\begin{aligned}
\|(\partial_y,|k|)w\|_{L^2}\lesssim \nu^{-1}\|F\|_{H^{-1}_k}
\end{aligned}
\end{equation}
and
\begin{equation}\nonumber
\begin{aligned}
\|w\|_{L^2}^2\lesssim \|(\partial_y,|k|)w\|_{L^2}\|u\|_{L^2}\lesssim \nu^{-\frac{3}{2}}|k_1|^{-\frac{1}{2}}\|F\|_{H^{-1}_k}^2.
\end{aligned}
\end{equation}

\textbf{Case 3.3.} $y_2-y_1\geq 4 |\nu /k_1|^{\frac{1}{4}}$ and $\nu|k|^3|k_1|^{-1}\leq |1-\lambda|^{\f12}+|\nu/k_1|^{\f14}$.

 Let us define $\delta=|\nu/k_1|^{\f13} (y_2-y_1)^{-\f13}$, then $0<\delta\leq |\nu /k_1|^{\frac{1}{4}}\leq \frac{y_2-y_1}{4}$ and $0<\delta\leq 2|k|^{-1}$. By Lemma \ref{H-1-wL2outside} and Lemma \ref{H-1-wL2inside}, we have
$$\|w\|_{L^2}^2\lesssim \mathscr{E}_3+\mathscr{E}_4\lesssim \mathscr{E}_4.$$

It remains to estimate each term of $\mathscr{E}_4$. Thanks to the definition of $\delta$ and \eqref{2.2}, we obtain
\begin{equation}\nonumber
\begin{aligned}
&\frac{\|F\|_{H^{-1}_k}^2}{|k_1|^2(y_2-y_1)^3\delta^3}=\frac{\|F\|_{H^{-1}_k}^2}{ \nu|k_1|(y_2-y_1)^2},\quad\frac{\epsilon^2\nu\|w\|_{L^2}^2}{|k_1|(y_2-y_1)^3\delta}
\lesssim\epsilon^2\|w\|_{L^2}^2,\\
&\frac{\nu^2\|w'\|_{L^2}^2}{|k_1|^2(y_2-y_1)^3\delta^3} \lesssim\frac{\|F\|_{H^{-1}_k}^2}{\nu|k_1|(y_2-y_1)^2}+\epsilon\|w\|_{L^2}^2,\\
&\frac{\|F\|_{H^{-1}_k}\|(\partial_y,|k|)w\|_{L^2}}{|k_1|(y_2-y_1)\delta}
\lesssim\frac{\|F\|_{H_k^{-1}}(\nu^{-1}\|F\|_{H^{-1}_k}+\sqrt{\epsilon}(\nu |k_1|^{-1})^{-\frac{1}{4}}\|w\|_{L^2})}{|k_1|(y_2-y_1)\delta} \\
&\qquad\qquad\qquad\qquad\qquad \lesssim \frac{\|F\|_{H^{-1}_k}^2}{\nu|k_1|(y_2-y_1)^2}+\epsilon\|w\|_{L^2}^2,\\
&\delta\|w\|_{L^\infty}^2\lesssim\delta\|w\|_{L^2}\|w'\|_{L^2}
\lesssim\delta\nu^{-1}\|F\|_{H^{-1}_k}\|w\|_{L^2}+\sqrt{\epsilon}\|w\|_{L^2}^2,
\end{aligned}
\end{equation}
and
\begin{equation}\nonumber
\begin{aligned}
&\frac{\nu^2\|w\|_{L^\infty}^2}{|k_1|^2(y_2-y_1)^3\delta^4}= \frac{\delta^6(y_2-y_1)^2\|w\|_{L^\infty}^2}{(y_2-y_1)^3\delta^4}\lesssim \delta\|w\|_{L^\infty}^2,\\
&\frac{\|F\|_{H_k^{-1}}\|w\|_{L^\infty}}{|k_1|(y_2-y_1)\delta^{\frac{3}{2}}} \lesssim\frac{\|F\|_{H_k^{-1}}^2}{|k_1|^2(y_2-y_1)^2\delta^4}+\delta\|w\|_{L^\infty}^2 \lesssim\frac{\|F\|_{H_k^{-1}}^2}{\nu|k_1|(y_2-y_1)\delta}+\delta\|w\|_{L^\infty}^2,\\
&\frac{\nu\|w'\|_{L^2}\|w\|_{L^\infty}}{|k_1|(y_2-y_1)\delta^{\frac{3}{2}}} \lesssim\frac{\nu^2\|w'\|_{L^2}^2}{|k_1|^2(y_2-y_1)^2\delta^4}+\delta\|w\|_{L^\infty}^2\\
&\qquad \qquad  \qquad \quad\lesssim\frac{\|F\|_{H^{-1}_k}^2}{\nu|k_1|(y_2-y_1)\delta}+\epsilon\|w\|_{L^2}^2+\delta\|w\|_{L^\infty}^2.
\end{aligned}
\end{equation}

Note that for $\lambda\in (0,1)$ and $\delta\lesssim y_2-y_1,\ \nu|k|^3|k_1|^{-1}\leq |1-\lambda|^{\f12}+|\nu/k_1|^{\f14}=(y_2-y_1)/2+|\nu/k_1|^{\f14}$ and
$$ y_2-y_1\sim y_2-y_1+\delta,\quad  \frac{\nu^2|k|^4}{|k_1|^2(y_2-y_1)^2\delta }\lesssim \delta,$$
then we get by Lemma \ref{streamL-infty-H-1} that
\begin{equation}\nonumber
\begin{aligned}
\frac{\|\varphi\|_{L^\infty(B(y_1,\delta)\cup B(y_2,\delta))}^2}{(y_2-y_1)^2\delta}\lesssim\mathscr{F}_3\lesssim&
\frac{\|F\|_{H_k^{-1}}^2}{|k_1|^2(y_2-y_1)^2\delta^4}
+\frac{\nu^2\|w'\|_{L^2}^2}{|k_1|^2(y_2-y_1)^2\delta^4}+\delta\|w\|_{L^\infty}^2\\
&+\frac{\nu^2|k|^4\|w\|_{L^\infty}^2}{|k_1|^2(y_2-y_1)^2\delta}+\frac{\epsilon\nu\|w\|_{L^2}^2}{|k_1|^3(y_2-y_1)^2\delta^2}\\
\lesssim&\frac{\|F\|_{H^{-1}_k}^2}{\nu|k_1|(y_2-y_1)\delta}+\delta\|w\|_{L^\infty}^2+(\epsilon+\sqrt{\epsilon})\|w\|_{L^2}^2.
\end{aligned}
\end{equation}
Putting the above estimates together, using Young's inequality, $\delta\leq \nu^{\frac{1}{4}}|k_1|^{-\frac{1}{4}}$ and the smallness of $\epsilon$, we obtain
\begin{equation}\label{est:w-FH-1-00}
\begin{aligned}
\|w\|_{L^2}^2\lesssim&(\nu|k_1|(y_2-y_1)\delta)^{-1}\|F\|_{H^{-1}_k}^2+\delta^2\nu^{-2}\|F\|_{H^{-1}_k}^2 \\ \lesssim&\nu^{-\frac{3}{2}}|k_1|^{-\frac{1}{2}}\|F\|_{H^{-1}_k}^2.
\end{aligned}
\end{equation}
Meanwhile, we have
\begin{equation}\nonumber
\begin{aligned}
\|w'\|_{L^2}\leq \nu^{-1}\|F\|_{H^{-1}_k}+\sqrt{2\epsilon}\nu^{-\frac{1}{4}}|k_1|^{\frac{1}{4}}\|w\|_{L^2}\lesssim \nu^{-1}\|F\|_{H^{-1}_k}.
\end{aligned}
\end{equation}

The proof is completed.
\end{proof}

\begin{corollary}\label{reslvent-H-1-sgap}
Under the same assumptions of Proposition \ref{resolvent-H-1}, it holds that
\begin{equation}\label{resolventH-1-sgap}
\begin{aligned}
\nu^{\frac{2}{3}}|k_1|^{\frac{1}{3}}\big(|\lambda-1|^{\f12}+|\nu/k_1|^{\f14}\big)^{\frac{1}{3}}\|w\|_{L^2}\le C\|F\|_{H^{-1}_k}.
\end{aligned}
\end{equation}
\end{corollary}
\begin{proof}
The estimate \eqref{resolventH-1-sgap} can be obtained by following similar arguments as in the proof of Corollary \ref{resol-lambda}. Here we omit the details (mainly use \eqref{est:w-FH-1-00}).
\end{proof}

\subsection{Weak-type resolvent estimate for $F\in H^{-1}_k$}
Here we denote
\beno
\delta=|\nu/k_1|^{\frac{1}{4}}\ll 1,\quad \delta_1 =|\nu /k_1|^{\frac{1}{3}}(|\lambda-1|^{\frac{1}{2}}+\delta)^{\frac{2}{3}}.
\eeno

\begin{proposition}\label{weaktype}
 Let $w\in H^1(I)$ be a solution of \eqref{OSH-1} with $F\in H^{-1}_k(I)$.
Assume that $\nu|k|^3|k_1|^{-1}\leq 2(|1-\lambda|^{\f12}+|\nu/k_1|^{\f14})$ and $f\in H^1(I),j\in \{\pm1\}$ and $f(-j)=0$.  Then it holds that
\begin{align*}
|\langle w,f\rangle|
\lesssim&|k_1|^{-1}{|f(j)|}\|F\|_{H_k^{-1}}(|\lambda|+\delta^{\frac{4}{3}})^{-\frac{3}{4}} \delta^{-1}+|k_1|^{-1}\|F\|_{H^{-1}_k}\bigg(\|\mathbf{Ray}_{\delta_1}^{-1}f\|_{H^1}
\\&+\delta^{-\frac{4}{3}}(|\lambda-1|^{\frac{1}{2}}+\delta)^{\frac{1}{3}}
\|\mathbf{Ray}_{\delta_1}^{-1}f\|_{L^2}\bigg).
\end{align*}

\end{proposition}
\begin{proof}
To derive the weak-type estimate, we consider the following problem
\begin{equation}\label{OSweak}
\left \{
\begin{array}{lll}
-\nu(\partial_y^2-|k|^2)w+ik_1[(1-y^2-\lambda-i\delta_1)w+2\varphi]\\
\qquad \qquad -\epsilon (\nu |k_1|)^{\frac{1}{2}}w =F+ k_1\delta_1w,\\
w(\pm 1)=\varphi(\pm 1)=\varphi''(\pm 1)=0, \ \ (\partial_y^2-|k|^2)\varphi=w,
\end{array}
\right.
\end{equation}
where $\lambda\in\mathbb{R}$. For $\phi\in H^1_0(-1,1)$, we have
\begin{equation}
\begin{aligned}\label{basicphi0}
\|F\|_{H_k^{-1}}\|\phi\|_{H^1_k}
\geq& |\langle F,\phi\rangle|=|-\nu(\partial_y^2-|k|^2)w\\
&\qquad \qquad+ik_1\mathbf{Ray}_{-\delta_1}w-\epsilon\nu^{\frac{1}{2}}|k|^{\frac{1}{2}}w,\phi\rangle|\\
\geq &-\nu\|w'\|_{L^2}\|\phi'\|_{L^2}-\big|\nu {|k|^2}-k_1\delta_1\big|\|w\|_{L^2}\|\phi\|_{L^2}\\
&\quad-\epsilon(\nu |k_1|)^{\frac{1}{2}}\|w\|_{L^2}\|\phi\|_{L^2}+|\langle k_1\mathbf{Ray}_{-\delta_1}w, \phi\rangle|.
\end{aligned}
\end{equation}

Due to $\nu|k|^3|k_1|^{-1}\leq 2(|1-\lambda|^{\frac{1}{2}}+|\nu/k_1|^{\frac{1}{4}})$, we have
$\left|\nu|k|^2-k_1\delta_1\right|\lesssim |k_1|\delta_1.$
 By \eqref{basicphi0},  Proposition \ref{resolvent-H-1} and Corollary \ref{reslvent-H-1-sgap}, we obtain
\begin{equation}
\begin{aligned}\label{basicphi}
&|k_1||\langle w, (1-y^2-\lambda+i\delta_1)\phi+2(\partial_y^2-|k|^2)^{-1}\phi\rangle|\\
&\lesssim  \|F\|_{H_k^{-1}}\|\phi\|_{H^1_k}+\nu\|w'\|_{L^2}\|\phi'\|_{L^2} \\
&\qquad +|k_1|\delta_1\|w\|_{L^2}\|\phi\|_{L^2}+\epsilon\nu^{\frac{1}{2}}|k_1|^{\frac{1}{2}}\|w\|_{L^2}\|\phi\|_{L^2}\\
&\lesssim \|F\|_{H_k^{-1}}\|\phi\|_{H_k^1}+ |\nu/k_1|^{-\frac{1}{3}}(|\lambda-1|^{\frac{1}{2}}+\delta)^{\frac{1}{3}}
\|F\|_{H_k^{-1}}\|\phi\|_{L^2}.
\end{aligned}
\end{equation}

 If $\phi\in H^1(-1,1)$ and $\phi(-1)=0$, for any $\delta_*\in (0,\delta^{\frac{4}{3}}]\subset (0,1]$, let $\chi(y)=\max\left\{1-(1-y)/\delta_*,0\right\},$ and $\phi_1(y)=\phi(y)-\phi(1)\chi(y)$ for $y\in [-1,1]$. Then $\phi_1\in H^1_0(-1,1)$ with $\mathrm{supp}\chi\subset [1-\delta_*,1]$ and
\begin{equation}
\begin{aligned}\nonumber
&\|\chi\|_{L^\infty}=1, \ \|\chi\|_{L^2}\leq \delta_*^{\frac{1}{2}}, \ \|\chi'\|_{L^2}\leq \delta_*^{-\frac{1}{2}}, \\
& \|(\partial_y^2-|k|^2)^{-1}\chi\|_{L^\infty} \leq \|\chi\|_{L^1(1-\delta_*,1)}\lesssim \delta_*,\\
&\|(1-y^2-\lambda)\chi+2(\partial_y^2-|k|^2)^{-1}\chi\|_{L^\infty}\lesssim |\lambda|+\delta_*.
\end{aligned}
\end{equation}
As $w(1)=0$, we have $|w(y)|=|\int_y^1 w'(z)\mathrm{d}z|\leq |1-y|^{\frac{1}{2}}\|w'\|_{L^2}\leq \delta_*^{\frac{1}{2}}\|w'\|_{L^2}$ for $y\in [1-\delta_*,1]$ and $\|w\|_{L^1(1-\delta_*,1)}\leq \delta_*^{\frac{3}{2}}\|w'\|_{L^2}.$ Therefore, we deduce that
\begin{equation}\nonumber
\begin{aligned}
|&\langle w,(1-y^2-\lambda+i\delta_1)\phi+2(\partial_y^2-k^2)^{-1}\phi\rangle|\\
&\leq |\phi(1)||\langle w,(1-y^2-\lambda+i\delta_1)\chi+2(\partial_y^2-|k|^2)^{-1}\chi\rangle|\\
&\quad+|\langle w,(1-y^2-\lambda+i\delta_1)\phi_1+2(\partial_y^2-|k|^2)^{-1}\phi_1\rangle|\\
&\lesssim|\phi(1)|\|w\|_{L^1(1-\delta_*,1)}\|(1-y^2-\lambda)\chi+2(\partial_y^2-|k|^2)^{-1}\chi\|_{L^\infty}+\delta_1|\phi(1)|\|w\|_{L^2}\|\chi\|_{L^2}\\
&\quad+|k_1|^{-1}\|F\|_{H_k^{-1}}\|\phi_1\|_{H_k^1}+|k_1|^{-1}
|\nu/k_1|^{-\frac{1}{3}}(|\lambda-1|^{\frac{1}{2}}+\delta)^{\frac{1}{3}}\|F\|_{H_k^{-1}}\|\phi_1\|_{L^2}\\
&\lesssim |\phi(1)|\delta_*^{\frac{3}{2}}\nu^{-1}\|F\|_{H_k^{-1}}(|\lambda|+\delta_*)\\
&\quad+|k_1|^{-1}|\phi(1)|\|F\|_{H^{-1}_k}(\|\chi\|_{H^1_k}
+|\nu/k_1|^{-\frac{1}{3}}(|\lambda-1|^{\frac{1}{2}}+\delta)^{\frac{1}{3}}\|\chi\|_{L^2})\\
&\quad+|k_1|^{-1}\|F\|_{H^{-1}_k}\big(\|\phi\|_{H^1_k}+
|\nu/k_1|^{-\frac{1}{3}}(|\lambda-1|^{\frac{1}{2}}+\delta)^{\frac{1}{3}}\|\phi\|_{L^2}\big)\\
&\lesssim|k_1|^{-1}|\phi(1)|(\delta_*^{\frac{3}{2}}(|\lambda|+\delta^{\frac{4}{3}})\delta^{-4}+\delta_*^{-\frac{1}{2}}
+\nu^{-\frac{1}{3}}|k_1|^{\frac{1}{3}}
(|\lambda-1|^{\frac{1}{2}}+\delta)^{\frac{1}{3}}\delta_*^{\frac{1}{2}})\|F\|_{H^{-1}_k}\\
&\quad+|k_1|^{-1}\|F\|_{H^{-1}_k}\big(\|\phi\|_{H^1_k}+
\nu^{-\frac{1}{3}}|k_1|^{\frac{1}{3}}(|\lambda-1|^{\frac{1}{2}}+\delta)^{\frac{1}{3}}\|\phi\|_{L^2}\big).
\end{aligned}
\end{equation}
Taking $\delta_*=(|\lambda|+\delta^{\frac{4}{3}})^{-\frac{1}{2}}\delta^2$ to yield that
\begin{align*}
&|\langle w,(1-y^2-\lambda+i\delta_1)\phi+2(\partial_y^2-k^2)^{-1}\phi\rangle|\\
&\lesssim |k_1|^{-1}|\phi(1)|\|F\|_{H_k^{-1}}
(|\lambda|+\delta^{\frac{4}{3}})^{\frac{1}{4}}\delta^{-1}\\
&\quad+|k_1|^{-1}\|F\|_{H_k^{-1}}(\|\phi\|_{H^1_k}
+\nu^{-\frac{1}{3}}|k_1|^{\frac{1}{3}}\big(
|\lambda-1|^{\frac{1}{2}}+\delta)^{\frac{1}{3}}\|\phi\|_{L^2}\big),
\end{align*}
here we have used $(|\lambda-1|^{\frac{1}{2}}+\delta)^{\frac{1}{3}}\lesssim\delta^{-\frac{2}{3}}(|\lambda|+\delta^{\frac{4}{3}})^{\frac{1}{2}}$.

For $f\in H^1(-1,1)$ with $f(-1)=0$, let $\phi=\mathbf{Ray}_{\delta_1}^{-1}f$, then
$(-\lambda+i\delta_1)\phi(1)=f(1)$. Then we obtain
\begin{equation}\nonumber
\begin{aligned}
&|\langle w,f\rangle|\leq |\langle w,(1-y^2-\lambda+i\delta_1)\phi+2(\partial_y^2-k^2)^{-1}\phi\rangle|\\
&\lesssim|k_1|^{-1}|\phi(1)|\|F\|_{H_k^{-1}} (|\lambda|+\delta^{\frac{4}{3}})^{\frac{1}{4}}\delta^{-1}\\
&\quad+|k_1|^{-1}\|F\|_{H_k^{-1}} \big(\|\phi\|_{H^1_k}+\nu^{-\frac{1}{3}}|k_1|^{\frac{1}{3}}(|\lambda-1|^{\frac{1}{2}}+\delta)^{\frac{1}{3}}\|\phi\|_{L^2}\big).
\end{aligned}
\end{equation}
Using the fact that $|\phi(1)|\leq (|\lambda|+\delta_1)^{-1}|f(1)|$ and  $|\lambda|+\delta^{\frac{4}{3}}\lesssim |\lambda|+\delta_1,$ we deduce that
\begin{equation}\nonumber
\begin{aligned}
|\langle w,f\rangle|
\lesssim&|k_1|^{-1}|f(1)|\|F\|_{H_k^{-1}}(|\lambda|+\delta^{\frac{4}{3}})^{\frac{1}{4}}(|\lambda|+\delta_1)^{-1}\delta^{-1}\\
&+|k_1|^{-1}\|F\|_{H_k^{-1}}\big(\|\mathbf{Ray}_{\delta_1}^{-1}f\|_{H^1}
+\nu^{-\frac{1}{3}}|k_1|^{\frac{1}{3}}(|\lambda-1|^{\frac{1}{2}}+\delta)^{\frac{1}{3}}\|\mathbf{Ray}_{\delta_1}^{-1}f\|_{L^2}\big)\\
\lesssim&|k_1|^{-1}|f(1)|\|F\|_{H_k^{-1}}(|\lambda|+\delta^{\frac{4}{3}})^{-\frac{3}{4}}\delta^{-1}\\
&+|k_1|^{-1}\|F\|_{H_k^{-1}}\big(\|\mathbf{Ray}_{\delta_1}^{-1}f\|_{H^1}+ \delta^{-\frac{4}{3}}(|\lambda-1|^{\frac{1}{2}}+\delta)^{\frac{1}{3}}\|\mathbf{Ray}_{\delta_1}^{-1}f\|_{L^2}\big).
\end{aligned}
\end{equation}
This proves the weak-type resolvent estimate. {The case of $f(1)=0$ can be achieved similarly.}
\end{proof}

The above weak-type resolvent estimate yields the following estimate for $\|u\|_{L^2}$.

\begin{corollary}\label{u-L2-H-1}
 Let $w\in H^1(I)$ be a solution of \eqref{OSH-1} with $F\in H^{-1}_k(I)$. Then it holds that
$$\|u\|_{L^2}\leq C(\nu |k_1|)^{-\frac{1}{2}}\|F\|_{H_k^{-1}}.$$
\end{corollary}
\begin{proof}
\textbf{Case 1.} $|1-\lambda|\geq 4$.

In this case, we get by \eqref{2.0-basic}(considering the imaginary part) and \eqref{2.2} that
\begin{align*}
   & \big(|1-\lambda|-1\big)\|w\|_{L^2}^2-\|u\|_{L^2}^2\leq \|F\|_{H_k^{-1}}\|(\partial_y,|k|)w\|_{L^2} \\
   &\quad\lesssim \nu^{-1}\|F\|_{H_k^{-1}}^2+\epsilon^{\f12}|\nu/k_1|^{-\f14}\|F\|_{H_k^{-1}}\|w\|_{L^2}.
\end{align*}
Then $\|w\|_{L^2}\lesssim (\nu^{-\f12}+|\nu/k_1|^{-\f14})\|F\|_{H^{-1}_k}$, hence
\begin{align*}
    &\|u\|_{L^2}\leq |k|^{-1}\|w\|_{L^2}\lesssim  (\nu^{-\f12}|k|^{-1}+|\nu/k_1|^{-\f14}|k|^{-1})\|F\|_{H^{-1}_k}\lesssim (\nu|k_1|)^{-\f12}\|F\|_{H^{-1}_k}.
 \end{align*}

\textbf{Case 2.}  $\nu|k|^3|k_1|^{-1}\leq |1-\lambda|^{\f12}+|\nu/k_1|^{\f14}$ and $|1-\lambda|\leq 4$.

In this case, we have $\delta_1\ll1$. Taking $f=\varphi=(\partial_y^2-|k|^2)^{-1}w$ with $\varphi(\pm 1)=0$, we get by Lemma \ref{lem:Ray-simple-1} that
\begin{align*}
   & \|\mathbf{Ray}_{\delta_1}^{-1}f\|_{H^1}+ \delta^{-\frac{4}{3}}(|\lambda-1|^{\frac{1}{2}}+\delta)^{\frac{1}{3}}\|\mathbf{Ray}_{\delta_1}^{-1}f\|_{L^2} \\
  &\lesssim \big(\delta_1^{-\f32}(|\lambda-1|+\delta_1)^{\f12}+ \delta_1^{-\f12}+\delta_1^{-\f12}\delta^{-\frac{4}{3}}(|\lambda-1|^{\frac{1}{2}} +\delta)^{\frac{1}{3}}\big)\|u\|_{L^2}\lesssim |\nu/k_1|^{-\f12}\|u\|_{L^2}.
\end{align*}
here we have used $|\lambda-1|^{\f12}+\delta \sim (|\lambda-1|+\delta_1)^{\f12}$. Therefore, it follows from Proposition \ref{weaktype} that
\begin{align*}
\|u\|_{L^2}^2=&|\langle w,\varphi\rangle|\\
 \lesssim& |k_1|^{-1}\|F\|_{H^{-1}_k} \big(\|\mathbf{Ray}_{\delta_1}^{-1}f\|_{H^1}+ \delta^{-\frac{4}{3}}(|\lambda-1|^{\frac{1}{2}}+\delta)^{\frac{1}{3}}\|\mathbf{Ray}_{\delta_1}^{-1}f\|_{L^2}\big) \\ \lesssim& (\nu|k_1|)^{-\f12}\|F\|_{H^{-1}_k}\|u\|_{L^2},
\end{align*}
which gives our result.\smallskip

\textbf{Case 3.} $\nu|k|^3|k_1|^{-1}\geq |1-\lambda|^{\f12}+|\nu/k_1|^{\f14}$.

In this case, it holds that $\nu |k|^3|k_1|^{-1}\geq  |\nu /k_1|^{\frac{1}{4}}$ and then
\begin{equation}\nonumber
\begin{aligned}
\|u\|_{L^2}\lesssim |k|^{-1}\|w\|_{L^2}\lesssim |k|^{-1}\nu^{-\frac{3}{4}}|k_1|^{-\frac{1}{4}}\|F\|_{H^{-1}_k}\lesssim ( \nu |k_1|)^{-\frac{1}{2}}\|F\|_{H^{-1}_k}.
\end{aligned}
\end{equation}
\end{proof}

\subsection{Resolvent estimate for $F\in  H^1_0$}
In this subsection, we consider the Orr-Sommerfeld equation with $F\in H^1_0(I)$
\begin{equation}\label{OSH-H1}
\left \{
\begin{array}{lll}
-\nu(\partial_y^2-|k|^2)w+ik_1[(1-y^2-\lambda)w+2\varphi]-\epsilon (\nu |k_1|)^{\frac{1}{2}}w=F,\\
w(\pm 1)=\varphi(\pm 1)=\varphi''(\pm 1)=0, \ \ (\partial_y^2-|k|^2)\varphi=w,
\end{array}
\right.
\end{equation}
where $0<\epsilon\ll 1, \lambda\in\mathbb{R}$ and $|k|^2=k_1^2+k_3^2$ with $k_1\not=0$.
\begin{proposition}\label{resolvent-H1}
Let $w\in H^3(I)$ be the solution to \eqref{OSH-H1} with $F\in H^1_0(I)$. Then it holds that
\begin{align}
   & |\nu/k_1|^{\frac{1}{2}} \|w'\|_{L^2} + |\nu/k_1|^{\frac{1}{4}}\|w\|_{L^2} +|\nu/k_1|^{\frac{1}{8}}\|u\|_{L^\infty} +\|u\|_{L^2}\le C|k_1|^{-1}\|(\partial_y,|k|)F\|_{L^2}.\nonumber
\end{align}
\end{proposition}

\begin{proof}
It is easy to see that
 \begin{align}\label{eq:int-w-H1-00}
  &\nu \|(\partial_y,|k|)w\|_{L^2}^2+ik_1\int_{-1}^1(1-y^2-\lambda)|w|^2-2ik_1\|u\|_{L^2}^2\nonumber\\
   &=\langle F, w\rangle+\epsilon(\nu|k_1|)^{\frac{1}{2}}\|w\|_{L^2}^2,
 \end{align}
 which gives $$\nu\|(\partial_y,|k|)w\|_{L^2}^2\leq \|(\partial_y, |k|) F\|_{L^2}\|u\|_{L^2}+\epsilon(\nu|k_1|)^{\frac{1}{2}}\|w\|_{L^2}^2.$$

Assume that $\nu|k|^3|k_1|^{-1}\leq |1-\lambda|^{\f12}+|\nu/k_1|^{\f14}$. Then the proof will be completed by three cases.\smallskip

\noindent\textbf{Case 1. $\lambda\geq 1$.}

If $\lambda\geq 1$, by taking the imaginary part of \eqref{eq:int-w-H1-00}, we have
$$\|\sqrt{\lambda-1+y^2}w\|_{L^2}^2+2\|u\|_{L^2}^2\leq |k_1|^{-1}\|(\partial_y, |k|) F\|_{L^2}\|u\|_{L^2},$$
which gives
\begin{align}\label{est:yw-lambda1-01}
 \|yw\|_{L^2}+\|u\|_{L^2}\lesssim |k_1|^{-1}\|(\partial_y,|k|)F\|_{L^2}.
\end{align}
Therefore, we have
\begin{equation}\nonumber
\begin{aligned}
\nu\|(\partial_y,|k|)w\|_{L^2}^2\leq &\|(\partial_y,|k|)F\|_{L^2}\|u\|_{L^2}+\epsilon (\nu|k_1|)^{\frac{1}{2}}\|w\|_{L^2}^2\\
\lesssim& |k_1|^{-1}\|(\partial_y,|k|)F\|_{L^2}^2+\epsilon (\nu|k_1|)^{\frac{1}{2}}\|(\partial_y,|k|)w\|_{L^2}\|u\|_{L^2}\\
\lesssim& |k_1|^{-1}\|(\partial_y,|k|)F\|_{L^2}^2+\epsilon |\nu/k_1|^{\frac{1}{2}}\|(\partial_y,|k|)w\|_{L^2}\|(\partial_y,|k|)F\|_{L^2}.
\end{aligned}
\end{equation}
Due to $\epsilon\ll 1$, we get by Young's inequality that
\begin{equation}\label{est:parw-lambda1-01}
\begin{aligned}
\|(\partial_y,|k|)w\|_{L^2}^2\lesssim& (\nu |k_1|)^{-1}\|(\partial_y,|k|)F\|_{L^2}^2.
\end{aligned}
\end{equation}

Note that
\begin{equation}\nonumber
\begin{aligned}
\|w\|_{L^2}^2=\langle 1,|w|^2\rangle=-\langle y,\partial_y|w|^2\rangle \leq 2\|yw\|_{L^2}\|\partial_yw\|_{L^2},
\end{aligned}
\end{equation}
then by \eqref{est:yw-lambda1-01} and \eqref{est:parw-lambda1-01}, we get
\begin{equation}\label{w-H01}
\|w\|_{L^2}^2\lesssim \nu^{-\frac{1}{2}}|k_1|^{-\frac{3}{2}}\|(\partial_y,|k|)F\|_{L^2}^2.
\end{equation}
Meanwhile, we have
$$\|u\|_{L^\infty}\lesssim \|u\|_{L^2}^{\frac{1}{2}}\|w\|_{L^2}^{\frac{1}{2}}\lesssim\nu^{-\frac{1}{8}}|k_1|^{-\frac{7}{8}}\|(\partial_y,|k|)F\|_{L^2}.$$

\noindent\textbf{Case 2. $\lambda \leq 0$.}

Indeed, it follows by taking the real part of \eqref{eq:int-w-H1-00} that
$$\nu\|(\partial_y,|k|)w\|_{L^2}^2\lesssim \|(\partial_y,|k|)F\|_{L^2}\|u\|_{L^2}+\epsilon(\nu|k_1|)^{\frac{1}{2}}\|w\|_{L^2}^2,$$
which along with $\epsilon(\nu|k_1|)^{\frac{1}{2}}\|w\|_{L^2}^2\leq (\nu^{\f12}\|(\partial_y,|k|)w\|_{L^2})(\epsilon|k_1|^{\f12}\|u\|_{L^2})$ gives
\begin{align}\label{eq:image-lambda00}
   \nu\|(\partial_y,|k|)w\|_{L^2}^2\lesssim \|(\partial_y,|k|)F\|_{L^2}\|u\|_{L^2}+\epsilon^2|k_1|\|u\|_{L^2}^2.
 \end{align}
By taking the imaginary part of \eqref{eq:int-w-H1-00}, we obtain
\begin{align}\label{eq:image-lambda0}
   |k_1|&\left(\int_{-1}^{1}(1-y^2-\lambda)|w|^2\mathrm{d}y +2\int_{-1}^{1}\varphi\overline{w}\mathrm{d}y\right)\\
   &\leq |\langle F,w\rangle|\leq \|(\partial_y,|k|)F\|_{L^2}\|u\|_{L^2},\nonumber
\end{align}
which along with Lemma \ref{positive} gives
\begin{align*}
    & \|u\|_{L^2}^2+|k|^2\|\varphi\|_{L^2}^2\leq 9|k_1|^{-1}\|(\partial_y,|k|)F\|_{L^2}\|u\|_{L^2}.
 \end{align*}
 Therefore, we have
 \begin{align}\label{eq:image-lambda000}
    & \|u\|_{L^2}\leq 9|k_1|^{-1}\|(\partial_y,|k|)F\|_{L^2}.
 \end{align}
This along with \eqref{eq:image-lambda00} gives
 \begin{align*}
    & \|(\partial_y,|k|)w\|_{L^2}\lesssim (\nu |k_1|)^{-\f12}\|(\partial_y,|k|)F\|_{L^2}.
 \end{align*}

Now we turn to establish the estimate for $\|w\|_{L^2}$. By \eqref{eq:image-lambda0}, we have
\begin{align*}
    \|(1-y^2)^{\f12}w\|_{L^2}^2\leq & |k_1|^{-1}\|(\partial_y,|k|)F\|_{L^2}\|u\|_{L^2}+2\|u\|_{L^2}^2\\
   \lesssim& |k_1|^{-2}\|(\partial_y,|k|)F\|_{L^2}^2,
\end{align*}
from which and  Lemma \ref{hardy-type}, we infer that
\begin{align*}
   & \|w\|_{L^2}\lesssim \nu^{-\f16}|k_1|^{-\f56}\|(\partial_y,|k|)F\|_{L^2}.
\end{align*}
Then we have
$$\|u\|_{L^\infty}\lesssim \|u\|_{L^2}^{\frac{1}{2}}\|w\|_{L^2}^{\frac{1}{2}}\lesssim\nu^{-\frac{1}{8}}|k_1|^{-\frac{7}{8}}\|(\partial_y,|k|)F\|_{L^2}.$$

\noindent\textbf{Case 3. $\lambda\in (0,1)$.}

Let $-1\leq y_1\leq 0\leq y_2\leq 1$ such that $1-y_i^2=\lambda(i=1,2)$ with $y_1=-y_2$. To achieve the desired estimates, we decompose $w=w_3+w_4+w_5$, where $w_3,\ w_4$ and $w_5$ solve
\begin{equation*}
\left\{\begin{aligned}
&\big(1-y^2-\lambda+i(2\delta_1-\nu|k|^2/k_1)\big)w_3+2\varphi_3=F/( ik_1) ,\\
&-\nu(\partial_y^2-|k|^2)w_4+ik_1\big((1-y^2-\lambda)w_4+2\varphi_4\big)= \nu\partial_y^2w_3,\\
&-\nu(\partial_y^2-|k|^2)w_5+ik_1\big((1-y^2-\lambda)w_5+2\varphi_5\big)=-2k_1\delta_1w_3,\\
&(\partial_y^2-|k|^2)\varphi_3=w_3,\ (\partial_y^2-|k|^2)\varphi_4=w_4,\ (\partial_y^2-|k|^2)\varphi_5=w_5,\\ &\varphi_3|_{\pm1}=\varphi_4|_{\pm1}=\varphi_5|_{\pm1}=0,
\end{aligned}
\right.
\end{equation*}
where $\delta_1=|\nu/k_1|^{\f13}\big(|1-\lambda|^{\f12}+|\nu/k_1|^{\f14}\big)^{\f23} =|\nu/k_1|^{\f13}\big((y_2-y_1)/2+|\nu/k_1|^{\f14}\big)^{\f23} $.

Thanks to $\nu|k|^3|k_1|^{-1}\leq |1-\lambda|^{\f12}+|\nu/k_1|^{\f14}$, we have $\nu|k|^2/|k_1|\leq \delta_1 $, $2\delta_1-\nu|k|^2/k_1 \sim \delta_1\lesssim |\nu/k_1|^{\f13}\ll 1$. Then  by Lemma \ref{lem:Ray-simple-1}, we get
\begin{align*}
  &  \|(\partial_y,|k|)\varphi_3\|_{L^2}+ \delta_1^{\f12}\|w_3\|_{L^2}\\
    &\qquad+\delta_1^{\f32}((y_2-y_1)^2+\delta_1)^{-\f12} \|w_3'\|_{L^2}\lesssim |k_1|^{-1}\|(\partial_y,|k|)F\|_{L^2},
\end{align*}
which along with $(y_2-y_1)^2+\delta_1\lesssim (y_2-y_1)^2+|\nu/k_1|^{\f12}$ yields that
\begin{equation}\label{est:varphi3-01}
  \begin{aligned}
  & \|(\partial_y,|k|)\varphi_3\|_{L^2}+ |\nu/k_1|^{\f16} \big((y_2-y_1)+|\nu/k_1|^{\f14}\big)^{\f13}\|w_3\|_{L^2}\\
  &\qquad+|\nu/k_1|^{\f12} \|w_3'\|_{L^2} \lesssim |k_1|^{-1}\|(\partial_y,|k|)F\|_{L^2}.
  \end{aligned}
\end{equation}
By Corollary \ref{reslvent-H-1-sgap}, Corollary \ref{u-L2-H-1} and \eqref{est:varphi3-01}, we get
\begin{equation}\label{est:varphi4-01}
  \begin{aligned}
  &(\nu|k_1|)^{\f12}\|(\partial_y,|k|)\varphi_4\|_{L^2}+ \nu^{\f23}|k_1|^{\f13}\big((y_2-y_1)+|\nu/k_1|^{\f14}\big)^{\f13} \|w_4\|_{L^2}
  \\& \lesssim \nu\|\partial_yw_3\|_{L^2}\lesssim |\nu/k_1|^{\f12}\|(\partial_y,|k|)F\|_{L^2}.
  \end{aligned}
\end{equation}
Thanks to Corollary \ref{resol-lambda} and \eqref{est:varphi3-01}, we have
\begin{align*}
   & \nu^{\f13}|k_1|^{\f23}\big((y_2-y_1)+|\nu/k_1|^{\f14}\big)^{\f23}\|w_5\|_{L^2}\lesssim |k_1|\delta_1\|w_3\|_{L^2}\\
   &\lesssim |\nu/k_1|^{\f16} \big((y_2-y_1)+|\nu/k_1|^{\f14}\big)^{\f13}\|(\partial_y,|k|)F\|_{L^2}.
\end{align*}
 By Corollary \ref{resol-lambda-00} and \eqref{est:varphi3-01}, we have
\begin{align*}
   & |\nu/k_1|^{\f16}\big((y_2-y_1)+|\nu/k_1|^{\f14}\big)^{\f13}\|(\partial_y,|k|)\varphi_5\|_{L^2} \lesssim |k_1|^{-1}\|2k_1\delta_1 w_3\|_{L^2}\\
   &\lesssim |\nu/k_1|^{\f16} \big((y_2-y_1)+|\nu/k_1|^{\f14}\big)^{\f13}|k_1|^{-1}\|(\partial_y,|k|)F\|_{L^2},
\end{align*}
which implies
\begin{align}\label{est:varphi5-01}
   & \|(\partial_y,|k|)\varphi_5\|_{L^2}+ |\nu/k_1|^{\f16} \big((y_2-y_1)+|\nu/k_1|^{\f14}\big)^{\f13}\|w_5\|_{L^2}\nonumber\\
   &\lesssim |k_1|^{-1}\|(\partial_y,|k|)F\|_{L^2}.
\end{align}

Thanks to $w=w_3+w_4+w_5$, we conclude by \eqref{est:varphi3-01}, \eqref{est:varphi4-01} and \eqref{est:varphi5-01} that
\begin{align}\label{est:u-w-FH1-01}
   & \|u\|_{L^2}+ |\nu/k_1|^{\f16} \big((y_2-y_1)+|\nu/k_1|^{\f14}\big)^{\f13}\|w\|_{L^2}\nonumber\\
   &\lesssim |k_1|^{-1}\|(\partial_y,|k|)F\|_{L^2},
\end{align}
which implies
\begin{align}\label{est:paw-FH1-01}
     \nu\|(\partial_y,|k|)w\|_{L^2}^2\leq &\|(\partial_y,|k|)F\|_{L^2}\|u\|_{L^2}+\epsilon (\nu|k_1|)^{\frac{1}{2}}\|w\|_{L^2}^2\\
     \lesssim& |k_1|^{-1}\|(\partial_y,|k_1|)F\|_{L^2}^2.\nonumber
\end{align}
Now \eqref{est:u-w-FH1-01} and \eqref{est:paw-FH1-01} show that
\begin{align*}
     \|u\|_{L^2}+& |\nu/k_1|^{\f16} \big((y_2-y_1)+|\nu/k_1|^{\f14}\big)^{\f13}\|w\|_{L^2}\\
    &+|\nu/k_1|^{\f12}\|w'\|_{L^2}\lesssim |k_1|^{-1}\|(\partial_y,|k|)F\|_{L^2}.
 \end{align*}
Then we finish the proof of Case 3.
\smallskip

It remains to consider the case of $\nu |k|^3|k_1|^{-1}\geq |1-\lambda|^{\f12}+|\nu/k_1|^{\f14}$. In this case, we have $\nu|k|^3|k_1|^{-1}\geq \nu^{\frac{1}{4}}|k_1|^{-\frac{1}{4}}$, that is, $(\nu |k_1|^3)^{\frac{1}{4}}\lesssim \nu |k|^3$. Therefore, for any $\lambda\in \mathbb{R}$, by Proposition \ref{resolvent-L2}, we have
\begin{equation}\nonumber
\begin{aligned}
&\|w\|_{L^2}\lesssim \nu^{-\frac{1}{2}}|k_1|^{-\frac{1}{2}}|k|^{-1}\|(\partial_y,|k|)F\|_{L^2}\lesssim \nu^{-\frac{1}{4}}|k_1|^{-\frac{3}{4}}\|(\partial_y,|k|)F\|_{L^2},\\
&\|w'\|_{L^2}\lesssim  \nu^{-\frac{3}{4}}|k_1|^{-\frac{1}{4}}|k|^{-1}\|(\partial_y,|k|)F\|_{L^2}\lesssim (\nu|k_1|)^{-\frac{1}{2}}\|(\partial_y,|k|)F\|_{L^2},\\
&\|u\|_{L^2}\lesssim\nu^{-\frac{3}{8}}|k_1|^{-\frac{5}{8}}|k|^{-\frac{3}{2}}\|(\partial_y,|k|)F\|_{L^2}
\lesssim|k_1|^{-1}\|(\partial_y,|k|)F\|_{L^2},\\
&\|u\|_{L^\infty}\lesssim \|u\|_{L^2}^{\frac{1}{2}}\|w\|_{L^2}^{\frac{1}{2}}\lesssim \nu^{-\frac{1}{8}}|k_1|^{-\frac{7}{8}}\|(\partial_y,|k|)F\|_{L^2},
\end{aligned}
\end{equation}
which give the desired results.

This completes the proof of the proposition.
\end{proof}

\subsection{Orr-Sommerfeld equation without nonlocal term}

To establish the space-time estimates  of $\om_2$, we need to consider the Orr-Sommerfeld equation without nonlocal term
\begin{equation}\label{eq:res-non}
\left \{
\begin{array}{lll}
-\nu(\partial_y^2-|k|^2)w+ik_1(1-y^2-\lambda)w-\epsilon (\nu |k_1|)^{\frac{1}{2}}w=F_1,\\
(\partial_y^2-|k|^2)\varphi=w, \quad \varphi(\pm 1)=0,\quad
w(\pm 1)=0,
\end{array}
\right.
\end{equation}
with $\epsilon$ small enough.

The proof of the following two propositions are very similar to the case with nonlocal term. Here we omit the proof.

\begin{proposition}\label{prop:res-non}
Let $w \in H^2(I)$ be the solution of (\ref{eq:res-non}) with $\lambda \in \mathbb{R}$ and $F_1 \in L^2(I)$. Then there holds that
\begin{equation}\nonumber
\begin{aligned}
\nu^{\frac{3}{4}}|k_1|^{\frac{1}{4}}\|(\pa_y, |k|)w\|_{L^2}+\nu^{\frac{1}{2}}|k_1|^{\frac{1}{2}} \Vert w \Vert_{L^2}
\leq C \Vert F_1 \Vert_{L^2}.
\end{aligned}
\end{equation}
\end{proposition}

\begin{proposition}\label{H-1-local}
Let $w\in H^1_0$ be the solution to (\ref{eq:res-non}) with $F_1\in H^{-1}_k$, then it holds that
$$\nu\|(\pa_y, |k|)w\|_{L^2}+\nu^{\frac{3}{4}}|k_1|^{\frac{1}{4}}\|w\|_{L^2}\leq C\|F_1\|_{H^{-1}_k}.$$
\end{proposition}

The following two corollaries can be obtained by following similar arguments as in Corollary \ref{resol-lambda} and  Corollary \ref{resol-lambda} respectively.

\begin{corollary}\label{pro-os-local}
It holds that
 \begin{align*}
&(\nu |k_1|^2)^{1/3}(|1-\lambda|^{\frac{1}{2}}+|\nu/k_1|^{1/4})^{2/3}\|w\|_{L^2}\\
&\qquad +(\nu^2 |k_1|)^{1/3}(|1-\lambda|^{\frac{1}{2}}+|\nu/k_1|^{1/4})^{1/3}\|(\partial_y,|k|)w\|_{L^2} \lesssim \|F_1\|_{L^2}.
 \end{align*}
\end{corollary}

\begin{corollary}\label{cor:OS-local}
It holds that
   \begin{align*}
&(\nu^2 |k_1|)^{1/3}(|1-\lambda|^{\frac{1}{2}}+|\nu/k_1|^{1/4})^{1/3}\|w\|_{L^2} \lesssim \|F_1\|_{H^{-1}_k}.
 \end{align*}
\end{corollary}

Next we consider the OS equation with the force $yG$, where the weight $y$ is crucial to cancel the singularity near the critical point $y=0$.
\begin{equation}\label{os-local-1}
\left \{
\begin{aligned}
&-\nu(\partial_y^2-|k|^2)w+ik_1(1-y^2-\lambda)w-\epsilon (\nu |k_1|)^{\frac{1}{2}}w=yG,\\
& (\partial_y^2-|k|^2)\varphi=w, \quad \varphi(\pm 1)=0, \quad w(\pm 1)=0.
\end{aligned}
\right.
\end{equation}

 \begin{lemma}\label{Ray-resol-local}
 Let $W$ solve the Rayleigh equation with $\delta>0$
 \begin{equation}\nonumber
 (1-y^2-\lambda+i\delta)W=yg, \ (\partial_y^2-|k|^2)\Phi=W,\, \Phi (\pm 1)=0.
 \end{equation}
 If $g\in H^1_0(-1,1)$, then we have
 $$\|(\partial_y,|k|)\Phi\|_{L^2}+\delta^{\frac{1}{2}}\|W\|_{L^2}+\delta^{\frac{3}{2}}(|1-\lambda|+\delta)^{-\frac{1}{2}}\|W'\|_{L^2}\lesssim |k|^{-1}\|(\partial_y,|k|)g\|_{L^2}.$$
 \end{lemma}
\begin{proof}
Thanks to
 $$W=\frac{yg}{1-y^2-\lambda+i\delta},$$
we infer for $c^2=1-\lambda+i\delta$ that
 \begin{align*}
\|(\partial_y,|k|)\Phi\|_{L^2}^2=&
\left|\int_{-1}^1 \frac{yg \overline{\Phi}}{1-y^2-\lambda+i\delta}\mathrm{d}y \right|\\
\leq& \frac{1}{2}\left(\left|\int_{-1}^1 \frac{g\overline{\Phi}}{y+c}\mathrm{d}y \right|+\left|\int_{-1}^1 \frac{g\overline{\Phi}}{y-c}\mathrm{d}y \right|\right)\\
\leq& C|k|^{-1/2}(\|\partial_y (g \Phi)\|_{L^2}+|k|\|g \Phi\|_{L^2})\\
\leq& C|k|^{-1/2}(\|\partial_y g\|_{L^2}\|\Phi\|_{L^\infty}+\|g\|_{L^\infty}\|\partial_y \Phi\|_{L^2}+|k|\|g\|_{L^\infty}\|\Phi\|_{L^2})\\
\lesssim& |k|^{-1}\|(\partial_y,|k|)g\|_{L^2}\|(\partial_y,|k|)\Phi\|_{L^2},
 \end{align*}
where we have used Lemma \ref{hardy-type-abp}. This shows
 \begin{align*}
\|(\partial_y,|k|)\Phi\|_{L^2}
\lesssim& |k|^{-1}\|(\partial_y,|k|)g \|_{L^2},
 \end{align*}

Moreover, we also have
 \begin{align*}
\|W\|_{L^2}^2\leq& \left|\left\langle\frac{g}{y+c},W \right\rangle\right|+\left|\left\langle\frac{g}{y-c},W \right\rangle\right|
\leq C\|g\|_{L^2} \delta^{-\frac{1}{2}}\|W\|_{L^2},
 \end{align*}
which gives
\begin{align*}
\|W\|_{L^2} \leq C |k|^{-1} \|(\partial_y,|k|) g\|_{L^2} \delta^{-\frac{1}{2}} .
 \end{align*}

 Notice that
  $$W'=\frac{yg'+g+2yW}{1-y^2-\lambda+i\delta},$$
  which gives
   \begin{align*}
\|W'\|_{L^2}\lesssim& ( \|g'\|_{L^2}+\|W\|_{L^2})\left\|\frac{y}{1-y^2-\lambda+i\delta}\right\|_{L^\infty}+\|g\|_{L^2}\left\|\frac{1}{1-y^2-\lambda+i\delta}\right\|_{L^\infty}\\
\lesssim&( \|g'\|_{L^2}+\|W\|_{L^2})\delta^{-1}(|1-\lambda|+\delta)^{\frac{1}{2}}+|k|^{-1}\|(\partial_y,|k|)g\|_{L^2}\delta^{-1}\\
\lesssim& \|g'\|_{L^2}\delta^{-1}(|1-\lambda|+\delta)^{\frac{1}{2}}+|k|^{-\f12}\|(\partial_y,|k|)g\|_{L^2}\delta^{-\f12}\\
&+|k|^{-1}\delta^{-\f32}(|1-\lambda|+\delta)^{\frac{1}{2}}\|(\partial_y,|k|)g\|_{L^2}\\
\lesssim&|k|^{-1}\delta^{-\f32}(|1-\lambda|+\delta)^{\frac{1}{2}}\|(\partial_y,|k|)g\|_{L^2}.
 \end{align*}
From the above proof, we can also deduce
  \begin{align}
\|W'\|_{L^2}\lesssim (\delta^{-\f12}+\delta^{-1}|k|^{-1})\|(\partial_y,|k|)g\|_{L^2}.\label{eq:W-new}
 \end{align}
 \end{proof}

\begin{proposition}\label{resol-H-1-local}
Let $w$ be the solution to \eqref{os-local-1} with $G|_{y=\pm 1}=0$. Then it holds that
 \begin{align*}
(\nu/|k_1|)^{\frac{1}{4}}\|w \|_{L^2}+(\nu/|k_1|)^{\frac{1}{2}} \|(\partial_y,|k|) w\|_{L^2} \lesssim (|k||k_1|)^{-1} \|(\partial_y,|k|)G\|_{L^2}.
 \end{align*}
\end{proposition}

\begin{proof}
We only consider the case of $\epsilon=0$. For $\epsilon\neq 0$, the following Case 1 can be deduced from a perturbation argument, and the other cases are obvious.

{\bf Case 1. $\lambda \geq 1.$}\smallskip

The standard energy gives
\begin{align}\label{basic-local}
\nu \|(\partial_y,|k|)w\|_{L^2}^2+ik_1\|(\lambda-1+y^2)^{\frac{1}{2}}w\|_{L^2}^2=\langle yG,w\rangle.
\end{align}
Taking the imaginary part of \eqref{basic-local}, we get
\begin{align*}
|k_1|\|(\lambda-1+y^2)^{\frac{1}{2}}w\|_{L^2}^2\leq \|G\|_{L^2}\|y w\|_{L^2},
\end{align*}
which implies
\begin{align*}
|k_1|\|y w\|_{L^2} \leq \|G\|_{L^2}.
\end{align*}
Therefore, we have
\begin{align*}
\nu\|(\partial_y,|k|)w\|_{L^2}^2\leq   \|G\|_{L^2}\|y w\|_{L^2}\leq |k_1|^{-1}\|G\|_{L^2}^2\leq |k|^{-2} |k_1|^{-1} \|(\partial_y,|k|)G\|_{L^2}^2.
\end{align*}

For $\delta =(\nu/|k_1|)^{1/4}\ll 1$, notice that
\begin{align*}
\delta^2|k_1|\|w\|_{L^2([-1,1]\setminus (-\delta,\delta))}^2\leq& |k_1|\|(\lambda-1+y^2)^{\frac{1}{2}}w\|_{L^2}^2\\
\leq& \|G\|_{L^2}\|y w\|_{L^2}\leq |k_1|^{-1}\|G\|_{L^2}^2,
\end{align*}
then we deduce
\begin{align*}
\|w\|_{L^2}^2\leq& 2\delta\|w\|_{L^\infty}^2+\delta^{-2}|k_1|^{-2}\|G\|_{L^2}^2\\
\lesssim&\delta \|w\|_{L^2}\|w'\|_{L^2}+\delta^{-2}|k_1|^{-2}\|G\|_{L^2}^2\\
\lesssim&\delta (\nu |k_1|)^{-\frac{1}{2}} \|w\|_{L^2}\|G\|_{L^2}+\delta^{-2}|k_1|^{-2}\|G\|_{L^2}^2,
\end{align*}
which yields
\begin{align*}
\|w\|_{L^2}^2
\lesssim&(\delta^2 (\nu |k_1|)^{-1} +\delta^{-2}|k_1|^{-2})\|G\|_{L^2}^2\lesssim \nu^{-1/2} |k_1|^{-3/2}\|G\|_{L^2}^2\\
\lesssim & \nu^{-1/2} |k_1|^{-3/2} |k|^{-2} \|(\partial_y,|k|)G\|_{L^2}^2.
\end{align*}

{\bf Case 2. $\lambda \leq 0$ and $\nu |k|^3|k_1|^{-1}\leq |1-\lambda|^{\frac{1}{2}}+|\nu/k_1|^{1/4}$.}\smallskip

In this case, we have
\begin{align*}
&|k_1|\left(\|(1-y^2)^{\frac{1}{2}}w\|_{L^2}^2 -\lambda \|w\|_{L^2}^2\right)
\leq |\langle yG,w\rangle|\leq \left\|\frac{yG}{(1-y^2)^{\frac{1}{2}}}\right\|_{L^2}\|(1-y^2)^{\frac{1}{2}}w\|_{L^2}\\
&\leq C \|G\|_{L^2}^{\frac{1}{2}}\left\|\frac{yG}{(1-y^2)}\right\|_{L^2}^{\frac{1}{2}}\|(1-y^2)^{\frac{1}{2}}w\|_{L^2}\leq |k|^{-\frac{1}{2}}\|(\partial_y,|k|)G\|_{L^2} \|(1-y^2)^{\frac{1}{2}}w\|_{L^2},
\end{align*}
which gives
\begin{align*}
 \|(1-y^2)^{\frac{1}{2}}w \|_{L^2}
 \leq |k|^{-1/2}|k_1|^{-1}  \|(\partial_y,|k|)G\|_{L^2}.
 \end{align*}
 Moreover, we have
 \begin{align*}
& \|(\partial_y,|k|)\varphi\|_{L^2}^2\leq |k|^{-2}|k_1|^{-2}  \|(\partial_y,|k|)G\|_{L^2}^2.
 \end{align*}
 Here we used the fact that
  $$ \|(\partial_y,|k|)\varphi \|_{L^2}\leq |k|^{-1/2}\|(1-y^2)^{\frac{1}{2}}w\|_{L^2}.$$

  To proceed, we decompose $w=w_3+w_4+w_5$ with
 \begin{equation}\nonumber
\left\{
\begin{aligned}
  &(1-y^2-\lambda+i(-2\delta_1-\nu|k|^2/k_1))w_3=yG/ik_1,\\
  &-\nu (\partial_y^2-|k|^2)w_4+ik_1(1-y^2-\lambda)w_4=\nu \partial_y^2 w_3,\\
  &-\nu (\partial_y^2-|k|^2)w_5+ik_1(1-y^2-\lambda)w_5=-2k_1\delta_1 w_3,
  \end{aligned}
\right.
\end{equation}
where $\delta_1=|\nu/k_1|^{1/3}(|1-\lambda|^{1/2}+|\nu/k_1|^{1/4})^{2/3}$.

Due to $\lambda\leq 0$ and $\nu |k|^3|k_1|^{-1}\leq |1-\lambda|^{\frac{1}{2}}+|\nu/k_1|^{1/4}\lesssim |1-\lambda|^{\frac{1}{2}}$, one has $\nu |k|^2/|k_1|\leq \delta_1$, and
by Lemma \ref{Ray-resol-local}, Proposition \ref{H-1-local} and Corollary \ref{pro-os-local}, we obtain
 \begin{align*}
&\|w_3\|_{L^2} \lesssim \delta_1^{-\frac{1}{2}} |k|^{-1} |k_1|^{-1}\|(\partial_y,|k|)G\|_{L^2},\\
&\|\partial_y w_3\|_{L^2}\lesssim \delta_1^{-\frac{3}{2}} (|1-\lambda|+\delta_1)^{\frac{1}{2}} |k|^{-1} |k_1|^{-1} \|(\partial_y,|k|)G\|_{L^2}\\
&\qquad \qquad \lesssim \nu^{-\frac{1}{2}} |k|^{-1}  |k_1|^{-\frac{1}{2}} \|(\partial_y,|k|)G\|_{L^2}  ,\\
&\nu \|(\partial_y,|k|)w_4\|_{L^2}\lesssim \nu \|\partial_y w_3\|_{L^2} \lesssim \nu^{\frac{1}{2}} |k|^{-1}  |k_1|^{-\frac{1}{2}} \|(\partial_y,|k|)G\|_{L^2},\\
&\|(\partial_y,|k|)w_5\|_{L^2}\lesssim \nu^{-\frac{2}{3}}|k_1|^{-\frac{1}{3}} |\lambda-1|^{-\frac{1}{6}} |k_1|\delta_1\|w_3\|_{L^2}\\
&\qquad \qquad \qquad \lesssim \nu^{-\frac{1}{2}}|k_1|^{-\frac{1}{2}} |k|^{-1} \|(\partial_y,|k|)G\|_{L^2}.
 \end{align*}
This shows
 \begin{align*}
&|\nu/k_1|^{\frac{1}{2}} \|(\partial_y,|k|)w\|_{L^2}\lesssim |k_1|^{-1} |k|^{-1} \|(\partial_y,|k|)G\|_{L^2},
 \end{align*}
and
 \begin{align*}
\|w\|_{L^2}\lesssim \|(\partial_y,|k|)w\|_{L^2}^{\frac{1}{2}} \|(\partial_y,|k|)\varphi\|_{L^2}^{\frac{1}{2}} \lesssim \nu^{-\frac{1}{4}}|k|^{-1} |k_1|^{-\frac{3}{4}}  \|(\partial_y,|k|)G\|_{L^2}.
 \end{align*}

 {\bf Case 3. $\lambda \in (0,1)$  and $\nu |k|^3|k_1|^{-1}\leq |1-\lambda|^{\frac{1}{2}}+|\nu/k_1|^{1/4}$.}\smallskip

 As in Case 2, we decompose $w=w_3+w_4+w_5$. Then by Lemma \ref{Ray-resol-local}, Proposition \ref{H-1-local}, Corollary \ref{cor:OS-local} and Corollary \ref{pro-os-local}, we obtain
   \begin{align*}
&\delta_1^{\f12}\|w_3 \|_{L^2}
+\delta_1^{\f32}(|1-\lambda|+\delta_1)^{-\f12}\|\partial_y w_3 \|_{L^2}\lesssim  |k|^{-1}|k_1|^{-1}\|(\partial_y,|k|)G\|_{L^2},\\
&(\nu^2 |k_1|)^{\f13}(|1-\lambda|^{\frac{1}{2}}+|\nu/k_1|^{1/4})^{\f13}\|w_4\|_{L^2}+\nu\|(\partial_y,|k|) w_4\|_{L^2}\lesssim  \nu \|\partial_y w_3\|_{L^2},\\
&(\nu |k_1|^2)^{\f13}(|1-\lambda|^{\frac{1}{2}}+|\nu/k_1|^{\f14})^{\f23}\|w_5\|_{L^2}\\
&\qquad +(\nu^2 |k_1|)^{\f13}(|1-\lambda|^{\frac{1}{2}} +|\nu/k_1|^{1/4})^{1/3}\|(\partial_y,|k|)w_5\|_{L^2}
\lesssim |k_1|\delta_1 \| w_3\|_{L^2},
 \end{align*}
 which imply
 \begin{align*}
 &\|w \|_{L^2}\lesssim  \nu^{-1/4} |k|^{-1}|k_1|^{-3/4} \|(\partial_y,|k|)G\|_{L^2},\\
&(\nu/|k_1|)^{\frac{1}{2}} \|(\partial_y,|k|) w\|_{L^2}\lesssim (|k||k_1|)^{-1} \|(\partial_y,|k|)G\|_{L^2}.
 \end{align*}

{\bf Case 4. $\nu |k|^3 |k_1|^{-1}\geq |1-\lambda|^{\frac{1}{2}}+|\nu/k_1|^{\frac{1}{4}}$.}\smallskip

In this case, it holds that $\nu |k|^4 |k_1|^{-1}\gtrsim  1$.  As in Case 2, we decompose $w=w_3+w_4+w_5$.  Then by Lemma \ref{Ray-resol-local}, \eqref{eq:W-new}, Proposition \ref{H-1-local},  and Corollary \ref{pro-os-local}, we obtain
 \begin{align*}
 \|w_3\|_{L^2}\lesssim&[\delta_1+ (\nu|k|^2/k_1)]^{-\f12}|k|^{-1}|k_1|^{-1} \|(\partial_y,|k|)G\|_{L^2}\\
\lesssim& \nu^{-\f14}|k_1|^{-\f34}|k|^{-1} \|(\partial_y,|k|)G\|_{L^2},\\
\|\partial_y w_3\|_{L^2}\lesssim&( (\nu|k|^2/k_1)^{-\f12}+ \delta_1^{-1} |k|^{-1} )|k_1|^{-1}  \|(\partial_y,|k|)G\|_{L^2}\\
\lesssim&\nu^{-\f12}|k|^{-1} |k_1|^{-\f12} \|(\partial_y,|k|)G\|_{L^2},\\
\|w_4\|_{L^2}\lesssim& \nu^{-\f34} |k_1|^{-\f14} \nu \|\partial_y w_3\|_{L^2}\lesssim \nu^{-\frac{1}{4}}|k|^{-1} |k_1|^{-\frac{3}{4}}\|(\partial_y,|k|)G\|_{L^2},\\
\|(\partial_y, |k|)w_4\|_{L^2}\lesssim& \nu^{-1} \nu \|\partial_y w_3\|_{L^2}\lesssim\nu^{-\f12}|k|^{-1} |k_1|^{-\f12} \|(\partial_y,|k|)G\|_{L^2},\\
\|w_5\|_{L^2} \lesssim& \nu^{-\f13}|k_1|^{-\f23} (|\lambda-1|^{\f12}+|\nu/k_1|^{\f14})^{-\f23}|k_1|\delta_1 \| w_3\|_{L^2}\\
\lesssim&\nu^{-\f14}|k_1|^{-\f34}|k|^{-1} \|(\partial_y,|k|)G\|_{L^2},\\
\|(\partial_y,|k|)w_5\|_{L^2} \lesssim& \nu^{-\f23}|k_1|^{-\f13} (|\lambda-1|^{\f12}+|\nu/k_1|^{\f14})^{-\f13}|k_1|\delta_1 \| w_3\|_{L^2}\\
\lesssim&\nu^{-\f12}|k_1|^{-\f12}|k|^{-1} \|(\partial_y,|k|)G\|_{L^2}.
 \end{align*}
Therefore, we arrive at
 \begin{align*}
&\|w\|\lesssim \nu^{-\f14}|k_1|^{-\f34}|k|^{-1} \|(\partial_y,|k|)G\|_{L^2},\\
&\|(\partial_y,|k|) w\|_{L^2}\lesssim \nu^{-\f12} |k|^{-1}|k_1|^{-\f12}   \|(\partial_y,|k|)G\|_{L^2}.
 \end{align*}

 The proof is completed.
\end{proof}

\begin{proposition}\label{resol-H-1-local-u}
Let $w$ be the solution to \eqref{os-local-1} with $G|_{y=\pm 1}=0$. Then it holds that
 \begin{align*}
\|(\partial_y,|k|)\varphi\|_{L^2}\lesssim ( \nu^{-\f18} |k_1|^{-\f78}|k|^{-\f32} +|k_1|^{-1}|k|^{-1})\|(\partial_y,|k|)G\|_{L^2}.
 \end{align*}
\end{proposition}
\begin{proof}
We only consider the case of $\epsilon=0$. The case of $\epsilon\neq 0$ can be easily seen from the proof below.
We have
\begin{equation}\nonumber
ik_1w=\frac{yG+\nu(\partial_y^2-|k|^2)w+ik_1\delta_2 w}{(1-y^2-\lambda+i\delta_2)},
\end{equation}
where $\delta_2=|\nu/k_1|^{\f12}$.
Then by integration by parts, we have
 \begin{align*}
 &|k_1|\|(\partial_y,|k|)\varphi\|_{L^2}^2\\
  \leq &\left|\int_{-1}^1 \frac{yG\overline{\varphi}}{(1-y^2-\lambda+i\delta_2)}\mathrm{d}y\right|+\nu \left| \int_{-1}^1 \frac{\partial_y w \partial_y\overline{\varphi}}{(1-y^2-\lambda+i\delta_2)}\mathrm{d}y\right|\\
 &+\nu |k|^2  \left| \int_{-1}^1 \frac{ w \overline{\varphi}}{(1-y^2-\lambda+i\delta_2)}\mathrm{d}y\right|+\nu  \left| \int_{-1}^1 \frac{\partial_y w (-2y) \overline{\varphi}}{(1-y^2-\lambda+i\delta_2)^2}\mathrm{d}y\right|\\
 &+|k_1|\delta_2 \left| \int_{-1}^1 \frac{ w  \overline{\varphi}}{(1-y^2-\lambda+i\delta_2)}\mathrm{d}y\right|\\
 \lesssim& |k|^{-\f12}\|(\partial_y,|k|)(G\varphi)\|_{L^2}+\nu\delta_2^{-1} \|(\partial_y,|k|)w\|_{L^2}\|(\partial_y,|k|)\varphi\|_{L^2}\\
 &+\nu \|\partial_y w\|_{L^2}\|\varphi\|_{L^\infty}\delta_2^{-\f54}+\delta_2 |k_1|\|w\|_{L^2} \|\varphi\|_{L^\infty}\delta_2^{-\f34}\\
 \lesssim& (\nu^{-\f18} |k_1|^{\f18}|k|^{-\f32}+|k|^{-1}) \|(\partial_y,|k|)G\|_{L^2}\|(\partial_y,|k|)\varphi\|_{L^2},
 \end{align*}
 Here we used Lemma \ref{hardy-type-abp} for the first term(see the proof of Lemma \ref{Ray-resol-local}).
 This gives
 \begin{align*}
 \|(\partial_y,|k|)\varphi\|_{L^2}
 \lesssim ( \nu^{-\f18} |k_1|^{-\f78}|k|^{-\f32} +|k_1|^{-1}|k|^{-1})\|(\partial_y,|k|)G\|_{L^2}.
 \end{align*}

\end{proof}

\section{Resolvent estimates with non-slip boundary condition}\label{sec-no-slip}

The goal of this section is to establish the resolvent estimates for the Orr-Sommerfeld equation with non-slip boundary condition
\begin{equation}\label{5.1}
\left \{
\begin{array}{lll}
&-\nu(\partial_y^2-|k|^2)w+ik_1[(1-y^2-\lambda)w+2\varphi]
-\epsilon\nu^{\frac{1}{2}}|k_1|^{\frac{1}{2}}w=F,\\
& (\partial_y^2-|k|^2)\varphi=w,\quad \varphi(\pm 1)=\varphi'(\pm 1)=0,
\end{array}
\right.
\end{equation}
or in terms of  the stream function
\begin{equation}\label{5.2}
\left \{
\begin{array}{lll}
&-\nu(\partial_y^2-|k|^2)^2 \varphi +ik_1[(1-y^2-\lambda)(\partial_y^2-|k|^2)\varphi+2\varphi]\\
&\qquad\qquad\qquad -\epsilon\nu^{\frac{1}{2}}|k_1|^{\frac{1}{2}}(\partial_y^2-|k|^2)\varphi=F,\\
&\varphi(\pm 1)=\varphi'(\pm 1)=0,
\end{array}
\right.
\end{equation}
where $\lambda\in\mathbb{R}$.\smallskip

Motivated by \cite{CLWZ}, we introduce the following decomposition
$$\varphi=\varphi_{Na}+c_1\psi_1+c_2\psi_2,$$
where $\varphi_{Na}, \psi_{i}(i=1,2)$ solve the following problems respectively
\begin{equation}\label{5.3}
\left \{
\begin{array}{lll}
-\nu(\partial_y^2-|k|^2)^2\varphi_{Na}+ik_1[(1-y^2-\lambda)(\partial_y^2-|k|^2)\varphi_{Na} +2\varphi_{Na}] \\ \qquad-\epsilon\nu^{\frac{1}{2}}|k_1|^{\frac{1}{2}}(\partial_y^2-|k|^2)\varphi_{Na}=F,\\
\varphi_{Na}(\pm 1)=\varphi_{Na}''(\pm 1)=0,
\end{array}
\right.
\end{equation}
and $\psi_1,\psi_2$ solve
\begin{equation}\label{5.3a}
\left \{
\begin{array}{lll}
-\nu(\partial_y^2-|k|^2)^2\psi_1+ik_1[(1-y^2-\lambda)(\partial_y^2-|k|^2)\psi_1+2\psi_1] \\ \qquad -\epsilon\nu^{\frac{1}{2}}|k_1|^{\frac{1}{2}}(\partial_y^2-|k|^2)\psi_1=0,\\
\psi_1(\pm 1)=0, \ \psi_1'(1)=1, \ \psi_1'(- 1)=0,
\end{array}
\right.
\end{equation}
and
\begin{equation}\label{5.3b}
\left \{
\begin{array}{lll}
-\nu(\partial_y^2-|k|^2)^2\psi_2+ik_1[(1-y^2-\lambda)(\partial_y^2-|k|^2)\psi_2+2\psi_2] \\ \qquad-\epsilon\nu^{\frac{1}{2}}|k_1|^{\frac{1}{2}}(\partial_y^2-|k|^2)\psi_2=0,\\
\psi_2(\pm 1)=0, \ \psi_2'(-1)=1, \ \psi_2'( 1)=0.
\end{array}
\right.
\end{equation}

Let $w_{Na}=(\partial_y^2-|k|^2)\varphi_{Na}$ and $w_i=(\partial_y^2-|k|^2)\psi_i$. Then we have
$$w=w_{Na}+c_1w_1+c_2w_2=(\partial_y^2-|k|^2)\varphi.$$
The non-slip boundary condition yields that
$$\langle e^{\pm |k|y},w\rangle=0,$$
which implies
\begin{equation}\label{5.17}
\left \{
\begin{aligned}
c_1(\lambda)=&-\int_{-1}^1\frac{\sinh \big(|k|(y+1)\big)}{\sinh (2|k|)}w_{Na}(y)\mathrm{d}y,\\ \
c_2(\lambda)=&\int_{-1}^1\frac{\sinh \big(|k|(1-y)\big)}{\sinh (2|k|)}w_{Na}(y)\mathrm{d}y.
\end{aligned}
\right.
\end{equation}


\subsection{Boundary correctors}\label{sub:Bd-corrector}

To solve \eqref{5.3a}-\eqref{5.3b}, we will construct the approximate solutions via the Airy function.

We denote $L=|2k_1/\nu|^{\frac{1}{3}}$. Without loss of generality, we always assume $k_1>0$ and the case of $k_1<0$ can be handled by similar arguments.
Let $\textbf{Ai}(y)$ be the Airy function, which solves $f''-yf=0.$ In addition, we define
\begin{align}\label{eq:A0-def-00}
  & A_0(z)=\int_{e^{\frac{i\pi}{6}z}}^\infty\textbf{Ai}(t)\mathrm{d}t=e^{\frac{i\pi}{6}}\int_z^\infty \textbf{Ai}(e^{\frac{i\pi}{6}}t)\mathrm{d}t,\nonumber \\
  &d=-\frac{\lambda}{2}-\frac{i\nu|k|^2}{2k_1}+\frac{i\epsilon(\nu|k_1|^{-1})^{\frac{1}{2}}}{2}.
\end{align}
 For $k_1\not=0$, we define
$$W_1(y)=\textbf{Ai}\left(e^{\frac{i\pi}{6}}L\left(y+1-\frac{\lambda}{2}-\frac{i\nu|k|^2}{2k_1}+\frac{i\epsilon(\nu|k_1|^{-1})^{\frac{1}{2}}}{2}\right)\right),$$ $$W_2(y)=\textbf{Ai}\left(e^{\frac{i\pi}{6}}L\left(1-y-\frac{\lambda}{2}-\frac{i\nu|k|^2}{2k_1}+\frac{i\epsilon(\nu|k_1|^{-1})^{\frac{1}{2}}}{2}\right)\right).$$
Then it is easy to verify that
\beno
&&-\nu(\partial_y^2-|k|^2)W_1+ik_1(2y-\lambda+2)W_1-
\epsilon\nu^{\frac{1}{2}}|k_1|^{\frac{1}{2}}W_1=0,\\
&&-\nu(\partial_y^2-|k|^2)W_2+ik_1(2-2y-\lambda)W_2-\epsilon\nu^{\frac{1}{2}}|k_1|^{\frac{1}{2}}W_2=0.
\eeno
Moreover, we define the stream functions $\Phi_i$ via
\beno
(\partial_y^2-|k|^2)\Phi_i=W_{i}, \ \ \Phi_i(\pm 1)=0\quad \text{for}\,\,i=1,2.
\eeno

\begin{lemma}\label{lem:W1W2-00}
It holds that for $m\geq0$,
\begin{equation}\nonumber
\begin{aligned}
&\|W_1\|_{L^\infty}+\|W_2\|_{L^\infty}\leq C(1+|Ld|)^{\f12}\left|A_0\left(Ld\right)\right|,\\
&\|(1+y)^mW_1\|_{L^1}+\|(1-y)^mW_2\|_{L^1}\leq CL^{-1-m}\left|A_0\left(Ld\right)\right|,\\
&\|(1+y)^mW_1\|_{L^2}+\|(1-y)^mW_2\|_{L^2}\leq C(1+|Ld|)^{\frac{1}{4}-\frac{m}{2}}L^{-\frac{1+2m}{2}}|A_0(Ld)|,\\
&\|(\partial_y\Phi_i,|k|\Phi_i)\|_{L^2}\leq CL^{-\frac{3}{2}}(1+|Ld|)^{-\frac{1}{4}}\left|A_0\left(Ld\right)\right|, \quad i=1,2,\\
&\|(\partial_y\Phi_i,|k|\Phi_i)\|_{L^\infty}\leq CL^{-1}\left|A_0\left(Ld\right)\right|, \quad i=1,2.
\end{aligned}
\end{equation}
\end{lemma}

\begin{proof}
These estimates can be obtained by following almost same arguments as in \cite{CLWZ, CWZ-mem, CWZ-cpam, CWZ-cmp}. Here we omit the details.
\end{proof}

Let $w_{err,i}(i=1,2)$ be the error, which solves
\begin{equation}\label{5.3d}
\left \{
\begin{array}{lll}
-\nu(\partial_y^2-|k|^2)w_{err,1}+ik_1[(1-y^2-\lambda)(\partial_y^2-|k|^2)\varphi_{err,1}+2\varphi_{err,1}]\\
\quad \quad \quad -\epsilon\nu^{\frac{1}{2}}|k_1|^{\frac{1}{2}}w_{err,1}=-ik_1(1+y)^2W_1+2ik_1\Phi_1,\\
(\partial_y^2-|k|^2)\varphi_{err,1}=w_{err,1},\\
w_{err,1}(\pm 1)=\varphi_{err,1}(\pm 1)=\varphi_{err,1}''(\pm 1)=0,
\end{array}
\right.
\end{equation}
and
\begin{equation}\label{5.3e}
\left \{
\begin{array}{lll}
-\nu(\partial_y^2-|k|^2)w_{err,2}+ik_1[(1-y^2-\lambda)(\partial_y^2-|k|^2)\varphi_{err,2}+2\varphi_{err,2}]\\
\quad \quad \quad -\epsilon\nu^{\frac{1}{2}}|k_1|^{\frac{1}{2}}w_{err,2}=-ik_1(1-y)^2W_2+2ik_1\Phi_2,\\
(\partial_y^2-|k|^2)\varphi_{err,2}=w_{err,2},\\
w_{err,2}(\pm 1)=\varphi_{err,2}(\pm 1)=\varphi_{err,2}''(\pm 1)=0,
\end{array}
\right.
\end{equation}
Therefore, the solutions $w_{1},w_{2}$ can be represented as
\begin{equation}\label{5.13}
\left \{
\begin{array}{lll}
w_{1}=C_{11}(W_1(y)+w_{err,1})+C_{12}(W_2(y)+w_{err,2}),\\
w_{2}=C_{21}(W_1(y)+w_{err,1})+C_{22}(W_2(y)+w_{err,2}),
\end{array}
\right.
\end{equation}
where $C_{ij},i,j=1,2$ are the constants, which could be determined by matching the boundary conditions. Indeed,  using the boundary conditions for $w_{1}$ (or equivalently, for $\varphi_{1}$), we have
$$\int_{-1}^1e^{|k|y}w_{1}(y)\mathrm{d}y=e^{|k|}, \quad \int_{-1}^1e^{-|k|y}w_{1}(y)\mathrm{d}y=e^{-|k|},$$
which gives
\begin{equation}\label{5.14}
\left \{
\begin{aligned}
\sinh(2|k|)=&C_{11}\int_{-1}^1\sinh(|k|(y+1))(W_1+w_{err,1})(y)\mathrm{d}y\\
&+C_{12}\int_{-1}^1\sinh(|k|(y+1))(W_2+w_{err,2})(y)\mathrm{d}y,\\
0=&C_{11}\int_{-1}^1\sinh(|k|(1-y))(W_1+w_{err,1})(y)\mathrm{d}y\\
&+C_{12}\int_{-1}^1\sinh(|k|(1-y))(W_2+w_{err,2})(y)\mathrm{d}y.
\end{aligned}
\right.
\end{equation}
If $A_1A_2-B_1B_2\not=0$ (to be verified later),  then we have
\begin{equation}\label{5.15}
C_{11}=\frac{\sinh(2|k|)A_2}{A_1A_2-B_1B_2},\quad
C_{12}=\frac{-\sinh(2|k|)B_1}{A_1A_2-B_1B_2},
\end{equation}
where
$$A_1=\int_{-1}^1\sinh(|k|(1+y))(W_1+w_{err,1})(y)\mathrm{d}y =:A_{a,1}+A_{err,1},$$
$$  A_2=\int_{-1}^1\sinh(|k|(1-y))(W_2+w_{err,2})(y)\mathrm{d}y =:A_{a,2}+A_{err,2},$$
$$B_1=\int_{-1}^1\sinh(|k|(1-y))(W_1+w_{err,1})(y)\mathrm{d}y =:B_{a,1}+B_{err,1},$$
$$B_2=\int_{-1}^1\sinh(|k|(1+y))(W_2+w_{err,2})(y)\mathrm{d}y =:B_{a,2}+B_{err,2}.$$
Similar arguments yield that
\begin{equation}\label{5.16}
C_{21}=\frac{\sinh(2|k|)B_2}{A_1A_2-B_1B_2},\quad
C_{22}=\frac{-\sinh(2|k|)A_1}{A_1A_2-B_1B_2}.
\end{equation}

The following lemma gives the estimates for the errors.

\begin{lemma}\label{estimates:error}
Let $d=-\f{\lambda}{2}-\f{i\nu |k|^2}{2k_1}-\f{i\epsilon |\nu/k_1|^{1/2}}{2} $. For $j=1,2$, it holds that
\begin{align*}
   & |\nu/k_1|^{\f14}\|w_{err,j}'\|_{L^2}+ \|w_{err,j}\|_{L^2}
+|\nu/k_1|^{-\f18} \|u_{err,j}\|_{L^\infty}+|\nu/k_1|^{-\f14} \|u_{err,j}\|_{L^2} \\
&\le CL^{-1}(1+|Ld|)^{-\frac{1}{4}}|A_0(Ld)|,
\end{align*}
where $A_0$ is defined by \eqref{eq:A0-def-00}.
\end{lemma}

\begin{proof}
We only give the proof for $w_{err,1}$ and the arguments for $w_{err,2}$ are similar. For this, we decompose the error as $w_{err,i}=w_{err,11}+w_{err,12}$,
where
\begin{equation}\nonumber
\left \{
\begin{array}{lll}
-\nu(\partial_y^2-|k|^2)w_{err,11}+ik_1[(1-y^2-\lambda)(\partial_y^2-|k|^2)\varphi_{err,11}+2\varphi_{err,11}]\\
\quad \quad \quad -\epsilon\nu^{\frac{1}{2}}|k_1|^{\frac{1}{2}}w_{err,11}=-ik_1(1+y)^2W_1,\\
-\nu(\partial_y^2-|k|^2)w_{err,12}+ik_1[(1-y^2-\lambda)(\partial_y^2-|k|^2)\varphi_{err,12}+2\varphi_{err,12}]\\
\quad \quad \quad -\epsilon\nu^{\frac{1}{2}}|k_1|^{\frac{1}{2}}w_{err,12}=2ik_1\Phi_1,\\
(\partial_y^2-|k|^2)\varphi_{err,1j}=w_{err,1j},\quad\varphi_{err,1j}(\pm 1)=\varphi_{err,1j}''(\pm 1)=0, j=1,2.
\end{array}
\right.
\end{equation}

By Corollary \ref{resolvent-L2-complex}, Corollary \ref{resol-lambda-00} and Lemma \ref{lem:W1W2-00}, we have
\begin{equation}\label{est:Werr11-00}
\begin{aligned}
&(\nu |k_1|)^{\frac{1}{2}}\|w_{err,11}\|_{L^2}+ \nu^{\frac{3}{4}}|k_1|^{\frac{1}{4}}\|w_{err,11}'\|_{L^2}
\\
&\quad+\big(\nu^{\frac{3}{8}}|k_1|^{\frac{5}{8}}|k|^{\frac{1}{2}}+\nu^{\f14}|k_1|^{\f34}\big)\|u_{err,11}\|_{L^2}+\nu^{\frac{3}{8}}|k_1|^{\frac{5}{8}}\|w_{err,11}\|_{L^1}\\
&\lesssim |k_1|\|(1+y)^2W_1\|_{L^2}\lesssim  (\nu|k_1|)^{\frac{1}{2}}L^{-1}(1+|Ld|)^{-\frac{3}{4}}|A_0(Ld)|,
\end{aligned}
\end{equation}
here we have used
\begin{equation}\nonumber
\begin{aligned}
\|(1-|y|)^2W_1\|_{L^2}\lesssim& L^{-\frac{5}{2}}(1+|Ld|)^{-\frac{3}{4}}|A_0(Ld)|\\
=& |\nu/k_1|^{\frac{1}{2}}L^{-1}(1+|Ld|)^{-\frac{3}{4}}|A_0(Ld)|.
\end{aligned}
\end{equation}
By Lemma \ref{lem:W1W2-00}, we have
$$ |k_1|\|(\partial_y,|k|)\Phi_1\|_{L^2}\lesssim |k_1|L^{-\frac{3}{2}} (1+|Ld|)^{-\frac{1}{4}}|A_0(Ld)|,$$
which gives by Proposition \ref{resolvent-H1} that
\begin{equation}\label{est:Werr12-00}
\begin{aligned}
&\|w_{err,12}\|_{L^2}\lesssim |\nu/k_1|^{\frac{1}{4}}(1+|Ld|)^{-\frac{1}{4}}|A_0(Ld)|,\\
&\|w_{err,12}'\|_{L^2}\lesssim (1+|Ld|)^{-\frac{1}{4}}|A_0(Ld)|,\\
&\|u_{err,12}\|_{L^2}\lesssim |\nu/k_1|^{\frac{1}{2}}(1+|Ld|)^{-\frac{1}{4}}|A_0(Ld)|.
\end{aligned}
\end{equation}

Hence, by \eqref{est:Werr11-00}, \eqref{est:Werr12-00} and the fact that $|\nu/k_1|^{\f14}\lesssim L^{-1}$, we obtain
\begin{equation}\nonumber
\begin{aligned}
   & |\nu/k_1|^{\f14}\|w_{err,1}'\|_{L^2}+ \|w_{err,1}\|_{L^2}+|\nu/k_1|^{-\f14} \|u_{err,1}\|_{L^2} \\
   &\lesssim L^{-1}(1+|Ld|)^{-\frac{1}{4}}|A_0(Ld)|.
\end{aligned}
\end{equation}
Thanks to $\|u_{err,1}\|_{L^\infty}\lesssim \|w_{err,1}\|_{L^2}^{\f12}\|u_{err,1}\|_{L^2}^{\f12}$,
we complete the proof of this lemma.
\end{proof}

With Lemma \ref{estimates:error} at hand, we can give the estimates for $C_{ij}$ as follows.

\begin{lemma}\label{C_ij}
If $L\geq k_0$ for some $k_0>0$, then it holds that
\begin{equation}\nonumber
\begin{aligned}
|C_{11}|+|C_{22}|\leq\frac{CL^{\frac{5}{8}}}{|A_0(Ld)|},\quad \ |C_{12}|+|C_{21}|\leq\frac{CL}{|A_0(Ld)|}.
\end{aligned}
\end{equation}
\end{lemma}
\begin{proof}
By Lemma \ref{estimates:error}, we have
\begin{equation}\label{eq:Ae1-00}
\begin{aligned}
|A_{err,1}|\leq &\left|\int_{-1}^1\sinh(|k|(1+y))(w_{err,11}+w_{err,12})(y)\mathrm{d}y\right|\\
\lesssim&\sinh (2|k|)(\|w_{err,11}\|_{L^1}+|\partial_y \varphi_{err,12}(1)|)\\
\lesssim&  \sinh (2|k|)(\|w_{err,11}\|_{L^1}+\|u_{err,12}\|_{L^\infty})\\
\lesssim& \sinh (2|k|)|\nu/k_1|^{\frac{1}{8}}\frac{|A_0(Ld)|}{L}\lesssim L^{-\f38}\frac{\sinh (2|k|)|A_0(Ld)|}{L}.
\end{aligned}
\end{equation}
Similarly, we can obtain
\begin{equation}\label{eq:Aei-Bei-00}
\begin{aligned}
|A_{err,2}|+|B_{err,1}|+|B_{err,2}|\lesssim L^{-\frac{3}{8}}\frac{\sinh (2|k|)|A_0(Ld)|}{L}.
\end{aligned}
\end{equation}
Following  similar arguments  in \cite{CLWZ, CWZ-mem}(section 8.3 in \cite{CLWZ} and section 5.3 in \cite{CWZ-mem}), for $i=1,2$, we have
\begin{equation}\label{eq:Bai-Aai-00}
\begin{aligned}
&|B_{a,i}|\geq c_2\frac{\sinh (2|k|)|A_0(Ld)|}{L},\qquad \left|\frac{A_{a,i}}{B_{a,i}}\right|\leq \frac{\sqrt{2}}{2}, \\
&|B_{a,i}|\leq \frac{2\sinh (2|k|)|A_0(Ld)|}{L},\qquad
|A_{a,i}|\leq \frac{3\sinh (2|k|)|A_0(Ld)|}{L^2},
\end{aligned}
\end{equation}
for some constant $c_2>0$. Then we obtain
\begin{equation}\nonumber
\begin{aligned}
|B_1|\geq c_2\big(1-C(c_2L^{\frac{3}{8}})^{-1}\big)\frac{\sinh (2|k|)|A_0(Ld)|}{L},\
\left|B_1/B_{a,1}\right|\geq1-C(c_2L^{\frac{3}{8}})^{-1}.
\end{aligned}
\end{equation}
Choosing $k_0$ large enough such that $C(c_2k_0^{\frac{3}{8}})^{-1}\leq 1/5$ gives
\begin{equation}\nonumber
\begin{aligned}
\left|B_1/B_{a,1}\right|\geq 4/5.
\end{aligned}
\end{equation}
Therefore, we arrive at
\begin{equation}\nonumber
\begin{aligned}
|A_1|\leq &|A_{err,1}|+\left|\frac{A_{a,1}}{B_{a,1}}\right|\left| \frac{B_{a,1}}{B_{1}}\right|\left|B_1\right|\\
\leq& CL^{-\frac{3}{8}}\frac{\sinh (2|k|)|A_0(Ld)|}{L}+\frac{\sqrt{2}}{2}\left(1-C(c_2L^{\frac{3}{8}})^{-1}\right)^{-1}|B_1|\\
\leq &\left(CL^{-\frac{3}{8}}+\frac{\sqrt{2}}{2}\cdot \frac{5}{4}\right)|B_1|.
\end{aligned}
\end{equation}
Choosing $k_0$ large enough such that $L\geq |k_1|\geq k_0$ with
$$CL^{-\frac{3}{8}}+\frac{\sqrt{2}}{2}\cdot \frac{5}{4}\leq \sqrt{\frac{7}{8}}.$$
Then we have
\begin{equation}\nonumber
\begin{aligned}
|A_1|\leq \sqrt{\frac{7}{8}} |B_1|, \quad\ |B_1|\geq\frac{4c_2\sinh (2|k|)|A_0(Ld)|}{5L}
\end{aligned}
\end{equation}
Similar arguments yield that
\begin{equation}\nonumber
\begin{aligned}
|A_2|\leq \sqrt{\frac{7}{8}}|B_2|, \quad\ |B_2|\geq\frac{4c_2\sinh (2|k|)|A_0(Ld)|}{5L}.
\end{aligned}
\end{equation}
Therefore, we deduce that
\begin{equation}\label{eq:A1A2-B1B2-00}
\begin{aligned}
|A_1A_2-B_1B_2|\geq& |B_1B_1|\left|1-\frac{A_1}{B_1}\frac{A_2}{B_2}\right|\\
\geq& \frac{1}{8}|B_1B_1|\geq \frac{1}{8}\left|\frac{4c_2\sinh (2|k|)|A_0(Ld)|}{5L}\right|^2>0.
\end{aligned}
\end{equation}
Furthermore, by \eqref{eq:Ae1-00}, \eqref{eq:Aei-Bei-00}, \eqref{eq:Bai-Aai-00} and $A_j=A_{a,j}+A_{err,j},\ B_j=B_{a,j}+B_{err,j}$, we obtain
\begin{equation}\label{eq:Aj-Bj-00}
\begin{aligned}
|A_j|\leq\frac{C_0\sinh(2|k|)}{L^{1+\frac{3}{8}}}|A_0(Ld)|, \ |B_j|\leq\frac{C_1\sinh(2|k|)}{L}|A_0(Ld)|, \quad j=1,2.
\end{aligned}
\end{equation}
Combining \eqref{eq:A1A2-B1B2-00}-\eqref{eq:Aj-Bj-00} and using \eqref{5.15}-\eqref{5.16}, the conclusion follows.
\end{proof}

Now we are in a position to give the estimates for the boundary correctors $w_i, i=1,2$.

\begin{proposition}\label{estimate-w-i}
If $L\geq k_0$ for some $k_0>0$, then for $i=1,2$, it holds that
\begin{equation}\nonumber
\begin{aligned}
&|\nu/k_1|^{-\f16}(1+|Ld|)^{\frac{1}{2}} \|(\partial_y,|k|)\psi_i\|_{L^2}+(1+|Ld|)^{\frac{1}{4}}\|(\partial_y,|k|)\psi_i \|_{L^\infty}\\
&\qquad\qquad +|\nu/k_1|^{\f16}\|w_i\|_{L^2} \le C(1+|Ld|)^{\frac{1}{4}}.
\end{aligned}
\end{equation}

\end{proposition}

\begin{proof}
We only consider the case of $i=1$ and the case of $i=2$ is similar.
Note that $L\sim |k_1/\nu|^{\frac{1}{3}}$,
 then by Lemma \ref{estimates:error}, Lemma \ref{C_ij}, Lemma \ref{lem:W1W2-00} and \eqref{5.13}, we obtain
\begin{equation*}
\begin{aligned}
&\|w_1\|_{L^2}\leq |C_{11}|\|W_1+w_{err,1}\|_{L^2}+ |C_{12}|\|W_2+w_{err,2}\|_{L^2}\lesssim |\nu/k_1|^{-\f16}(1+|Ld|)^{\frac{1}{4}},\\
&\|(\partial_y,|k|)\psi_1\|_{L^2}\leq |C_{11}|(\|(\partial_y,|k|)\Phi_1\|_{L^2}+\|u_{err,1}\|_{L^2})\\
&\qquad\qquad\qquad\qquad +|C_{12}|(\|(\partial_y,|k|)\Phi_2\|_{L^2}+\|u_{err,2}\|_{L^2})
\\ &\qquad\qquad\qquad\ \lesssim |\nu
/k_1|^{\frac{1}{6}}(1+|Ld|)^{-\frac{1}{4}},\\
&\|(\partial_y,|k|)\psi_1\|_{L^\infty}\leq \|w_1\|_{L^2}^{\f12}\|(\partial_y,|k|)\psi_1\|_{L^2}^{\f12}
\lesssim 1.
\end{aligned}
\end{equation*}
\end{proof}

\subsection{Estimates for the coefficients $c_i\ (i=1,2)$}\label{sub:coefficients-00}

The following lemmas show the estimates for the coefficient $c_i$ when $F\in L^2(I), H^1_k(I)$ and $H^{-1}_k(I)$ respectively.

\begin{lemma}\label{coe-L2}
Assume that $\nu |k|^3 |k_1|^{-1}\leq 2\big(|1-\lambda|^{\frac{1}{2}}+|\nu/k_1|^{\f14}\big)$ and $F\in L^2(I)$. Then it holds that
\begin{equation}\label{ciL2}
\begin{aligned}
(1+|k\lambda|)^{\frac{1}{4}}(|c_1|+|c_2|)\le C\nu^{-\frac{3}{8}}|k_1|^{-\frac{5}{8}}|k|^{-\frac{1}{4}}\|F\|_{L^2}.
\end{aligned}
\end{equation}
\end{lemma}
\begin{proof}
We only consider the estimate for $c_1$ and the proof for $c_2$ is similar. We divide the proof into several cases.\smallskip

\noindent\textbf{Case 1.} $\lambda\geq 1$.

By Corollary \ref{resol-lambda}, we have
\begin{equation}\label{delta-w}
\begin{aligned}
\|w_{Na}\|_{L^2}\lesssim \nu^{-\frac{1}{3}}|k_1|^{-\frac{2}{3}}(|\lambda-1|^{\frac{1}{2}}+|\nu/k_1|^{\f14})^{-\frac{2}{3}}\|F\|_{L^2}.
\end{aligned}
\end{equation}
By taking the inner product between \eqref{5.3} and $w_{Na}$, then taking the imaginary part to the resulting equation, we obtain
\begin{equation}\label{delta-ww}
\begin{aligned}
& \|(\lambda-1+y^2)^{\frac{1}{2}}w_{Na}\|_{L^2}^2+2\|(\partial_y,|k|)\varphi_{Na}\|_{L^2}^2\\
 &\leq  |k_1|^{-1}\|F\|_{L^2}\|w_{Na}\|_{L^2}
 \lesssim \nu^{-\frac{1}{3}}|k_1|^{-\frac{5}{3}} (|\lambda-1|^{\frac{1}{2}}+|\nu/k_1|^{\f14})^{-\frac{2}{3}}\|F\|_{L^2}^2,
 \end{aligned}
\end{equation}
here we have used \eqref{delta-w}.\smallskip

{\bf Case 1.1. $\lambda\geq 1+\frac{1}{2|k|}$.}

It is easy to check that $|k(\lambda-1)|\geq \frac{|k\lambda|}{1+2|k|}\gtrsim (1+|k\lambda|)/|k|$ and $|k(\lambda-1)|\geq 1/2$. Therefore, we get by \eqref{delta-ww} that
\begin{equation}\label{est:c1-casei-00}
\begin{aligned}
|c_1|\lesssim &\left(\int_{-1}^1\left|\frac{\sinh ( |k|(y+1))}{(\lambda-1+y^2)^{\frac{1}{2}}\sinh ( 2|k|)}\right|^2\mathrm{d}y\right)^{\frac{1}{2}}\|(\lambda-1+y^2)^{\frac{1}{2}}w_{Na}\|_{L^2}\\
\lesssim &|k(\lambda-1)|^{-\f12}\|(\lambda-1+y^2)^{\frac{1}{2}}w_{Na}\|_{L^2}\\
\lesssim& |k(\lambda-1)|^{-\frac{1}{2}} \nu^{-\frac{1}{6}}|k_1|^{-\frac{5}{6}} \big(|\lambda-1|^{\frac{1}{2}}+|\nu/k_1|^{\f14}\big)^{-\frac{1}{3}}\|F\|_{L^2}\\
\lesssim &|k(\lambda-1)|^{-\frac{1}{4}} \big( \nu^{-\frac{1}{6}}|k_1|^{-\frac{5}{6}}\big) \big(\nu|k|^3|k_1|^{-1}\big)^{-\f16} \big(|\lambda-1|^{\frac{1}{2}}+|\nu/k_1|^{\f14}\big)^{-\frac{1}{6}}\|F\|_{L^2}\\
\lesssim &\big( (1+|k\lambda|)^{-\f14}|k|^{\f14}\big)\times \big( \nu^{-\frac{1}{6}}|k_1|^{-\frac{5}{6}}\big) \big(\nu|k|^3|k_1|^{-1}\big)^{-\f16} |\nu/k_1|^{-\frac{1}{24}}\|F\|_{L^2}\\
= &\nu^{-\frac{3}{8}}|k_1|^{-\frac{5}{8}}(1+|k\lambda|)^{-\f14}|k|^{-\frac{1}{4}}\|F\|_{L^2},
 \end{aligned}
\end{equation}
which yields the desired estimate.\smallskip

{\bf Case 1.2. $1\leq \lambda\leq 1+\frac{1}{2|k|}$.}

For $\delta_*=|\nu /k_1|^{\frac{1}{6}}\big((\lambda-1)^{\frac{1}{2}}+|\nu/k_1|^{\f14}\big)^{\frac{1}{3}}<1$, we get by \eqref{delta-w} and \eqref{delta-ww} that
\begin{equation*}
\begin{aligned}
\|w_{Na} \|_{L^1}\leq &\|w_{Na}\|_{L^1(-\delta_*,\delta_*)}+\|w_{Na}\|_{L^1\big((-1,1)\setminus (-\delta_*,\delta_*)\big)}\\
\lesssim& \delta_*^{\frac{1}{2}}\|w_{Na}\|_{L^2}+\|y^{-1}\|_{L^2\big((-1,1)\setminus (-\delta_*,\delta_*)\big)}\|(\lambda-1+y^2)^{\frac{1}{2}}w_{Na}\|_{L^2}\\
\lesssim& \nu^{-\frac{1}{4}}|k_1|^{-\frac{3}{4}} \big(|\lambda-1|^{\frac{1}{2}}+|\nu/k_1|^{\f14}\big)^{-\frac{1}{2}}\|F\|_{L^2}\lesssim \nu^{-\f38}|k_1|^{-\f58}\|F\|_{L^2}.
 \end{aligned}
\end{equation*}
Notice that
$$\left|\frac{\sinh (|k|(y+1))}{\sinh( 2|k|)}\right|\leq 4e^{-|k|(1-y)},\quad |k|\sim 1+|k\lambda|,$$
then for $0<\beta\ll 1$ with $1-(100|k|)^{-1+\beta}\geq \frac{3}{4}$, we  obtain
\begin{equation}\label{est:c1-caseii-00}
\begin{aligned}
|c_1|\leq& \left|\int_{-1}^{1-(100|k|)^{-1+\beta}}\frac{\sinh (|k|(y+1))}{\sinh (2|k|)} w_{Na} (y)\mathrm{d}y\right|\\
&\qquad+\left|\int_{1-(100|k|)^{-1+\beta}}^1\frac{\sinh( |k|(y+1))}{\sinh (2|k|)} w_{Na} (y)\mathrm{d}y\right|\\
\lesssim& e^{-|k|^\beta/100}\|w_{Na}\|_{L^1}+\left(\int_{1-(100|k|)^{-1+\beta}}^1\frac{1}{\lambda-1+y^2} \mathrm{d}y\right)^{\frac{1}{2}}\|(\lambda-1+y^2)^{\frac{1}{2}}w_{Na}\|_{L^2}\\
\lesssim& e^{-|k|^\beta/100} \nu^{-\frac{3}{8}}|k_1|^{-\frac{5}{8}}\|F\|_{L^2}+(1-(100|k|)^{-1+\beta})^{-1} |k|^{-\frac{1}{2}+\frac{\beta}{2}} \nu^{-\frac{1}{6}}|k_1|^{-\frac{5}{6}}
\\&\times\big(|\lambda-1|^{\frac{1}{2}}+|\nu/k_1|^{\f14}\big)^{-\frac{1}{3}}\|F\|_{L^2}\\
\lesssim& e^{-|k|^\beta/100} \nu^{-\frac{3}{8}}|k_1|^{-\frac{5}{8}}\|F\|_{L^2}+\nu^{-\frac{1}{4}}|k_1|^{-\frac{3}{4}} |k|^{-\frac{1}{2}+\frac{\beta}{2}}\|F\|_{L^2}\\
\lesssim& e^{-|k|^\beta/100} \nu^{-\frac{3}{8}}|k_1|^{-\frac{5}{8}}\|F\|_{L^2}+\nu^{-\frac{1}{4}}|k_1|^{-\frac{3}{4}} |k|^{-\frac{1}{2}+\frac{\beta}{2}}(\nu|k|^3|k_1|^{-1})^{-\f18}
\\
&\times\big(|\lambda-1|^{\frac{1}{2}}+|\nu/k_1|^{\f14}\big)^{\f18}\|F\|_{L^2}\\
\lesssim &\nu^{-\f38}|k_1|^{-\f58}|k|^{-\f78+\f{\beta}{2}}\|F\|_{L^2}\lesssim  \nu^{-\frac{3}{8}} |k_1|^{-\frac{5}{8}} (1+|k\lambda|)^{-\frac{1}{4}} |k|^{-\frac{5}{8}+\frac{\beta}{2}}\|F\|_{L^2},
 \end{aligned}
\end{equation}
where we have used $|\lambda-1|^{\frac{1}{2}}+|\nu/k_1|^{\f14}\lesssim 1$.\smallskip

\noindent\textbf{Case 2.}  $\lambda\leq 0$.

By \eqref{delta-w}, we have
\begin{align*}
   |c_1|\lesssim& \left\|\frac{\sinh ( |k|(y+1))}{\sinh ( 2|k|)}\right\|_{L^2}\|w_{Na}\|_{L^2} \\\lesssim& \nu^{-\frac{1}{3}}|k_1|^{-\f23}|k|^{-\f12}
   (|\lambda-1|^{\frac{1}{2}}+|\nu/k_1|^{\f14})^{-\frac{2}{3}}\|F\|_{L^2}\\
   \lesssim &\nu^{-\frac{1}{3}}|k_1|^{-\f23}|k|^{-\f12}\big(\nu|k|^3|k_1|^{-1}\big)^{-\f{1}{24}}
   (|\lambda-1|^{\frac{1}{2}}+|\nu/k_1|^{\f14})^{-\frac{5}{8}}\|F\|_{L^2}\\
   = &\nu^{-\frac{3}{8}}|k_1|^{-\f58}|k|^{-\f58}
   (|\lambda-1|^{\frac{1}{2}}+|\nu/k_1|^{\f14})^{-\frac{5}{8}}\|F\|_{L^2}.
\end{align*}
Thanks to $\lambda\leq 0$, we have $|\lambda-1|^{\frac{1}{2}}+|\nu/k_1|^{\f14} \gtrsim |k|^{-\f12}(|k\lambda|+1)^{\f12}$ and $|\lambda-1|^{\frac{1}{2}}+|\nu/k_1|^{\f14} \geq 1$, and then
\begin{align}\label{est:c1-caseII-00}
   |c_1|&\lesssim  \nu^{-\frac{3}{8}}|k_1|^{-\f58}|k|^{-\f58}
   \big(|k|^{-\f12}(|k\lambda|+1)^{\f12}\big)^{-\f12}\|F\|_{L^2} \\&=\nu^{-\frac{3}{8}}|k_1|^{-\f58}|k|^{-\f38}(|k\lambda|+1)^{-\f14}\|F\|_{L^2}.\nonumber
\end{align}

\noindent\textbf{Case 3.} $\lambda\in (0,1)$.

{\bf Case 3.1. $4|1-\lambda|^{\frac{1}{2}}\leq |\nu/k_1|^{\frac{1}{4}}.$}

In this case, it holds that $\nu |k|^3 |k_1|^{-1}\lesssim |1-\lambda|^{\frac{1}{2}}+|\nu/k_1|^{\f14}\lesssim |\nu/k_1|^{\f14}$, i.e., $\nu |k|^4 |k_1|^{-1}\lesssim 1$. Taking the inner product between \eqref{5.3} and $w_{Na}$, and then taking the imaginary part to the resulting equation, we get by \eqref{delta-w} that
\begin{align*}
\|yw_{Na}\|_{L^2}^2+2\|(\partial_y,|k|)\varphi_{Na}\|_{L^2}^2\leq& |k_1|^{-1}\|F\|_{L^2}\|w_{Na}\|_{L^2}+|1-\lambda|\|w_{Na}\|_{L^2}^2\\
\lesssim & \nu^{-\frac{1}{2}}|k_1|^{-\frac{3}{2}}\|F\|_{L^2}^2.
\end{align*}
Due to $\lambda \in (0,1)$, we have $1+|k\lambda|\lesssim |k|$ and
\begin{equation}\label{est:c1-caseIII-00}
\begin{aligned}
|c_1|\leq & \left|\int_{-1}^{1-(100|k|)^{-1+\beta}}\frac{\sinh (|k|(y+1))}{\sinh (2|k|)} w_{Na} (y)\mathrm{d}y\right|\\
&+\left|\int_{1-(100|k|)^{-1+\beta}}^1\frac{\sinh (|k|(y+1))}{\sinh (2|k|)} w_{Na} (y)\mathrm{d}y\right|\\
\lesssim& e^{-|k|^\beta/100}\|w_{Na}\|_{L^1}+\left(\int_{1-(100|k|)^{-1+\beta}}^1\frac{1}{y^2} \mathrm{d}y\right)^{\frac{1}{2}}\|yw_{Na}\|_{L^2}\\
\lesssim&e^{-|k|^\beta/100}\nu^{-\frac{3}{8}}|k_1|^{-\frac{5}{8}}\|F\|_{L^2}
+\nu^{-\frac{1}{4}}|k_1|^{-\frac{3}{4}}|k|^{-\f12-\f{\beta}{2}}\|F\|_{L^2}\\
\lesssim& \nu^{-\frac{3}{8}} |k_1|^{-\frac{5}{8}} (1+ |k\lambda|)^{-\frac{1}{4}} (1+|k\lambda|)^{\frac{1}{4}}e^{-|k|^\beta/100}\|F\|_{L^2}\\
&+ \nu^{-\frac{3}{8}}|k_1|^{-\frac{5}{8}} |k|^{-1-\f{\beta}{2}}\|F\|_{L^2}\\
\lesssim &\nu^{-\frac{3}{8}}|k_1|^{-\frac{5}{8}} (1+|k\lambda|)^{-\frac{1}{4}} |k|^{-1-\f{\beta}{2}}\|F\|_{L^2},
 \end{aligned}
\end{equation}
here we have used $\nu^{-\frac{1}{4}}|k_1|^{-\frac{3}{4}}\lesssim \nu^{-\frac{3}{8}}|k_1|^{-\frac{5}{8}} |k|^{-\f12}$.\smallskip

{\bf Case 3.2. $4|1-\lambda|^{\frac{1}{2}}\geq  |\nu/k_1|^{\frac{1}{4}}$.}

Note that $|1-\lambda|^{\frac{1}{2}}\geq |\nu/k_1|^{\frac{1}{4}}/4$. We get by \eqref{delta-w} that
\begin{equation}\nonumber
\begin{aligned}
&\|yw_{Na}\|_{L^2}^2+2\|(\partial_y,|k|)\varphi_{Na}\|_{L^2}^2\leq |k_1|^{-1}\|F\|_{L^2}\|w_{Na}\|_{L^2}+|1-\lambda|\|w_{Na}\|_{L^2}^2\\
&\quad\lesssim \nu^{-\frac{1}{3}}|k_1|^{-\frac{5}{3}}|1-\lambda|^{-\frac{1}{3}}\|F\|_{L^2}^2+\nu^{-\frac{2}{3}}|k_1|^{-\frac{4}{3}}|1-\lambda|^{\frac{1}{3}}\|F\|_{L^2}^2\\
&\quad \lesssim \nu^{-\frac{2}{3}}|k_1|^{-\frac{4}{3}}|1-\lambda|^{\frac{1}{3}}\|F\|_{L^2}^2,
 \end{aligned}
\end{equation}
Thus, we obtain
\begin{equation}\nonumber
\begin{aligned}
\|w_{Na}\|_{L^1}\leq &\left\|(y+|1-\lambda|^{1/2})w_{Na}\right\|_{L^2} \left\|(y+|1-\lambda|^{1/2})^{-1}\right\|_{L^2}\\
\lesssim& \nu^{-\frac{1}{3}}|k_1|^{-\frac{2}{3}}|1-\lambda|^{-\frac{1}{12}}\|F\|_{L^2},
 \end{aligned}
\end{equation}
here we have used $ \|yw_{Na}\|_{L^2}+|1-\lambda|^{\f12}\|w_{Na}\|_{L^2}\lesssim  \nu^{-\frac{1}{3}}|k_1|^{-\frac{2}{3}}|1-\lambda|^{\frac{1}{6}}\|F\|_{L^2}$.

Therefore,  we deduce that
\begin{equation}\nonumber
\begin{aligned}
|c_1|\leq& \left|\int_{-1}^{1-(100|k|)^{-1+\beta}}\frac{\sinh (|k|(y+1))}{\sinh (2|k|)} w_{Na} (y)\mathrm{d}y\right|\\
&+\left|\int_{1-(100|k|)^{-1+\beta}}^1\frac{\sinh (|k|(y+1))}{\sinh (2|k|)} w_{Na} (y)\mathrm{d}y\right|\\
\lesssim& e^{-|k|^\beta/100}\|w_{Na}\|_{L^1}+\left(\int_{1-(100|k|)^{-1+\beta}}^1\frac{1}{y^2} \mathrm{d}y\right)^{\frac{1}{2}}\|yw_{Na}\|_{L^2}\\
\lesssim&e^{-|k|^\beta/100} \nu^{-\frac{1}{3}}|k_1|^{-\frac{2}{3}}|1-\lambda|^{-\frac{1}{12}}\|F\|_{L^2}+|k|^{-\frac{1}{2}+\frac{\beta}{2}}
\nu^{-\frac{1}{3}}|k_1|^{-\frac{2}{3}}|1-\lambda|^{\frac{1}{6}}\|F\|_{L^2}\\
\lesssim&e^{-|k|^\beta/100} \nu^{-\frac{3}{8}}|k_1|^{-\frac{5}{8}}\|F\|_{L^2}+
 |k|^{-\frac{1}{2}+\frac{\beta}{2}}
\nu^{-\frac{1}{3}}|k_1|^{-\frac{2}{3}}\|F\|_{L^2}\\
\lesssim&e^{-|k|^\beta/100} \nu^{-\frac{3}{8}}|k_1|^{-\frac{5}{8}}\|F\|_{L^2}+
 |k|^{-\frac{1}{2}+\frac{\beta}{2}}
\nu^{-\frac{1}{3}}|k_1|^{-\frac{2}{3}}(\nu |k|^3|k_1|^{-1})^{-\f{1}{24}}\\
&\times\big(|1-\lambda|^{\f12}+|\nu/k_1|^{\f14}\big)^{\f{1}{24}}\|F\|_{L^2}\\
\lesssim& \nu^{-\frac{3}{8}} |k_1|^{-\frac{5}{8}}  |k|^{-\frac{5}{8}+\frac{\beta}{2}}\|F\|_{L^2},
 \end{aligned}
\end{equation}
here we have used $\nu|k|^3|k_1|^{-1}\leq 2\big(|1-\lambda|^{\f12}+|\nu/k_1|^{\f14}\big)\lesssim 1$.
Thanks to $1+|k\lambda|\lesssim |k|$, we get
\begin{align}\label{est:c1-caseiv-00}
  &|c_1|\lesssim  \nu^{-\frac{3}{8}} |k_1|^{-\frac{5}{8}}  |k|^{-\frac{3}{8}+\frac{\beta}{2}}(1+|k\lambda|)^{-\f14}\|F\|_{L^2}.
 \end{align}

Finally, the desired bound \eqref{ciL2} follows from \eqref{est:c1-casei-00}, \eqref{est:c1-caseii-00}, \eqref{est:c1-caseII-00}, \eqref{est:c1-caseIII-00} and \eqref{est:c1-caseiv-00}.
\end{proof}

\begin{lemma}\label{ci:-H1}
Assume that $\nu |k|^3|k_1|^{-1}\leq 2\big(|1-\lambda|^{\frac{1}{2}}+|\nu/k_1|^{\f14}\big)$ and $F\in H^1_k(I)$. Then we have
\begin{equation}\label{ciH1-2}
\begin{aligned}
 &|c_1|+|c_2|\le C\nu^{-\frac{1}{8}}|k_1|^{-\frac{7}{8}}
 \big(1+|\lambda|+|\nu/k_1|^{\frac{1}{3}}\big)^{-\frac{1}{4}}\|(\partial_y,|k|)F\|_{L^2}.
\end{aligned}
\end{equation}
\end{lemma}

\begin{proof} {\bf Case 1}. $|\lambda|\geq 4.$

Taking the inner product between \eqref{5.3} and $w_{Na}$, and then taking the imaginary part to the resulting equation, we get by Proposition \ref{resolvent-H1} that
\begin{align*}
   &\big|\lambda\big| \|w_{Na}\|_{L^2}^2\\
   &\leq  \left\|(1-y^2)^{1/2}w_{Na}\right\|_{L^2}^2+ 2\|(\partial_y,|k|)\varphi_{Na}\|_{L^2}^2\\
   &\quad+|k_1|^{-1}\|(\partial_y,|k|)F\|_{L^2}\|(\partial_y,|k|)\varphi_{Na}\|_{L^2}\\
   &\leq \|w_{Na}\|_{L^2}^2+C|k_1|^{-2}\|(\partial_y,|k|)F\|_{L^2}^2,
\end{align*}
which gives
\begin{align*}
   &|\lambda|^{\f12}\|w_{Na}\|_{L^2}\lesssim |k_1|^{-1}\|(\partial_y,|k|)F\|_{L^2}.
\end{align*}
Thanks to $|\lambda|\geq 4$,  we have$\big(1+|\lambda|+|\nu/k_1|^{\f13}\big)^{\f12}\|w_{Na}\|_{L^2}\lesssim |k_1|^{-1}\|(\partial_y,|k|)F\|_{L^2}$. Therefore, we have
\begin{align*}
   |c_1|+|c_2| \leq& 2\|w_{Na}\|_{L^1}\lesssim \|w_{Na}\|_{L^2}\lesssim |k_1|^{-1}\big(1+|\lambda|+|\nu/k_1|^{\f13}\big)^{-\f12}\|(\partial_y,|k|)F\|_{L^2}\\
   \lesssim &\nu^{-\frac{1}{8}}|k_1|^{-\frac{7}{8}}
 \big(1+|\lambda|+|\nu/k_1|^{\frac{1}{3}}\big)^{-\frac{1}{4}}\|(\partial_y,|k|)F\|_{L^2}.
 \end{align*}

{\bf Case 2}. $|\lambda|\leq  4$.

Thanks to $1\lesssim (1+|\lambda|+|\nu/k_1|^{\frac{1}{3}})^{-\frac{1}{4}}$ and $c_1=\pa_y\varphi_{Na}(1), c_2=-\pa_y\varphi_{Na}(-1)$, we get by Proposition \ref{resolvent-H1} that
\begin{equation}\nonumber
\begin{aligned}
|c_1|+|c_2|\lesssim& \|(\partial_y,|k|)\varphi_{Na}\|_{L^\infty}\\
\lesssim& \big(1+|\lambda|+|\nu/k_1|^{\frac{1}{3}}\big)^{-\frac{1}{4}} \nu^{-\frac{1}{8}}|k_1|^{-\frac{7}{8}}\|(\partial_y,|k|)F\|_{L^2}.
 \end{aligned}
\end{equation}

Combining the above estimates, the conclusion follows.
\end{proof}

\begin{lemma}\label{ci:H-1}
Assume that $\nu |k|^3|k_1|^{-1}\leq 2\big(|1-\lambda|^{\frac{1}{2}}+|\nu/k_1|^{\f14}\big)$ and  $F\in H_k^{-1}(I)$. Then we have
\begin{equation}\label{ciH-1}
\begin{aligned}
|c_1|+|c_2| \le C|\nu k_1|^{-\f12}\big(|\lambda|+|\nu/k_1|^{\f13}\big)^{-\f14}\|F\|_{H^{-1}_k}.
\end{aligned}
\end{equation}
\end{lemma}

\begin{proof}
We only give the estimate for $c_1$, and the estimate for $c_2$ is similar.
We denote $\delta=|\nu/k_1|^{\f14}$ and  $\delta_1=\delta^{\frac{4}{3}}(|\lambda-1|^{\frac{1}{2}}+\delta)^{\frac{2}{3}}=|\nu/k_1|^{\f13} \big(|1-\lambda|^{\f12}+|\nu/k_1|^{\f14}\big)^{\f23}$.

By Proposition \ref{weaktype} with $f=\frac{\sinh (|k|(1+y))}{\sinh( 2|k|)}$\big($f(-1)=0,\ f(1)=1,$ and $\|f\|_{L^2}\lesssim |k|^{-\frac{1}{2}},\ \|f\|_{H^{1}_k}:=\|(\partial_y,|k|)f\|_{L^2}\lesssim |k|^{\frac{1}{2}}$\big), we have
\begin{equation}\label{est:c1-prior-FH-1-00}
  \begin{aligned}
||c_1|&=\left|\left\langle w,\frac{\sinh (|k|(1+y))}{\sinh (2|k|)}\right\rangle\right|\\
&\lesssim |k_1|^{-1}\|F\|_{H^{-1}_k}\big(|\lambda|+\delta^{\frac{4}{3}}\big)^{-\frac{3}{4}}\delta^{-1}
+|k_1|^{-1}\|F\|_{H^{-1}_k}\\
&\qquad\times(\|\mathbf{Ray}_{\delta_1}^{-1}f\|_{H^1}
+\delta^{-\frac{4}{3}}\big(|\lambda-1|^{\frac{1}{2}}+\delta)^{\frac{1}{3}}\|\mathbf{Ray}_{\delta_1}^{-1}f\|_{L^2}\big)\\
&\lesssim |k_1|^{-1}\|F\|_{H^{-1}_k}\big(|\lambda|+\delta^{\frac{4}{3}}\big)^{-\frac{1}{4}} \delta^{-\f53}+|k_1|^{-1}\|F\|_{H^{-1}_k}\\
&\qquad\times(\|\mathbf{Ray}_{\delta_1}^{-1}f\|_{H^1}
+\delta^{-\frac{4}{3}}\big(|\lambda-1|^{\frac{1}{2}}+\delta)^{\frac{1}{3}} \|\mathbf{Ray}_{\delta_1}^{-1}f\|_{L^2}\big).\\
&\lesssim |\nu k_1|^{-\f12}\big(|\lambda|+\delta^{\frac{4}{3}}\big)^{-\frac{1}{4}} \|F\|_{H^{-1}_k} +|k_1|^{-1}\|F\|_{H^{-1}_k}\\
&\qquad\times(\|\mathbf{Ray}_{\delta_1}^{-1}f\|_{H^1}
+\delta^{-\frac{4}{3}}\big(|\lambda-1|^{\frac{1}{2}}+\delta)^{\frac{1}{3}}\|\mathbf{Ray}_{\delta_1}^{-1}f\|_{L^2}\big).
\end{aligned}
\end{equation}

It remains to bound the following two terms
$$\left\|\mathbf{Ray}_{\delta_1}^{-1}f\right\|_{L^2} \quad \text{and}\quad \left\|\mathbf{Ray}_{\delta_1}^{-1} f \right\|_{H^1}.$$
As $f(1)=1\neq 0$, we can not apply Lemma \ref{lem:Ray-simple-1} directly.  Recall that
$$(1-y^2-\lambda+i\delta_1)W+2\Phi=f, \ (\partial_y^2-|k|^2)\Phi=W, \ \Phi(\pm 1)=0.$$

\noindent{\bf Case 1.} $|\lambda|\geq 4 $.

In this case, it holds that $\delta_1=\delta^{\frac{4}{3}}(|\lambda-1|^{\frac{1}{2}}+\delta)^{\frac{2}{3}}\approx \delta^{\f43}|\lambda|^{\f13}\approx \delta^{\f43}\big(|\lambda|+\delta^{\f43}\big)^{\f13}$ and $(|\lambda-1|^{\frac{1}{2}}+\delta)^{\frac{1}{3}}\approx \big(|\lambda|+\delta^{\f43}\big)^{\f16}$.
By the energy estimate, we have
\begin{align*}
   \big|\lambda-i\delta_1\big| \|W\|_{L^2}^2\leq & \left\|(1-y^2)^{1/2}W\right\|_{L^2}^2+ 2\|(\partial_y,|k|)\Phi\|_{L^2}^2+\|f\|_{L^2}\|W\|_{L^2}\\
   \leq &\|W\|_{L^2}^2+2|k|^{-2}\|W\|_{L^2}+\|f\|_{L^2}\|W\|_{L^2}.
\end{align*}
Due to $|\lambda|\geq 4$ and $|k|\geq 1$, we get
\begin{align*}
   & (|\lambda|+\delta_1) \|W\|_{L^2}\lesssim \|f\|_{L^2}\lesssim |k|^{-\f12},
\end{align*}
which yields that
\begin{align*}
   & \|\mathbf{Ray}_{\delta_1}^{-1}f\|_{L^2}=\|W\|_{L^2}\lesssim \big(|\lambda|+\delta^{\f43}\big)^{-1}|k|^{-\f12}.
\end{align*}
Therefore, we have
\begin{equation}\label{est:WL2-lambdabig-00}
  \begin{aligned}
&|k_1|^{-1}\delta^{-\frac{4}{3}}
    (|\lambda-1|^{\frac{1}{2}}+\delta)^{\frac{1}{3}}\|\mathbf{Ray}_{\delta_1}^{-1}f\|_{L^2}\\ \lesssim &|k_1|^{-1}\delta^{-\frac{4}{3}}\big(|\lambda|+\delta^{\f43}\big)^{\f16} \cdot  \big(|\lambda|+\delta^{\f43}\big)^{-1}|k|^{-\f12}
    \lesssim |\nu k_1|^{-\f12}\big(|\lambda|+\delta^{\f43}\big)^{-\f14}.
 \end{aligned}
\end{equation}

Note that $(1-y^2-\lambda+i\delta_1)W'+2\Phi'=f'+2yW$. Then the energy estimate gives
\begin{align*}
   &\big|\lambda-i\delta_1\big| \|W'\|_{L^2}^2\\
   &\leq \left\|(1-y^2)^{1/2}W'\right\|_{L^2}^2+ 2\|\partial_y\Phi\|_{L^2}\|W'\|_{L^2}+\|f'\|_{L^2}\|W'\|_{L^2} +2\|yW\|_{L^2}\|W'\|_{L^2}\\
   &\leq \|W\|_{L^2}^2+ 2\|\partial_y\Phi\|_{L^2}\|W'\|_{L^2}+\|f'\|_{L^2}\|W'\|_{L^2} +2\|yW\|_{L^2}\|W'\|_{L^2}.
\end{align*}
Due to $|\lambda|\geq 4$ and $|k|\geq 1$,  we get
\begin{align*}
   \big|\lambda-i\delta_1\big| \|W'\|_{L^2}\lesssim & \|\partial_y\Phi\|_{L^2}+\|f'\|_{L^2}+\|yW\|_{L^2}\\
   \lesssim& |k|^{-1}\|W\|_{L^2}+|k|^{\f12}+\|W\|_{L^2}\lesssim |k|^{\f12}.
\end{align*}
which yields that
\begin{align*}
   & \|\mathbf{Ray}_{\delta_1}^{-1}f\|_{H^1}\lesssim\|W'\|_{L^2}+\|W\|_{L^2}\lesssim \big(|\lambda|+\delta^{\f43}\big)^{-1}|k|^{\f12}.
\end{align*}
Thus, we obtain
\begin{equation}\label{est:WH1-lambdabig-00}
  \begin{aligned}
&|k_1|^{-1}\|\mathbf{Ray}_{\delta_1}^{-1}f\|_{H^1}\\
 &\lesssim |k_1|^{-1}\cdot  \big(|\lambda|+\delta^{\f43}\big)^{-1}|k|^{\f12}\\
&\lesssim |k_1|^{-1}  \big(|\lambda|+\delta^{\f43}\big)^{-1}|k|^{\f12}\cdot(\nu|k|^3|k_1|^{-1})^{-\f16} (|1-\lambda|^{\f12}+\delta)^{\f16}\\
&\lesssim |k_1|^{-1}  \big(|\lambda|+\delta^{\f43}\big)^{-1}|\nu/k_1|^{-\f16} \big(|\lambda|+\delta^{\f43}\big)^{\f{1}{12}}\\
    &\lesssim |\nu k_1|^{-\f12}\big(|\lambda|+\delta^{\f43}\big)^{-\f14}.
 \end{aligned}
\end{equation}

Thus, the desired conclusion follows from \eqref{est:c1-prior-FH-1-00}, \eqref{est:WL2-lambdabig-00} and \eqref{est:WH1-lambdabig-00}.\smallskip

\noindent{\bf Case 2.} $|\lambda|\leq 4 $.\smallskip

{\textbf{Step 1}. Estimate of $\|(\partial_y,|k|)\Phi\|_{L^2}$}.

Due to $\delta_1=\delta^{\frac{4}{3}}(|\lambda-1|^{\frac{1}{2}}+\delta)^{\frac{2}{3}}\ll 1$, we get by Proposition \ref{est-2} that
\begin{align}\label{est:Phi-k-12-00}
\|(\partial_y,|k|)\Phi\|_{L^2}\lesssim |k|^{-\frac{1}{2}}.
\end{align}

{\textbf{Step 2}. Estimate of $\|W\|_{L^2}$}.

For $f(-1)=0$ and $f(1)\not=0$, let us introduce the decomposition $W=W_1+W_2$ and $(\partial_y^2-|k|^2)\Psi_i=W_i,\ \Psi_i(\pm 1)=0 \ (i=1,2)$ with
$$W_1=\frac{f(1)\rho (y)}{1-y^2-\lambda+i\delta_1}, \quad \rho(y)=\max\{1-4|k|(1-y),0\}.$$
Then we have
$$(1-y^2-\lambda+i\delta_1)W_2+2\Psi_2=\tilde{f},\quad \Psi_2(\pm 1)=0, $$
where $\tilde{f}=f-f(1)\rho(y)-2\Psi_1$ and $\tilde{f}(\pm 1)=0$. It is easy to check that
\begin{equation}\nonumber
\begin{aligned}
\|W_1\|_{L^2}\lesssim &|f(1)|\left\|\frac{\rho}{1-y^2-\lambda+i\delta_1}\right\|_{L^2}\\
\lesssim& |f(1)|\left\|\frac{y}{1-y^2-\lambda+i\delta_1}\right\|_{L^2}
\lesssim \delta_1^{-\frac{1}{2}}|f(1)|\lesssim\delta_1^{-\frac{1}{2}}.
\end{aligned}
\end{equation}
and by Lemma \ref{hardy-type-abp},
\begin{equation}\nonumber
\begin{aligned}
\|(\partial_y,|k|)\Psi_1\|_{L^2}^2=&\langle W_1,-\Psi_1\rangle=-\int_{-1}^1\frac{f(1)
\rho(y)\overline{\Psi}_1(y)}{1-y^2-\lambda+i\delta_1}\mathrm{d}y\\
\leq & \left|\int_{1-(4|k|)^{-1}}^1\frac{f(1)
\rho(y)\overline{\Psi}_1(y)}{y^2-(1-\lambda+i\delta_1)}\mathrm{d}y\right|\\
 \lesssim&|f(1)| c_0^{-1}|k|^{-\frac{1}{2}}(\|\partial_y (\rho \Psi_1)\|_{L^2(1-(4|k|)^{-1},1)}\\
 &+|k|\|\rho \Psi_1\|_{L^2(1-(4|k|)^{-1},1)}+c_0^{-1}\|\rho \Psi_1\|_{L^2(1-(4|k|)^{-1},1)})\\
 \lesssim&|f(1)| |k|^{-\frac{1}{2}}\big(\|\partial_y \rho \|_{L^2(1-(4|k|)^{-1},1)}\|\Psi_1\|_{L^\infty}\\
 &+\|\partial_y \Psi_1 \|_{L^2}+|k|^{\frac{1}{2}} \| \Psi_1\|_{L^\infty}\big)\\
\lesssim&  |f(1)||k|^{-\frac{1}{2}}\|(\partial_y,|k|)\Psi_1\|_{L^2}\lesssim |k|^{-\frac{1}{2}} \|(\partial_y,|k|)\Psi_1\|_{L^2},
\end{aligned}
\end{equation}
 Therefore, we obtain
\begin{align}\label{est:W1-Psi1-00}
  & \delta_1^{\frac{1}{2}}\|W_1\|_{L^2}+|k|^{\frac{1}{2}}\|(\partial_y,|k|)\Psi_1\|_{L^2} \lesssim 1.
\end{align}

By \eqref{est:Phi-k-12-00} and \eqref{est:W1-Psi1-00}, we have
\begin{align*}
   \|(\partial_y,|k|)\Psi_2\|_{L^2}&\leq \|(\partial_y,|k|)\Phi\|_{L^2}+ \|(\partial_y,|k|)\Psi_1\|_{L^2}\lesssim |k|^{-\frac{1}{2}}.
\end{align*}
Thanks to $\|(\partial_y,|k|)\tilde{f}\|_{L^2}\lesssim \|(\partial_y,|k|)f\|_{L^2}+\|(\partial_y,|k|)\rho\|_{L^2} +\|(\partial_y,|k|)f\|_{L^2}\lesssim |k|^{\f12}$, we obtain
\begin{equation}\nonumber
\begin{aligned}
\delta_1\|W_2\|_{L^2}^2 \lesssim \big|\langle\tilde{f},W_2 \rangle\big|\lesssim   \|(\partial_y,|k|)\tilde{f}\|_{L^2} \|(\partial_y,|k|)\Psi_2\|_{L^2}
\lesssim 1,
\end{aligned}
\end{equation}
which along with \eqref{est:W1-Psi1-00} gives
\begin{equation}\label{est:WL2-k12-001}
\begin{aligned}
\|\mathbf{Ray}_{\delta_1}^{-1}f\|_{L^2}=\|W\|_{L^2}\leq  \|W_1\|_{L^2}+\|W_2\|_{L^2} \lesssim \delta_1^{-\frac{1}{2}} .
\end{aligned}
\end{equation}
Due to $1\lesssim \big(|\lambda|+\delta^{\f43}\big)^{-\f14}$, we have
\begin{align}\label{est:WL2-k12-00}
|k_1|^{-1}\delta^{-\frac{4}{3}}
(|\lambda-1|^{\frac{1}{2}}+\delta)^{\frac{1}{3}}\|\mathbf{Ray}_{\delta_1}^{-1}f\|_{L^2}
&\lesssim  |k_1|^{-1}\delta^{-\frac{4}{3}}(|\lambda-1|^{\frac{1}{2}}+\delta)^{\frac{1}{3}}\delta_1^{-\frac{1}{2}}\nonumber\\
&\lesssim |k_1|^{-1}\delta^{-2}\lesssim |\nu k_1|^{-\f12}\big(|\lambda|+\delta^{\f43}\big)^{-\f14}.
\end{align}

{\textbf{Step 3}. Estimate of $\|W'\|_{L^2}$}.

Thanks to $f'\leq 4|k|e^{|k|(1+y)}/(e^{2|k|}-e^{-2|k|})\lesssim |k|e^{-|k|(1-y)}$, we obtain
\begin{align*}
   &\left\|\frac{f'}{1-y^2-\lambda+i\delta_1}\right\|_{L^2}\\
   &\leq \left\|\frac{f'}{1-y^2-\lambda+i\delta_1}\right\|_{L^2((-1/2,1/2))}+
   \left\|\frac{f'}{1-y^2-\lambda+i\delta_1}\right\|_{L^2((-1,1)\setminus (-1/2,1/2))}\\
&\lesssim \delta_1^{-1}|k|\left\|e^{-|k|(1-y)}\right\|_{L^2((-1/2,1/2))}+
   |k|\left\|\frac{y}{1-y^2-\lambda+i\delta_1}\right\|_{L^2((-1,1)\setminus (-1/2,1/2))}\\
  &\lesssim \delta_1^{-1}|k|^{-\f12}+\delta_1^{-\f12}|k|.
\end{align*}

Notice that
$$W'=\frac{f'-2\Phi'}{1-y^2-\lambda+i\delta_1}+\frac{2yW}{1-y^2-\lambda+i\delta_1},$$
then by \eqref{est:Phi-k-12-00} and \eqref{est:WL2-k12-001}, we get
\begin{align*}
&\|\partial_y(\mathbf{Ray}_{\delta_1}^{-1}f)\|_{L^2}=\|W' \|_{L^2}\\
 &\leq \left\|\frac{f'-2\Phi'}{1-y^2-\lambda+i\delta_1}\right\|_{L^2}+
\|W\|_{L^2}\left\|\frac{2y}{1-y^2-\lambda+i\delta_1}\right\|_{L^\infty}\\
&\lesssim \delta_1^{-1}|k|^{-\frac{1}{2}}+\delta_1^{-\frac{1}{2}}|k|+\delta_1^{-\frac{3}{2}}(|1-\lambda|+\delta_1)^{\frac{1}{2}}
\end{align*}
As $\nu|k|^3|k_1|^{-1}\leq 2\big(|1-\lambda|^{\f12}+\delta\big)$, we have $|k|\lesssim \delta^{-\f43}\big(|1-\lambda|^{\f12}+\delta\big)^{\f13}$, and then
\begin{equation}\nonumber
  \begin{aligned}
  &|k_1|^{-1}\|\mathbf{Ray}_{\delta_1}^{-1}f\|_{H^1}\\
&\lesssim |k_1|^{-1}\Big( \delta_1^{-1}|k|^{-\frac{1}{2}}+\delta_1^{-\frac{1}{2}}|k|+
\delta_1^{-\frac{3}{2}}(|1-\lambda|+\delta_1)^{\frac{1}{2}}+\delta_1^{-\f12}\Big)\\
&\lesssim |k_1|^{-1}\Big(\delta_1^{-1}|k|^{-\frac{1}{2}}+\delta_1^{-\frac{1}{2}}
\delta^{-\frac{4}{3}}(|\lambda-1|^{\frac{1}{2}}+\delta)^{\frac{1}{3}}
+\delta_1^{-\frac{3}{2}}(|1-\lambda|^{\frac{1}{2}}+\delta)+\delta_1^{-\f12}\Big)\\
&\lesssim |k_1|^{-1} \delta^{-2} \lesssim |\nu k_1|^{-\f12}\big(|\lambda|+\delta^{\f43}\big)^{-\f14},
  \end{aligned}
\end{equation}
where we have used $\delta_1^{-1}\lesssim \delta^{-2}$ and $1\lesssim \big(|\lambda|+\delta^{\f43}\big)^{-\f14}$. Then the desired conclusion follows from \eqref{est:c1-prior-FH-1-00} and \eqref{est:WL2-k12-00}.
\end{proof}

\subsection{Resolvent estimates}

Now we are in a position to establish the resolvent estimates for the Orr-Sommerfeld equation with non-slip boundary condition.

\begin{proposition}\label{resolvent-noslip}
Let $(w,\varphi)$ be the solution to \eqref{5.1}. It holds that

(1) If $\nu |k|^3|k_1|^{-1}\leq 2\big(|1-\lambda|^{\frac{1}{2}}+|\nu/k_1|^{\f14}\big)$, then we have
\begin{align}
&\nu^{\frac{5}{8}} |k_1|^{\frac{3}{8}} |k|^{\frac{1}{2}} \|w\|_{L^2}+\nu^{\f38}|k_1|^{\f58}\|u\|_{L^\infty}
+\nu^{\f14}
|k_1|^{\f34}\|u\|_{L^2}\lesssim\|F\|_{L^2},\label{estimate-noslipL21}\\
 &\nu^{\frac{3}{4}}|k_1|^{\frac{1}{4}} \|w\|_{L^2}+\nu^{\frac{5}{8}}|k_1|^{\frac{3}{8}} \|u\|_{L^\infty}+\nu^{\frac{1}{2}}|k_1|^{\frac{1}{2}}
 \|u\|_{L^2}\lesssim\|F\|_{H^{-1}_k},\label{estimate-noslipH-11}\\
&|\nu/k_1|^{\f38}\|w\|_{L^2} +|\nu/k_{1}|^{\f18}\|u\|_{L^\infty}+\|u\|_{L^2}\lesssim |k_1|^{-1}\|(\partial_y,|k|)F\|_{L^2}.\label{estimate-noslipH11-uL2}
\end{align}

(2) If $\nu |k|^3|k_1|^{-1}\geq 2\big(|1-\lambda|^{\frac{1}{2}}+|\nu/k_1|^{\frac{1}{4}}\big)$, then we have
\begin{align}
&\nu^{\frac{5}{8}} |k_1|^{\frac{3}{8}} |k|^{\frac{1}{2}} \|w\|_{L^2}+\nu^{\f38}|k_1|^{\f58}\|u\|_{L^\infty}+
\nu^{\f14}|k_1|^{\f34}\|u\|_{L^2}\nonumber\\
&\quad\lesssim \nu|k|^2\|w\|_{L^2}\lesssim\|F\|_{L^2},\label{estimate-noslipL2-2}\\
&\nu^{\frac{3}{4}}|k_1|^{\frac{1}{4}}\|w\|_{L^2}+\nu^{\frac{5}{8}}|k_1|^{\frac{3}{8}} \|u\|_{L^\infty}+
\nu^{\frac{1}{2}}|k_1|^{\frac{1}{2}}\|u\|_{L^2}\nonumber\\
&\quad\lesssim \nu|k|\|w\|_{L^2}\lesssim \|F\|_{H^{-1}_k},\label{estimate-noslipH-1-2}\\
  &\nu^{\f38}|k_1|^{\f{5}{8}}\|w\|_{L^2}+\nu^{\f18}|k_{1}|^{\f78}\|u\|_{L^\infty}+|k_1|\|u\|_{L^2}\nonumber\\
  &\quad\lesssim\nu |k|^3 \|w\|_{L^2}\lesssim\|(\partial_y,|k|)F\|_{L^2}.\label{estimate-noslipL2-H1}
\end{align}

\end{proposition}

\begin{proof}
\textbf{Case 1.} $\nu |k|^3|k_1|^{-1}\leq 2\big(|1-\lambda|^{\frac{1}{2}}+|\nu/k_1|^{\f14}\big)$. \smallskip

In this case, we have
\begin{align*}
   & \big(\nu|k|^3|k_1|^{-1}\big)^{\f23} \lesssim |1-\lambda|^{\f13}+|\nu/k_1|^{\f16}\lesssim 1+|\lambda|^{\f13}\lesssim 1+|\lambda|\lesssim   1+|\lambda||k_1/\nu|^{\f13}.
\end{align*}
Recall that $d=-\f{\lambda}{2}-\f{i\nu|k|^2}{2k_1}+\f{i\epsilon|\nu/k_1|^{1/2}}{2}$, $L=|2k_1/\nu|^{\f13}$. Then we have
\begin{align}\label{eq:1+Ld-sim-00}
    1+|Ld|&\approx 1+|\lambda||k_1/\nu|^{\f13}+\big(\nu|k|^3|k_1|^{-1}\big)^{\f23}+ |\nu/k_1|^{\f16}\nonumber\\
   &\approx 1+|\lambda||k_1/\nu|^{\f13}.
\end{align}

For $F\in L^2$ and $i=1,2$, it follows from Proposition \ref{estimate-w-i} and \eqref{eq:1+Ld-sim-00} that
\begin{align}
&\|w_i\|_{L^2}\lesssim L^{\frac{1}{2}}(1+|Ld|)^{\frac{1}{4}}\approx |\nu/k_1|^{-\f14}\big((\nu |k|^3|k_1|^{-1})^{\f13}+|k\lambda|\big)^{\f14}|k|^{-\f14}\label{eq:wi-FL2-est-00}\\ &\qquad\quad\lesssim |\nu/k_1|^{-\frac{1}{4}}(1+|k\lambda|)^{\frac{1}{4}}|k|^{-\f14},\nonumber\\
&\|(\partial_y,|k|)\psi_i\|_{L^2}+|\nu/k_1|^{\f16}\|(\partial_y,|k|)\psi_i\|_{L^\infty}\lesssim |\nu/k_1|^{\f16},\label{eq:parpsi-FL2-est-00}
\end{align}
where we have used $(\nu |k|^3|k_1|^{-1})^{\f13}\lesssim |1-\lambda|^{\f16}+|\nu/k_1|^{\f{1}{12}}\lesssim 1+|k\lambda|$.
Therefore, by Lemma \ref{coe-L2}, we get
\begin{equation}\nonumber
\begin{aligned}
&|c_1|\|w_1\|_{L^2}+|c_2|\|w_2\|_{L^2}\\
&\lesssim  \nu^{-\frac{3}{8}}|k_1|^{-\frac{5}{8}}|k|^{-\frac{1}{4}}
(1+|k\lambda|)^{-\frac{1}{4}}|\nu/k_1|^{-\frac{1}{4}}
(1+|k\lambda|)^{\frac{1}{4}}|k|^{-\f14}\|F\|_{L^2}\\
&=  \nu^{-\frac{5}{8}}|k_1|^{-\f38} |k|^{-\frac{1}{2}}\|F\|_{L^2}.
\end{aligned}
\end{equation}
By Corollary \ref{resol-lambda}, we have
\begin{align*}
   & \|w_{Na}\|_{L^2}\lesssim \nu^{-\frac{1}{3}}|k_1|^{-\frac{2}{3}}\big(|\lambda-1|^{\frac{1}{2}}+|\nu/k_1|^{\f14}\big)^{-\frac{2}{3}}\|F\|_{L^2}.
\end{align*}
Notice that
\begin{align*}
  \nu^{-\frac{1}{3}}|k_1|^{-\frac{2}{3}}\big(|\lambda-1|^{\frac{1}{2}}+|\nu/k_1|^{\f14}\big)^{-\frac{2}{3}} &\lesssim \nu^{-\frac{1}{3}}|k_1|^{-\frac{2}{3}}
  \big(\nu |k|^3|k_1|^{-1}\big)^{-\f16}\big(|\nu/k_1|^{\f14}\big)^{-\frac{1}{2}} \\&=\nu^{-\f58}|k_1|^{-\f38}|k|^{-\f12}.
\end{align*}
Then we obtain
\begin{equation}\label{est:w-nonslip-FL2-00}
\begin{aligned}
\|w\|_{L^2}\leq& \|w_{Na}\|_{L^2}+|c_1|\|w_1\|_{L^2}+|c_2|\|w_2\|_{L^2}\\
\lesssim& \nu^{-\frac{5}{8}}|k_1|^{-\f38} |k|^{-\frac{1}{2}}\|F\|_{L^2}.
\end{aligned}
\end{equation}

By Corollary \ref{resol-lambda-00}, Lemma \ref{coe-L2} and \eqref{eq:parpsi-FL2-est-00}, we have
\begin{equation}\label{est:u-nonslip-FL2-00}
\begin{aligned}
\|u\|_{L^2}\leq& \|u_{Na}\|_{L^2}+|c_1|\|(\partial_y,|k|)\psi_1\|_{L^2}+|c_2|\|(\partial_y,|k|)\psi_2\|_{L^2}\\
\lesssim&(\nu^{-\f14}|k_1|^{-\f34}+
\nu^{-\frac{3}{8}}|k_1|^{-\frac{5}{8}}|k|^{-\frac{1}{4}}  |\nu/k_1|^{\f16})\|F\|_{L^2}\\
\lesssim& \nu^{-\frac{1}{4}}|k_1|^{-\f34}\|F\|_{L^2}
\end{aligned}
\end{equation}
and
\begin{equation}\label{est:uLinfty-nonslip-FL2-00}
\begin{aligned}
\|u\|_{L^\infty}\lesssim& \|u_{Na}\|_{L^2}^{\f12}\|w_{Na}\|_{L^2}^{\f12}+|c_1|\|(\partial_y,|k|)\psi_1\|_{L^\infty}\\
&+ |c_2|\|(\partial_y,|k|)\psi_2\|_{L^\infty}\\
\lesssim&(\nu^{-\f38}|k_1|^{-\f58}+
\nu^{-\frac{3}{8}}|k_1|^{-\frac{5}{8}}|k|^{-\frac{1}{4}} )\|F\|_{L^2}
\lesssim \nu^{-\f38}|k_1|^{-\f58}\|F\|_{L^2}.
\end{aligned}
\end{equation}
Then \eqref{estimate-noslipL21} follows form \eqref{est:w-nonslip-FL2-00}, \eqref{est:u-nonslip-FL2-00} and \eqref{est:uLinfty-nonslip-FL2-00}.
\smallskip

For $F\in H^{-1}_k$, it follows from \eqref{eq:wi-FL2-est-00} and Proposition \ref{resolvent-H-1} that
\begin{equation}\nonumber
\begin{aligned}
&\|w_i\|_{L^2}\lesssim L^{\frac{1}{2}}(1+|Ld|)^{\frac{1}{4}}\lesssim |\nu/k_1|^{-\frac{1}{4}}\big(|\nu /k_1|^{\f13}+|\lambda|\big)^{\frac{1}{4}},\\
&\|w_{Na}\|_{L^2}\lesssim \nu^{-\f34}|k_1|^{-\f14}\|F\|_{H^{-1}_k}.
\end{aligned}
\end{equation}
Then by Lemma \ref{ci:H-1}, we get
\begin{equation}\nonumber
\begin{aligned}
\|w\|_{L^2}\leq \|w_{Na}\|_{L^2}+|c_1|\|w_1\|_{L^2}+|c_2|\|w_2\|_{L^2}\lesssim \nu^{-\frac{3}{4}}|k_1|^{-\frac{1}{4}} \|F\|_{H^{-1}_k}.
\end{aligned}
\end{equation}

By Corollary \ref{u-L2-H-1}, Lemma \ref{ci:H-1} and \eqref{eq:parpsi-FL2-est-00}, we obtain
\begin{equation}\nonumber
\begin{aligned}
\|u\|_{L^2}\leq & \|u_{Na}\|_{L^2}+|c_1|\|(\partial_y,|k|)\psi_1\|_{L^2}+|c_2|\|(\partial_y,|k|)\psi_2\|_{L^2}\\
\lesssim &\big(\nu^{-\frac{1}{2}}|k_1|^{-\frac{1}{2}}+|\nu k_1|^{-\f12}\big(|\lambda|+|\nu/k_1|^{\f13}\big)^{-\f14}|\nu/k_1|^{\f16} \big)\|F\|_{H^{-1}_k}\\
\lesssim &\nu^{-\frac{1}{2}}|k_1|^{-\frac{1}{2}}\|F\|_{H^{-1}_k}
\end{aligned}
\end{equation}
and
\begin{equation}\nonumber
\begin{aligned}
\|u\|_{L^\infty}\lesssim & \|u\|_{L^2}^{\f12}\|w\|_{L^2}^{\f12}\lesssim \nu^{-\frac{5}{8}}|k_1|^{-\frac{3}{8}}\|F\|_{H^{-1}_k}.
\end{aligned}
\end{equation}
Hence, \eqref{estimate-noslipH-11} follows.
\smallskip

For $F\in H^1_0$, by Proposition \ref{resolvent-H1}, Lemma \ref{ci:-H1} and \eqref{eq:parpsi-FL2-est-00}, we get
\begin{equation}\nonumber
\begin{aligned}
\|u\|_{L^2}\lesssim&  \|u_{Na}\|_{L^2}+|c_1|\|(\partial_y,|k|)\psi_1\|_{L^2}+|c_2|\|(\partial_y,|k|)\psi_2\|_{L^2}\\
\lesssim &|k_1|^{-1}\|(\partial_y,|k|)F\|_{L^2}\\
&+ (\nu^{-\frac{1}{8}}|k_1|^{-\frac{7}{8}}\big(1+|\lambda|+|\nu /k_1|^{\f13}\big)^{-\frac{1}{4}}\|(\partial_y,|k|)F\|_{L^2}) |\nu/k_1|^{\f16}\\
\lesssim&  |k_1|^{-1}\|(\partial_y,|k|)F\|_{L^2}
\end{aligned}
\end{equation}
and
\begin{equation}\nonumber
\begin{aligned}
\|u\|_{L^\infty}\lesssim&  \|u_{Na}\|_{L^\infty}+|c_1|\|(\partial_y,|k|)\psi_1\|_{L^\infty} +|c_2|\|(\partial_y,|k|)\psi_2\|_{L^\infty}\\
\lesssim &\nu^{-\frac{1}{8}}|k_1|^{-\frac{7}{8}}\|(\partial_y,|k|)F\|_{L^2} \\
&+ \nu^{-\frac{1}{8}}|k_1|^{-\frac{7}{8}}\big(1+|\lambda|+|\nu /k_1|^{\f13}\big)^{-\frac{1}{4}}\|(\partial_y,|k|)F\|_{L^2} \\
\lesssim&  \nu^{-\frac{1}{8}}|k_1|^{-\frac{7}{8}}\|(\partial_y,|k|)F\|_{L^2}.
\end{aligned}
\end{equation}
By Proposition \ref{resolvent-H1}, Lemma \ref{ci:-H1} and \eqref{eq:wi-FL2-est-00}, we have
\begin{equation}\nonumber
\begin{aligned}
\|w\|_{L^2}\lesssim& \|w_{Na}\|_{L^2}+|c_1|\|w_1\|_{L^2}+|c_2|\|w_2\|_{L^2}\\
\lesssim &\nu^{-\f14}|k_1|^{-\f34}\|(\partial_y,|k|)F\|_{L^2}\\
&+ |\nu/k_1|^{-\frac{1}{4}}\big(|\nu /k_1|^{\f13}+|\lambda|\big)^{\frac{1}{4}}\nu^{-\f18}|k_1|^{-\f78}\big(|\nu /k_1|^{\f13}+|\lambda|\big)^{-\frac{1}{4}}\|(\partial_y,|k|)F\|_{L^2}\\
\lesssim &\nu^{-\frac{3}{8}}|k_1|^{-\frac{5}{8}}\|(\partial_y,|k|)F\|_{L^2}.
\end{aligned}
\end{equation}
Hence, \eqref{estimate-noslipH11-uL2} follows.\smallskip

\textbf{Case 2.}  $\nu |k|^3|k_1|^{-1}\geq 2( |1-\lambda|^{\f12}+|\nu/k_1|^{\frac{1}{4}})$. \smallskip

Taking the inner product between \eqref{5.1} and $-\varphi$ and taking the imaginary part to the resulting equation,  we obtain
\begin{equation}\label{y-phi}
\begin{aligned}
&\int_{-1}^1y^2(|\varphi'|^2+|k|^2|\varphi|^2)\mathrm{d}y+\int_{-1}^1|\varphi|^2\mathrm{d}y \\&\leq |k_1|^{-1}|\langle F,-\varphi\rangle|+|1-\lambda| (\|\varphi'\|_{L^2}^2+|k|^2\|\varphi\|_{L^2}^2).
\end{aligned}
\end{equation}
Taking the real part, we obtain
\begin{equation}\label{w-phi}
\begin{aligned}
\nu&(\|\varphi''\|_{L^2}^2+2|k|^2\|\varphi'\|_{L^2}^2+|k|^4\|\varphi\|_{L^2}^2)\\
\leq &|\langle F,-\varphi\rangle|+\epsilon \nu^{\frac{1}{2}}|k_1|^{\frac{1}{2}}\|u\|_{L^2}^2+2|k_1|\|y\varphi'\|_{L^2}\|\varphi\|_{L^2}\\
\leq &|\langle F,-\varphi\rangle|+\epsilon C \nu |k|^{2}(\|\varphi'\|_{L^2}^2+|k|^2\|\varphi\|_{L^2}^2)+
|k_1|\big(\|y\varphi'\|_{L^2}^2+\|\varphi\|_{L^2}^2\big)\\
\leq &2|\langle F,-\varphi\rangle|+\epsilon C\nu |k|^{2}(\|\varphi'\|_{L^2}^2+|k|^2\|\varphi\|_{L^2}^2)\\
&+|k_1||1-\lambda|(\|\varphi'\|_{L^2}^2+|k|^2\|\varphi\|_{L^2}^2\big),
\end{aligned}
\end{equation}
here we have used $\nu^{\f12}|k_1|^{\f12}\lesssim\nu|k|^2$ and \eqref{y-phi} in the last inequality.
Due to $\epsilon\ll 1$ and $\nu |k|^3|k_1|^{-1}\geq 2( |1-\lambda|^{\f12}+|\nu/k_1|^{\frac{1}{4}})$, we get by \eqref{w-phi} that
\begin{equation}\nonumber
\begin{aligned}
\nu\|w\|_{L^2}^2=\nu(\|\varphi''\|_{L^2}^2+2|k|^2\|\varphi'\|_{L^2}^2+|k|^4\|\varphi\|_{L^2}^2)
\leq C|\langle F,-\varphi\rangle|.
\end{aligned}
\end{equation}
Then by noticing that  $$|\langle F,-\varphi\rangle| \lesssim |k|^{-1}\|F\|_{H^{-1}_k}\|w\|_{L^2}\lesssim|k|^{-2}\|F\|_{L^2}\|w\|_{L^2} \lesssim|k|^{-3}\|(\partial_y,|k|)F\|_{L^2}\|w\|_{L^2} $$ and using the fact that $\nu|k|^4|k_1|^{-1}\gtrsim 1$, we can deduce \eqref{estimate-noslipL2-2}, \eqref{estimate-noslipH-1-2} and \eqref{estimate-noslipL2-H1}.
\end{proof}

\section{Space-time estimates for the linearized NS system}\label{sec:sp}

In this section, we will establish the space-time estimates for the linearized Navier-Stokes system around the plane Poiseuille flow, which will play key roles in nonlinear stability.

\subsection{Space-time estimates with Navier-slip boundary condition}

We consider the linearized Navier-Stokes system without nonlocal term
\begin{equation}\label{linear NS-without-nonlocal}
\left \{
\begin{array}{lll}
(\partial_t+\mathscr{H}_k)\omega=-ik_1g_1-\partial_yg_2-ik_3g_3-g_4,\\
(\partial_y^2-|k|^2)\varphi=\omega,\quad \varphi(\pm 1)=0,\\
\omega(\pm 1)=0,\ \
\omega|_{t=0}=\omega_0(k_1,y,k_3),
\end{array}
\right.
\end{equation}
where $|k|^2=k_1^2+k_3^2,\ k_1\neq0$ and $\mathscr{H}_k=-\nu(\partial_y^2-|k|^2)+ik_1(1-y^2)$.

\begin{theorem}\label{thm:sp-without-nonlocal}
Let $\omega$ solve \eqref{linear NS-without-nonlocal}. Assume $g_4|_{y=\pm 1}=0$. Then there exists a constant $c$ with $0<c\ll1$ such that
\begin{equation*}
\begin{aligned}
&\nu \|e^{c\nu^{\frac{1}{2}}t} (\partial_y,|k|)\omega\|_{L^2 L^2}^2+(\nu |k_1|)^{\frac{1}{2}}\|e^{c\nu^{\frac{1}{2}}t} \omega\|_{L^2 L^2}^2\\
&\lesssim   \|\omega_0\|_{L^2}^2 + |k_1|^{-1}|k|^{-2}\| e^{c\nu^{\frac{1}{2}}t} (\partial_y,|k|)g_4\|_{L^2 L^2}^2+\nu^{-1}\|e^{c \nu^{\frac{1}{2}}t} g_2\|_{L^2 L^2}^2\\
&\quad+\min\{(\nu |k|^2)^{-1},(\nu |k_1|)^{-1/2} \}\|e^{c \nu^{\frac{1}{2}}t} (k_1g_1+k_3g_3)\|_{L^2 L^2}^2,
\end{aligned}
\end{equation*}
and
\begin{equation*}
\begin{aligned}
&\|e^{c\nu^{\frac{1}{2}}t} \omega\|_{L^\infty L^2}^2\lesssim   \|\omega_0\|_{L^2}^2 + ( \nu^{-\f14} |k_1|^{-\f34} |k|^{-2}+|k_1|^{-1}|k|^{-1})\| e^{c\nu^{\frac{1}{2}}t} (\partial_y,|k|)g_4\|_{L^2 L^2}^2\\
&\quad+\min\{(\nu |k|^2)^{-1},(\nu |k_1|)^{-1/2} \}\|e^{c \nu^{\frac{1}{2}}t} (k_1g_1+k_3g_3)\|_{L^2 L^2}^2+\nu^{-1}\|e^{c \nu^{\frac{1}{2}}t} g_2\|_{L^2 L^2}^2.
\end{aligned}
\end{equation*}
Moreover, it holds that
\begin{align*}
& \|e^{c\nu^{\frac{1}{2}}t} \partial_y \omega\|_{L^\infty L^2}^2+\nu \|e^{c\nu^{\frac{1}{2}}t} (\partial_y,|k|)\omega'\|_{L^2L^2}^2+(\nu |k_1|^3)^{\frac{1}{4}} \|e^{c\nu^{\frac{1}{2}}t} \partial_y \omega\|_{L^2 L^2}^2 \\
&\lesssim   \|\omega_0'\|_{L^2}^2+ \nu^{-\frac{3}{4}} |k_1|^{\frac{3}{4}}\Big(  \|\omega_0\|_{L^2}^2+ |k_1|^{-1}|k|^{-2}\| e^{c\nu^{\frac{1}{2}}t} (\partial_y,|k|)g_4\|_{L^2 L^2}^2\\
&\quad+\nu^{-1}\|e^{c \nu^{\frac{1}{2}}t} g_2\|_{L^2 L^2}^2+{\min\{(\nu |k|^2)^{-1},(\nu |k_1|)^{-1/2}\}} \|e^{c \nu^{\frac{1}{2}}t} (k_1g_1+k_3g_3)\|_{L^2 L^2}^2\Big)\\
&\quad+\nu^{-1}\|e^{c\nu^{\frac{1}{2}}t} \partial_y g_2\|_{L^2 L^2}^2+\nu^{-\frac{3}{4}}|k|^{-\frac{4}{3}}|k_1|^{-\frac{1}{4}}\|e^{c\nu^{\frac{1}{2}}t} \partial_y g_4\|_{L^2 L^2}^2.
\end{align*}

\end{theorem}

\begin{proof}
The proof is based on the resolvent estimates in section 2.6 for the operator $\mathscr{H}_k-\epsilon(\nu|k_1|)^\f12$. Thus, it suffices to consider the case of $c=0$.

Let us decompose $\omega=\omega_I+\omega_H$ with
\begin{equation}\label{local-ns-I}
\left \{
\begin{array}{lll}
(\partial_t+\mathscr{H}_k)\omega_I=-ik_1g_1-\partial_yg_2-ik_3g_3-yg_4,\\
(\partial_y^2-|k|^2)\varphi_I=\omega_I,\quad \varphi_I(\pm 1)=0,\\
\omega_I(\pm 1)=0,\ \
\omega_I|_{t=0}=0,
\end{array}
\right.
\end{equation}
and
\begin{equation}\label{local-ns-H}
\left \{
\begin{array}{lll}
(\partial_t+\mathscr{H}_k)\omega_H=0,\ (\partial_y^2-|k|^2)\varphi_H=\omega_H,\\
 \varphi_H(\pm 1)=0,\
\omega_H(\pm 1)=0,\ \
\omega_H|_{t=0}=\omega_0.
\end{array}
\right.
\end{equation}

For $\omega_H$, based on Gearhart-Pr\" uss type lemma in \cite{Wei} and the resolvent estimate in Proposition \ref{prop:res-non}, we can deduce that
$$\|\omega_H(t)\|_{L^2}\leq Ce^{-c_1 (\nu |k_1|)^{\frac{1}{2}}t}\|\omega_0\|_{L^2},$$
which yields
$$(\nu |k_1|)^{\frac{1}{2}} \|\omega_H\|_{L^2 L^2}^2+\|\omega_H\|_{L^\infty L^2}^2\leq \|\omega_0\|_{L^2}^2.$$
Meanwhile, the energy  estimate gives
$$\nu \|(\partial_y,|k|)\omega_H\|_{L^2 L^2}^2 \leq \|\omega_0\|_{L^2}^2.$$

We further introduce the decomposition
$\omega_I=\omega_I^{(1)}+\omega_I^{(2)}$ with
\begin{equation}\label{local-ns-I-1}
\left \{
\begin{array}{lll}
(\partial_t+\mathscr{H}_k)\omega_I^{(1)}=-ik_1g_1-\partial_y g_2-ik_3g_3,\\
(\partial_t+\mathscr{H}_k)\omega_I^{(2)}=-yg_4,\\
 (\partial_y^2-|k|^2)\varphi_I^{(i)}=\omega_I^{(i)}, i=1,2,\\
 \varphi_I^{(i)}(\pm 1)=0,\
\omega_I^{(i)}(\pm 1)=0,\ \
\omega_I^{(i)}|_{t=0}=0.
\end{array}
\right.
\end{equation}

By Proposition \ref{prop:res-non} and Proposition \ref{H-1-local}, we have
\begin{equation*}
\begin{aligned}
& (\nu |k_1|)^{\frac{1}{2}} \|\omega_I^{(1)}\|_{L^2 L^2}^2+\nu \|(\partial_y,|k|)\omega_I^{(1)}\|_{L^2 L^2}^2 \\
&  \qquad \lesssim \nu^{-1}\|g_2\|_{L^2 L^2}^2+\min\{(\nu |k|^2)^{-1},(\nu |k_1|)^{-1/2}\}\|(k_1g_1+k_3g_3)\|_{L^2 L^2}^2,\\
\end{aligned}
\end{equation*}
Here we used(by a direct energy estimate)
\beno
\nu \|(\partial_y,|k|)\omega_I^{(1)}\|_{L^2 L^2}^2 \lesssim \nu^{-1}\|g_2\|_{L^2 L^2}^2+(\nu |k|^2)^{-1}\|(k_1g_1+k_3g_3)\|_{L^2 L^2}^2,
\eeno
and
\begin{align*}
\nu |k|^2\|\omega_I^{(1)}\|_{L^2 L^2}^2 &\lesssim \nu^{-1}\|g_2\|_{L^2 L^2}^2+\|(k_1g_1+k_3g_3)\|_{L^2 L^2}\|\omega_I^{(1)}\|_{L^2 L^2}\\
&\lesssim \nu^{-1}\|g_2\|_{L^2 L^2}^2+\|(k_1g_1+k_3g_3)\|_{L^2 L^2}^2(\nu |k_1|)^{-\frac{1}{2}} .
\end{align*}
By Proposition \ref{resol-H-1-local} and Proposition \ref{resol-H-1-local-u}, we have
\begin{equation*}
\begin{aligned}
&(\nu |k_1|)^{\frac{1}{2}} \|\omega_I^{(2)}\|_{L^2 L^2}^2+\nu \|(\partial_y,|k|)\omega_I^{(2)}\|_{L^2 L^2}^2 \lesssim |k|^{-2}|k_1|^{-1} \|(\partial_y,|k|)g_4)\|_{L^2 L^2}^2,\\
&|k_1||k|\|(\partial_y,|k|)\varphi^{(2)}\|_{L^2 L^2}^2\lesssim ( \nu^{-\f14} |k_1|^{-\f34} |k|^{-2}+|k_1|^{-1}|k|^{-1}) \|(\partial_y,|k|)g_4)\|_{L^2 L^2}^2.
\end{aligned}
\end{equation*}

Summing up, we conclude the first inequality of the theorem.

The energy estimate gives
\begin{equation*}
\begin{aligned}
&\frac{1}{2}\frac{\mathrm{d}}{\mathrm{d}t}\|\omega^{(2)}_I\|_{L^2}^2+\nu \|(\partial_y,|k|)\omega^{(2)}_I\|_{L^2}^2\\
&\leq | \mathrm{Re}\langle yg_4,\omega_I^{(2)}\rangle|
\leq C \|(\partial_y,|k|)g_4\|_{L^2}\|_{L^2}\|(\partial_y,|k|)\varphi^{(2)}\|_{ L^2}\\
&\lesssim |k||k_1|\|(\partial_y,|k|)\varphi^{(2)}\|_{ L^2}^2+|k_1|^{-1}|k|^{-1}\|(\partial_y,|k|)g_4\|_{L^2}^2,
\end{aligned}
\end{equation*}
which yields
\begin{equation*}
\begin{aligned}
&\|\omega^{(2)}_I\|_{L^\infty L^2}^2+\nu \|(\partial_y,|k|)\omega^{(2)}_I\|_{L^2 L^2}^2\\
&\lesssim |k||k_1|\|(\partial_y,|k|)\varphi^{(2)}\|_{ L^2 L^2}^2+ |k_1|^{-1}|k|^{-1}\|(\partial_y,|k|)g_4\|_{L^2 L^2}^2\\
&\lesssim   ( \nu^{-\f14} |k_1|^{-\f34} |k|^{-2}+|k_1|^{-1}|k|^{-1}) \|(\partial_y,|k|)g_4\|_{L^2 L^2}^2.
\end{aligned}
\end{equation*}
Then the second inequality follows by making a direct energy estimate for $\om-\om_{I}^{(2)}$.

The energy estimate gives
\begin{equation*}
\begin{aligned}
&\langle \pa_y\omega_t,\pa_y\omega\rangle+\nu \|(\partial_y,|k|)\omega'\|_{L^2}^2+ik_1\langle (1-y^2)\omega,\omega''\rangle\\
&=\langle -ik_1 g_1-\partial_y g_2-ik_3 g_3-yg_4,\omega''\rangle,
\end{aligned}
\end{equation*}
which yields by taking the real part that
\begin{equation*}
\begin{aligned}
\frac{1}{2}\frac{\mathrm{d}}{\mathrm{d}t}\|\partial_y \omega\|_{L^2}^2&+\nu \|(\partial_y,|k|)\omega'\|_{L^2}^2
\leq 2|k_1|\|\omega\|_{L^2}\|\partial_y\omega\|_{L^2}+\frac{\nu}{4}\|\omega''\|_{L^2}^2\\
&+\nu^{-1}\|(k_1 g_1+\partial_y g_2+k_3 g_3)\|_{L^2}^2+\|\partial_y g_4\|_{L^2}\|\omega'\|_{L^2}.
\end{aligned}
\end{equation*}
Therefore, we obtain
\begin{equation*}
\begin{aligned}
&\frac{1}{2}\frac{\mathrm{d}}{\mathrm{d}t}\|e^{c\nu^{\frac{1}{2}}t} \partial_y \omega\|_{L^2}^2+\nu \|e^{c\nu^{\frac{1}{2}}t} (\partial_y,|k|)\omega'\|_{L^2}^2\\
&\leq \left(C(\nu |k_1|^3)^{\frac{1}{4}}+\frac{\nu |k|^2}{2}\right)\|e^{c\nu^{\frac{1}{2}}t} \omega'\|_{L^2}^2+\nu^{-\frac{1}{4}}|k_1|^{\frac{5}{4}}\|e^{c\nu^{\frac{1}{2}}t}\omega\|_{L^2}^2+c\nu^{\frac{1}{2}}\|e^{c\nu^{\frac{1}{2}}t} \omega'\|_{L^2}^2\\
&\quad+\frac{\nu}{4}\|e^{c\nu^{\frac{1}{2}}t} \omega''\|_{L^2}^2+\nu^{-1}\|e^{c\nu^{\frac{1}{2}}t} (k_1 g_1+\partial_y g_2+k_3 g_3)\|_{L^2}^2\\
&\quad+(\nu |k|^2 )^{-\frac{2}{3}} (\nu |k_1|^3)^{-\frac{1}{4}\cdot \frac{1}{3}}\|e^{c\nu^{\frac{1}{2}}t} \partial_y g_4\|_{L^2}^2 ,
\end{aligned}
\end{equation*}
from which, we infer that
\begin{align*}
&\|e^{c\nu^{\frac{1}{2}}t} \partial_y \omega\|_{L^\infty L^2}^2+\nu \|e^{c\nu^{\frac{1}{2}}t} (\partial_y,|k|)\omega'\|_{L^2L^2}^2\\
&\lesssim \|\omega_0'\|_{L^2}^2+ (\nu |k_1|^3)^{\frac{1}{4}}\|e^{c\nu^{\frac{1}{2}}t} \omega'\|_{L^2 L^2}^2+\nu^{-\frac{1}{4}}|k_1|^{\frac{5}{4}}\|e^{c\nu^{\frac{1}{2}}t}\omega\|_{ L^2 L^2}^2\\
&\quad+\nu^{-1}\|e^{c\nu^{\frac{1}{2}}t} (k_1 g_1+\partial_y g_2+k_3 g_3)\|_{L^2 L^2}^2+\nu^{-\frac{3}{4}}|k|^{-\frac{4}{3}}|k_1|^{-\frac{1}{4}}\|e^{c\nu^{\frac{1}{2}}t} \partial_y g_4\|_{L^2 L^2}^2 \\
&\lesssim  \|\omega_0'\|_{L^2}^2+ \nu^{-\frac{3}{4}} |k_1|^{\frac{3}{4}}\big(\nu \|e^{c\nu^{\frac{1}{2}}t} \omega'\|_{L^2 L^2}^2+\nu^{\frac{1}{2}}|k_1|^{\frac{1}{2}}\|e^{c\nu^{\frac{1}{2}}t}\omega\|_{ L^2 L^2}^2\big)\\
&\quad+\nu^{-1}\|e^{c\nu^{\frac{1}{2}}t} (k_1 g_1+\partial_y g_2+k_3 g_3)\|_{L^2 L^2}^2+\nu^{-\frac{3}{4}}|k|^{-\frac{4}{3}}|k_1|^{-\frac{1}{4}}\|e^{c\nu^{\frac{1}{2}}t} \partial_y g_4\|_{L^2 L^2}^2 \\
&\lesssim  \|\omega_0'\|_{L^2}^2+ \nu^{-\frac{3}{4}} |k_1|^{\frac{3}{4}}\Big(\|\omega_0\|_{L^2}^2+\nu^{-1}\|e^{c \nu^{\frac{1}{2}}t} g_2\|_{L^2 L^2}^2 + |k_1|^{-1}|k|^{-1}\| e^{c\nu^{\frac{1}{2}}t} (\partial_y,|k|)g_4\|_{L^2 L^2}^2\\
&\quad +{\min\{(\nu |k|^2)^{-1},(\nu |k_1|)^{-1/2}\}} \|e^{c \nu^{\frac{1}{2}}t} (k_1g_1+k_3g_3)\|_{L^2 L^2}^2\Big)\\
&\quad+\nu^{-1}\|e^{c\nu^{\frac{1}{2}}t} (k_1 g_1+\partial_y g_2+k_3 g_3)\|_{L^2 L^2}^2+\nu^{-\frac{3}{4}}|k|^{-\frac{4}{3}}|k_1|^{-\frac{1}{4}}\|e^{c\nu^{\frac{1}{2}}t} \partial_y g_4\|_{L^2 L^2}^2,
\end{align*}
which implies the third inequality of the theorem.
\end{proof}

\subsection{Space-time estimates with non-slip boundary condition}\label{subsec:sp-noslip}

We consider the linearized Navier-Stokes system with non-slip boundary condition
\begin{equation}\label{toy model}
\left \{
\begin{array}{lll}
(\partial_t+\mathscr{L}_k)\omega=F ,\\
  (\partial_y^2-|k|^2)\varphi=\omega,\\ \omega|_{t=0}=\omega_0, \ \varphi(\pm 1)=\varphi'(\pm 1)=0,\\
\end{array}\right.\end{equation}
where the linear operator $\mathscr{L}_k=-\nu (\partial_y^2-|k|^2)+ik_1(1-y^2)+2ik_1 (\partial_y^2-|k|^2)^{-1}$ with $k_1\neq 0$. We denote $\|u\|_{L^2}^2=\|(\partial_y,|k|)\varphi\|_{L^2}^2$.

\begin{theorem}\label{thm:sp-no-slip}
Let $\omega$ be the solution to \eqref{toy model} with  $\langle \omega_0, e^{\pm |k|y}\rangle=0$ and $F=-ik_1f_1-\partial_y f_2-ik_3f_3$. Then there exists $0<c\ll1$ such that
\begin{equation*}
\begin{aligned}
&(|k|+\nu^{\f12} |k|^2) \|e^{c\nu^{\frac{1}{2}}t} u\|_{L^\infty L^2}^2+(\nu|k|+ \nu^{\f32} |k|^2) \|e^{c\nu^{\frac{1}{2}}t} \omega\|_{ L^2 L^2 }^2\\
 &\quad+(|k_1|+\nu^{\f12} |k_1| |k|) \|e^{c\nu^{\frac{1}{2}}t} u\|_{L^2 L^2}^2+{\nu^{\f34} |k|^{\f34} |k_1|^{\f14} \|e^{c\nu^{\frac{1}{2}}t} \omega\|_{ L^2 L^2 }^2}\\
 &\quad{+\nu^{\f7{12}} (\|e^{c\nu^{\frac{1}{2}}t} \omega\|_{L^\infty L^2}^2+\nu \|e^{c\nu^{\frac{1}{2}}t} \partial_y \omega\|_{L^2 L^2}^2)}\\
 &\lesssim \|(\partial_y^2-|k|^2) \omega_0\|_{L^2}^2 +\nu^{-1} \|e^{c\nu^{\frac{1}{2}}t} (f_1,f_2,f_3)\|_{L^2 L^2}^2 .
\end{aligned}
\end{equation*}
\end{theorem}

First of all, we consider the following inhomogeneous problem
\begin{equation}\label{toy model-inhomo}
\left \{
\begin{array}{lll}
(\partial_t+\mathscr{L}_k)\omega_I=F,\\ (\partial_y^2-|k|^2)\varphi_I=\omega_I,\\ \omega_I|_{t=0}=0, \ \varphi_I(\pm 1)=\varphi_I'(\pm 1)=0.\end{array}\right.\end{equation}
For $0<c\ll 1$ small enough, let us introduce
$$\tilde{\omega}_I=e^{c\nu^{\frac{1}{2}}t}\omega_I,\quad \tilde{\varphi}_I=e^{c \nu^{\frac{1}{2}}t}\varphi_I,\quad \tilde{F}=e^{c \nu^{\frac{1}{2}}t}F.$$
Then we have
\begin{equation}\label{toy model1}
\left \{
\begin{array}{lll}
(\partial_t+\mathscr{L}_k)\tilde{\omega}_I-c \nu^{\frac{1}{2}}\tilde{\omega}_I=\tilde{F},\\ (\partial_y^2-|k|^2) \tilde{\varphi}_I=\tilde{\omega}_I,\\ \tilde{\omega}_I|_{t=0}=0, \ \tilde{\varphi}_I(\pm 1)=\tilde{\varphi}_I'(\pm 1)=0.
\end{array}\right.\end{equation}

\begin{lemma}\label{sp-inhomo}
Assume that $\nu |k|^3 |k_1|^{-1}\leq 2,\ k_1\neq 0$ and $F=-ik_1f_1-\partial_y f_2-ik_3f_3-f_4-f_5$.  Then we have
\begin{equation*}
\begin{aligned}
\|\tilde{\omega}_I\|_{L^2 L^2}^2\lesssim & {\nu^{-\frac{5}{4}}|k_1|^{-\frac{3}{4}} |k|^{-1}} \| \tilde{f}_4\|_{L^2 L^2}^2+\nu^{-\frac{3}{2}}|k_1|^{-\frac{1}{2}}\|(\tilde{f}_1,\tilde{f}_2,\tilde{f}_3)\|_{L^2 L^2}^2\\
&+\nu^{-\frac{3}{4}}|k_1|^{-\frac{5}{4}}\|(\partial_y,|k|)\tilde{f}_5\|_{L^2 L^2}^2,\\
\|\tilde{u}_I\|_{L^2 L^2}^2\lesssim&\nu^{-\frac{1}{2}}|k_1|^{-\frac{3}{2}}\|\tilde{f}_4\|_{L^2 L^2}^2+\nu^{-1} |k_1|^{-1}\|(\tilde{f}_1,\tilde{f}_2,\tilde{f}_3)\|_{L^2 L^2}^2\\
&+|k_1|^{-2}\|(\partial_y,|k|)\tilde{f}_5\|_{L^2 L^2}^2,
\end{aligned}
\end{equation*}
where $\|\tilde{u}_I\|_{L^2}^2=\|(\partial_y,|k|)  \tilde{\varphi}_I\|_{L^2}^2$.
\end{lemma}

\begin{proof}
Let us introduce
\beno
&&w(\lambda,y)=\int_{\mathbb{R}_+}\tilde{\omega}_I(t,k_1,y,k_3) e^{-i\lambda t}\mathrm{d}t,\quad \varphi(\lambda,y)=\int_{\mathbb{R}_+}\tilde{\varphi}_I(t,k_1,y,k_3) e^{-i\lambda t}\mathrm{d}t,\\
&&F_j(\lambda,y)=\int_{\mathbb{R}_+}\tilde{f}_j(t,k_1,y,k_3) e^{-i\lambda t}\mathrm{d}t,\quad j=1,2,3,4,5.
\eeno
Then we have
\begin{equation*}
\left\{
\begin{array}{lll}
&-\nu(\partial_y^2-|k|^2)w+ik_1(1-y^2+\lambda/k_1)w+2ik_1\varphi -c(\nu|k_1|)^{\frac{1}{2}}w\\&\qquad=-ik_1F_1-\partial_y F_2-ik_3F_3-F_4-F_5,\\
&(\varphi,\partial_y\varphi)|_{y=\pm 1}=(0,0).
\end{array}
\right.
\end{equation*}
It follows from Proposition \ref{resolvent-noslip} that
\begin{equation*}
\begin{aligned}
\|w\|_{L^2}^2\lesssim &  {\nu^{-\frac{5}{4}}|k_1|^{-\frac{3}{4}} |k|^{-1}}  \|F_4\|_{L^2}^2+\nu^{-\frac{3}{2}}|k_1|^{-\frac{1}{2}}\|(F_1,F_2,F_3)\|_{L^2}^2\\
&+\nu^{-\frac{3}{4}}|k_1|^{-\frac{5}{4}}\|(\partial_y,|k|)F_5\|_{L^2}^2,\\
\|u\|_{L^2}^2\lesssim&\nu^{-\frac{1}{2}}|k_1|^{-\frac{3}{2}}\|F_4\|_{L^2}^2\\
&
+\nu^{-1}|k_1|^{-1}\|(F_1,F_2,F_3)\|_{L^2}^2+
|k_1|^{-2}\|(\partial_y,|k|)F_5\|_{L^2}^2,
\end{aligned}
\end{equation*}
which give our results by  Plancherel's formula.
\end{proof}

Next we consider the following homogeneous problem
\begin{equation}\label{omega-H}
\left \{
\begin{array}{lll}
(\partial_t+\mathscr{L}_k)\omega_H=0,\\
(\partial_y^2-|k|^2)\varphi_H=\omega_H,\\
\omega_H|_{t=0}=\omega_0(k_1,y,k_3), \quad \varphi_H(\pm 1)=\varphi_H'(\pm 1)=0.
\end{array}
\right.
\end{equation}

\begin{lemma}\label{sp-Homo}
Assume that $\nu |k|^3|k_1|^{-1}\leq 2,\ k_1\neq 0, \ \langle \omega_0,e^{\pm |k|y}\rangle=0$. Then it holds that
\begin{equation}\nonumber
\begin{aligned}
&\nu^{\f34}|k_1|^{\f14}|k|^{-\f14} \|\omega_H\|_{L^2 L^2}^2+\|u_H\|_{L^2L^2}^2\le C|k|^{-1} \|(\partial_y^2-|k|^2)\omega_0\|_{L^2}^2.
\end{aligned}
\end{equation}
Moreover,  there exists a constant $c$ with $0<c\ll1$ such that
\begin{equation}\nonumber
\begin{aligned}
&\nu^{\f34}|k_1|^{\f14}|k|^{-\f14} \| e^{c \nu^{\frac{1}{2}}t}\omega_H\|_{L^2 L^2}^2+\|e^{c \nu^{\frac{1}{2}}t} u_H\|_{L^2L^2}^2\le  C|k|^{-1} \|(\partial_y^2-|k|^2)\omega_0\|_{L^2}^2.
\end{aligned}
\end{equation}
\end{lemma}

\begin{proof}
Let $\omega_H^{(0)}$ solve
\begin{equation*}
\begin{aligned}
&\partial_t\omega_H^{(0)}+ik_1(1-y^2)\omega_H^{(0)}+2ik_1\varphi_H^{(0)}=0,\\
&(\partial_y^2-|k|^2) \varphi_H^{(0)}=\omega_H^{(0)},\ \omega_H^{(0)}|_{t=0}=\omega_0.
\end{aligned}
\end{equation*}
It follows from Proposition \ref{lem:Rayleigh-000} that
\begin{align}\label{psiL2H1}
&\sup_{t>0}\|\om_{H}^{(0)}(t)\|_{L^2}^2+|k_1|\int_0^{+\infty}(\|\partial_y\varphi_{H}^{(0)} (t)\|_{L^2}^2+|k|^2\|\varphi_{H}^{(0)}(t)\|_{L^2}^2)dt\\
\nonumber&\leq C\|(\partial_y,|k|)\om_0\|_{L^2}^2,
\end{align}
and
\begin{align}\label{psiL2bd}
&\int_0^{+\infty}(|\partial_y\varphi_{H}^{(0)}(t,0)|^2
+|\partial_y\varphi_{H}^{(0)}(t,1)|^2)dt\leq C \|(\partial_y,|k|)\om_0\|_{L^2}^2,
\end{align}
and
\begin{align}\label{psiL2-t-000}
     &\int_{0}^{+\infty}\langle t\rangle^2\|k\varphi_H^{(0)}(t)\|_{L^2}^2dt \leq C|k|^{-1}(\|\omega_0\|_{H^2}^2+|k|^2\|\omega_0\|_{H^1}^2)\\
     &\leq C|k|^{-1}\|(\partial_y^2-|k|^2)\omega_0\|_{L^2}^2.\nonumber
\end{align}

We introduce
\begin{equation*}
\begin{aligned}
&g(t,y)=e^{-\nu |k|^2t-4\nu k_1^2y^2t^3/3-|\nu k_1|^{1/2}t},\quad \omega_H^{(1)}(t,k_1,y,k_3)=g(t,y)\omega_H^{(0)},\\
& (\partial_y^2-|k|^2)\varphi_H^{(i)}=\omega_H^{(i)}, \quad \varphi_H^{(i)}(\pm 1)=0, \quad i=0,1.
\end{aligned}
\end{equation*}
It is easy to find that
\begin{equation*}
\begin{aligned}
&\partial_t\omega_{H}^{(1)}+\nu |k|^2\omega_{H}^{(1)}+ (4\nu k_1^2y^2t^2+(\nu |k_1|)^{1/2})\omega_{H}^{(1)}+ik_1(1-y^2)\omega_{H}^{(1)}=-2ik_1g(t,y)\varphi_{H}^{(0)},\\
& \omega_H^{(0)}(0,k_1,y,k_3)=\omega_H^{(1)}(0,k_1,y,k_3)=\omega_0(k_1,y,k_3).
\end{aligned}
\end{equation*}
Thus, we have
\begin{align*}
   &\partial_t\omega_{H}^{(1)}-\nu(\partial_y^2-|k|^2)\omega_{H}^{(1)} +ik_1(1-y^2)\omega_{H}^{(1)}+2ik_1\varphi_{H}^{(1)}\\
    &=-\nu\partial_y^2\omega_{H}^{(1)}- 4\nu k_1^2y^2t^2\omega_{H}^{(1)} -(\nu| k_1|)^{1/2}\omega_{H}^{(1)}-2ik_1(g\varphi_{H}^{(0)}-\varphi_{H}^{(1)}).
 \end{align*}

 Now we decompose $\omega_H$ as follows
\begin{align}\nonumber
   & \omega_H=\omega_H^{(1)}+\omega_H^{(2)}+\omega_H^{(3)},
\end{align}
where $\omega_H^{(2)}$ satisfies
\begin{equation*}
\left \{
\begin{array}{lll}
(\partial_t+\mathcal{L}_k)\omega_H^{(2)}=\nu\partial_y^2\omega_{H}^{(1)}+ 4\nu k_1^2y^2t^2\omega_{H}^{(1)} +|\nu k_1|^{1/2}\omega_{H}^{(1)}\\
    \qquad\qquad\quad+2ik_1(g(t,y)\varphi_{H}^{(0)}-\varphi_{H}^{(1)}),\\
 (\partial_y^2-|k|^2) \varphi_H^{(2)}=\omega_H^{(2)},\
\omega_H^{(2)}|_{t=0}=0, \  \langle \omega_H^{(2)},e^{\pm |k|y}\rangle=0,
\end{array}\right.
\end{equation*}
and $\omega_H^{(3)}$ satisfies
\begin{equation*}
\left \{
\begin{array}{lll}
(\partial_t+\mathcal{L}_k)\omega_H^{(3)}=0
,\\ (\partial_y^2-|k|^2) \varphi_H^{(3)}=\omega_H^{(3)},\
\omega_H^{(3)}|_{t=0}=0, \  \langle \omega_H^{(1)}+\omega_H^{(3)},e^{\pm |k|y}\rangle=0.
\end{array}\right.
\end{equation*}

\textbf{Step 1.} Estimates of $\omega_H^{(0)}$ and $\omega_{H}^{(1)}$.\smallskip

We introduce the vector field $X=\partial_y-2ik_1y t$, which commutes with $\partial_t+ik_1(1-y^2)$.
We denote
\begin{align*}
&\om_1=X{\om}^{(0)}_H,\ \psi_1=(\partial^2_y-|k|^2)^{-1}\om_1,\ \psi_2=X\varphi_H^{(0)}.
\end{align*}
Then we find that
\beno
\psi_1=\psi_2+4ik_1t(\partial_y^2-|k|^2)^{-1}\partial_y\varphi_{H}^{(0)}+ \psi_3,
\eeno
where
\begin{align*}
    &(\partial_y^2-|k|^2)\psi_3=0,\quad \psi_3(t,\pm1)=-\psi_{2}(t,\pm1)=-\partial_y\varphi_H^{(0)}(t,\pm1).
 \end{align*}
 Then we have
\begin{align*}
\partial_t{\om_1}+ik_1(1-y^2){\om_1}+2ik_1\psi_1&
=-8k_1^2t(\partial_y^2-|k|^2)^{-1}\partial_y\varphi_{H}^{(0)}+2ik_1\psi_3:=\psi_4.
\end{align*}

By Proposition \ref{lem:Rayleigh-000}, we have
\begin{align}\label{psi1L2H1}
&\sup_{t>0}\|\om_1(t)\|_{L^2}^2+|k_1|\int_0^{+\infty}(\|\partial_y\psi_1(t)\|_{L^2}^2+|k|^2\|\psi_1(t)\|_{L^2}^2)dt\\ \nonumber
&\leq C\|(\partial_y,|k|)\om_1(0)\|_{L^2}^2+C|k_1|^{-1}\int_0^{+\infty}\big(\|\partial_y\psi_4(t)\|_{L^2}^2+|k|^2\|\psi_4(t)\|_{L^2}^2\big)dt.
\end{align}
A direct calculation shows
\begin{align*}
   & \|\partial_y\psi_4(t)\|_{L^2}+|k|\|\psi_4(t)\|_{L^2}
 \\
 &\leq Ck_1^2t\|\varphi_{H}^{0}(t)\|_{L^2}+ C|k|^{\f12}(|\partial_y\varphi_H^{(0)}(t,-1)| +|\partial_y\varphi_H^{(0)}(t,1)|).
\end{align*}
 To proceed, we get by \eqref{psiL2bd}, \eqref{psiL2-t-000} and  \eqref{psi1L2H1} that
\begin{align}\label{psi1L2H1a}
&\sup_{t>0}\|\om_1(t)\|_{L^2}^2+|k_1|\int_0^{+\infty}(\|\partial_y\psi_1(t)\|_{L^2}^2+|k|^2\|\psi_1(t)\|_{L^2}^2)dt\\ \nonumber
&\leq C\|(\partial_y,|k|)\om_1(0)\|_{L^2}^2+C\|(\partial_y^2-|k|^2)\omega_0\|_{L^2}^2\leq C\|(\partial_y^2-|k|^2)\omega_0\|_{L^2}^2.
\end{align}
Thanks to $\omega^{(1)}_{H}=g\omega_{H}^{(0)},$ $g=e^{-\nu |k|^2t-4\nu k_1^2y^2t^3/3-|\nu k_1|^{1/2}t},$ we have
\begin{align}
&\om_1^{(1)}=X\omega^{(1)}_{H} =g\omega_1
-(8\nu k_1^2 yt^3/3)g\omega_{H}^{(0)},\nonumber\\
&\nu k_1^2yt^3 g(t,y)\leq C\nu^{-\f14}|k_1|^{\f14} (\nu k_1^2y^2t^3)^{\f12}(|\nu k_1|^{\f12}t)^{\f32}g(t,y)\label{est:g-bunds}\\
&\qquad\qquad\quad\quad\leq C \nu^{-\f14}|k_1|^{\f14} e^{-|\nu k_1|^{\f12}t/2-\nu k_1^2y^2t^3}.\nonumber
\end{align}
 Then we get by \eqref{psiL2H1}, \eqref{psi1L2H1a} and \eqref{est:g-bunds} that
\begin{align}\label{om1c}
\|\omega^{(1)}_{H}(t)\|_{L^2}\leq& Ce^{-|\nu k_1|^{1/2}t}\|\omega^{(0)}_{H}(t)\|_{L^2}\leq Ce^{-(\nu|k_1|)^{1/2}t}\|\omega_{0}\|_{L^2},
\end{align}
and
\begin{align}\label{om1d}
&\|\omega_1^{(1)}(t)\|_{L^2}\\
&\leq  C(e^{-|\nu k_1|^{1/2}t}\| \omega_1(t)\|_{L^2}+\nu^{-\f14}|k_1|^{\f14}e^{-|\nu k_1|^{1/2}t/2}\|\omega^{(0)}_{H}(t)\|_{L^2})\nonumber
\\ &\leq C\nu^{-\f14}|k_1|^{\f14}e^{-|\nu k_1|^{1/2}t/2}\|(\partial_y^2-|k|^2)\omega_{0}\|_{L^2},\nonumber
\end{align}
and
\begin{align}
\label{om1d1}
\|ty\omega_1^{(1)}(t)\|_{L^2}\leq & C(\nu^{-\f13}|k_1|^{-\f23}e^{-|\nu k_1|^{1/2}t}\| \omega_1(t)\|_{L^2}\\
&\qquad+\nu^{-\f12}|k_1|^{-\f12}e^{-|\nu k_1|^{1/2}t/2}\|\omega^{(0)}_{H}(t)\|_{L^2})\notag
\\ \notag\leq& C\nu^{-\f12}|k_1|^{-\f12}e^{-|\nu k_1|^{1/2}t/2}\|(\partial_y^2-|k|^2)\omega_{0}\|_{L^2}.
\end{align}

Let
\begin{align}\nonumber
   \phi_1^{(1)}(t,y)=g(t,y)\varphi_H^{(0)}(t,y)-\varphi_H^{(1)}(t,y).
 \end{align}
Applying Lemma \ref{lem5.3} with
\beno
\psi_1=\varphi_{H}^{(1)},\quad \psi_2=\varphi_{H}^{(0)},\quad  g=e^{-\nu |k|^2t-4\nu k_1^2y^2t^3/3-|\nu k_1|^{1/2}t},
\eeno
where
\beno
\|g\|_{C^1}\leq Ce^{-|\nu k_1|^{1/2}t}, \quad \|\partial_yg\|_{H^1}\leq
C\nu^{\f14}|k_1|^{-\f14}te^{-|\nu k_1|^{1/2} t/2}.
\eeno
here we used $\nu k_1^2yt^3=(\nu k_1^2y^2t^3)^{\f12}(|\nu k_1|^{\f12}t)^{\f12}\nu^{\f14}|k_1|^{-\f14}t$,  we conclude that
\begin{align}\label{om1a}
&|k|^{\f12}\|(\partial_y,|k|)\varphi_H^{(1)}(t)\|_{L^2} +(|\partial_y\varphi_H^{(1)}(t,-1)|+|\partial_y\varphi_H^{(1)}(t,1)|) \\
\nonumber&\leq Ce^{-|\nu k_1|^{1/2}t}\big(|k|^{\frac{1}{2}}\|(\partial_y,|k|)\varphi_H^{(0)}\|_{L^2}
+|\partial_y\varphi_H^{(0)}(-1)|
+|\partial_y\varphi_H^{(0)}(1)|\big),\\
 \label{phi11}&\|(\partial_y,|k|)\phi_1^{(1)}(t)\|_{L^2}\leq C\nu^{\f14}|k_1|^{-\f54} e^{-|\nu k_1|^{1/2} t/2}\|k_1t\varphi^{(0)}_H(t)\|_{L^2}.
\end{align}
Then we get by \eqref{psiL2H1} and \eqref{psiL2bd} that
\begin{align}\label{est:varphi(1)L2-00}
   \|(\partial_y,|k|)\varphi_{H}^{(1)}\|_{L^2L^2}^2\leq C|k_1|^{-1}\|(\partial_y,|k|)\omega_0\|_{L^2}^2,
 \end{align}
 and by \eqref{om1c}, we get
 \begin{align}\label{est:omega(1)L2-00}
   \|\omega_{H}^{(1)}\|_{L^2L^2}^2\leq C|\nu k_1|^{-\f12}\|\omega_0\|_{L^2}^2.
 \end{align}

\textbf{Step 2.} Estimates of $\omega_{H}^{(2)}$.\smallskip

Thanks to the definition of $\omega_1$, we have
\begin{align*}
\partial_y^2\omega^{(1)}_H &=\partial_y(\omega_1^{(1)}+2ik_1 ty\omega^{(1)}_{H})=\partial_y\omega_1^{(1)}+2ik_1t\omega_{H}^{(1)}+2ik_1ty\omega_{H}^{(1)}\\
&=\partial_y\omega_1^{(1)}+2ik_1t\omega_{H}^{(1)}+2ik_1ty(\omega_1^{(1)}+2ik_1 ty\omega^{(1)}_{H})
 ,\\ \partial_y^2\omega_{H}^{(1)}&+4k_1^2y^2t^2\omega^{(1)}_{H} =\partial_y\omega_1^{(1)}+2ik_1t\omega^{(1)}_{H}+2ik_1ty\omega_1^{(1)}.
\end{align*}
Then by the definition of $\phi_1^{(1)}$, we find that
\begin{align*}
   &\partial_t\omega^{(2)}_H-\nu(\partial_y^2-|k|^2)\omega_{H}^{(2)}
   +ik_1\big( (1-y^2)\omega_H^{(2)}-\varphi^{(2)}_H\big)\\
   &=\nu\partial_y^2\omega_{H}^{(1)}+ 4\nu k_1^2y^2t^2\omega_{H}^{(1)} +\epsilon_0|\nu k_1|^{1/2}\omega_{H}^{(1)}+2ik_1(g(t,y)\varphi_{H}^{(0)}-\varphi_{H}^{(1)}).\\
   &=\nu(\partial_y\omega_1^{(1)}+2ik_1t\omega^{(1)}_{H}+2ik_1ty\omega_1^{(1)})+ |\nu k_1|^{1/2}\omega_{H}^{(1)}
   +2ik_1\phi_1^{(1)}.
\end{align*}
Then we get by Lemma \ref{sp-inhomo} with $f_2=-\nu \omega_{1}^{(1)}$, $f_4=-2\nu ik_1ty\omega_{1}^{(1)}-2\nu ik_1t\omega_H^{(1)}-(\nu |k_1|)^{\frac{1}{2}}\omega_H^{(1)}$ and $f_5=-2ik_1\phi_1^{(1)}$ that
\begin{align*}
   &\|\omega_{H}^{(2)}\|_{L^2L^2}^2 \leq C\nu^{-\f32}|k_1|^{-\f12}\|\nu \omega_1^{(1)}\|_{L^2L^2}^2+C\nu^{-\f54}|k_1|^{-\f34}|k|^{-1} \|\nu k_1ty\omega_{1}^{(1)}\|_{L^2L^2}^2\\
   &\qquad+ C\nu^{-\f54}|k_1|^{-\f34}|k|^{-1}\|(\nu |k_1|t+|\nu k_1|^{\f12})\omega_H^{(1)}\|_{L^2L^2}^2\\
   &\qquad+C\nu^{-\f34}|k_1|^{-\f54}\|k_1(\partial_y,|k|)\phi_1^{(1)}\|_{L^2L^2}^2.\\
   &\quad\leq C\nu^{\f12}|k_1|^{-\f12}\| \omega_1^{(1)}\|_{L^2L^2}^2 +C\nu^{\f34}|k_1|^{\f54}|k|^{-1}\|yt\omega_{1}^{(1)}\|_{L^2L^2}^2\\
   &\qquad+C \nu^{-\f14}|k_1|^{\f14}|k|^{-1}\|(|\nu k_1|^{\f12}t+1)\omega_H^{(1)}\|_{L^2L^2}^2 +C\nu^{-\f34}|k_1|^{\f34}\|(\partial_y,|k|)\phi_1^{(1)}\|_{L^2L^2}^2.
\end{align*}
Using \eqref{om1c}, \eqref{om1d}, \eqref{om1d1}, \eqref{phi11} and \eqref{psiL2-t-000}, we arrive at
\begin{align}\label{est:omega(2)L2-00}
  \|\omega_H^{(2)}\|_{L^2L^2}^2\leq& C(|\nu k_1|^{-\f12}+
   \nu^{-\f34}|k_1|^{-\f14}|k|^{-1}+
   \nu^{-\f14}|k_1|^{-\f74}|k|^{-1})\|(\partial_y^2-|k|^2)\omega_0\|_{L^2}^2\\
   \leq &C\nu^{-\f34}|k_1|^{-\f14}|k|^{-\f34}\|(\partial_y^2-|k|^2)\omega_0\|_{L^2}^2.\nonumber
\end{align}
Applying Lemma \ref{sp-inhomo}  with $f_2=-\nu \omega_{1}^{(1)}$, $-f_4=2\nu ik_1ty\omega_{1}^{(1)}+2\nu ik_1t\omega_H^{(1)}+(\nu |k_1|)^{\frac{1}{2}}\omega_H^{(1)}$, $f_5=-2ik_1(g(t,y)\varphi_{H}^{(0)}-\varphi_{H}^{(1)})=-2ik_1\phi_1^{(1)}$, we obtain
\begin{align*}
  &\|(\partial_y\varphi_{H}^{(2)},|k|\varphi_{H}^{(2)})\|_{L^2L^2}^2 \leq C\nu^{-1}|k_1|^{-1}\|\nu \omega_1^{(1)}\|_{L^2L^2}^2+C\nu^{-\f12}|k_1|^{-\f32} \|\nu k_1ty\omega_{1}^{(1)}\|_{L^2L^2}^2\\
   &\quad\quad+ C\nu^{-\f12}|k_1|^{-\f32} \|(|\nu k_1t|+|\nu k_1|^{1/2})\omega_H^{(1)}\|_{L^2L^2}^2+C|k_1|^{-2}\|2k_1(\partial_y,|k|)\phi_1^{(1)}\|_{L^2L^2}^2\\
   &\leq C\nu|k_1|^{-1}\| \omega_1^{(1)}\|_{L^2L^2}^2 +C\nu^{\f32}|k_1|^{\f12}\|ty\omega_{1}^{(1)}\|_{L^2L^2}^2+C \nu^{\f12}|k_1|^{-\f12}\|(|\nu k_1|^{\f12}t+1)\omega_H^{(1)}\|_{L^2L^2}^2\\
   &\quad+C|k_1|^{-2}\|2k_1(\partial_y,|k|)\phi_1^{(1)}\|_{L^2L^2}^2.
\end{align*}
Using \eqref{om1c}, \eqref{om1d}, \eqref{phi11} and \eqref{psiL2-t-000}, we obtain
\begin{align}\label{est:varphi(2)L2-00}
  \|(\partial_y\varphi_{H}^{(2)},&|k|\varphi_{H}^{(2)})\|_{L^2L^2}^2\leq C|k_1|^{-1}\|(\partial_y^2-|k|^2)\omega_0\|_{L^2}^2\\
  &\quad+C\nu^{\f12}|k_1|^{-\f12}\|(\partial_y^2-|k|^2)\omega_0\|_{L^2}^2+ C|k_1|^{-1}\|\omega_0\|_{L^2}^2\nonumber\\
   &\leq C|k_1|^{-1}\|(\partial_y^2-|k|^2)\omega_0\|_{L^2}^2.\nonumber
\end{align}

\textbf{Step 3}. Estimate for $\omega_H^{(3)}$.

We introduce $\gamma_1(y)=\dfrac{\sinh(|k| (1+y))}{\sinh(2|k|) }$, $\gamma_0(y)=\dfrac{\sinh(|k|(1- y))}{\sinh(2|k|) }$ and
\begin{align*}
   &a_1(t)=\langle\gamma_0,\omega_H^{(1)}\rangle=-\langle\gamma_0,\omega_H^{(3)}\rangle =-\partial_y\varphi_H^{(1)}(t,-1),\\
   &a_2(t)=\langle\gamma_1,\omega_H^{(1)}\rangle=-\langle\gamma_1,\omega_H^{(3)}\rangle
   =\partial_y\varphi_H^{(1)}(t,1).
\end{align*}
By \eqref{om1a}, we get
\begin{align}\label{eq:a-est}
   & |a_1(t)|+|a_2(t)|\\
   &\leq Ce^{-\epsilon_0|\nu k_1|^{1/2}t}\big(|k|^{\frac{1}{2}}\|(\partial_y,|k|)\varphi_H^{(0)}\|_{L^2}
   +|\partial_y\varphi_H^{(0)}(-1)|
+|\partial_y\varphi_H^{(0)}(1)|\big).\nonumber
\end{align}

We define
\begin{equation*}
\begin{aligned}
w(\lambda,y)=\int_0^{+\infty} \omega_H^{(3)} (t,k_1,y,k_3)e^{-i\lambda t}\mathrm{d}t,\\
\varphi (\lambda,y)=\int_0^{+\infty} \varphi_H^{(3)} (t,k_1,y,k_3)e^{-i\lambda t}\mathrm{d}t,\\
c_j(\lambda)=\int_0^{+\infty} a_j(t)e^{-i\lambda t}\mathrm{d}t, \quad j=1,2.
\end{aligned}
\end{equation*}
Then we have
\begin{equation*}
\begin{aligned}
&\big(i\lambda-\nu (\partial_y^2-|k|^2)+ik_1(1-y^2)\big)w+2ik_1\varphi=0,\\
&c_1(\lambda)=-\int_{-1}^1\frac{\sinh (|k|(1+y))}{\sinh( 2|k|)}w(\lambda,k_1,y,k_3)\mathrm{d}y, \\ &c_2(\lambda)=-\int_{-1}^1\frac{\sinh (|k|(1-y))}{\sinh (2|k|)}w(\lambda,k_1,y,k_3)\mathrm{d}y.
\end{aligned}
\end{equation*}
Therefore, we obtain
$$w(\lambda,y)=-c_1(\lambda)w_1(-\lambda/k_1,y)-c_2(\lambda)w_2(-\lambda/k_1,y),$$
where $w_1,\ w_2$ are constructed in section \ref{sec-no-slip} with $\epsilon=0$ and $\lambda$ replaced $-\lambda/k_1$.

Let us first claim that
\begin{equation}\label{c_i:omega-3}
\begin{aligned}
\|(1+|\lambda|)c_1\|_{L^2_{\lambda}(\mathbb{R})}^2+\|(1+|\lambda|)c_2\|_{L^2_{\lambda}(\mathbb{R})}^2\lesssim |k|\|(\partial_y,|k|)\omega_0\|_{L^2}^2.
\end{aligned}
\end{equation}

Thanks to $d=\f{\lambda/k_1}{2}-\f{i\nu |k|^2}{2k_1}$, $L=|2k_1/\nu|^{1/3}$, we get by Proposition \ref{estimate-w-i} that
\begin{equation*}
\begin{aligned}
\|w_1\|_{L^2_y}+\|w_2\|_{L^2_y}\lesssim &|k_1/\nu|^{\f16}\big(1+|Ld|\big)^{\f14}\lesssim  \nu^{-\f14}\big((\nu |k_1|^2)^{1/3}+\nu |k|^2+ |\lambda|\big)^{\f14}\\
\lesssim &\nu^{-\f14}\big(((\nu |k|^3|k_1|^{-1})^{1/3}+\nu |k|^3|k_1|^{-1})|k_1/k|+ |\lambda|\big)^{\f14}\\
\lesssim &\nu^{-\f14}\big(|k_1/k|+ |\lambda|\big)^{\f14},
\end{aligned}
\end{equation*}
which yields
\begin{equation*}
\begin{aligned}
\|w(\lambda)\|_{L^2}\leq |c_1|\|w_1\|_{L^2}+|c_2|\|w_2\|_{L^2}\lesssim\nu^{-\frac{1}{4}}(1+|\lambda|)^{\frac{1}{4}}  (|c_1(\lambda)|+|c_2(\lambda)|).
\end{aligned}
\end{equation*}
Similarly, we have
\begin{equation*}
\begin{aligned}
&\|u(\lambda)\|_{L^2}\leq |\nu /k_1|^{\frac{1}{6}} (|c_1(\lambda)|+|c_2(\lambda)|).
\end{aligned}
\end{equation*}
Therefore, we conclude
\begin{equation}\label{est:omega(3)L2-00}
\begin{aligned}
(\nu |k_1|)^{\frac{1}{2}}\|\omega_H^{(3)}\|_{L^2L^2}^2\sim& (\nu |k_1|)^{\frac{1}{2}}\left\|\|w(\lambda)\|_{L^2}\right\|_{L^2(\mathbb{R})}^2 \\
\lesssim& |k_1|^{\frac{1}{2}}(\|(1+|\lambda|)^{\frac{1}{4}} c_1(\lambda)\|_{L^2(\mathbb{R})}^2+\|(1+|\lambda|)^{\frac{1}{4}} c_2(\lambda)\|_{L^2(\mathbb{R})}^2)\\
\lesssim&|k_1|^{\frac{1}{2}}(\|(1+|\lambda|) c_1(\lambda)\|_{L^2(\mathbb{R})}^2+\|(1+|\lambda|) c_2(\lambda)\|_{L^2(\mathbb{R})}^2)\\
\lesssim&|k_1|^{\f12}|k|\|(\partial_y,|k|)\omega_0\|_{L^2}^2,
\end{aligned}
\end{equation}
and
\begin{equation}\label{est:u(3)L2-00}
\begin{aligned}
{ \|u_H^{(3)}\|_{L^2L^2}^2}
\lesssim& |\nu/k_1|^{\f13}|k|\|(\partial_y,|k|)\omega_0\|_{L^2}^2
\lesssim  \|(\partial_y,|k|)\omega_0\|_{L^2}^2.
\end{aligned}
\end{equation}

Now we conclude by \eqref{est:varphi(1)L2-00}, \eqref{est:varphi(2)L2-00} and \eqref{est:u(3)L2-00}  that
\begin{align}\label{est:varphiL2-000}
   & \|(\partial_y,|k|)\varphi_{H}\|_{L^2L^2}^2\leq C|k_1|^{-1}\|(\partial_y^2-|k|^2)\omega_0\|_{L^2}^2,
\end{align}
and by \eqref{est:omega(1)L2-00}, \eqref{est:omega(2)L2-00} and \eqref{est:omega(3)L2-00},  we obtain
 \begin{align}\label{est:omegaHL2-000}
    \|\omega_H\|_{L^2L^2}^2\leq C \nu^{-\f34}|k_1|^{-\f14}|k|^{-\f34}\|(\partial_y^2-|k|^2)\omega_0\|_{L^2}^2.
  \end{align}

{\bf Step 4.} Proof of  the claim  \eqref{c_i:omega-3}. \smallskip

Thanks to $a_1=\langle \gamma_0,\omega_H^{(1)}\rangle
$ and
\begin{align*}
   & (\partial_t+ik_1(1-y^2)+ 4\nu |k_1|^2y^2t^2+\nu |k|^2)\omega^{(1)}_H=-2ik_1g(t,y)\varphi_{H}^{(0)},
\end{align*}
we infer that
\begin{align*}
  |\partial_ta(t)|&{\leq} \bigg(\left|\langle \gamma_0,k_1(1-y^2)\omega_H^{(1)}\rangle\right|+\left|\langle \gamma_0,(4\nu |k_1|^2y^2t^2+\nu |k|^2)\omega_H^{(1)}\rangle\right|\bigg)\\
   &\quad+|\langle\gamma_0, 2k_1g(t,y)\varphi_{H}^{(0)}\rangle|.
\end{align*}
As $(\gamma_0(1-y^2))|_{y=0,1}=0$, we have
\begin{align*}
\left|\langle \gamma_0,k_1(1-y^2)\omega_H^{(1)}\rangle\right| \leq& |k_1| \|(\partial_y,|k|)\varphi_H^{(1)}\|_{L^2}\|(\partial_y,|k|)(\gamma_0 (1-y^2))\|_{L^2}\\
   \leq& C|k_1||k|^{-\f12}\|(\partial_y,|k| )\varphi_{H}^{(1)}\|_{L^2},
\end{align*}
here we used the fact that
\beno
  \|(\partial_y,|k|)(\gamma_0 (1-y^2))\|_{L^2}\leq C|k|^{-\f12}.
\eeno
 Then we get by \eqref{om1a} that
\begin{align}\nonumber
   & \left|\langle \gamma_0,k_1(1-y^2)\omega^{(1)}_H\rangle\right|\\
    \nonumber&\leq Ce^{-|\nu k_1|^{1/2}t}|k_1||k|^{-1}\big(|k|^{\frac{1}{2}}\|(\partial_y,|k|) \varphi_H^{(0)}\|_{L^2}+|\partial_y\varphi_H^{(0)}(-1)|
+|\partial_y\varphi_H^{(0)}(1)|\big).
\end{align}
Similarly, we  have
\begin{align*}
   &\left|\langle \gamma_0,( 4\nu |k_1|^2 y^2t^2+\nu |k|^2)\omega_H^{(1)}\rangle\right|\\
   &\leq 4\nu |k_1|^2|\langle \gamma_0, (y^2-1)t^2\omega^{(1)}_H\rangle|+4\nu |k_1|^2|\langle \gamma_0, t^2\omega^{(1)}_H\rangle|+\nu |k|^2|\langle \gamma_0,\omega^{(1)}_H\rangle|\\
   &\leq C\nu |k_1|^2t^2|k|^{-\f12}\|(\partial_y,|k|)\varphi_H^{(1)}\|_{L^2}+ C(\nu |k_1|^2t^2+\nu |k|^2)|a_1(t)|.
\end{align*}
which along with\eqref{om1a} and \eqref{eq:a-est} gives
\begin{align}\nonumber
   & |\langle \gamma_0,-(4\nu |k|^2y^2t^2+\nu|k|^2)\omega_H^{(1)}\rangle|\\
   \nonumber&\leq Ce^{-|\nu k_1|^{1/2}t}(\nu |k_1|^2t^2+\nu |k|^2)\big(|k|^{\frac{1}{2}}\|(\partial_y,|k|)\varphi_H^{(0)}\|_{L^2}
   +|\partial_y\varphi_H^{(0)}(-1)|
+|\partial_y\varphi_H^{(0)}(1)|\big).
\end{align}
We also have
\begin{align*}
   & \left|\langle\gamma_0,k_1g(t,y)\varphi_{H}^{(0)}\rangle\right|\leq C{|k_1|}e^{-|\nu k_1|^{1/2}t}\|\gamma_0\|_{L^2}\|\varphi_H^{(0)}\|_{L^2}.
\end{align*}
This along with $\|\gamma_0\|_{L^2}\leq C|k|^{-\f12}$ gives
\begin{align}\nonumber
   & \left|\langle\gamma_0,k_1g \varphi_H^{(0)}\rangle\right| \leq C {|k_1||k|^{-\f12}}e^{-|\nu k_1|^{1/2}t}\|\varphi_H^{(0)}\|_{L^2}.
\end{align}

Summing up, we arrive at
\begin{align}
   & \left|\partial_ta_1(t)\right|+ |a_1(t)|\nonumber\\
   &\leq Ce^{-|\nu k_1|^{1/2}t}(1+\nu |k_1|^2t^2 +\nu |k|^2)\nonumber\\
   &\quad\times\big(|k|^{\frac{1}{2}}\|(\partial_y,|k|)\varphi_H^{(0)}\|_{L^2}
   +|\partial_y\varphi_H^{(0)}(-1)|
+|\partial_y\varphi_H^{(0)}(1)|\big).\label{est:pr-t-a1}
\end{align}

By \eqref{psiL2H1} and \eqref{psiL2bd}, we have
\begin{align*}
   &\|a_1(t)\|_{H^1(\mathbb{R}_+)}\leq C(|k_1|+\nu |k|^2)\|\omega_0\|_{L^2}^2\leq C|k|\|(\partial_y,|k|)\omega_0\|_{L^2}^2.
\end{align*}
Now $a_1\in H^1(\mathbb{R}_+)$. We define {$\widetilde{a}_1(t)=a_1(t)$ for $t\geq0$ and} $\tilde{a}_1(t)=0 $ for $t\leq 0$. Since $\langle\omega_0,e^{\pm |k| y}\rangle=0$, $a_1(0)=0$. Thus, $\tilde{a}_1\in H^1(\mathbb{R})$ and
\begin{align*}
 & \left\|{\tilde{a}_1(t)}\right\|_{H^1(\mathbb{R})}\leq C|k|\|(\partial_y,|k|)\omega_0\|_{L^2}^2.
\end{align*}
 Notice that
\begin{align*}
   &\int_{\mathbb{R}}\widetilde{a}_1(t)e^{-it\lambda}dt
   =\int_{0}^{+\infty}
   a_{1}(t)\mathrm{e}^{-it\lambda}dt=c_1(\lambda),
\end{align*}
which gives
\begin{align*}
\|\widetilde{a}_1(t)\|_ {H^1(\mathbb{R})}^2\sim\int_{\mathbb{R}}(1+|\lambda|)^2|c_1(\lambda)|^2
d\lambda=\|(1+|\lambda|)c_1\|_{L^2(\mathbb{R})}^2 .
\end{align*}
Thus, we obtain
\begin{align*}
\|(1+|\lambda|)c_1\|_{L^2}\leq C|k|^{\f12}\|\omega_0\|_{L^2}.
\end{align*}

Similar argument can yield the estimate for $\|(1+|\lambda|)c_2\|_{L^2(\mathbb{R})}$. Then we completes the proof of \eqref{c_i:omega-3}

Finally, let us consider the estimates with the weight $e^{c\nu^{\frac{1}{2}}t}$, where $0<c\ll 1$ is a suitable small constant. Consider
$\widetilde{\omega}=e^{c\nu^{\frac{1}{2}}t}\omega_H$. Then we have
$$(\partial_t+\widetilde{\mathcal{L}}_k)\widetilde{\omega}=0, \quad \widetilde{\mathcal{L}}_k=\mathcal{L}_k-c\nu^{\frac{1}{2}}.$$
Then following similar arguments as before (by taking $c>0$ sufficiently small such that the estimates for $\widetilde{\mathcal{L}}_k$ holds as $\mathcal{L}_k$) and using $|k_1|\geq 1$, we can conclude the second statement of the lemma.
\end{proof}

\begin{lemma}\label{energy-w-u}
Let $\omega$ be the solution to \eqref{toy model} with $F=-ik_1f_1-\partial_y f_2-ik_3f_3$ and $\langle \omega_0,e^{\pm |k|y}\rangle=0$. Then it holds that
\begin{equation}\nonumber
\begin{aligned}
&|k|^2 \|e^{c\nu^{\frac{1}{2}}t} u\|_{L^\infty L^2}^2+\frac{3}{2} \nu |k|^2 \|e^{c\nu^{\frac{1}{2}}t} \omega\|_{ L^2 L^2 }^2\\
&\leq |k|^2\|u_0\|_{L^2}^2+\big((5/2) |k_1|+2c\nu^{\frac{1}{2}}|k|\big)|k|\|e^{c\nu^{\frac{1}{2}}t} u\|_{L^2 L^2}^2\\
&\quad+4\min\{(\nu |k|^3)^{-1},|k_1|^{-1}\}|k| \|e^{c\nu^{\frac{1}{2}}t} f_4\|_{L^2 L^2}^2+4\nu^{-1}\|e^{c\nu^{\frac{1}{2}}t} (f_1,f_2,f_3)\|_{L^2 L^2}^2.
\end{aligned}
\end{equation}

\end{lemma}

\begin{proof}
Taking the inner product between \eqref{toy model} and $-\varphi$, we obtain
\begin{equation}\label{est:d/dt-u-00}
\begin{aligned}
&\langle u_t, u\rangle+\nu \|\omega\|_{L^2 }^2+ik_1\int_{-1}^1 (1-y^2)(|\varphi'|^2+|k|^2|\varphi|^2)\mathrm{d}y\\
&\quad+ik_1\int_{-1}^1 (-2y)\varphi'\overline{\varphi}\mathrm{d}y-2ik_1\int_{-1}^1 |\varphi|^2\mathrm{d}y\\
&=\langle f_4,\varphi\rangle+\langle ik_1f_1+\partial_y f_2+ik_3f_3,\varphi\rangle.
\end{aligned}
\end{equation}
Taking the real part of \eqref{est:d/dt-u-00}, we have
\begin{equation*}
\begin{aligned}
&\frac{1}{2}\frac{\mathrm{d}}{\mathrm{d}t} \|u\|_{ L^2}^2+\nu \|\omega\|_{L^2 }^2\\
&\leq 2|k_1|\|\varphi\|_{L^2}\|\varphi'\|_{L^2}+
\|f_4\|_{L^2}\|\varphi\|_{L^2}+\|(f_1,f_2,f_3)\|_{L^2}\|(\partial_y,|k|)\varphi\|_{L^2}\\
&\leq |k_1||k|^{-1}(|k|^2\|\varphi\|_{L^2}^2+\|\varphi'\|_{L^2}^2)+
\|f_4\|_{L^2}\|\varphi\|_{L^2}+\|(f_1,f_2,f_3)\|_{L^2}\|(\partial_y,|k|)\varphi\|_{L^2}\\
&\leq |k_1||k|^{-1}\|u\|_{L^2}^2+\dfrac{2\|f_4\|_{L^2}^2}{\nu|k|^4}+
\dfrac{2\|(f_1,f_2,f_3)\|_{L^2}^2}{\nu |k|^2}+\frac{\nu |k|^2}{8}(2|k|^2\|\varphi\|_{L^2}^2+\|\varphi'\|_{L^2}^2)\\
\text{or} &\leq 2|k_1|\|\varphi'\|_{L^2}\|\varphi\|_{L^2}+\dfrac{\|f_4\|_{L^2}^2}{|k_1||k|}+
\dfrac{\|(f_1,f_2,f_3)\|_{L^2}^2}{\nu |k|^2}+\frac{|k_1||k|\|\varphi\|_{L^2}^2+\nu |k|^2\|u\|_{L^2}^2}{4},
\end{aligned}
\end{equation*}
which yields that
\begin{equation}\label{est:nusmall-d/dtu-00}
\begin{aligned}
\f{1}{2}\frac{\mathrm{d}}{\mathrm{d}t} \|u\|_{ L^2}^2+\frac{3}{4}\nu \|\omega\|_{L^2 }^2\leq& |k_1||k|^{-1} \|u\|_{L^2}^2+2(\nu |k|^4)^{-1} \|f_4\|_{L^2}^2\\
&+2(\nu |k|^2)^{-1} \|(f_1,f_2,f_3)\|_{L^2}^2,
\end{aligned}
\end{equation}
or
\begin{equation}\label{est:nubig-d/dtu-00}
\begin{aligned}
\frac{1}{2}\frac{\mathrm{d}}{\mathrm{d}t} \|u\|_{ L^2}^2+\f{3}{4}\nu \|\omega\|_{L^2 }^2
\leq& \f{5}{4}|k_1|\|u\|_{L^2}\|\varphi\|_{L^2}+|k_1|^{-1}|k|^{-1}\| f_4\|_{L^2}^2 \\
&+(\nu |k|^2)^{-1}\| (f_1,f_2,f_3)\|_{L^2}^2.
\end{aligned}
\end{equation}
Then we deduce from \eqref{est:nusmall-d/dtu-00} and \eqref{est:nubig-d/dtu-00} that
\begin{equation}\nonumber
\begin{aligned}
\frac{\mathrm{d}}{\mathrm{d}t}\|e^{c\nu^{\frac{1}{2}}t} u\|_{ L^2}^2+\frac{3}{2}\nu \|e^{c\nu^{\frac{1}{2}}t} \omega\|_{L^2 }^2
\leq& \big((5/2)|k_1||k|^{-1}+2c\nu^{\frac{1}{2}}\big)\|e^{c\nu^{\frac{1}{2}}t} u\|_{L^2}^2\\
&+4(\nu |k|^2)^{-1}\|e^{c\nu^{\frac{1}{2}}t}(f_1, f_2,f_3)\|_{L^2}^2\\&+4\min\{(\nu|k|^4)^{-1},|k_1|^{-1}|k|^{-1}\}\|e^{c\nu^{\frac{1}{2}}t}f_4\|_{L^2}^2.
\end{aligned}
\end{equation}

This completes the proof of the lemma.
\end{proof}

\begin{lemma}\label{lemma-w-L-infty}
Let $\omega$ be the solution to \eqref{toy model} with $F=-ik_1f_1-\partial_y f_2-ik_3f_3-f_4$ and $\langle \omega_0,e^{\pm |k|y}\rangle=0$. Then it holds that
 \begin{equation}\nonumber
\begin{aligned}
&\|e^{c\nu^{\frac{1}{2}}t} \omega\|_{L^\infty L^2}^2+\nu\|e^{c\nu^{\frac{1}{2}}t} (\partial_y,|k|)\omega\|_{L^2 L^2}^2\\
&\lesssim \|\omega_0\|_{L^2}^2+\nu^{-1}\|e^{c\nu^{\frac{1}{2}}t} (|k|^{-1}k_1f_1+|k|^{-1}k_3f_3,f_2,|k|^{-1}f_4)\|_{L^2 L^2}^2\\
&\quad+(\nu^{\f12}+\nu|k|^2)\|e^{c\nu^{\frac{1}{2}}t}\omega\|_{L^2 L^2}^2 + \nu^{-\f13}|k_1|^{\f43}|k|^{-\f23} \|e^{c\nu^{\frac{1}{2}}t}u\|_{L^2L^2}^{\f43}\|e^{c\nu^{\frac{1}{2}}t}\omega\|_{L^2L^2}^{\f23}.
\end{aligned}
\end{equation}
\end{lemma}

\begin{proof}
The standard energy estimate gives
\begin{equation*}
\begin{aligned}
&\frac{1}{2}\frac{\mathrm{d}}{\mathrm{d}t}\|\omega\|_{L^2}^2+\nu \|(\partial_y,|k|)\omega\|_{L^2}^2\\
&=\mathrm{Re}\Big(\langle -ik_1f_1-ik_3 f_3-f_4,\omega\rangle+\langle f_2,\partial_y\omega\rangle+(\nu \partial_y\omega-f_2)\overline{\omega}|_{- 1}^1\Big),
\end{aligned}
\end{equation*}
which gives
\begin{equation*}
\begin{aligned}
\frac{1}{2}\frac{\mathrm{d}}{\mathrm{d}t}\|\omega\|_{L^2}^2+\nu \|(\partial_y,|k|)\omega\|_{L^2}^2&\leq \|(ik_1f_1+ik_3f_3+f_4)\|_{L^2}\|\omega\|_{L^2}+\|f_2\|_{L^2}\|\partial_y\omega\|_{L^2}\\
&\quad+\|\nu\partial_y\omega-f_2\|_{l^1(\{y=\pm 1\})}\|\omega\|_{L^\infty}.
\end{aligned}
\end{equation*}

To handle the boundary term, we introduce
$$s_{1}(y)=\frac{\sinh (|k|(1+y))}{\sinh (2|k|)}, \quad s_{-1}(y)=\frac{\sinh (|k|(1-y))}{\sinh (2|k|)}.$$
It is easy to see that $\langle \partial_t^j \omega,s_i\rangle=0, \  i=\pm 1,j=0,1.$ Therefore, we have
 \begin{equation*}
\begin{aligned}
0=&\langle \partial_t \omega,s_i\rangle=\langle \nu (\partial_y^2-|k|^2)\omega-ik_1(1-y^2)\omega-2ik_1\varphi\\
&\qquad\qquad\qquad -ik_1f_1-\partial_y f_2-ik_3f_3-f_4,s_i\rangle\\
=&\langle \nu\omega,(\partial_y^2-|k|^2)s_i\rangle+ \langle  -ik_1(1-y^2)\omega,s_i\rangle+\langle -ik_1f_1-ik_3f_3-f_4,s_i\rangle\\
&+\langle -2ik_1\varphi, s_i\rangle+\langle f_2,\partial_ys_i\rangle +(-\nu \omega\partial_y s_i+(\nu\partial_y \omega-f_2)s_i)|_{-1}^1.
\end{aligned}
\end{equation*}
Note that
$$ \langle-ik_1(1-y^2) \omega-2ik_1\varphi,s_i\rangle=\langle ik_1\varphi,4y\partial_y s_i\rangle+\langle -4ik_1\varphi,s_i\rangle,$$
then we obtain
 \begin{equation*}
\begin{aligned}
&\left|(\nu\partial_y\omega-f_2)|_{y=i}\right|\\
&=|\langle ik_1\varphi,4y\partial_y s_i\rangle+ \langle -4ik_1\varphi,s_i\rangle+ \langle
-ik_1f_1-ik_3f_3-f_4,s_i\rangle\\
&\qquad+\langle f_2,\partial_ys_i\rangle-(\nu\omega\partial_y s_i)|_{-1}^1|\\
&\leq 4|k_1|\|\varphi\|_{L^2}(\|\partial_ys_i\|_{L^2}+\|s_i\|_{L^2})+\|k_1f_1+k_3f_3-if_4\|_{L^2}\|s_i\|_{L^2}\\
&\qquad+\|f_2\|_{L^2}\|\partial_ys_i\|_{L^2}+\nu\|\partial_y s_i\|_{l^1(\{y=\pm 1\})}\|\omega\|_{L^\infty}\\
&\lesssim |k_1||k|^{\frac{1}{2}}\|\varphi\|_{L^2}+|k|^{\frac{1}{2}}\||k|^{-1}(k_1f_1+k_3f_3-if_4),f_2\|_{L^2}+\nu|k|\|\omega\|_{L^\infty},
\end{aligned}
\end{equation*}
where we have used the following facts
\begin{align*}
&|\partial_ys_i(\pm i)|=|k|\coth (2|k|)\leq C|k|,\\
&\|(\partial_y,|k|)s_i\|_{L^2}^2=-\langle s_i,(\partial_y^2-|k|^2)s_i\rangle+(\partial_ys_i\cdot s_i)|_{-1}^1\leq C|k|,\ i=\pm 1.
\end{align*}
Therefore, we deduce that
 \begin{equation*}
\begin{aligned}
\frac{\mathrm{d}}{\mathrm{d}t}&\|\omega\|_{L^2}^2+\f32\nu\|(\partial_y,|k|)\omega\|_{L^2}^2\\
\lesssim& \nu^{-1}\|(|k|^{-1}k_1f_1+|k|^{-1}k_3f_3,f_2,|k|^{-1}f_4)\|_{L^2}^2+\|\nu\partial_y\omega+f_2\|_{l^1(\{y=\pm 1\})}\|\omega\|_{L^\infty}\\
\lesssim& \nu^{-1}\|(|k|^{-1}k_1f_1+|k|^{-1}k_3f_3,f_2,|k|^{-1}f_4)\|_{L^2}^2+\|\omega\|_{L^\infty}
\big( |k_1||k|^{\frac{1}{2}}\|\varphi\|_{L^2}\\
&+|k|^{\frac{1}{2}}\|(|k|^{-1}(k_1f_1+k_3f_3),f_2,|k|^{-1}f_4)\|_{L^2}+\nu|k|\|\omega\|_{L^\infty}\big)\\
\lesssim&\nu^{-1}\|(|k|^{-1}k_1f_1+|k|^{-1}k_3f_3,f_2,|k|^{-1}f_4)\|_{L^2}^2\\
&+
\nu|k|\|\omega\|_{L^\infty}^2 +|k_1||k|^{-\frac{1}{2}}\|u\|_{L^2}\|\omega\|_{L^\infty}\\
\lesssim& \nu^{-1}\|(|k|^{-1}k_1f_1+|k|^{-1}k_3f_3,f_2,|k|^{-1}f_4)\|_{L^2}^2\\
&
+\nu|k|\|\omega\|_{L^2}\|(\partial_y,|k|)\omega\|_{L^2} +|k_1||k|^{-\f12}\|u\|_{L^2}\|\omega\|_{L^2}^{\f12}\|(\partial_y,|k|)\omega\|_{L^2}^{\f12},
\end{aligned}
\end{equation*}
and
 \begin{equation*}
\begin{aligned}
&\frac{\mathrm{d}}{\mathrm{d}t}\|e^{c\nu^{\frac{1}{2}}t} \omega\|_{L^2}^2+ \nu\|e^{c\nu^{\frac{1}{2}}t}(\partial_y,|k|)\omega\|_{L^2}^2\\
&\lesssim \nu^{-1}\|e^{c\nu^{\frac{1}{2}}t}  (|k|^{-1}k_1f_1+|k|^{-1}k_3f_3,f_2,|k|^{-1}f_4)\|_{L^2}^2\\
&\quad+(\nu^{\f12}+\nu|k|^2)\|e^{c\nu^{\frac{1}{2}}t} \omega\|_{L^2}^2 + \nu^{-\f13}|k_1|^{\f43}|k|^{-\f23} \|e^{c\nu^{\frac{1}{2}}t}u\|_{L^2}^{\f43}\|e^{c\nu^{\frac{1}{2}}t}\omega\|_{L^2}^{\f23},
\end{aligned}
\end{equation*}
which gives our result.
\end{proof}

Now we are in a position to prove  Theorem  \ref{thm:sp-no-slip}.

\begin{proof}[Proof of Theorem  \ref{thm:sp-no-slip}]
By Lemma \ref{energy-w-u} with $f_4=0$, we have
\begin{equation}
\begin{aligned}\label{energy-u}
&|k|^2 \|e^{c\nu^{\frac{1}{2}}t} u\|_{L^\infty L^2}^2 +\frac{3}{2} \nu |k|^2 \|e^{c\nu^{\frac{1}{2}}t} \omega \|_{L^2 L^2}^2\\
&\leq   |k|^2 \|u_0\|_{L^2}^2 +\big((5/2)|k_1|+2c\nu^{\frac{1}{2}}|k|\big)  |k| \|e^{c\nu^{\frac{1}{2}}t} u\|_{L^2 L^2}^2\\
&\qquad+4\nu^{-1} \|e^{c\nu^{\frac{1}{2}}t} (f_1,f_2,f_3)\|_{L^2 L^2}^2.
\end{aligned}
\end{equation}

 We divide the proof into two cases. \smallskip

\textbf{Case 1.} $\nu |k|^3 |k_1|\leq 2$.\smallskip

By Lemma \ref{sp-inhomo} and Lemma \ref{sp-Homo}, we have
\begin{align}\label{est:wL2L2-nus-000}
\nu^{\f34} |k_1|^{\f14} \|e^{c\nu^{\frac{1}{2}}t} \omega\|_{L^2 L^2}^2 \lesssim& |k|^{-\f34}\|(\partial_y^2-|k|^2)\omega_0\|_{L^2}^2 \nonumber+ \nu^{-\f34} |k_1|^{-\f14} \|e^{c\nu^{\frac{1}{2}}t} (f_1,f_2,f_3)\|_{L^2 L^2}^2\\
\lesssim & |k|^{-\f34}\|(\partial_y^2-|k|^2)\omega_0\|_{L^2}^2 + \nu^{-1 } |k|^{-\f34}  \|e^{c\nu^{\frac{1}{2}}t} (f_1,f_2,f_3)\|_{L^2 L^2}^2,
\end{align}
and
\begin{align}\label{est:uL2L2-nus-000}
|k_1| \|e^{c\nu^{\frac{1}{2}}t} u\|_{L^2 L^2}^2\lesssim&  \|(\partial_y^2-|k|^2)\omega_0\|_{L^2}^2+ \nu^{-1} \|e^{c\nu^{\frac{1}{2}}t}(f_1,f_2,f_3)\|_{L^2 L^2}^2,
\end{align}
 therefore we get by \eqref{energy-u}, \eqref{est:uL2L2-nus-000} and $\nu^{\f12}|k|\leq C|k|^{-\f12}|k_1|^{\f12}\leq C|k_1|$ that
 \begin{equation}\label{est:k2u-LinftyL2-00}
\begin{aligned}
|k|& \|e^{c\nu^{\frac{1}{2}}t} u\|_{L^\infty L^2}^2+ \nu |k| \|e^{c\nu^{\frac{1}{2}}t} \omega\|_{ L^2 L^2 }^2+|k_1| \|e^{c\nu^{\frac{1}{2}}t} u\|_{L^2 L^2}^2\\
 \lesssim& \|(\partial_y^2-|k|^2)\omega_0\|_{L^2}^2+\nu^{-1} \|e^{c\nu^{\frac{1}{2}}t} (f_1,f_2,f_3)\|_{L^2 L^2}^2.
\end{aligned}
\end{equation}

On the other hand, note that $\nu |k|^3 |k_1|^{-1}\leq 2 \Rightarrow \nu|k|^2\leq C$, then we get by \eqref{est:k2u-LinftyL2-00} that
 \begin{equation}\label{est:k2u-LinftyL2-01}
\begin{aligned}
&\nu^{\f12} |k|^2 \|e^{c\nu^{\frac{1}{2}}t} u\|_{L^\infty L^2}^2+ \nu^{\f32} |k|^2 \|e^{c\nu^{\frac{1}{2}}t} \omega\|_{ L^2 L^2 }^2+\nu^{\f12} |k_1| |k| \|e^{c\nu^{\frac{1}{2}}t} u\|_{L^2 L^2}^2\\
 &=(\nu|k|^2)^{\f12}\big(|k| \|e^{c\nu^{\frac{1}{2}}t} u\|_{L^\infty L^2}^2+ \nu |k| \|e^{c\nu^{\frac{1}{2}}t} \omega\|_{ L^2 L^2 }^2+|k_1| \|e^{c\nu^{\frac{1}{2}}t} u\|_{L^2 L^2}^2\big)\\
 &\lesssim\|(\partial_y^2-|k|^2)\omega_0\|_{L^2}^2 +\nu^{-1} \|e^{c\nu^{\frac{1}{2}}t} (f_1,f_2,f_3)\|_{L^2 L^2}^2 .
\end{aligned}
\end{equation}
Therefore, we conclude by \eqref{est:wL2L2-nus-000}, \eqref{est:k2u-LinftyL2-00} and \eqref{est:k2u-LinftyL2-01} that
 \begin{equation}\nonumber
\begin{aligned}
 &(|k|+\nu^{\f12} |k|^2) \|e^{c\nu^{\frac{1}{2}}t} u\|_{L^\infty L^2}^2+(\nu|k|+ \nu^{\f32} |k|^2) \|e^{c\nu^{\frac{1}{2}}t} \omega\|_{ L^2 L^2 }^2\\
 &\quad+(|k_1|+\nu^{\f12} |k_1| |k|) \|e^{c\nu^{\frac{1}{2}}t} u\|_{L^2 L^2}^2+{\nu^{\f34} |k|^{\f34} |k_1|^{\f14} \|e^{c\nu^{\frac{1}{2}}t} \omega\|_{ L^2 L^2 }^2}\\
 &\lesssim \|(\partial_y^2-|k|^2)\omega_0\|_{L^2}^2 +\nu^{-1} \|e^{c\nu^{\frac{1}{2}}t} (f_1,f_2,f_3)\|_{L^2 L^2}^2 .
\end{aligned}
\end{equation}
{Here we used $\nu^{-\f34} |k_1|^{-\f14}\leq C\nu^{-1}$.}

By Lemma \ref{lemma-w-L-infty}, \eqref{est:wL2L2-nus-000} and \eqref{est:uL2L2-nus-000}, we have
 \begin{equation}\nonumber
\begin{aligned}
&\|e^{c\nu^{\frac{1}{2}}t} \omega\|_{L^\infty L^2}^2+\nu\|e^{c\nu^{\frac{1}{2}}t} (\partial_y,|k|)\omega\|_{L^2 L^2}^2\\
&\lesssim \|\omega_0\|_{L^2}^2+\nu^{-1}\|e^{c\nu^{\frac{1}{2}}t} (|k|^{-1}k_1f_1+|k|^{-1}k_3f_3,f_2)\|_{L^2 L^2}^2\\
&\quad+(\nu^{\f12}+\nu|k|^2)\|e^{c\nu^{\frac{1}{2}}t}\omega\|_{L^2 L^2}^2 + \nu^{-\f13}|k_1|^{\f43}|k|^{-\f23} \|e^{c\nu^{\frac{1}{2}}t}u\|_{L^2L^2}^{\f43}\|e^{c\nu^{\frac{1}{2}}t}\omega\|_{L^2L^2}^{\f23}\\
&\lesssim \|\omega_0\|_{L^2}^2+\nu^{-1}\|e^{c\nu^{\frac{1}{2}}t} (f_1,f_2,f_3))\|_{L^2 L^2}^2\\
&\quad+\nu^{\f12}  \|e^{c\nu^{\frac{1}{2}}t}u\|_{L^2 L^2}\|e^{c\nu^{\frac{1}{2}}t}(\partial_y,|k|)\omega\|_{L^2 L^2}+\nu|k|^2\|e^{c\nu^{\frac{1}{2}}t}\omega\|_{L^2 L^2}^2 \\
&\quad + \nu^{-\f7{12}}|k_1|^{\f12}\|e^{c\nu^{\frac{1}{2}}t}u\|_{L^2L^2}^2+\nu^{\f16} |k|^{\f34} |k_1|^{\f14} \|e^{c\nu^{\frac{1}{2}}t}\omega\|_{L^2L^2}^{2}\\
&\leq \|\omega_0\|_{L^2}^2+\nu^{-1}\|e^{c\nu^{\frac{1}{2}}t} (f_1,f_2,f_3)\|_{L^2 L^2}^2+\f12 \|e^{c\nu^{\frac{1}{2}}t}(\partial_y,|k|)\omega\|_{L^2 L^2}^2\\
&\quad +C\|e^{c\nu^{\frac{1}{2}}t}u\|_{L^2 L^2}^2+ |k|\|(\partial_y^2-|k|^2)\omega_0\|_{L^2}^2 +\nu^{-1} |k|\|e^{c\nu^{\frac{1}{2}}t} (f_1,f_2,f_3)\|_{L^2 L^2}^2\\
&\quad + \nu^{-\f7{12}}(|k_1|^{\f12}\|e^{c\nu^{\frac{1}{2}}t}u\|_{L^2L^2}^2+\nu^{\f34} |k|^{\f34} |k_1|^{\f14} \|e^{c\nu^{\frac{1}{2}}t}\omega\|_{L^2L^2}^{2}),
\end{aligned}
\end{equation}
which implies the desired estimates.

\smallskip

\textbf{Case 2.} $\nu |k|^3 |k_1|\geq 2$.\smallskip

For $0\leq c\leq \frac{\sqrt{2}}{16}$, we have
\begin{equation*}
\begin{aligned}
&\big((5/2) |k_1||k|^{-1}+2c\nu^{\frac{1}{2}}\big)\|e^{c\nu^{\frac{1}{2}}t} u\|_{L^2}^2\\
&\leq \big((5/2) |k_1||k|^{-3}+2c\nu^{\frac{1}{2}}|k|^{-2}\big) \|e^{c\nu^{\frac{1}{2}}t} \omega\|_{L^2}^2
\leq \frac{11}{8}\nu \|e^{c\nu^{\frac{1}{2}}t} \omega\|_{L^2}^2,
\end{aligned}
\end{equation*}
which along with \eqref{energy-u} implies
\begin{equation}\label{est:nubig-uw-000}
\begin{aligned}
&|k|^2\|e^{c\nu^{\frac{1}{2}}t} u\|_{L^\infty L^2 }^2+ (\nu |k|^2/8) \|e^{c\nu^{\frac{1}{2}}t} \omega\|_{L^2 L^2}^2\\
&\leq \|\omega_0\|_{L^2}^2+4\nu^{-1}\|e^{c\nu^{\frac{1}{2}}t} (f_1,f_2,f_3)\|_{L^2L^2}^2.
\end{aligned}
\end{equation}
Thanks to $\nu |k|^3|k_1|^{-1}\geq 2$, $\nu\ll 1$, we have
\begin{align*}
   &|k|+\nu^{\f12} |k|^2 \leq C|k|^2,\quad
    \nu|k|+ \nu^{\f32} |k|^2+ { \nu^{\f34}|k|^{\f34} |k_1|^{\f14} }\leq C\nu |k|^2,\\
    &|k|^{-2}(|k_1|+\nu^{\f12}|k_1||k|)\leq C\nu |k|^2,
\end{align*}
then by \eqref{est:nubig-uw-000}, we have
 \begin{equation}\nonumber
\begin{aligned}
 &(|k|+\nu^{\f12} |k|^2) \|e^{c\nu^{\frac{1}{2}}t} u\|_{L^\infty L^2}^2+(\nu|k|+ \nu^{\f32} |k|^2) \|e^{c\nu^{\frac{1}{2}}t} \omega\|_{ L^2 L^2 }^2\\
 &\quad+(|k_1|+\nu^{\f12} |k_1| |k|) \|e^{c\nu^{\frac{1}{2}}t} u\|_{L^2 L^2}^2+{\nu^{\f34}|k|^{\f34} |k_1|^{\f14} \|e^{c\nu^{\frac{1}{2}}t} \omega\|_{ L^2 L^2 }^2}\\
&\lesssim |k|^2\|e^{c\nu^{\frac{1}{2}}t} u\|_{L^\infty L^2 }^2+ (\nu |k|^2/8) \|e^{c\nu^{\frac{1}{2}}t} \omega\|_{L^2 L^2}^2\\
 &\lesssim \| \omega_0\|_{L^2}^2 +\nu^{-1} \|e^{c\nu^{\frac{1}{2}}t} (f_1,f_2,f_3)\|_{L^2 L^2}^2 .
\end{aligned}
\end{equation}

Moreover, by Lemma \ref{lemma-w-L-infty}, we have
\begin{equation*}
\begin{aligned}
&\|e^{c\nu^{\frac{1}{2}}t} \omega\|_{L^\infty L^2}^2+\nu\|e^{c\nu^{\frac{1}{2}}t} (\partial_y,|k|)\omega\|_{L^2 L^2}^2\\
&\lesssim\|\omega_0\|_{L^2}^2+\nu^{-1}\|e^{c\nu^{\frac{1}{2}}t} (|k|^{-1}k_1f_1+|k|^{-1}k_3f_3,f_2)\|_{L^2 L^2}^2\\
&\quad+\nu|k|^2\|e^{c\nu^{\frac{1}{2}}t}\omega\|_{L^2 L^2}^2+\nu^{-\f13}|k_1|^{\f43}|k|^{-\f23}\| e^{c\nu^{\frac{1}{2}}t} u\|_{L^2 L^2}^{\f43}\| e^{c\nu^{\frac{1}{2}}t}\omega\|_{L^2 L^2}^{\f23}\\
&\lesssim\|\omega_0\|_{L^2}^2+\nu^{-1}\|e^{c\nu^{\frac{1}{2}}t} (|k|^{-1}k_1f_1+|k|^{-1}k_3f_3,f_2)\|_{L^2 L^2}^2\\
&\quad+\nu|k|^2\|e^{c\nu^{\frac{1}{2}}t}\omega\|_{L^2 L^2}^2+\nu^{-1}|k_1|^{2}|k|^{-2}\| e^{c\nu^{\frac{1}{2}}t} u\|_{L^2 L^2}^2\\
&\lesssim\|\omega_0\|_{L^2}^2+
 \nu^{-1}\|e^{c\nu^{\frac{1}{2}}t} (f_1,f_2,f_3)\|_{L^2L^2}^2,
\end{aligned}
\end{equation*}
which yields the desired estimates.

This completes the proof of Theorem  \ref{thm:sp-no-slip}.
\end{proof}

\section{Nonlinear interaction}\label{sec:iter}

This section is devoted to the estimates for nonlinear terms.

\subsection{Anisotropic product laws}
In this subsection, we denote $\partial_1=\partial_x,\ \partial_2=\partial_2,\ \partial_3=\partial_z.$
The following lemmas come from  Lemma 11.1, Lemma 11.2 and Lemma 11.3 in \cite{CWZ-mem} respectively.

\begin{lemma}\label{lemma-A.4}
For  $\{i,j\}=\{1,3\},\ i\neq j$, it holds that
\begin{align}
\label{f1}&\|f_1f_2\|_{L^2}\lesssim \big(\|\partial_j f_1\|_{H^1}+\|f_1\|_{H^1}\big)\big(\|\pa_if_2\|_{L^2}+
\|f_2\|_{L^2}\big),\\
\label{f2}&\|f_1f_2\|_{L^2}+\|\partial_i (f_1f_2)\|_{L^2}\lesssim \|(\partial_x\partial_zf_1,\partial_xf_1,\partial_zf_1,f_1)\|_{H^1} \|(\partial_i f_2,f_2)\|_{L^2},\\
\label{f3}&\|f_1f_2\|_{L^2}+\|\partial_i (f_1f_2)\|_{L^2}\lesssim \|(\partial_i f_1,f_1)\|_{H^1}\|(\partial_x\partial_zf_2, \partial_xf_2,\partial_zf_2,f_2)\|_{L^2},\\
\label{f4}&\|\nabla(f_1f_2)\|_{L^2}\lesssim \|(\partial_x\partial_zf_1,\partial_xf_1,\partial_zf_1,f_1)\|_{H^1}\|f_2\|_{H^1},\\ \label{f5}&\|\nabla(f_1f_2)\|_{L^2}\lesssim
\big(\|\partial_j f_1\|_{H^1}+\|f_1\|_{H^1})(\|\partial_i f_2\|_{H^1}+\|f_2\|_{H^1}\big),
\end{align}
and
\begin{align}
 \label{f6}\|\nabla(f_1f_2)\|_{L^2}\lesssim \big(&\|(\partial_j f_1,f_1)\|_{H^2}\|(\partial_i f_2,f_2)\|_{L^2}\\
 &+
\|(\partial_x\partial_zf_1,\partial_zf_1,\partial_zf_1,f_1)\|_{L^2}\|f_2\|_{H^2}\big).\nonumber
\end{align}
\end{lemma}

\begin{lemma}\label{Lem: bilinear zero nonzero}
  If $\partial_xf_1=0$, then it holds that
  \begin{align*}
    &\|f_1f_2\|_{L^2}\lesssim\|f_1\|_{H^1}(\|f_2\|_{L^2}+\|\partial_zf_2\|_{L^2}),\\
    &\|(\partial_x,\partial_z)(f_1f_2)\|_{L^2} \lesssim\big(\|f_1\|_{H^1}+\|\partial_zf_1\|_{H^1}\big)\big(\|f_2\|_{L^2}
    +\|(\partial_x,\partial_z)f_2\|_{L^2}\big),\\
    &\|(\partial_x,\partial_z)(f_1f_2)\|_{L^2} \lesssim
\big(\|f_1\|_{L^2}+\|\partial_zf_1\|_{L^2}\big)\big(\|f_2\|_{H^1}
    +\|(\partial_x,\partial_z)f_2\|_{H^1}\big),\\
    &\|\partial_x(f_1f_2)\|_{L^2}\lesssim\|f_1\|_{H^1}\big(\|\partial_xf_2\|_{L^2}
    +\|\partial_z\partial_x f_2\|_{L^2}\big).
  \end{align*}
\end{lemma}

\begin{lemma}\label{lemma-A.1}If $\partial_xf_1=0$, then it holds that
\begin{align*}
&\|f_1\|_{L^{\infty}}\lesssim\big(\|f_1\|_{H^1}+\|\partial_zf_1\|_{H^1}\big),\\
&\|\nabla(f_1f_2)\|_{L^2}\lesssim(\|f_1\|_{H^1}+\|\partial_zf_1\|_{H^1})\|f_2\|_{H^1},\\
&\|\nabla(f_1f_2)\|_{L^2}\lesssim\|f_1\|_{H^1}(\|f_2\|_{H^1}+\|\partial_zf_2\|_{H^1}).
\end{align*}
\end{lemma}

\subsection{The velocity estimates in terms of the energy}

\begin{lemma}\label{lem:u-relation}
  It holds that for $j \ge 0$,
  \begin{align*}
  &\|\nabla^j (\partial_x,\partial_z)\partial_xu_{\neq}\|_{L^{2}}\lesssim\| \nabla^j (\partial_x^2+\partial_z^2)u_{3,\neq}\|_{L^2}+\|  \nabla^{j+1}(\partial_x,\partial_z) u_{2,\neq}\|_{L^2},\\&
 \|\partial_x\nabla u_{\neq}\|_{L^2} +\|\partial_z\nabla u_{\neq}\|_{L^2}\lesssim \|\nabla\omega_{2,\neq}\|_{L^2}+
  \|\Delta u_{2,\neq}\|_{L^2},\\
&\|\Delta(\partial_x,\partial_z)u_{\neq}\|_{L^2}\leq \|\Delta\omega_{2,\neq}\|_{L^2}+
  \|\nabla\Delta u_{2,\neq}\|_{L^2},\\
  &\|\Lambda_{x,z}^{1/2}\Delta u_{\neq}\|_{L^2}\lesssim \|\Lambda_{x,z}^{-1/2}\Delta\omega_{2,\neq}\|_{L^2}+
  \|\Lambda_{x,z}^{-1/2}\nabla\Delta u_{2,\neq}\|_{L^2}.  \end{align*}
\end{lemma}

\begin{proof}
Thanks to $\text{div}u_{\neq}= \partial_xu_{1,\neq}+\partial_yu_{2,\neq}+\partial_zu_{3,\neq}=0,$  we have
\begin{align*}
&\|\nabla^j (\partial_x,\partial_z)\partial_xu_{\neq}\|_{L^{2}}\\
&\leq
\|\nabla^j (\partial_x,\partial_z)\partial_xu_{1,\neq}\|_{L^{2}}
+\|\nabla^j (\partial_x,\partial_z)\partial_xu_{2,\neq}\|_{L^{2}}
+\|\nabla^j (\partial_x,\partial_z)\partial_xu_{3,\neq}\|_{L^{2}}\\
&\leq\|\nabla^j (\partial_x,\partial_z)\partial_yu_{2,\neq}\|_{L^{2}}
+\|\nabla^j (\partial_x,\partial_z)\partial_zu_{3,\neq}\|_{L^{2}}
\\&\quad+\|\nabla^j (\partial_x,\partial_z)\partial_xu_{2,\neq}\|_{L^{2}}
+\|\nabla^j (\partial_x,\partial_z)\partial_xu_{3,\neq}\|_{L^{2}}\\ &\lesssim
\|\nabla^{j +1}(\partial_x,\partial_z)u_{2,\neq}\|_{L^{2}}
+\|\nabla^j (\partial_x^2,\partial_x\partial_z,\partial_z^2)u_{3,\neq}\|_{L^{2}}\\ &\lesssim \| \nabla^j (\partial_x^2+\partial_z^2)u_{3,\neq}\|_{L^2}+\| \nabla^{j+1}(\partial_x,\partial_z) u_{2,\neq}\|_{L^2}.
\end{align*}

Using the formula $\|(\partial_x,\partial_z)(f_1,f_2)\|^2_{L^2}=\|(\partial_zf_1-\partial_xf_2,\partial_xf_1+\partial_zf_2)\|^2_{L^2}$, we can deduce that
  \begin{align*}
   &\|(\partial_x,\partial_z)( \nabla u_{1,\neq},\nabla u_{3,\neq})\|^2_{L^2}= \|\nabla\omega_{2,\neq}\|^2_{L^2}+
  \|\nabla\partial_y u_{2,\neq}\|^2_{L^2},\\
  &\|\Delta(\partial_x,\partial_z)(u_{1,\neq},u_{3,\neq})\|^2_{L^2}= \|\Delta\omega_{2,\neq}\|^2_{L^2}+
  \|\Delta\partial_y u_{2,\neq}\|^2_{L^2},\\
   &\|\Lambda_{x,z}^{1/2}\Delta(u_{1,\neq},u_{3,\neq})\|^2_{L^2}\lesssim \|\Lambda_{x,z}^{-1/2}\Delta\omega_{2,\neq}\|^2_{L^2}+
  \|\Lambda_{x,z}^{-1/2}\Delta\partial_y u_{2,\neq}\|^2_{L^2},
\end{align*}
which give the other inequalities.
\end{proof}

\begin{lemma}\label{lem:u23-zero}
  It holds that for $j \in\{2,3\}$,
  \begin{align*}
    & \|\overline{u}_2\|_{H^2}+\|\nabla\overline{u}_2\|_{H^1}
    +\|\overline{u}_3\|_{H^1}+\|\partial_z\overline{u}_3\|_{H^1} \lesssim E_2,\\
    &\|\overline{u}_j\|_{L^\infty
    L^\infty}+\nu^{\f12}\|\nabla\overline{u}_j\|_{L^2L^\infty} \lesssim E_2,\\
    &\|\nabla(\overline{u}_j f)\|_{L^2}+ \|\overline{u}_j \nabla f\|_{L^2} \lesssim
    E_2\| f\|_{H^1}.
    \end{align*}
\end{lemma}
\begin{proof}
Thanks to $\partial_y\overline{u}_2+\partial_z\overline{u}_3=0 $,  we have
\begin{align*}
&\|\overline{u}_2\|_{H^2}+\|\nabla\overline{u}_2\|_{H^1}
    +\|\overline{u}_3\|_{H^1}+\|\partial_z\overline{u}_3\|_{H^1}\\
    &\lesssim
    \|\Delta\overline{u}_2\|_{L^2}+\|\nabla\overline{u}_3\|_{L^2}
    +\|\partial_y\overline{u}_2\|_{H^1}\\
&\lesssim\|\Delta\overline{u}_2\|_{L^2}+\|\nabla\overline{u}_3\|_{L^2} \lesssim E_2.
\end{align*}
By Lemma \ref{lemma-A.1},  we get
\begin{align*}
&\|\overline{u}_2\|_{L^{\infty}}+\|\overline{u}_3\|_{L^{\infty}}\lesssim \|\overline{u}_2\|_{H^1}+\|\partial_z\overline{u}_2\|_{H^1}+\|\overline{u}_3\|_{H^1}
+\|\partial_z\overline{u}_3\|_{H^1}\lesssim E_2
\end{align*} for any $t\in[0,T]$,  and
\begin{align*}
&\|\nabla\overline{u}_2\|_{L^2L^{\infty}}\lesssim\|\nabla\overline{u}_2\|_{L^2H^2}\lesssim\|\nabla\Delta\bar{u}^2\|_{L^2L^2}\lesssim\nu^{-\f12}E_2,\\
&\|\nabla\overline{u}_3\|_{L^{\infty}}
\lesssim \|\partial_z\nabla\overline{u}_3\|_{H^1}+\|\nabla\overline{u}_3\|_{H^1}\lesssim \|\nabla\Delta\overline{u}_2\|_{L^2}+\|\Delta\overline{u}_3\|_{L^2}
\end{align*}
for  any $t\in[0,T],$ and then
\begin{align*}
&\|\nabla\overline{u}_3\|_{L^2L^{\infty}}\lesssim\|\nabla\Delta\overline{u}_2\|_{L^2L^2}+\|\Delta\overline{u}_3\|_{L^2L^2}
\lesssim \nu^{-\f12}E_2.
\end{align*}

By Lemma \ref{lemma-A.1} again, we have
\begin{align*}
   \|\nabla(\overline{u}_j f)\|_{L^2}+ \|\overline{u}_j \nabla f\|_{L^2} &\leq\big(\|\overline{u}_j\|_{H^1}+\|\partial_z\overline{u}_j\|_{H^1}\big)\| f\|_{H^1}+\|\overline{u}_j \|_{L^\infty
    }\|\nabla f\|_{L^2}\\&\lesssim E_2\| f\|_{H^1},\qquad
    j \in\{2,3\}.
  \end{align*}
  \end{proof}

\begin{lemma}\label{lem:u-nonzero}
  It holds that
  \begin{align*}
   &\|e^{c\nu^{\f12}t}\partial_x\nabla u_{2,\neq}\|_{L^2L^2}+\|e^{c\nu^{\f12}t}\partial_x(\partial_x,\partial_z) u_{j,\neq}\|_{L^2L^2}\leq  \nu^{-\f{1}{4}}E_3,\ \ j\in\{2,3\},\\
   &\|e^{c\nu^{\f12}t}\partial_x^2 u_{1,\neq}\|_{L^2L^2}\lesssim\nu^{-\f{1}{4}}E_3,\\
   &{\nu^{1/2}\Big(\|e^{c\nu^{\f12}t}\partial_z\partial_x  u_{1,\neq}\|_{L^2L^2}+\|e^{c\nu^{\f12}t}\partial_z\nabla (u_{2,\neq},u_{3,\neq})\|_{L^2L^2}\Big)\lesssim E_3}.
  \end{align*}
\end{lemma}
\begin{proof}
The first inequality of this lemma can be deduced by using the definitions of $E_{3,0}$ and $E_{3,2}$ in \eqref{eq:E30-define-00} and \eqref{eq:E32-define-00} respectively.  Due to  $\partial_x^2 u_{1,\neq}=-\partial_x\partial_y u_{2,\neq}-\partial_x\partial_z u_{3,\neq} $, the second inequality follows from the first inequality.

Note that $\partial_z\partial_x u_{1,\neq}=-\partial_z\partial_y u_{2,\neq}-\partial_z\partial_z u_{3,\neq}$, then we have
\begin{align*}
\|e^{c\nu^{\f12}t} \partial_z \partial_x u_{1,\neq} \|_{ L^2 L^2}^2&\leq \|e^{c\nu^{\f12}t} \partial_z\partial_y u_{2,\neq} \|_{L^2 L^2}^2+ \|e^{c\nu^{\f12}t} \partial_z \partial_z  u_{3,\neq} \|_{L^2 L^2}^2\\
&\lesssim \|e^{c\nu^{\f12}t} \partial_z\nabla u_{2,\neq} \|_{L^2 L^2}^2+ \|e^{c\nu^{\f12}t} \Lambda_{x,z}^{3/2}\nabla u_{3,\neq}\|_{L^2 L^2}^2\\\
&\lesssim \nu^{-1} E_{3,0}^2+\nu^{-1} E_{3,2}^2\lesssim \nu^{-1} E_{3}^2,
 \end{align*}
 which yields the first part of the last inequality.

Notice that
\begin{align*}
&\|\partial_z\nabla (u_{2,\neq},u_{3,\neq})\|_{L^2}^2\leq\|\Lambda_{x,z}^{3/2}\nabla u_{3,\neq}\|_{L^2}^2+ \|\Lambda_{x,z}^{1/2}\Delta u_{2,\neq}\|_{L^2}^2,
      \end{align*}
 which along with Lemma \ref{lem:u-relation} gives
  \begin{align*}&\nu \|e^{c\nu^{\f12}t}\partial_z\nabla (u_{2,\neq},u_{3,\neq})\|_{L^2L^2}^2 \\
      &\lesssim \nu\|e^{c\nu^{\f12}t}\Lambda_{x,z}^{3/2}\nabla u_{3,\neq}\|_{L^2L^2}^2+ \nu\|e^{c\nu^{\f12}t}\Lambda_{x,z}^{1/2}\Delta u_{2,\neq}\|_{L^2L^2}^2\\
      &\lesssim E_{3,2}^2+E_{3,1}^2\lesssim E_3^2.
  \end{align*}

This completes the proof of the lemma.
  \end{proof}

\subsection{Interaction between non-zero modes}
\begin{lemma}\label{lem:int-nn}
It holds that
\begin{equation}\label{lemma2.1-1}
\begin{aligned}
& \|e^{2c\nu^{\frac{1}{2}}t} |u_{\not=}|^2 \|_{L^2 L^2}+ \|e^{2c\nu^{\frac{1}{2}}t} u_{\not=}\cdot \nabla u_{\not=} \|_{L^2 L^2}+ \|e^{2c\nu^{\frac{1}{2}}t} \partial_z(u_{\not=}\cdot \nabla u_{3,\not=}) \|_{L^2 L^2}\\
&\qquad + \|e^{2c\nu^{\frac{1}{2}}t} \partial_x(u_{\not=}\cdot \nabla u_{\not=}) \|_{L^2 L^2}+ \|e^{2c\nu^{\frac{1}{2}}t} \nabla (u_{\not=}\cdot \nabla u_{2,\not=})\|_{L^2 L^2}\lesssim \nu^{-1} E_3^2,
\end{aligned}
\end{equation}
and
\begin{equation}\label{lemma2.1-3}
\begin{aligned}
& \|e^{2c\nu^{\frac{1}{2}}t} \nabla (u_{\not=}\cdot \nabla u_{\not=}) \|_{L^2 L^2}\lesssim \nu^{-\frac{11}{8}} E_3^2.
\end{aligned}
\end{equation}
\end{lemma}

\begin{proof}
By \eqref{f1}, we have
\begin{align*}
\||u_{\neq}|^2\|_{L^2}\lesssim& \big(\|\partial_xu_{\neq}\|_{H^1}+\|u_{\neq}\|_{H^1}\big)\big(\|\partial_zu_{\neq}\|_{L^2}+\|u_{\neq}\|_{L^2}\big)\\
\lesssim& \|\nabla\partial_x^2 u_{\neq}\|_{L^2} \|(\partial_x,\partial_z)u_{\neq}\|_{L^2}.
\end{align*}
For $j\in\{1,3\}$, we get by \eqref{f2}  that
\begin{align*}
&\|u_{j,\neq}\partial_j u_{\neq}\|_{L^2}+\|\partial_x(u_{j,\neq}\partial_j u_{\neq})\|_{L^2}\\
& \lesssim \big\|(\partial_x\pa_zu_{j,\neq},\partial_xu_{j,\neq},\partial_zu_{j,\neq},u_{j,\neq})\big\|_{H^1}
\|(\partial_x\partial_j u_{\neq},\partial_j u_{\neq})\|_{L^2}\\
&\lesssim\big\|
\nabla(\partial_x,\partial_z)\partial_xu_{\neq}\big\|_{L^2}
\|(\partial_x,\partial_z)\partial_xu_{\neq}\|_{L^2},
\end{align*}
and by \eqref{f2}  again, we have
\begin{align*}
&\|\partial_z(u_{j,\neq}\partial_j u_{3,\neq})\|_{L^2}\\
&\lesssim \|(\partial_x\partial_zu_{j,\neq},\partial_xu_{j,\neq},
\partial_zu_{j,\neq},u_{j,\neq})\|_{H^1}
\|(\partial_z\partial_j u_{3,\neq},\partial_j u_{3,\neq})\|_{L^2}\\
&\lesssim \|\nabla(\partial_x,\partial_z)\partial_xu_{\neq}\|_{L^2}\|(\partial_x^2+\partial_z^2)
u_{3,\neq}\|_{L^2}.
\end{align*}
For $j=2$, by \eqref{f3}, we have
\begin{align*}
&\|u_{j,\neq}\partial_j u_{\neq}\|_{L^2}+\|(\partial_x,\partial_z)(u_{j,\neq}\partial_j u_{\neq})\|_{L^2}\\ &\lesssim\|(\partial_xu_{j,\neq},\partial_zu_{j,\neq},u_{j,\neq}) \|_{H^{1}}\|(\partial_x\partial_z\partial_j u_{\neq},\partial_x\partial_j u_{\neq},\partial_z\partial_j u_{\neq},
\partial_j u_{\neq})\|_{L^2}\\
&\lesssim \|(\partial_x,\partial_z)\nabla u_{\neq}^2\|_{L^{2}}\|\nabla(\partial_x,\partial_z)\partial_xu_{\neq}\|_{L^2}.
\end{align*}
Summing up,  we get by Lemma \ref{lem:u-relation}  that
\begin{align*}
&\||u_{\neq}|^2\|_{L^2}+\|u_{\neq}\cdot\nabla u_{\neq}\|_{L^2}+\|\partial_x(u_{\neq}\cdot\nabla u_{\neq})\|_{L^2}+\|\partial_z(u_{\neq}\cdot\nabla u_{3,\neq})\|_{L^2}\\
&\lesssim \|\nabla(\partial_x,\partial_z)\partial_xu_{\neq}\|_{L^2}
\big(\|(\partial_x,\partial_z)\partial_xu_{\neq}\|_{L^2}\\
&\quad+
\|(\partial_x^2+\partial_z^2)u_{3,\neq}\|_{L^2}+\|(\partial_x,\partial_z)\nabla u_{2,\neq}\|_{L^{2}}\big)\\
&\lesssim \big(\|(\partial_x,\partial_z)\Delta u_{2,\neq}\|_{L^2}+\|(\partial_x^2+\partial_z^2)\nabla u_{3,\neq}\|_{L^2}\big)\\
&\quad\times \big(\|(\partial_x,\partial_z)\nabla u_{2,\neq}\|_{L^{2}}+\|(\partial_x^2+\partial_z^2) u_{3,\neq}\|_{L^2}\big).
\end{align*}

For $j \in\{1,3\}$, by \eqref{f4}, we have
\begin{align*}
\|\nabla(u_{j,\neq}\partial_j u_{2,\neq})\|_{L^2}\lesssim &\|(\partial_x\partial_zu_{j,\neq},\partial_xu_{j,\neq},\partial_zu_{j,\neq},u_{j,\neq})\|_{H^1}
\|\partial_j u_{2,\neq}\|_{H^1}\\
\lesssim&\|\nabla(\partial_x,\partial_z)\partial_x u_{\neq}\|_{L^2}
\|(\partial_x,\partial_z)\nabla u_{2,\neq}\|_{L^2},
\end{align*}
and for $j=2$, by \eqref{f5}, we have
\begin{align*}
\|\nabla(u_{j,\neq}\partial_j u_{2,\neq})\|_{L^2}\lesssim& \|(\partial_zu_{j,\neq},u_{j,\neq})\|_{H^1}\|(\partial_x\partial_j u_{2,\neq}, \partial_j u_{2,\neq})\|_{H^1}\\
\lesssim &\|(\partial_x,\partial_z)\nabla u_{2,\neq}\|_{L^{2}}\|\partial_x\Delta u_{2,\neq}\|_{L^2}.
\end{align*}
Then it follows from Lemma \ref{lem:u-relation} that
\begin{align*}
&\|\nabla(u_{\neq}\cdot\nabla u_{2,\neq})\|_{L^2}\\
\lesssim&  ( \|(\partial_x,\partial_z)\Delta u_{2,\neq}\|_{L^2}+\|(\partial_x^2+\partial_z^2)\nabla u_{3,\neq}\|_{L^2}) \|(\partial_x,\partial_z)\nabla u_{2,\neq}\|_{L^{2}}.
\end{align*}
This shows that
\begin{align*}
&\|e^{2c\nu^{\f12}t}|u_{\neq}|^2\|_{L^2L^2}^2+\|e^{2c\nu^{\f12}t}u_{\neq}\cdot\nabla u_{\neq}\|_{L^2L^2}^2+\|e^{2c\nu^{\f12}t}\partial_x(u_{\neq}\cdot\nabla u_{\neq})\|_{L^2L^2}^2\\
\nonumber&\quad+\|e^{2c\nu^{\f12}t}\partial_z(u_{\neq}\cdot\nabla u_{3,\neq})\|_{L^2L^2}^2+\|e^{2c\nu^{\f12}t}\nabla(u_{\neq}\cdot\nabla u_{2,\neq})\|_{L^2L^2}^2\\
&\lesssim \big( \|e^{c\nu^{\f12}t}(\partial_x,\partial_z)\Delta u_{2,\neq}\|_{L^2L^2}^2+\|e^{c\nu^{\f12}t}\nabla(\partial_x^2+\partial_z^2) u_{3,\neq}\|_{L^2L^2}^2\big)\\
&\quad\times \big(\|e^{c\nu^{\f12}t}(\partial_x,\partial_z)\nabla u_{2,\neq}\|_{L^{\infty}L^2}^2+\|e^{c\nu^{\f12}t}(\partial_x^2+\partial_z^2) u_{3,\neq}\|_{L^{\infty}L^2}^2\big)\\
&\lesssim (\nu^{-\f32}E_{3,0}^2+\nu^{-\f32}E_{3,3}^2)(\nu^{-\f12}E_{3,0}^2+\nu^{-\f12}E_{3,3}^2)\lesssim \nu^{-2}E_3^4.
\end{align*}
This proves \eqref{lemma2.1-1}.

For $j \in\{1,3\}$, we get by \eqref{f6} and Lemma \ref{lem:u-relation} that
\begin{align*}
\|\nabla(u_{j,\neq}\partial_j u_{\neq})\|_{L^2}^2
\leq& C\Big(\|(\partial_zu_{j,\neq},u_{j,\neq})\|_{H^2}^2\|(\partial_x\partial_j u_{\neq},\partial_j u_{\neq})\|_{L^2}^2\\&\quad+
\|(\partial_x\partial_zu_{j,\neq},\partial_xu_{j,\neq},\partial_zu_{j,\neq},u_{j,\neq})\|_{L^2}^2\|\partial_j u_{\neq}\|_{H^2}^2\Big)\\ \lesssim& \Big(\|\Delta(\partial_x,\partial_z)u_{j,\neq}\|_{L^2}^2\|(\partial_x,\partial_z) \partial_x u_{\neq}\|_{L^2}^2\\
&+
\|(\partial_x,\partial_z) \partial_xu_{j,\neq}\|_{L^2}^2\|\Delta\partial_j u_{\neq}\|_{L^2}^2\Big)\\ \lesssim& \big(\|\nabla\Delta u_{2,\neq}\|_{L^2}^2+\|\Delta\omega_{2,\neq}\|_{L^2}^2\big)\\
&\times \big(\|(\partial_x^2+\partial_z^2 )u_{3,\neq}\|_{L^2}^2+\|(\partial_x,\partial_z)\nabla u_{2,\neq}\|_{L^2}^2\big),
\end{align*}
and for $j=2$,  by \eqref{f5}  and  Lemma \ref{lem:u-relation}, we have
\begin{align*}
\|\nabla(u_{j,\neq}\partial_j u_{\neq})\|_{L^2}^2
\lesssim& \|(\partial_zu_{j,\neq},u_{j,\neq})\|_{H^1}\|(\partial_x\partial_j u_{\neq},\partial_j u_{\neq})\|_{H^1}\\ \lesssim& \|(\partial_x,\partial_z)\nabla u_{2,\neq}\|_{L^2}^2\|\Delta\partial_x u_{\neq}\|_{L^2}^2\\ \leq&C\|(\partial_x,\partial_z)\nabla u_{2,\neq}\|_{L^2}^2\big(\|\nabla\Delta u_{2,\neq}\|_{L^2}^2+\|\Delta\omega_{2,\neq}\|_{L^2}^2\big).
\end{align*}
Thus, we arrive at
\begin{align*}
&\|e^{2c\nu^{\f12}t}\nabla(u_{\neq}\cdot\nabla u_{\neq})\|_{L^2L^2}^2\\
&\lesssim\Big(\|e^{c\nu^{\f12}t}\nabla\Delta u_{2,\neq}\|_{L^2L^2}^2+\|e^{c\nu^{\f12}t}\Delta\omega_{2,\neq}\|_{L^2L^2}^2\Big)\\
&\quad\times
\Big(\|e^{c\nu^{\f12}t}(\partial_x^2+\partial_z^2 )u_{3,\neq}\|_{L^{\infty}L^2}^2+\|e^{c\nu^{\f12}t}(\partial_x,\partial_z)\nabla u_{2,\neq}\|_{L^{\infty}L^2}^2\Big)\\
&{ \lesssim \big(\nu^{-\f{19}{12}}E_{3,0}^2+\nu^{-\f{9}{4}}E_{3,1}^2\big) (\nu^{-\f12}E_{3,3}^2+\nu^{-\f12}E_{3,0}^2)\lesssim \nu^{-\f{11}4}E_3^4},
\end{align*}
which gives  \eqref{lemma2.1-3}.
\end{proof}

\subsection{Interaction between zero mode and non-zero mode}

The following lemma gives the intereaction between $\overline{u}_1$ and $u_{2,\neq},u_{3,\neq}$.

\begin{lemma}\label{lem:int-zn-1-23}
  It holds that
  \begin{align*}
    &\|e^{c\nu^{\f12}t}(\partial_x,\partial_z)
    \big(\overline{u}_1\partial_xu_{3,\neq}\big)\|^2_{L^2L^2}
    +\|e^{c\nu^{\f12}t}(\partial_x,\partial_z )
    (\overline{u}_1\partial_xu_{2,\neq})\|^2_{L^2L^2}\\&\quad
    +\|e^{c\nu^{\f12}t}\partial_x
    ((u_{2,\neq}\partial_y+u_{3,\neq}\partial_z)\overline{u}_1)\|^2_{L^2L^2}
    \lesssim \nu^{-\f12}E^2_1E_3^2.
  \end{align*}
\end{lemma}

\begin{proof}
For  $j \in\{2,3\}$,  thanks to Lemma \ref{Lem: bilinear zero nonzero} and the definition of $E_1$,  we obtain
   \begin{align*}
      & \|(\partial_x,\partial_z)
    \big(\overline{u}_1\partial_xu_{j,\neq}\big)\|^2_{L^2} +\|\partial_x
    (u_{j,\neq}\partial_j \overline{u}_1)\|^2_{L^2} \\
      &\lesssim\big(\|\overline{u}_1 \|^2_{H^1}+\|\nabla\overline{u}_1\|^2_{H^1}\big) \big(\|\partial_x(\partial_x,\partial_z)u_{j,\neq}\|^2_{L^2}+ \|\partial_xu_{j,\neq}\|^2_{L^2}+ \|\partial_z\partial_xu_{j,\neq}\|^2_{L^2}\big)\\
      &\lesssim\|\overline{u}_1\|^2_{H^2} \|\partial_x(\partial_x,\partial_z)u_{j,\neq}\|^2_{L^2}
      \lesssim  {E_1^2} \|\partial_x(\partial_x,\partial_z)u_{j,\neq}\|^2_{L^2}.
   \end{align*}
Then, by the definitions of $E_{3,0}$ and $E_{3,2}$ in \eqref{eq:E30-define-00}, \eqref{eq:E32-define-00} respectively, we deduce that
\begin{align*}
   \|e^{c\nu^{\f12}t}\partial_x(\partial_x,\partial_z)u_{2,\neq}\|^2_{L^2L^2}& \lesssim  \nu^{-\f12}E_{3,0}^2\lesssim \nu^{-\f12}E_3^2,
\end{align*}
and
\begin{align*}
   \|e^{c\nu^{\f12}t}\partial_x(\partial_x,\partial_z)u_{3,\neq}\|^2_{L^2L^2} \lesssim  \|e^{c\nu^{\f12}t} |\partial_x|^{\f12} \Lambda_{x,z}^{3/2} u_{3,\neq}\|^2_{L^2L^2} \lesssim \nu^{-\f12}   E_{3,2}^2\lesssim \nu^{-\f12} E_{3}^2.
\end{align*}
\end{proof}

The following lemma describes the intereaction between $\overline{u}_1$ and $u_{1,\neq}$.

\begin{lemma}\label{lem:int-zn-11}
  It holds that
  \begin{align*}
     \|e^{c\nu^{\f12}}\partial_x
     (\overline{u}_1\partial_xu_{1,\neq})\|^2_{L^2L^2}
     \lesssim \nu^{-\f12}E_1^2E_3^2.
  \end{align*}
\end{lemma}

\begin{proof}
Thanks to the definition of $E_1$, we have
\begin{align*}
     \|\partial_x
     (\overline{u}_1\partial_xu_{1,\neq})\|^2_{L^2}
     &\leq \|\overline{u}_1\|^2_{L^{\infty}}\|\partial_x^2u_{1,\neq}\|^2_{L^2}\\
     &\lesssim \|\overline{u}_1\|^2_{H^{2}}\|\partial_x^2u_{1,\neq}\|^2_{L^2}\lesssim E^2_1\|\partial_x^2u_{1,\neq}\|^2_{L^2},
  \end{align*}
 which along with Lemma \ref{lem:u-nonzero}  gives
\begin{align*}
   \|e^{c\nu^{\f12}t}\partial_x
     (\overline{u}_1\partial_xu_{1,\neq})\|^2_{L^2L^2}
   \lesssim& E^2_1\|e^{c\nu^{\f12}t}\partial_x^2u_{1,\neq}\|_{L^2L^2}^2
   \lesssim \nu^{-\f12}E^2_1E_3^2.
   \end{align*}
 \end{proof}

The following lemma provides the intereactions between $\overline{u}_2$, $\overline{u}_3$ with nonzero modes, which suggests that $\overline{u}_2$ and $\overline{u}_3$ are good components.

\begin{lemma}\label{lem:int-nz-23}
  It holds that for $j \in\{2,3\}$,
  \begin{align*}
     &\|e^{c\nu^{\f12}t}(\partial_x,\partial_z) (\overline{u}_j \nabla
     u_{\neq})\|_{L^2L^2} +\|e^{c\nu^{\f12}t}(\partial_x,\partial_z)
     (u_{\neq}\cdot\nabla\overline{u}_j)\|_{L^2L^2}
     \lesssim \nu^{-1}E_2E_3.
  \end{align*}
  \end{lemma}

\begin{proof}
  By Lemma \ref{Lem: bilinear zero nonzero}, Lemma \ref{lem:u23-zero} and Lemma \ref{lem:u-relation}, we get
  \begin{align*}
     &\|(\partial_x,\partial_z) (\overline{u}_j\nabla
     u_{\neq})\|_{L^2} +\|(\partial_x,\partial_z)
     (u_{\neq}\cdot\nabla\overline{u}_j)\|_{L^2}\\
     &\lesssim \big(\|\overline{u}_j\|_{H^1} +\|\partial_z\overline{u}_j \|_{H^1}\big)\big(\| \nabla
     u_{\neq}\|_{L^2}+\|(\partial_x,\partial_z) \nabla
     u_{\neq}\|_{L^2}\big)\\&\quad+\big(\|\nabla\overline{u}_j \|_{L^2} +\|\partial_z\nabla\overline{u}_j \|_{L^2}\big)\big(\|
     u_{\neq}\|_{H^1}+\|(\partial_x,\partial_z)
     u_{\neq}\|_{H^1}\big)\\
     &\lesssim\big(\|\overline{u}_j \|_{H^1} +\|\partial_z\overline{u}_j \|_{H^1}\big)\big(\|\partial_x\nabla
     u_{\neq}\|_{L^2} +\|\partial_z\nabla
     u_{\neq}\|_{L^2} \big)\\
     &\lesssim E_2\big(\|\nabla\omega_{2,\neq}\|_{L^2}+
  \|\Delta u_{2,\neq}\|_{L^2}\big)\\
  &\lesssim E_2\big(\|\Lambda_{x,z}^{-1/2}\nabla\omega_{2,\neq}\|_{L^2}^{\f12}
  \|\Lambda_{x,z}^{-1/2}\Delta\omega_{2,\neq}\|_{L^2}^{\f12}+
  \|\Lambda_{x,z}^{1/2}\Delta u_{2,\neq}\|_{L^2}\big),
  \end{align*}
which gives
\begin{align*}
     &\|e^{c\nu^{\f12}t}(\partial_x,\partial_z) (\overline{u}_j \nabla
     u_{\neq})\|_{L^2L^2} +\|e^{c\nu^{\f12}t}(\partial_x,\partial_z)
     (u_{\neq}\cdot\nabla\overline{u}_j )\|_{L^2L^2}\\
    & \lesssim E_2\Big( \|e^{c\nu^{\f12}t}\Lambda_{x,z}^{-1/2}\nabla \omega_{2,\neq}\|_{L^2L^2}^{\f12}\|e^{c\nu^{\f12}t}\Lambda_{x,z}^{-1/2}\Delta \omega_{2,\neq}\|_{L^2L^2}^{\f12}\\
     &\quad+\|e^{c\nu^{\f12}t}\Lambda_{x,z}^{1/2}\Delta u_{2,\neq}\|_{L^2L^2}\Big)\\
     &{ \lesssim E_2\big((\nu^{-\f12}E_{3,2})^{\f12}(\nu^{-\f78} E_{3,2})^{\f12}+\nu^{-\f12}E_{3,0}\big)}\\
     &\lesssim \nu^{-\f{11}{16}}E_2E_3\lesssim \nu^{-1}E_2E_3.
  \end{align*}
The proof is completed.
 \end{proof}
  The following lemma will be used to estimate $E_{3,1}$.

\begin{lemma}\label{lem:int-zz-11-pz}
  It holds that
   \begin{align*}
    &\|e^{c\nu^{\f12}t}\partial_z(\overline{u}_1\partial_xu_{1,\neq})\|^2_{L^2L^2}+
    \|e^{c\nu^{\f12}t}\partial_z((u_{2,\neq}\partial_y+u_{3,\neq}\partial_z)\overline{u}_1) \|^2_{L^2L^2}\lesssim  \nu^{-1}E^2_1E^{2}_3.
  \end{align*}
  \end{lemma}

\begin{proof}
Thanks to Lemma \ref{Lem: bilinear zero nonzero} and the definition od $E_1$, we have
  \begin{align*}
     &\|\partial_z(\overline{u}_1\partial_xu_{1,\neq})\|^2_{L^2}+
     \|\partial_z((u_{2,\neq}\partial_y+u_{3,\neq}\partial_z)\overline{u}_1)\|^2_{L^2}\\ &\leq  C\big(\|\overline{u}_1\|_{H^1}^2+\|\partial_z\overline{u}_1\|_{H^1}^2\big)
     \big(\|\partial_xu_{1,\neq}\|^2_{L^2}+
     \|(\partial_x,\partial_z)\partial_xu_{1,\neq}\|^2_{L^2}\big)\\&\quad+
     C\big(\|\partial_y\overline{u}_1\|_{L^2}^2+
     \|\partial_z\partial_y\overline{u}_1\|_{L^2}^2\big)\big(\|u_{2,\neq}\|^2_{H^1}+
     \|(\partial_x,\partial_z)u_{2,\neq}\|^2_{H^1}\big)\\&\quad+
     C\big(\|\partial_z\overline{u}_1\|_{L^2}^2+
     \|\partial_z^2\overline{u}_1\|_{L^2}^2\big)\big(\|u_{3,\neq}\|^2_{H^1}+
     \|(\partial_x,\partial_z)u_{3,\neq}\|^2_{H^1}\big)\\ &\leq C\|\overline{u}_1\|_{H^2}^2\big(   {  \|(\partial_x,\partial_z)\partial_xu_{1,\neq}\|^2_{L^2}}
     +\|\nabla(\partial_x,\partial_z)u_{2,\neq}\|^2_{L^2}
     +\|\nabla(\partial_x,\partial_z)u_{3,\neq}\|^2_{L^2}\big)\\
     &\leq E^2_1({ \|(\partial_x,\partial_z)\partial_xu_{1,\neq}\|^2_{L^2}+
     \|\nabla (\partial_x,\partial_z)(u_{2,\neq},u_{3,\neq})\|^2_{L^2}}),
  \end{align*}
  which along with Lemma \ref{lem:u-nonzero}  gives
   \begin{align*}
   &\|e^{c\nu^{\f12}t}\partial_z(\overline{u}_1\partial_xu_{1,\neq})\|^2_{L^2L^2}+
    \|e^{c\nu^{\f12}t}\partial_z((u_{2,\neq}\partial_y+u_{3,\neq}\partial_z)\overline{u}_1)\|^2_{L^2L^2}\\
   &\lesssim E^2_1\Big({ \| e^{c\nu^{\f12}t} (\partial_x,\partial_z)\partial_xu_{1,\neq}\|^2_{L^2 L^2}+
     \| e^{c\nu^{\f12}t} \nabla (\partial_x,\partial_z)(u_{2,\neq},u_{3,\neq})\|^2_{L^2 L^2}}\Big)\\
     &{\lesssim E^2_1\Big(\nu^{-1}E_3^2+
     \| e^{c\nu^{\f12}t} \nabla \partial_z(u_{2,\neq},u_{3,\neq})\|^2_{L^2 L^2} } \\
     &{\qquad +   \| e^{c\nu^{\f12}t} \nabla \partial_x u_{2,\neq}\|^2_{L^2 L^2}+     \| e^{c\nu^{\f12}t} \nabla \partial_x u_{3,\neq}\|^2_{L^2 L^2}\Big)}\\
   &\lesssim {E^2_1(\nu^{-1}E_3^2+\nu^{-\f12}E_3^2+\nu^{-1} E_{3,2}^2)}\lesssim \nu^{-1}E^2_1E^{2}_3.
   \end{align*}
  \end{proof}

\section{Energy estimates for zero modes}\label{sec:3d-nonlinear-estimates}

\subsection{Estimate of $E_1$}
\begin{proposition}\label{prop:E1}
It holds that
\begin{align}
&E_{1}\lesssim \|{u}(0)\|_{H^2}+\nu^{-1}E_2+\nu^{-1}E_2E_{1}+{\nu^{-\f{15}{8}}}E_3^2.\label{eq:E1n}
\end{align}
\end{proposition}
\begin{proof}

Recall that $\overline{u}_{1}$ satisfies
\begin{align*}
\left\{
\begin{aligned}
&(\partial_t-\nu\Delta )\overline{u}_{1}-2y\overline{u}_2+\overline{u}_2\partial_y\overline{u}_{1}+\overline{u}_3
\partial_z\overline{u}_{1}+\overline{u_{\neq}\cdot\nabla u_{1,\neq}}=0,\\
&\overline{u}_{1}|_{t=0}=\overline{u}_1(0),\quad \Delta\overline{u}_{1}|_{y=\pm1}=\overline{u}_{1}|_{y=\pm1}=0.
\end{aligned}
\right.
\end{align*}
Thanks to $\Delta\overline{u}_{1}|_{y=\pm1}=0$,  the energy estimate gives
\begin{align*}
&\frac{ \mathrm{d} }{\mathrm{d}t}\| \Delta\overline{u}_{1}\|_{L^2}^2+2\nu\| \nabla\Delta\overline{u}_{1}\|_{L^2}^2+4\big\langle
\nabla(y\overline{u}_2),\nabla\Delta\overline{u}_{1}\big\rangle
\\&\quad-2\Big\langle \nabla(\overline{u}_2\partial_y\overline{u}_{1}+\overline{u}_3\partial_z\overline{u}_{1}+
\overline{u_{\neq}\cdot\nabla u_{1,\neq}}), \nabla\Delta\overline{u}_{1}\Big\rangle=0,
\end{align*}
which gives
\begin{align*}
&\frac{ \mathrm{d} }{\mathrm{d}t}\| \Delta\overline{u}_{1}\|_{L^2}^2+\nu\| \nabla\Delta\overline{u}_{1}\|_{L^2}^2\\
&\lesssim\nu^{-1}\Big(\|\overline{u}_{2}\|_{H^1}^2+\|\nabla (\overline{u}_2\partial_y\overline{u}_{1}+\overline{u}_3\partial_z\overline{u}_{1})\|_{L^2}^2+\|\nabla (\overline{u_{\neq}\cdot\nabla u_{1,\neq}})\|_{L^2}^2\Big).
\end{align*}
As $\Delta\overline{u}_{1}|_{y=\pm1}=\overline{u}_{1}|_{y=\pm1}=0$,  we have $\|\overline u_1\|_{L^2}\lesssim \|\na\overline u_1\|_{L^2}\lesssim \|\Delta \overline u_1\|_{L^2}\lesssim \|\na \Delta\overline u_1\|_{L^2}$. Then it is easy to see that
\begin{align*}
\|\nabla(\overline{u}_{j}\partial_j\overline{u}_{1})\|_{L^2}^2\lesssim& \|\overline{u}_{j}\|_{H^1}^2\|\partial_j\overline{u}_{1}\|_{H^2}^2\lesssim \|\nabla\overline{u}_{j}\|_{L^2}^2 \|\nabla\Delta\overline{u}_{1}\|_{L^2}^2,
\end{align*}
which gives
\begin{align*}
&\|\nabla (\overline{u}_{2}\partial_y\overline{u}_{1}+
\overline{u}_{3}\partial_z\overline{u}_{1})\|_{L^2}^2
\lesssim \big(\|\nabla\overline{u}_{2}\|_{L^2}^2+\|\nabla\overline{u}_{3}\|_{L^2}^2\big)
\|\nabla\Delta\overline{u}_{1}\|_{L^2}^2,
\end{align*}
from which and  Lemma \ref{lem:int-nn}, we infer that
\begin{align*}
&\| \Delta\overline{u}_{1}\|_{L^{\infty}L^2}^2+\nu\| \nabla\Delta\overline{u}_{1}\|_{L^2L^2}^2\\ &\lesssim \| u(0)\|_{H^2}^2+\nu^{-1}\Big(\|\overline{u}_{2}\|_{L^2H^1}^2+\|\nabla (\overline{u}_{2}\partial_y\overline{u}_{1}+\overline{u}_{3}\partial_z\overline{u}_{1}
)\|_{L^2L^2}^2\\
&\qquad+\|\nabla (\overline{u_{\neq}\cdot\nabla u_{1,\neq}})\|_{L^2L^2}^2\Big)\\ &\lesssim \| u(0)\|_{H^2}^2+\nu^{-2}E_2^2+\nu^{-1}E_2^2\|\nabla\Delta\overline{u}_{1}\|_{L^2L^2}^2
+{ \nu^{-\f{15}{4}}}E_3^4.
\end{align*}
Thus, we obtain
\begin{align*}
E_{1}^2&=\big(\|\overline{u}_{1}\|_{L^{\infty}H^2}
+\nu^{\f12}\|\nabla\overline{u}_{1}\|_{L^2H^2}\big)^2\lesssim
\| \Delta\overline{u}_{1}\|_{L^{\infty}L^2}^2+\nu\| \nabla\Delta\overline{u}_{1}\|_{L^2L^2}^2 \\&\lesssim \|u(0)\|_{H^2}^2+\nu^{-2}E_2^2+\nu^{-2}E_2^2E_{1}^2+{ \nu^{-\f{15}{4}}}E_3^4.
\end{align*}
This proves \eqref{eq:E1n}.
\end{proof}

\subsection{Estimate of $E_2$}

\begin{proposition}\label{prop:E2}
It holds that
\beno
E_2\lesssim \big(1+\nu^{-1}E_2\big)^2\big(\|u(0)\|_{H^2}+\nu^{-\f32}E_3^2\big).
\eeno
\end{proposition}

The proposition is a direct consequence of the following lemmas.

\begin{lemma}\label{lem:u23-H1}
There exists a small constant $c>0$ such that
\begin{align*}
&\|e^{c\nu t}(\overline{u}_{2},\overline{u}_{3})\|_{L^{\infty}L^2}^2+\nu\|e^{c\nu t}(\nabla\overline{u}_{2},\nabla\overline{u}_{3})\|_{L^2L^2}^2\lesssim
\|u(0)\|_{H^2}^2+\nu^{-3}E_3^4,\\
&\|e^{c\nu t}(\nabla\overline{u}_{2},\nabla\overline{u}_{3})\|_{L^{\infty}L^2}^2+\nu^{-1}\|e^{c\nu t}(\partial_t\overline{u}_{2},\partial_t\overline{u}_{3})\|_{L^2L^2}^2\\
&\qquad\qquad \qquad \lesssim\big(1+\nu^{-2}E_2^2\big)\big(\|u(0)\|_{H^2}^2+\nu^{-3}E_3^4\big).
\end{align*}
\end{lemma}

\begin{proof}
Recall that $\overline{u}_{j}$ for $j=2,3$ satisfies
\begin{align*}
(\partial_t-\nu\Delta )\overline{u}_{j}+\partial_j \overline{p}+(\overline{u}_{2}\partial_y+\overline{u}_{3}\partial_z)\overline{u}_{j}
+\overline{u_{\neq}\cdot\nabla u_{j,\neq}}=0.
\end{align*}
As $\overline{u}_{j}|_{y=\pm1}=0$, $L^2$ energy estimate gives
\begin{align*}
&\frac{ \mathrm{d} }{\mathrm{d}t}\big(\|\overline{u}_{2}\|_{L^2}^2+\|\overline{u}_{3}\|_{L^2}^2\big)
+2\nu\big(\|\nabla\overline{u}_{2}\|_{L^2}^2+\|\nabla\overline{u}_{3}\|_{L^2}^2\big)\\
&=2\langle \overline{p},\partial_y\overline{u}_{2}\rangle+2\langle \overline{p},\partial_z\overline{u}_{3}\rangle-2\sum\limits_{j=2,3}
\big\langle(\overline{u}_{2}\partial_y+\overline{u}_{3}\partial_z)\overline{u}_{j}
,\overline{u}_{j}\big\rangle\\&\quad-2\big\langle \overline{u_{\neq}\cdot\nabla u_{2,\neq}},\overline{u}_{2}\rangle-2\big\langle \overline{u_{\neq}\cdot\nabla u_{3,\neq}},\overline{u}_{3}\big\rangle.
\end{align*}
As $\partial_y\overline{u}_{2}+\partial_z\overline{u}_{3}=0$, we have
\begin{align*}
&\frac{ \mathrm{d} }{\mathrm{d}t}\big(\|\overline{u}_{2}\|_{L^2}^2+\|\overline{u}_{3}\|_{L^2}^2\big)
+2\nu\big(\|\nabla\overline{u}_{2}\|_{L^2}^2+\|\nabla\overline{u}_{3}\|_{L^2}^2\big)\\
&=-2\sum\limits_{j=2,3}\langle\overline{u_{\neq}\cdot\nabla u_{j,\neq}},\overline{u}_{j}\rangle
=2\sum\limits_{j=2,3}\langle u_{\neq}\cdot\nabla \overline{u}_{j},u_{j,\neq}\rangle\\
&\leq2\||u_{\neq}|^2\|_{L^2}\|(\nabla\overline{u}_{2},\nabla\overline{u}_{3})\|_{L^2},
\end{align*}
from which and the fact that $\|\nabla\overline{u}_j\|_{L^2}^2\geq(\pi/2)^2\|\overline{u}_j\|_{L^2}^2$, we infer that
\begin{align*}
&e^{2c\nu t}\big(\|\overline{u}_{2}\|_{L^2}^2+\|\overline{u}_{3}\|_{L^2}^2\big)
+\nu\int_0^te^{2c\nu s}\big(\|\nabla\overline{u}_{2}(s)\|_{L^2}^2+\|\nabla\overline{u}_{3}(s)\|_{L^2}^2\big)\mathrm{d}s\\
&\lesssim \|\overline{u}_{2}(0)\|_{L^2}^2+\|\overline{u}_{3}(0)\|_{L^2}^2+\nu^{-1}\int_0^te^{2c\nu s}\||u_{\neq}(s)|^2\|_{L^2}^2\mathrm{d}s.
\end{align*}
This  along with Lemma \ref{lem:int-nn} gives
\begin{align*}
&\|e^{c\nu t}(\overline{u}_{2},\overline{u}_{3})\|_{L^{\infty}L^2}^2+\nu\|e^{c\nu t}(\nabla\overline{u}_{2},\nabla\overline{u}_{3})\|_{L^2L^2}^2\\
&\lesssim \|u(0)\|_{L^2}^2+\nu^{-1}\|e^{c\nu t}|u_{\neq}|^2\|_{L^2L^2}^2 \lesssim\|u(0)\|_{L^2}^2+\nu^{-3}E_3^4.
\end{align*}

Now $H^1$ energy estimate gives
\begin{align*}
&\nu\frac{ \mathrm{d} }{\mathrm{d}t}\big(\|\nabla\overline{u}_{2}\|_{L^2}^2+\|\nabla\overline{u}_{3}\|_{L^2}^2\big)
+2\big(\|\partial_t\overline{u}_{2}\|_{L^2}^2+\|\partial_t\overline{u}_{3}\|_{L^2}^2\big)\\
&=-2\sum\limits_{j=2,3}\big\langle(\overline{u}_{2}\partial_y+\overline{u}_{3}\partial_z)
\overline{u}_{j},\partial_t\overline{u}_{j}\big\rangle\\
&\qquad -2\big\langle \overline{u_{\neq}\cdot\nabla u_{2,\neq}},\partial_t\overline{u}_{2}\big\rangle-2\big\langle \overline{u_{\neq}\cdot\nabla u_{3,\neq}},\partial_t\overline{u}_{3}\big\rangle.
\end{align*}
By Lemma \ref{lem:u23-zero}, we have
\begin{align}\label{eq:u2}
\|e^{c\nu t}(\overline{u}_{2}\partial_y+\overline{u}_{3}\partial_z)\overline{u}_{j}\|_{L^2}\lesssim E_2\|e^{c\nu t}\nabla\overline{u}_{j}\|_{L^2},
\end{align}
and by Lemma \ref{lem:int-nn},
\beno
\|e^{c\nu t}(u_{\neq}\cdot \nabla u_{\neq})\|_{L^2L^2}^2\lesssim \nu^{-2}E_3^4.
\eeno
Thus, we obtain
\begin{align*}
&\|e^{c\nu t}(\nabla\overline{u}_{2},\nabla\overline{u}_{3})\|_{L^{\infty}L^2}^2+\nu^{-1} \|e^{c\nu t}(\partial_t\overline{u}_{2},\partial_t\overline{u}_{3})\|_{L^2L^2}^2\\
&\lesssim\|u(0)\|_{H^2}^2+\nu \|e^{c\nu t}(\nabla\overline{u}_{2},\nabla\overline{u}_{3})\|_{L^2L^2}^2+\nu^{-3}E_3^4\\
&\quad+\nu^{-1}E_2^2 \|e^{c\nu t}(\nabla\overline{u}_{2},\nabla\overline{u}_{3})\|_{L^2L^2}^2\lesssim \big(1+\nu^{-2}E_2^2\big)\big(\|u(0)\|_{H^2}^2+\nu^{-3}E_3^4\big).
\end{align*}
\end{proof}

\begin{lemma}\label{lem:u2-H2}
There exists a small constant $c>0$ such that
\begin{align*}
&\|e^{c\nu t}\Delta\overline{u}_{2}\|_{L^{\infty}L^2}^2+\nu^{-1}\|e^{c\nu t}\nabla\partial_t\overline{u}_{2}\|_{L^2L^2}^2 \lesssim \big(1+\nu^{-2}E_2^2\big)^2\big(\|u(0)\|_{H^2}^2+\nu^{-3}E_3^4\big).
\end{align*}
\end{lemma}

\begin{proof}
Recall that $\Delta \overline{u}_{2}$ satisfies
\begin{align}\label{2.14.1}
(\partial_t-\nu\Delta )\Delta\overline{u}_{2}+\Delta\partial_y \overline{p}+\Delta(\overline{u\cdot\nabla u_2})=0,\quad \nabla\overline{u}_{2}|_{y=\pm1}=0.
\end{align}
Taking the $L^2$ inner product with $-2\bar{u}^2$, we get
\begin{align*}
\frac{ \mathrm{d} }{\mathrm{d}t} \|\nabla\overline{u}_{2}\|_{L^2}^2&+2\nu\|\Delta\overline{u}_{2}\|_{L^2}^2
+2\langle\Delta \bar{p},\partial_y\overline{u}_{2}\rangle\\
&-2\big\langle\overline{u}_{2}
\partial_y\overline{u}_{2}+\overline{u}_{3}\partial_z\overline{u}_{2}
+\overline{u_{\neq}\cdot\nabla u_{2,\neq}},\Delta\overline{u}_{2}\big\rangle=0,
\end{align*}
which  gives
\begin{align}
\nu\|e^{c\nu t}\Delta\overline{u}_{2}\|_{L^2L^2}^2\lesssim&\|u(0)\|_{H^2}^2+ \nu^{-1}\|e^{c\nu t}\Delta \overline{p}\|_{L^2L^2}^2+\nu\| e^{c\nu t}\nabla\overline{u}_{2}\|_{L^2L^2}^2\nonumber\\
&+\nu^{-1}\|e^{c\nu t}(\overline{u}_{2}\partial_y+\overline{u}_{3}\partial_z)\overline{u}_{2}\|_{L^2L^2}^2
+\nu^{-1}\|e^{c\nu t} \overline{u_{\neq}\cdot \nabla u_{2,\neq}}\|_{L^2L^2}^2.\label{eq:Du2-est1}
\end{align}
By Lemma \ref{lemma-A.1} and Lemma \ref{lem:u23-H1}, we have
\begin{align*}
\|e^{c\nu t}\nabla(\overline{u}_{i}\partial_i \overline{u}_{j})\|_{L^2L^2}^2
\leq& \|e^{c\nu t}\overline{u}_{i}\|_{L^{\infty}H^1}^2\|(\partial_i \overline{u}_{j},\partial_z\partial_i \overline{u}_{j})\|_{L^2H^1}^2\\
\lesssim&  \|e^{c\nu t}(\nabla\overline{u}_{2},\nabla\overline{u}_{3})\|_{L^{\infty}L^2}^2
\|(\Delta\overline{u}_{2},\Delta\overline{u}_{3},\nabla\Delta\overline{u}_{2})\|_{L^2L^{2}}^2\\ \lesssim& \big(\|u(0)\|_{H^2}^2+\nu^{-3}E_3^4\big)\big(1+\nu^{-2}E_2^2\big)\nu^{-1}E_2^2.
\end{align*}
This shows that
\begin{align}\label{eq:zz-est}
\nu^{-1}\|e^{c\nu t}\nabla(\overline{u}\cdot\nabla \overline{u}_{j})\|_{L^2L^2}^2\lesssim
\big(1+\nu^{-2}E_2^2\big)^2\big(\|u(0)\|_{H^2}^2+\nu^{-3}E_3^4\big),\quad j\in\{2,3\}.
\end{align}

Notice that
\begin{align*}
\Delta \overline{p}=-\sum_{i,j=1,2,3}\overline{\partial_ju_i\partial_iu_j}=
-\sum_{i,j=1,2,3}\partial_j\overline{u}_{i}\partial_i\overline{u}_{j}
-\sum_{i,j=1,2,3}\overline{\partial_ju_{i,\neq}\partial_iu_{j,\neq}}.
\end{align*}
Then by Lemma \ref{lem:int-nn} and $\text{div} u_{\neq}=0$, we have
\begin{align*}
&\Big\|\sum_{i,j=1,2,3}e^{c\nu t}\overline{\partial_ju^i_{\neq}\partial_iu^j_{\neq}}\Big\|_{L^2L^2}^2=\Big\|\sum_{j=1,2,3}e^{c\nu t}\overline{\partial_j(u_{\neq}\cdot\nabla u^j_{\neq}})\Big\|_{L^2L^2}^2\\
&\leq \Big\|\sum_{j=1,2,3}e^{c\nu t}\partial_j(u_{\neq}\cdot\nabla u_{j,\neq})\Big\|_{L^2L^2}^2\lesssim\nu^{-2}E_3^4,
\end{align*}
and by \eqref{eq:zz-est}, we have
\begin{align*}
&\Big\| \sum_{j=1,2,3}e^{c\nu t}\partial_j\overline{u}_i\partial_i\overline{u}_j\Big\|_{L^2L^2}^2=\Big\|\sum_{\alpha=1,2,3} e^{c\nu t}\partial_{\alpha}(\overline{u}\cdot\nabla \overline{u}_{\alpha})\Big\|_{L^2L^2}^2\\
&\lesssim \nu\big(\|u(0)\|_{H^2}^2+\nu^{-3}E_3^4\big)\big(1+\nu^{-2}E_2^2\big)^2.
\end{align*}
This shows that
\begin{align}\label{eq:Dp-est}
\nu^{-1}\|e^{c\nu t}\Delta \overline{p}\|_{L^2L^2}^2\lesssim (1+\nu^{-2}E_2^2)^2\big(\|u(0)\|_{H^2}^2+\nu^{-3}E_3^4\big).
\end{align}
Then it follows from \eqref{eq:Du2-est1}, \eqref{eq:zz-est}, \eqref{eq:Dp-est} and Lemma \ref{lem:u23-H1} that
\begin{align}\label{eq:Du2-est2}
\nu\|e^{\nu t}\Delta\overline{u}_2\|_{L^2L^2}^2\lesssim\big(1+\nu^{-2}E_2^2\big)^2\big(\|u(0)\|_{H^2}^2+\nu^{-3}E_3^4\big).
\end{align}

Next we  take the $L^2$ inner product between \eqref{2.14.1}  and $-2\partial_t\overline{u}_2$ to obtain
\begin{align*}
2\|\partial_t\nabla\overline{u}_2\|_{L^2}^2+\nu\frac{ \mathrm{d} }{\mathrm{d}t} \|\Delta\overline{u}_2\|_{L^2}^2+
2\langle\Delta\overline{p},\partial_t\partial_y\overline{u}_2\rangle+2\big\langle
\nabla(\overline{u\cdot\nabla u_2}),\partial_t\nabla\overline{u}_2\big\rangle=0,
\end{align*}
which gives
\begin{align*}
\|\partial_t\nabla\overline{u}_2\|_{L^2}^2
+\nu\partial_t\|\Delta\overline{u}_2\|_{L^2}^2\lesssim\big(\|\Delta \overline{p}\|_{L^2}^2+\|\nabla(\overline{u\cdot\nabla u_2})\|_{L^2}^2\big),
\end{align*}
and then
\begin{align*}
&\nu^{-1}e^{2c\nu t}\|\partial_t\nabla\overline{u}_2\|_{L^2}^2+\frac{ \mathrm{d} }{\mathrm{d}t}\big(e^{2c\nu t}\|\Delta\overline{u}_2\|_{L^2}^2\big)\\
&\lesssim\nu^{-1}\Big(e^{2c\nu t}\|\Delta \overline{p}\|_{L^2}^2+e^{2c\nu t}\|\nabla(\overline{u\cdot\nabla u_2})\|_{L^2}^2+\nu^2e^{2c\nu t}\|\Delta\overline{u}_2\|_{L^2}^2\Big).
\end{align*}
Then we infer from \eqref{eq:Du2-est2}, \eqref{eq:zz-est}, \eqref{eq:Dp-est} and Lemma \ref{lem:int-nn} that
\begin{align*}
&\|e^{c\nu t}\Delta\overline{u}_2\|_{L^{\infty}L^2}^2+\nu^{-1}\|e^{c\nu t}\partial_t\nabla\overline{u}_2\|_{L^2L^2}^2 \\
&\lesssim (\|u(0)\|_{H^2}^2+\nu^{-1} \|e^{c\nu t}\Delta \overline{p}\|_{L^2L^2}^2\\
&\quad+\nu^{-1}\|e^{c\nu t}\nabla(\overline{u\cdot\nabla u_2})\|_{L^2L^2}^2+\nu \|e^{c\nu t}\Delta\overline{u}_2\|_{L^2L^2}^2)\\
&\lesssim \big(1+\nu^{-2}E_2^2\big)^2\big(\|u(0)\|_{H^2}^2+\nu^{-3}E_3^4\big).
\end{align*}
\end{proof}

\begin{lemma}\label{lem:u2-H3}
There exists a small constant $c>0$ such that
\begin{align*}
&\nu\|e^{c\nu t}\nabla\Delta\overline{u}_{2}\|_{L^2L^2}^2+
\nu\|e^{c\nu t}\Delta\overline{u}_{3}\|_{L^2L^2}^2\leq C\big(1+\nu^{-2}E_2^2\big)^2\big(\|u(0)\|_{H^2}^2+\nu^{-3}E_3^4\big).
\end{align*}
\end{lemma}

\begin{proof}
Using the equation
\begin{align*}
(\partial_t-\nu\Delta )\partial_z\overline{u}_{2}+\partial_z\partial_y \overline{p}+\partial_z(\overline{u\cdot\nabla u_2})=0,\quad \partial_z\overline{u}_{2}|_{y=\pm1}=\nabla\overline{u}_{2}|_{y=\pm1}=0,
\end{align*}
we get by integration by parts that
\begin{align*}
&\|\partial_t \partial_z\overline{u}_{2}+\partial_z(\overline{u\cdot\nabla u_2})\|_{L^2}^2=\| \nu\Delta \partial_z\overline{u}_{2}-\partial_z\partial_y \overline{p}\|_{L^2}^2\\
&= \nu^2\| \Delta \partial_z\overline{u}_{2}\|_{L^2}^2+\|\partial_z\partial_y \overline{p}\|_{L^2}^2-2\nu\big\langle \Delta \partial_z\overline{u}_{2}, \partial_z\partial_y \overline{p}\big\rangle\\
&= \nu^2\| \Delta \partial_z\overline{u}_{2}\|_{L^2}^2+\|\partial_z\partial_y \overline{p}\|_{L^2}^2+2\nu\big\langle \nabla \partial_z\overline{u}_2, \nabla \partial_z\partial_y \overline{p}\big\rangle\\
&=\nu^2\| \Delta \partial_z\overline{u}_2\|_{L^2}^2+\|\partial_z\partial_y \overline{p}\|_{L^2}^2-2\nu\big\langle  \partial_z^2\partial_y\overline{u}_{2}, \Delta \overline{p}\big\rangle,
\end{align*}
which shows that
\begin{align}\label{eq:u2p-H3-est}
&\nu^2\| \Delta \partial_z\overline{u}_{2}\|_{L^2}^2+\|\partial_z\partial_y \overline{p}\|_{L^2}^2\nonumber\\
&\leq C\big(\| \Delta \overline{p} \|_{L^2}^2+\|\partial_t \partial_z\overline{u}_{2}+\partial_z(\overline{u\cdot\nabla u_2})\|_{L^2}^2\big).
\end{align}
Then by Lemma \ref{lem:u2-H2}, \eqref{eq:zz-est}, \eqref{eq:Dp-est} and Lemma \ref{lem:int-nn}, we get
\begin{align*}
&\nu\| e^{c\nu t}\Delta \partial_z\overline{u}_{2}\|_{L^2L^2}^2+\nu^{-1}\|e^{c\nu t}\partial_z\partial_y \overline{p}\|_{L^2L^2}^2\\
&\lesssim\nu^{-1}\Big(\| e^{c\nu t}\Delta \overline{p} \|_{L^2L^2}^2+\|e^{c\nu t}\partial_t \partial_z\overline{u}_{2}\|_{L^2L^2}^2+\|e^{c\nu t}\partial_z(\overline{u}\cdot\nabla \overline{u}_{2})\|_{L^2L^2}^2\\
&\qquad\quad+\|e^{c\nu t}\partial_z(\overline{u_{\neq}\cdot\nabla u_{2,\neq}})\|_{L^2L^2}^2\Big) \\
&\lesssim\big(1+\nu^{-2}E_2^2\big)^2\big(\|u(0)\|_{H^2}^2+\nu^{-3}E_3^4\big).
\end{align*}
Thanks to Lemma \ref{delta-bar{p}}, we find that
\begin{align*}
\|\partial_y^2\overline{p}\|_{L^2}^2+\|\partial_z^2\overline{p}\|_{L^2}^2\lesssim
\|\partial_z\partial_y\overline{p}\|_{L^2}^2+\| \Delta \overline{p} \|_{L^2}^2,
\end{align*}
which gives
\begin{align}\label{eq:p-H2}
\nu^{-1}\big(\|e^{c\nu t}\partial_y^2\overline{p}\|_{L^2L^2}^2+\|e^{c\nu t}\partial_z^2\overline{p}\|_{L^2L^2}^2\big)\lesssim& \nu^{-1}\big(\|e^{c\nu t}\partial_z\partial_y\overline{p}\|_{L^2L^2}^2+\| e^{c\nu t}\Delta \overline{p} \|_{L^2L^2}^2\big)\nonumber\\
\lesssim& \big(1+\nu^{-2}E_2^2\big)^2\big(\|u(0)\|_{H^2}^2+\nu^{-3}E_3^4\big).
\end{align}

Using the equation
\begin{align*}
(\partial_t-\nu\Delta )\partial_y\overline{u}_2+\partial_y^2 \overline{p}+\partial_y(\overline{u}_2\partial_y\overline{u}_2+\overline{u}_3\partial_z\overline{u}_2)
+\partial_y(\overline{u_{\neq}\cdot\nabla u_{2,\neq}})=0,
\end{align*}
we deduce that
\begin{align*}
&\nu\|\Delta \partial_y\overline{u}_2\|_{L^2}^2\\
&\lesssim \nu^{-1}\Big(\|\partial_t\partial_y\overline{u}_2\|_{L^2}^2+\|\partial_y^2 \overline{p}\|_{L^2}^2+\|\partial_y(\overline{u}\cdot\nabla\overline{u}_2)\|_{L^2}^2
+\|\partial_y(\overline{u_{\neq}\cdot\nabla u_{2,\neq}})\|_{L^2}^2\Big).
\end{align*}
Thus, we have
\begin{align*}
\nu\|e^{c\nu t}\partial_y\Delta \overline{u}_2\|_{L^2L^2}^2
\lesssim\big(1+\nu^{-2}E_2^2\big)^2\big(\|u(0)\|_{H^2}^2+\nu^{-3}E_3^4\big).
\end{align*}

Using the fact that $\|\partial_z \overline{p}\|_{L^2}\leq \|\partial_z^2 \overline{p}\|_{L^2}$ and the equation $(\partial_t-\nu\Delta)\overline{u}_3+\partial_z\overline{p}+\overline{u\cdot\nabla u_3}=0$, we deduce that
\begin{align*}
\nu\|\Delta\overline{u}_3\|_{L^2}^2\lesssim& \nu^{-1}\Big(\|\partial_t\overline{u}_3\|_{L^2}^2
+\|\partial_z\overline{p}\|_{L^2}^2\\
&+\|(\overline{u}_2\partial_y\overline{u}_2+\overline{u}_3\partial_z\overline{u}_3)\|_{L^2}^2
+\|\overline{u_{\neq}\cdot\nabla u_{3,\neq}}\|_{L^2}^2\Big).
\end{align*}
As above, we can obtain
\begin{align*}
\nu\|e^{c\nu t}\Delta\overline{u}_3\|_{L^2L^2}^2\lesssim \big(1+\nu^{-2}E_2^2\big)^2\big(\|u(0)\|_{H^2}^2+\nu^{-3}E_3^4\big).
\end{align*}

This completes the proof of this lemma.
\end{proof}

\section{Energy estimates for non-zero modes}

For non-zero modes, the energy estimate is based on the formulation in terms of $(\Delta u_2,\om_2)$:
\begin{equation*}
\left \{
\begin{array}{lll}
&(\partial_t-\nu \Delta+(1-y^2)\partial_x)\Delta u_2+2\partial_x u_2=-\Delta_{x,z}  (u\cdot\nabla u_2)  \\
&\quad\qquad+\partial_y[\partial_x (u\cdot \nabla u_1)+\partial_z (u\cdot \nabla u_3)],\\
&(\partial_t-\nu \Delta+(1-y^2)\partial_x)\omega_2-2y\partial_z u_2=-\partial_z (u\cdot \nabla u_1)+\partial_x (u\cdot\nabla u_3),\\
&\omega_2:=\partial_zu_1-\partial_xu_3,\qquad (u_2,\partial_y u_2,\omega_2)|_{y=\pm 1}=(0,0,0).
\end{array}\right.
\end{equation*}
We denote
\beno f_{k}(y)=\f{1}{2\pi}\int_{\mathbb{T}^2}f(x,y,z)e^{-ik_1x-ik_3 z}\mathrm{d}x\mathrm{d}z.
\eeno
Taking the Fourier transform in $(x,z)$,  we obtain
\begin{equation*}\left\{
\begin{aligned}
  &\partial_t(\Delta u_2)_k-\nu(\partial_y^2-|k|^2)(\Delta u_2)_{k}+ik_1(1-y^2)( u_2)_k+2ik_1{( u_2)_k}\\
  &\quad\qquad= |k|^2(u\cdot\nabla
  u_2)_{k}+\partial_y\big[\partial_x(u\cdot\nabla u^1)+\partial_z(u\cdot\nabla
  u^3)\big]_{k},\\
  &\partial_t (\omega_2)_ k- \nu(\partial_y^2-|k|^2)(\omega_2)_ k
  +ik_1(1-y^2)(\omega_2)_ k-2ik_3y(u_2)_{k}\\
  &\qquad\quad=-ik_3(u\cdot\nabla u_1)_{k}+ik_1(u\cdot\nabla u_3)_{k},\\
  &\partial_y(u_2)_k|_{y=\pm1}=(u_2)_{k}|_{y=\pm1}=0,\quad
  (u_2)_{k}|_{t=0}=(u_2)_{k}(0),\\
  &(\omega_2)_k|_{y=\pm1}=0,\quad
  (\omega_2)_k|_{t=0}=ik_3 (u_1)_{k}(0)-ik_1(u_3)_{k}(0).
  \end{aligned}\right.
\end{equation*}
Let $a\geq0$, we introduce the following norms:
\begin{equation}\label{eq:Xa}
\begin{aligned}
\|f\|_{X_{k}^{a}}^2
  =&(|k_1|+\nu^{\f12}|k_1||k|)\|e^{a\nu^{\f12}t}(\partial_y,|k|)f\|^2_{L^2 L^2}\\
  &+(\nu|k|+\nu^{\f32}|k|^2)
  \|e^{a\nu^{\f12}t}(\partial^2_y-|k|^2)f\|^2_{L^2 L^2}\\
  &+(|k|+\nu^{\f12}|k|^2)\|e^{a\nu^{\f12}t}(\partial_y,|k|)f\|^2_{L^\infty L^2}\\
&{+\nu^{\f7{24}}\big(\|e^{a\nu^{\f12}t}(\partial^2_y-|k|^2)f\|_{L^{\infty} L^2}+\nu^{\f12}\|e^{a\nu^{\f12}t}\partial_y(\partial^2_y-|k|^2)f\|_{L^2 L^2}\big)},
\end{aligned}
\end{equation}
and
\begin{align}\label{eq:Ya}
  \|f\|_{Y_{k}^{a}}^2=&
  \nu
  \|e^{a\nu^{\f12}t}\partial_yf\|^2_{L^2 L^2}+((\nu |k_1|)^{1/2}+\nu|k|^2)
  \|e^{a\nu^{\f12}t}f\|^2_{L^2 L^2},\nonumber
\end{align}
and
\begin{align*}
  \|f\|^2_{X_a} &=\sum_{k\in\mathbb{Z}^2 ;k_1\neq0}
  \|{f}_k\|^2_{X^a_{k}},\quad
  \|f\|^2_{Y_a} = \sum_{k\in\mathbb{Z}^2 ;k_1\neq0}
  \|{f}_k\|^2_{Y^a_{k}}.
\end{align*}
It is easy to see that
\begin{equation}\label{fX}
\begin{aligned}
  &E_{3,0}^2 \leq \|f\|^2_{X_c}, \\
 \end{aligned}
 \end{equation}
 and
 \begin{align}
 \label{fY}
 &{\nu \|e^{a\nu^{\f12}t}\nabla f_{\neq}\|_{L^2L^2}^2+\nu^{\f12} \|e^{a\nu^{\f12}t}|\pa_x|^\f14 f_{\neq}\|_{L^2L^2}^2\leq\|f\|^2_{Y_a}}.
\end{align}

\begin{lemma}\label{lem:E30}
It holds that
\begin{align*}
E^2_{3,0}+E^2_{3,2}+E_{3,3}^2\lesssim& \|u_{2,\neq}\|^2_{X_{c}}
   +\|\partial_x \omega_{2,\neq}\|_{Y_c}^2 +\nu^{\f12} \|e^{c\nu^{\f12}t}\partial_x\omega_{2,\neq}\|^2_{L^\infty L^2}\\
   &+\nu^{\f14} \|e^{c\nu^{\f12}t}\partial_x \Lambda_{x,z}^{-1/2}\omega_{2,\neq}\|^2_{L^\infty L^2}.
   \end{align*}
\end{lemma}

\begin{proof}
By the definition of $E_{3,0}$ in \eqref{eq:E30-define-00}, we have $E_{3,0}\lesssim \|u_{2,\neq}\|_{X_c} $.

Using the fact $\omega_2=\partial_zu_1-\partial_xu_3$ and $\partial_xu_1+\partial_y
  u_2+\partial_zu_3=0$, we know that $(\partial_x^2+\partial_z^2) u_{3,\neq}
  =-\partial_x\omega_{2,\neq}-\partial_z\partial_yu_{2,\neq} $ and
  \begin{align*}
&\nu^{\f18}\|e^{c\nu^{\f12}t}\Lambda_{x,z}^{3/2} u_{3,\neq}\|_{L^{\infty}L^2}+\nu^{\f12}\|e^{c\nu^{\f12}t}\Lambda_{x,z}^{3/2}
\nabla u_{3,\neq}\|_{L^2L^2} +\nu^{\f14}\|e^{c\nu^{\f12}t}|\pa_x|^\f12\Lambda_{x,z}^{3/2} u_{3,\neq}\|_{L^2L^2}\\
&\leq \nu^{\f18} \big( \|e^{c\nu^{\f12}t}\partial_x\Lambda_{x,z}^{-1/2}\omega_{2,\neq}\|_{L^{\infty}L^2}+
\|e^{c\nu^{\f12}t}\Lambda_{x,z}^{1/2}\partial_yu_{2,\neq}\|_{L^{\infty}L^2}\big)\\
&
\quad+\nu^{\f12}\|e^{c\nu^{\f12}t}\partial_x\Lambda_{x,z}^{-1/2}\nabla\omega_{2,\neq}\|_{L^2L^2}+\nu^{\f12}\|e^{c\nu^{\f12}t}\Lambda_{x,z}^{1/2}\partial_y\nabla u_{2,\neq}\|_{L^2L^2}\\
&\quad+\nu^{\f14}\|e^{c\nu^{\f12}t} \partial_x \omega_{2,\neq}\|_{L^2L^2}
+\nu^{\f14}\|e^{c\nu^{\f12}t}|\partial_x|^{1/2} \Lambda_{x,z}^{1/2}\nabla u_{2,\neq}\|_{L^2L^2},
\end{align*}
from which, \eqref{fX}, \eqref{fY}, we deduce that
\begin{align*}
    E_{3,2}^2\lesssim& \|u_{2,\neq}\|^2_{X_{c}}
   +\|\partial_x\Lambda_{x,z}^{-1/2}\omega_{2,\neq}\|^2_{Y_{c}}\\
   &+ \| \partial_x\omega_{2,\neq}\|_{Y_c}^2 +\nu^{\f14}  \|e^{c\nu^{\f12}t}\partial_x\Lambda_{x,z}^{-1/2}\omega_{2,\neq}\|_{L^{\infty}L^2}^2.
 \end{align*}

 By using $(\partial_x^2+\partial_z^2) u_{3,\neq}
  =-\partial_x\omega_{2,\neq}-\partial_z\partial_yu_{2,\neq} $, we have
   \begin{align*}
E_{3,3}&= \nu^{\f14}\|e^{c\nu^{\f12}t}(\partial_x^2+\partial_z^2) u_{3,\neq}\|_{L^{\infty}L^2}+\nu^{\f34}\|e^{c\nu^{\f12}t}(\partial_x^2+\partial_z^2)
\nabla u_{3,\neq}\|_{L^2L^2}\\ &\leq \nu^{\f14}\|e^{c\nu^{\f12}t}\partial_x\omega_{2,\neq}\|_{L^{\infty}L^2}+
\nu^{\f14}\|e^{c\nu^{\f12}t}\partial_z\partial_yu_{2,\neq}\|_{L^{\infty}L^2}\\&\quad
+\nu^{\f34}\|e^{c\nu^{\f12}t}\partial_x\nabla\omega_{2,\neq}\|_{L^2L^2}
+\nu^{\f34}\|e^{c\nu^{\f12}t}\partial_z\partial_y\nabla u_{2,\neq}\|_{L^2L^2}\\
&\lesssim  \nu^{\f14}\|e^{c\nu^{\f12}t}\partial_x\omega_{2,\neq}\|_{L^{\infty}L^2} +\nu^{\f14}\|\partial_x\omega_{2,\neq}\|_{Y_c}+ \|u_{2,\neq}\|_{X_c}.
\end{align*}

This completes the proof of this lemma.
\end{proof}

\subsection{Estimates of $E_{3,0}$, $E_{3,2}$ and $E_{3,3}$}

\begin{proposition}\label{prop:E30}
It holds that
\begin{align*}
 & E^2_{3,0}+E_{3,2}^2+E_{3,3}^2 +\|u_{2,\neq}\|^2_{X_{c}}\\
  & \lesssim \|u(0)\|^2_{H^{3}}+{\nu^{-3}E_3^4+\nu^{-3}E_2^2E_3^2}+\nu^{-\f32} E_1^2E_3^2.
  \end{align*}
\end{proposition}

\begin{proof}
 By Theorem \ref{thm:sp-no-slip} and Theorem \ref{thm:sp-without-nonlocal}, we obtain
  \begin{equation}\label{est:u2k-00}
    \begin{aligned}
    &\|(u_{2})_k\|^2_{X_{k}^{c}}\leq  C\|(\partial_y^2-|k|^2)(\Delta u_2)_{k}(0)\|_{L^2}^2\\
   &\qquad + \nu^{-1}\Big(\|e^{c\nu^{\f12}t}[\partial_x(u\cdot\nabla
   u_1)+\partial_z(u\cdot\nabla
  u_3)]_{k}\|_{L^2L^2}^2+ \|e^{c\nu^{\f12}t}|k|(u\cdot\nabla
  u_2)_{k}\|^2_{L^2L^2}\Big),
    \end{aligned}
  \end{equation}
and
\begin{equation}\label{est:w2k-00}
    \begin{aligned}
& \|(\omega_2)_{k}\|^2_{Y^{c}_{k}}
   \leq C
  \Big(\|(\omega_2)_{k}(0)\|^2_{L^2}
 \\&\quad+\min\big\{(\nu |k|^2)^{-1},(\nu
  |k_1|)^{-1/2}\big\}\|e^{c\nu^{\f12}t}(k_1(u\cdot\nabla u_3)_{k}-k_3(u\cdot\nabla u_1)_{k})\|^2_{L^2L^2}\\
  &\quad+(k_3^2(|k_1||k|^2)^{-1})\|e^{c\nu^{\f12}t}(\partial_y,|k|)(u_2)_{k}\|^2_{L^2L^2}\Big)\\
  &\leq C
  \Big(\|(\omega_2)_{k}(0)\|^2_{L^2}+\nu^{-1}\big(\|e^{c\nu^{\f12}t}(u\cdot\nabla u_3)_{k}\|^2_{L^2L^2}+\|e^{c\nu^{\f12}t}(u\cdot\nabla u_1)_{k}\|^2_{L^2L^2}\big)\\
  &\quad+|k_1|^{-1}\|e^{c\nu^{\f12}t}(\partial_y,|k|)(u_2)_{k}\|^2_{L^2L^2}\Big).
 \end{aligned}
  \end{equation}
Therefore, we have 
\begin{align*}
& \|\partial_x \omega_{2,\neq} \|^2_{Y_{c}}
   \leq C\sum_{k\in\mathbb{Z}^2;k_1\neq 0}
  \Big(|k_1|^2\| (\omega_2)_{k}(0)\|^2_{L^2}
 \\&\quad+\nu^{-1} |k_1|^2\|e^{c\nu^{\f12}t}((u\cdot\nabla u_3)_{k}-(u\cdot\nabla u_1)_{k})\|^2_{L^2L^2}+|k_1|\|e^{c\nu^{\f12}t}(\partial_y,|k|)(u_2)_{k}\|^2_{L^2L^2}\Big)\\
  &\leq C
  \Big(\|(\partial_x \omega_2)(0)\|^2_{L^2}+\|u_{2,\neq}\|_{X_c}^2\\
  &\quad+\nu^{-1}\big(\|e^{c\nu^{\f12}t}\partial_x (u\cdot\nabla u_3)_{k}\|^2_{L^2L^2}+\|e^{c\nu^{\f12}t}\partial_x (u\cdot\nabla u_1)_{k}\|^2_{L^2L^2}\big)
  \Big).
\end{align*}

By Theorem \ref{thm:sp-without-nonlocal} again, we have 
\begin{align*}
\nu^{\f14} \|e^{c\nu^{\frac{1}{2}}t} (\omega_2)_k\|_{L^\infty L^2}^2 \lesssim& \nu^{\f14} \|(\omega_2)_k(0)\|_{L^2}^2+|k_1|^{-\f34} |k|^{-2} |k_3|^2\|e^{c\nu^{\f12}t} (\partial_y,|k|) (u_2)_k\|_{L^2 L^2}^2\\
&+\nu^{\f14} |k_1|^{-1} |k|^{-1}|k_3|^2 \|e^{c\nu^{\f12}t} (\partial_y,|k|) (u_2)_k\|_{L^2 L^2}^2\\
&+\nu^{-\f34}\big(\|e^{c\nu^{\f12}t}(u\cdot\nabla u_3)_{k}\|^2_{L^2L^2}+\|e^{c\nu^{\f12}t}(u\cdot\nabla u_1)_{k}\|^2_{L^2L^2}\big),
\end{align*}
which gives 
\begin{align*}
&\nu^{\f12}|k_1|^2 \|e^{c\nu^{\frac{1}{2}}t} (\omega_2)_k\|_{L^\infty L^2}^2\\
 &\lesssim \nu^{\f12} |k_1|^2 \|(\omega_2)_k(0)\|_{L^2}^2+\nu^{\f14}|k_1|^{-\f34} |k|^{-2} |k_3|^2  |k_1|^2\|e^{c\nu^{\f12}t} (\partial_y,|k|) (u_2)_k\|_{L^2 L^2}^2\\
&\quad+\nu^{\f12} |k_1|^{-1} |k|^{-1} |k_3|^2 |k_1|^2 \|e^{c\nu^{\f12}t} (\partial_y,|k|) (u_2)_k\|_{L^2 L^2}^2\\
&\quad+\nu^{-\f12}|k_1|^2\big(\|e^{c\nu^{\f12}t}(u\cdot\nabla u_3)_{k}\|^2_{L^2L^2}+\|e^{c\nu^{\f12}t}(u\cdot\nabla u_1)_{k}\|^2_{L^2L^2}\big)\\
&\lesssim |k_1|^2 \|(\omega_2)_k(0)\|_{L^2}^2+ \|u_{2,\neq}\|_{X^c_k}^2\\
&\quad+\nu^{-\f12}|k_1|^2\big(\|e^{c\nu^{\f12}t}(u\cdot\nabla u_3)_{k}\|^2_{L^2L^2}+\|e^{c\nu^{\f12}t}(u\cdot\nabla u_1)_{k}\|^2_{L^2L^2}\big),
\end{align*}
and 
\begin{align*}
&\nu^{\f14}|k_1|^2 |k|^{-1} \|e^{c\nu^{\frac{1}{2}}t} (\omega_2)_k\|_{L^\infty L^2}^2\\
 &\lesssim \nu^{\f14} |k_1|^2 |k|^{-1} \|(\omega_2)_k(0)\|_{L^2}^2+|k|^{-1} |k_1|^2 |k_1|^{-\f34} |k|^{-2} |k_3|^2\|e^{c\nu^{\f12}t} (\partial_y,|k|) (u_2)_k\|_{L^2 L^2}^2\\
&\quad+\nu^{\f14} |k|^{-1} |k_1|^2 |k_1|^{-1} |k|^{-1}|k_3|^2 \|e^{c\nu^{\f12}t} (\partial_y,|k|) (u_2)_k\|_{L^2 L^2}^2\\
&\quad+\nu^{-\f34}|k_1|^2 |k|^{-1} \big(\|e^{c\nu^{\f12}t}(u\cdot\nabla u_3)_{k}\|^2_{L^2L^2}+\|e^{c\nu^{\f12}t}(u\cdot\nabla u_1)_{k}\|^2_{L^2L^2}\big)\\
 &\lesssim \nu^{\f14} |k_1| \| (\omega_2)_k(0)\|_{L^2}^2+\|u_{2,\neq}\|_{X^c_k}^2\\
&\quad+\nu^{-\f34}|k_1| \big(\|e^{c\nu^{\f12}t}(u\cdot\nabla u_3)_{k}\|^2_{L^2L^2}+\|e^{c\nu^{\f12}t}(u\cdot\nabla u_1)_{k}\|^2_{L^2L^2}\big).
\end{align*}

Therefore, we obtain 
\begin{align*}
  & \nu^{\f12} \|e^{c\nu^{\f12}t} \partial_x \omega_{2,\neq}\|_{L^\infty L^2}^2+ \nu^{\f14} \|e^{c\nu^{\f12}t} \partial_x \Lambda_{x,z}^{-1/2} \omega_{2,\neq}\|_{L^\infty L^2}^2\\
  &\leq C
  \Big(\|(\partial_x \omega_2)(0)\|^2_{L^2}+\|u_{2,\neq}\|_{X_c}^2\\
  &\quad+\nu^{-1}\big(\|e^{c\nu^{\f12}t}\partial_x (u\cdot\nabla u_3)_{k}\|^2_{L^2L^2}+\|e^{c\nu^{\f12}t}\partial_x (u\cdot\nabla u_1)_{k}\|^2_{L^2L^2}\big)
  \Big). 
  \end{align*}

It follows from \eqref{est:u2k-00} that
\begin{align}
    \|u_{2,\neq}\|_{X_c}^2&\lesssim {\|u_{\neq}(0)\|^2_{H^{4}}} +
   \nu^{-1}\Big(\|e^{c\nu^{\f12}t}\partial_z(u\cdot\nabla u_3)_{\neq}\|^2_{L^2L^2}\nonumber\\
  &\quad+\|e^{c\nu^{\f12}t}\partial_x(u\cdot\nabla u_1)_{\neq}\|^2_{L^2L^2}+ \|e^{c\nu^{\f12}t}(\partial_x,\partial_z)(u\cdot\nabla
  u_2)_{\neq}\|^2_{L^2L^2}\Big).\label{est:u2neqX-000}
\end{align}
Therefore, we infer from Lemma \ref{lem:E30} that
\begin{align*}
&E^2_{3,0}+E_{3,2}^{2}+E_{3,3}^2+\|u_{2,\neq}\|^2_{X_{c}}\\
    &\lesssim { \|u_{\neq}(0)\|^2_{H^{4}}}+
   \nu^{-1}\Big(\|e^{c\nu^{\f12}t}(\partial_x,\partial_z)(u\cdot\nabla u_3)_{\neq}\|^2_{L^2L^2}\\
   &\qquad+\|e^{c\nu^{\f12}t}\partial_x(u\cdot\nabla u_1)_{\neq}\|^2_{L^2L^2}+ \|e^{c\nu^{\f12}t}(\partial_x,\partial_z)(u\cdot\nabla
  u_2)_{\neq}\|^2_{L^2L^2}\Big).
\end{align*}

Let us estimate each term on the right hand side. For $j\in\{2,3\}$,  we get by Lemma \ref{lem:int-nn}  that
\begin{align*}
   & \|e^{c\nu^{\f12}t}(\partial_x,\partial_z)(u_{\neq}\cdot\nabla u_{j,\neq})\|_{L^2L^2}^2\leq C\nu^{-2}E_3^4.
\end{align*}
By Lemma \ref{lem:int-nz-23}, we get
\begin{align*}
   & \|e^{c\nu^{\f12}t}(\partial_x,\partial_z)(u_{\neq}\cdot\nabla \overline{u}_j)\|_{L^2L^2}^2\lesssim {\nu^{-2}E_2^2E_3^2}.
\end{align*}
And by Lemma \ref{lem:int-zn-1-23} and Lemma \ref{lem:int-nz-23}, we have
\begin{align*}
   &\|e^{c\nu^{\f12}t}(\partial_x,\partial_z)(\overline{u}\cdot\nabla u_{j,\neq})\|_{L^2L^2}^2\\ &\lesssim \|e^{c\nu^{\f12}t}(\partial_x,\partial_z)(\overline{u}_1\partial_x u_{j,\neq})\|_{L^2L^2}^2+\|e^{c\nu^{\f12}t}(\partial_x,\partial_z)
   ((\overline{u}_2\partial_y+\overline{u}_3\partial_z) u_{j,\neq})\|_{L^2L^2}^2\\ &\lesssim \nu^{-\f12} E_1^2E_3^2+{\nu^{-2}E_2^2E_3^2}.
\end{align*}
Notice that  $(fg)_{\neq}=\overline{f}g_{\neq}+f_{\neq}\overline{g}+(f_{\neq}g_{\neq})_{\neq}$, which shows that for $j=2,3$,
\begin{align}\label{u32}
   & \|e^{c\nu^{\f12}t}(\partial_x,\partial_z)(u\cdot\nabla u_j)_{\neq}\|_{L^2L^2}^2\lesssim \nu^{-2}E_3^4+{\nu^{-2}E_2^2E_3^2}+\nu^{-\f12} E_1^2E_3^2.
\end{align}

By Lemma \ref{lem:int-nn}, we have
\begin{align*}
   & \|e^{c\nu^{\f12}t}\partial_x(u_{\neq}\cdot\nabla u_{1,\neq})\|_{L^2L^2}^2\leq C\nu^{-2}E_3^4,
\end{align*}
and  by Lemma \ref{lem:int-zn-1-23}, we have
\begin{align*}
   & \|e^{c\nu^{\f12}t}\partial_x(u_{\neq}\cdot\nabla \overline{u}_1)\|_{L^2L^2}^2\lesssim\nu^{-\f12} E_1^2E_3^2.
\end{align*}
Meanwhile, by Lemma \ref{lem:int-zn-11} and Lemma \ref{lem:int-nz-23}, we get
\begin{align*}
   & \|e^{c\nu^{\f12}t}\partial_x(\overline{u}\cdot\nabla u_{1,\neq})\|_{L^2L^2}^2\\ &\lesssim\|e^{c\nu^{\f12}t}\partial_x(\overline{u}_1\partial_x u_{1,\neq})\|_{L^2L^2}^2+\|e^{c\nu^{\f12}t}\partial_x((\overline{u}_2\partial_y
   +\overline{u}_3\partial_z) u_{1,\neq})\|_{L^2L^2}^2\\ &\lesssim \nu^{-\f12} E_1^2E_3^{2}+{\nu^{-2}E_2^2E_3^2},
\end{align*}
which yields
\begin{align*}
   & \|e^{c\nu^{\f12}t}\partial_x(u\cdot\nabla u_1)_{\neq}\|_{L^2L^2}^2\lesssim\nu^{-2}E_3^4+{\nu^{-2}E_2^2E_3^2}+\nu^{-\f12} E_1^2E_3^2.
\end{align*}

Summing up, we conclude that
\begin{align*}
 & E^2_{3,0}+E_{3,2}^2+E_{3,3}^2 +\|u_{2,\neq}\|^2_{X_{c}}\\
  & \lesssim {\|u(0)\|^2_{H^{4}}}+{\nu^{-3}E_3^4+\nu^{-3}E_2^2E_3^2}+\nu^{-\f32} E_1^2E_3^2.
  \end{align*}

This completes the proof of the proposition.
\end{proof}

\subsection{Estimate of $E_{3,1}$}

\begin{proposition}\label{prop:E31}
It holds that
\begin{align*}
  E^2_{3,1} \lesssim {\|u(0)\|^2_{H^4}}+{\nu^{-\f{13}4}E_3^4+\nu^{-3}E_2^2E_3^2+\nu^{-\f32} E_1^2E_3^2}.
\end{align*}
\end{proposition}
\begin{proof}
Recall that
\begin{equation*}\left\{
\begin{aligned}
  &\partial_t (\omega_2)_ k- \nu(\partial_y^2-|k|^2)(\omega_2)_ k
  +ik_1(1-y^2)(\omega_2)_ k-2ik_3y(u_2)_{k}\\
  &\qquad\quad=-ik_3(u\cdot\nabla u_1)_{k}+ik_1(u\cdot\nabla u_3)_{k},\\
  &(\omega_2)_k|_{y=\pm1}=0,\quad
  (\omega_2)_k|_{t=0}=ik_3 (u_1)_{k}(0)-ik_1(u_3)_{k}(0).
  \end{aligned}\right.
\end{equation*}
By \eqref{est:w2k-00}, we have
\begin{equation*}
\begin{aligned}
&|k|\|(\omega_{2})_{k}\|_{Y^c_{k}}
\lesssim |k|
  \Big(\|(\omega_2)_{k}(0)\|_{L^2}
 \\&\quad+\min\big\{(\nu |k|^2)^{-1/2},(\nu
  |k_1|)^{-1/4}\big\}\|e^{c\nu^{\f12}t}(k_1(u\cdot\nabla u_3)_{k}-k_3(u\cdot\nabla u_1)_{k})\|_{L^2L^2}\\
  &\quad+\min\big\{(\nu |k|^2)^{-1/2},(\nu
  |k_1|)^{-1/4}\big\}|k_3| \|e^{c\nu^{\f12}t} (u_2)_{k}\|_{L^2L^2}\Big)\\
&\lesssim  |k|\|(\omega_{2})_{k}(0)\|_{L^2}+{\nu^{-\f38} |k_1|^{-\f18} |k|^{\f12} |k_3|\| e^{c\nu^{\frac{1}{2}}t} (u_2)_k \|_{L^2 L^2}}\\
&\quad+{ \nu^{-\f12} \|e^{c\nu^{\frac{1}{2}}t}(ik_3(u\cdot \nabla u_1)_{k}-ik_1(u\cdot \nabla u_3)_{k})\|_{L^2 L^2}}
\end{aligned}
\end{equation*}
and
\begin{equation*}
\begin{aligned}
&\|e^{c\nu^{\frac{1}{2}}t} \partial_y (\omega_2)_k\|_{L^\infty L^2}+\nu^{\frac{1}{2}}\|e^{c\nu^{\frac{1}{2}}t} (\partial_y,|k|)  \partial_y (\omega_2)_k\|_{L^2 L^2} +{\nu^{\f18}|k_1|^{\f38}\|e^{c\nu^{\frac{1}{2}}t} \partial_y (\omega_2)_k\|_{L^2 L^2}}\\
&\lesssim \|\partial_y (\omega_2)_k(0)\|_{ L^2}+{\nu^{-\f38 } |k_1|^{\frac{3}{8}}\| (\omega_2)_k(0)\|_{ L^2}}\\
&\quad+{\min\{\nu^{-\frac{5}{8}}|k_1|^{\f18},\nu^{-\f18}\} \|e^{c\nu^{\frac{1}{2}}t}(ik_3(u\cdot \nabla u_1)_{k}-ik_1(u\cdot \nabla u_3)_{k})\|_{L^2 L^2}}\\
&\quad+{|\nu /k_1|^{-\frac{3}{8}}}\Big(|k_3|(|k_1||k|)^{-\frac{1}{2}}\|e^{c\nu^{\frac{1}{2}}t}\partial_y (u_2)_k\|_{L^2 L^2}\\
&\quad+|k_3|(|k| |k_1|^{-1})^{\frac{1}{2}}\|e^{c\nu^{\frac{1}{2}}t} (u_2)_k\|_{L^2 L^2}\Big).
\end{aligned}
\end{equation*}
Here we have used the following fact
\begin{align*}
\nu^{-\f38}|k_1|^{\f38}\min\{|\nu k_1|^{-1/2},(\nu |k|^2)^{-1}\}^{\f12}=&\min\{\nu^{-\f58}|k_1|^{\f18},\nu^{-\f78}|k_1|^{\f38}|k|^{-1}\}\\
\leq& \min\{\nu^{-\f58}|k_1|^{\f18},\nu^{-\f78}\}.
\end{align*}

Therefore, we deduce that
\begin{equation*}
\begin{aligned}
&\|e^{c\nu^{\frac{1}{2}}t} (\partial_y,|k|) (\omega_2)_k\|_{L^\infty L^2}+\nu^{\frac{1}{2}}\|e^{c\nu^{\frac{1}{2}}t} (\partial_y^2-|k|^2)   (\omega_2)_k\|_{L^2 L^2}\\
&\qquad+(\nu |k_1|^3)^{\f18} \|e^{c\nu^{\frac{1}{2}}t}  (\partial_y,|k|)(\omega_2)_k\|_{L^2 L^2}\\
&\lesssim (1+{\nu^{-\frac{3}{8}}|k|^{-\f58}})\|(\partial_y,|k|) (\omega_2)_k (0)\|_{L^2}\\
&\quad+{\nu^{-\f38} |k_1|^{\f38} |k|^{\f12} |k_1|^{-\f12} } \|e^{c\nu^{\frac{1}{2}}t} (\partial_y, |k|) (u_2)_k\|_{L^2 L^2}\\
&\quad+{\min\{\nu^{-\f58}|k_1|^{\f18},\nu^{-\f78}\}\|e^{c\nu^{\frac{1}{2}}t}(ik_3(u\cdot \nabla u_1)_{k}-ik_1(u\cdot \nabla u_3)_{k})\|_{L^2 L^2}}\\
&\lesssim {\nu^{-\f38}} \|(\partial_y,|k|)(\omega_2)_k (0)\|_{L^2}+{\nu^{-\f38} |k|^{\f12} |k_1|^{-\f18} } \|e^{c\nu^{\frac{1}{2}}t} (\partial_y, |k|) (u_2)_k\|_{L^2 L^2}\\
&\quad+{\min\{\nu^{-\f58}|k_1|^{\f18},\nu^{-\f78}\}\|e^{c\nu^{\frac{1}{2}}t}(ik_3(u\cdot \nabla u_1)_{k}-ik_1(u\cdot \nabla u_3)_{k})\|_{L^2 L^2}},
\end{aligned}
\end{equation*}
which implies
\begin{equation}\label{est:Lambdaxzomega-00}
\begin{aligned}
& { \nu^{\f34}} \big(\|e^{c\nu^{\frac{1}{2}}t} \Lambda_{x,z}^{-1/2}\nabla\omega_{2,\neq}\|_{L^\infty L^2}+\nu^{\frac{1}{2}}\|e^{c\nu^{\frac{1}{2}}t} \Lambda_{x,z}^{-1/2}\Delta\omega_{2,\neq}\|_{L^2 L^2}\\
&\quad+\nu^{\frac{3}{4}}\|e^{c\nu^{\frac{1}{2}}t} \Delta\omega_{2,\neq}\|_{L^2 L^2}+\nu^{\f 18}\|e^{c\nu^{\frac{1}{2}}t} \Lambda_{x,z}^{-1/2}\nabla\omega_{2,\neq}\|_{L^2 L^2}\big)^2\\
&\lesssim \|u(0)\|_{H^3}^2+ \|u_{2,\neq}\|_{X_c}^2\\
&\quad+{\nu^{-\f12}} \|e^{c\nu^{\frac{1}{2}}t}(\partial_z(u\cdot \nabla u_1)_{\neq}-\partial_x(u\cdot \nabla u_3)_{\neq})\|_{L^2 L^2}^2,
\end{aligned}
\end{equation}
where we have used
\begin{align*}
    & \|e^{c\nu^{\frac{1}{2}}t} (\partial_y, |k|) (u_2)_k\|_{L^2 L^2}+\nu^{\f14}{|k|^{\f12}|k_1|^{-\f18}}\|e^{c\nu^{\frac{1}{2}}t} (\partial_y, |k|) (u_2)_k\|_{L^2 L^2} \leq \|(u_2)_k\|_{X_k^c}.
 \end{align*}

It remains to estimate the terms on the right hand side.  By \eqref{u32}, we have
\begin{align}\label{est:uu3-00}
   & \|e^{c\nu^{\f12}t}\partial_x(u\cdot\nabla u_3)_{\neq}\|_{L^2L^2}^2\lesssim \nu^{-2}E_3^4+{\nu^{-2}E_2^2E_3^2}+\nu^{-\f12} E_1^2E_3^2.
\end{align}
By Lemma \ref{lem:int-nn}, we have
\begin{align*}
   & \|e^{c\nu^{\f12}t}\partial_z(u_{\neq}\cdot\nabla u_{1,\neq})_{\neq}\|_{L^2L^2}^2\leq C{\nu^{-\f{11}4}E_3^4},
\end{align*}
and by Lemma \ref{lem:int-zz-11-pz}, we have
\begin{align*}
   & \|e^{c\nu^{\f12}t}\partial_z(u_{\neq}\cdot\nabla \overline{u}_1)_{\neq}\|_{L^2L^2}^2\lesssim { \nu^{-1}E^2_1E_{3}^2} ,
\end{align*}
and by Lemma \ref{lem:int-zz-11-pz} and Lemma \ref{lem:int-nz-23},
\begin{align*}
   & \|e^{c\nu^{\f12}t}\partial_z(\overline{u}\cdot\nabla u_{1,\neq})\|_{L^2L^2}^2\\ &\lesssim\|e^{c\nu^{\f12}t}\partial_z(\overline{u}_1\partial_x u_{1,\neq})\|_{L^2L^2}^2
   +\|e^{c\nu^{\f12}t}\partial_z((\overline{u}_2\partial_y+\overline{u}_3\partial_z) u_{1,\neq})\|_{L^2L^2}^2\\ &\lesssim {\nu^{-1}E^2_1E_{3}^2 +\nu^{-2}E_2^2E_3^2}.
\end{align*}
This shows that
\begin{align}\label{est:uu1-00}
   & \|e^{c\nu^{\f12}t}\partial_z(u\cdot\nabla u_1)_{\neq}\|_{L^2L^2}^2\lesssim { \nu^{-\f{11}4}E_3^4+\nu^{-1}E^2_1E_{3}^2 +\nu^{-2}E_2^2E_3^2}.
\end{align}

By \eqref{est:Lambdaxzomega-00}, \eqref{est:uu3-00} and \eqref{est:uu1-00}, we conclude that
\begin{align*}
   E^2_{3,1}=&\nu^{\frac{3}{4}} \big(\|e^{c\nu^{\frac{1}{2}}t} \Lambda_{x,z}^{-1/2}\nabla\omega_{2,\neq}\|_{L^\infty L^2}+\nu^{\frac{1}{2}}\|e^{c\nu^{\frac{1}{2}}t} \Lambda_{x,z}^{-1/2}\Delta\omega_{2,\neq}\|_{L^2 L^2}\\
   &+\nu^{\frac{3}{4}}\|e^{c\nu^{\frac{1}{2}}t} \Delta\omega_{2,\neq}\|_{L^2 L^2}+{ \nu^{\f18}}\|e^{c\nu^{\frac{1}{2}}t} \Lambda_{x,z}^{-1/2}\nabla\omega_{2,\neq}\|_{L^2 L^2}\big)^2\\
   \lesssim & \|u(0)\|_{H^3}^2+\|u_{2,\neq}\|^2_{X_{c}}+{\nu^{-\f{13}4}E_3^4+\nu^{-\f{3}{2}}E^2_1E_{3}^2 +\nu^{-\f52}E_2^2E_3^2},
\end{align*}
which along with Proposition \ref{prop:E30} gives
\begin{align*}
  E^2_{3,1} \lesssim { \|u(0)\|^2_{H^4}}+{ \nu^{-\f{13}4}E_3^4+\nu^{-3}E_2^2E_3^2+\nu^{-\f32} E_1^2E_3^2}.
\end{align*}

This completes the proof of the proposition.
\end{proof}

\section{Nonlinear stability}\label{nonlinear-stable}

In this section, we will prove Theorem \ref{main-result}. Let us suppose that $\nu \in (0,\nu_0], \ \nu_0^{\frac{1}{2}}\leq 4\epsilon<c$ for some $\epsilon\in (0,1)$, then $e^{\nu t}\leq e^{c\nu^{\frac{1}{2}}t}$ for $t\geq 0$.

\subsection{Global well-posedness result}
The local well-posedness for the solution to \eqref{pertur-ins} with \eqref{no-slip} is classical and we will extend the local solution to a global one.

The proof  is based on a bootstrap argument. Let us first assume that
\begin{align}\label{ass:boot}
E_1\leq \varepsilon_1\nu^{\f34}, \quad E_2\leq \varepsilon_1\nu^{\f74},\quad  E_3\leq \varepsilon_1\nu^{\f74}.
\end{align}
where $\varepsilon_1$ is determined later.

Now it follows from Proposition \ref{prop:E1}-\ref{prop:E2} and Proposition \ref{prop:E30}-\ref{prop:E31}
that
\begin{align*}
&E_{1}\leq C\big(\|\overline{u}(0)\|_{H^2}+\nu^{-1}E_2+\nu^{-1}E_2E_{1}+{\nu^{-\f{15}{8}}E_3^2}\big),\\
&E_2\leq C\big(1+\nu^{-1}E_2\big)^2\big(\|u(0)\|_{H^2}+\nu^{-\frac{3}{2}}E_3^2\big),\\
& E^2_{3,0} +E_{3,2}^2+E_{3,3}^2 \leq C\Big({\|u(0)\|^2_{H^{4}}}+{\nu^{-3}E_3^4+\nu^{-3}E_2^2E_3^2}+\nu^{-\f32} E_1^2E_3^2\Big),\\
&E^2_{3,1} \leq C\Big( { \|u(0)\|_{H^4}^2}+{ \nu^{-\f{13}4}E_3^4+\nu^{-3}E_2^2E_3^2+\nu^{-\f32} E_1^2E_3^2}\Big).
\end{align*}
By taking $\varepsilon_1$ small enough, we deduce from \eqref{ass:boot} that
\begin{align*}
E_3 \leq  E_{3,0} +E_{3,1}+E_{3,2}+E_{3,3}
 \leq  C { \|u(0)\|_{H^{4}}}\leq c_0C \nu^{\frac{7}{4}}.
\end{align*}
Then we have
\begin{align*}
E_2\leq c_0C \nu^{\frac{7}{4}},\quad  E_{1}\leq   c_0C\nu^{\frac{3}{4}}.\end{align*}
Now taking $c_0>0$ small enough such that $c_0C<\varepsilon_1/2$, the continuity argument ensures that the Navier-Stokes equations \eqref{pertur-ins} with \eqref{no-slip} admit a global unique solution.

\subsection{Global stability estimates}
Let us turn to give the long-time dynamic behavior of the solution, i.e., \eqref{sta-est-1}-\eqref{streak-2}.

By Lemma \ref{lem:u23-H1} and Lemma \ref{lem:u2-H2}, we get
\begin{align*}
&e^{c\nu t}
\|\nabla\overline{u}_{3}(t)\|_{L^2} \lesssim\big(1+\nu^{-1}E_2\big)\big(\|u(0)\|_{H^2}+\nu^{-\frac{3}{2}}E_3^2\big)\leq C {\|u(0)\|_{H^{4}}},\\
&e^{c\nu t}
\|\Delta \overline{u}_{2}(t)\|_{L^2} \lesssim\big(1+\nu^{-1}E_2\big)^2\big(\|u(0)\|_{H^2}+\nu^{-\frac{3}{2}}E_3^2\big)\leq C {\|u(0)\|_{H^{4}}}.
\end{align*}
Due to $\partial_y \overline{u}_2+\partial_z \overline{u}_3=0$, we have
\begin{align*}
&\|\overline{u}_2(t) \|_{H^2}+\|\overline{u}_3(t)\|_{H^1}+\|\partial_z \overline{u}_3(t)\|_{H^1}\\
\leq&  C(\|\Delta \overline{u}_{2}(t)\|_{L^2}+ \|\nabla\overline{u}_{3}(t)\|_{L^2})
\leq Ce^{-c\nu t}C {\|u(0)\|_{H^{4}}},
\end{align*}
from which and  Lemma \ref{lemma-A.1}, we infer that
\begin{align*}
\|\overline{u}_2(t)\|_{L^\infty}+\|\overline{u}_3(t)\|_{L^\infty}\leq& C(\|\overline{u}_2(t) \|_{H^2}+\|\overline{u}_3(t)\|_{H^1}+\|\partial_z \overline{u}_3(t)\|_{H^1})\\
\leq & Ce^{-c\nu t} {\|u(0)\|_{H^{4}}}.
\end{align*}
Then \eqref{sta-est-2} follows.

Next, we prove the estimates \eqref{streak-1}-\eqref{streak-2} . Firstly, we have
\begin{align*}
&e^{c\nu^{\frac{1}{2}}t}\|\Lambda_{x,z}^{\frac{1}{2}}\nabla u_{2,\not=}(t)\|_{L^2}\\
&\qquad +
\nu^{\frac{1}{4}} e^{c\nu^{\frac{1}{2}}t}( \|(\partial_x,\partial_z)\nabla u_{2,\not=}(t)\|_{L^2}+\|(\partial_x^2+\partial_z^2)u_{3,\not=}(t)\|_{L^2})\\
&\leq E_{3,0}+E_{3,3}\leq E_3\leq C {\|u(0)\|_{H^{4}}},
\end{align*}
which along with Lemma \ref{lem:u-relation} gives
\begin{align*}
\|(\partial_x,\partial_z)\partial_x u_{\not=}(t)\|_{L^2}\leq& C\big(\|(\partial_x,\partial_z)\nabla u_{2,\not=}(t)\|_{L^2}+\|(\partial_x^2+\partial_z^2)u_{3,\not=}(t)\|_{L^2}\big)\\
\leq& C\nu^{-\frac{1}{4}} e^{-c\nu^{\frac{1}{2}}t}  { \|u(0)\|_{H^{4}}}.
\end{align*}
Due to $\partial_x u_{1,\not=}+\partial_y u_{2,\not=}+\partial_z u_{3,\not=}=0$,  we have
\begin{align*}
\|\partial_x u_{\not=}\|_{L^2}\leq&  \|\partial_x u_{1,\not=}\|_{L^2}+\|\partial_x u_{2,\not=}\|_{L^2}+\|\partial_x u_{3,\not=}\|_{L^2}\\
\leq &  C\big(\|\nabla u_{2,\not=}\|_{L^2}+\|(\partial_x,\partial_z)u_{3,\not=}\|_{L^2}\big),
\end{align*}
which gives
\begin{align*}
\|\partial_x u_{\not=} (t)\|_{L^2}\leq&    C(\|\nabla u_{2,\not=}(t)\|_{L^2}+\|(\partial_x,\partial_z)u_{3,\not=}(t) \|_{L^2})\\
\leq & Ce^{-c\nu^{\frac{1}{2}}t} (E_{3,0}+{\nu^{-\f18}E_{3,2}})\leq C{\nu^{-\f18}}e^{-c\nu^{\frac{1}{2}}t} { \|u(0)\|_{H^{4}}},
\end{align*}
and
\begin{align*}
\|u_{\not=}(t)\|_{L^2}\leq \|\partial_x u_{\not=}(t)\|_{L^2}
\leq C{\nu^{-\f18}}e^{-c\nu^{\frac{1}{2}}t} { \|u(0)\|_{H^{4}}}.
\end{align*}
Moreover, we have
\begin{align*}
&\|u_{\not=}\|_{L^\infty L^2}+\nu^{\frac{3}{4}}\|t(u_{1,\not=},u_{3,\not=})\|_{L^2 L^2}\\
&\leq  C\|e^{-c\nu^{\frac{1}{2}}t}\|_{L^\infty(0,+\infty)}{\nu^{-\f18}}e^{-c\nu^{\frac{1}{2}}t} {\|u(0)\|_{H^{4}}}\\
&\quad+C\nu^{\frac{3}{4}} \|te^{-c\nu^{\frac{1}{2}}t}\|_{L^2(0,+\infty)}{ \nu^{-\f18}}e^{-c\nu^{\frac{1}{2}}t} { \|u(0)\|_{H^{4}}}\\
&\leq C{\nu^{-\f18}}e^{-c\nu^{\frac{1}{2}}t} { \|u(0)\|_{H^{4}}}.
\end{align*}

By \eqref{fX} and Proposition \ref{prop:E30}, we get
\begin{align*}
  { \nu^{\frac{7}{12}} }&e^{2c\nu^{\frac{1}{2}}t}\|\Delta u_{2,\not=}(t)\|_{L^2}^2 \leq \|u_{2,\neq}\|^2_{X_{c}}\\
   \lesssim&  { \|u(0)\|_{H^{4}}^2}+\nu^{-3}E_3^4+{ \nu^{-3}E_2^2E_3^2}+\nu^{-\f32} E_1^2E_3^2
 \leq  C {\|u(0)\|_{H^{4}}^2},
\end{align*}
which yields
\begin{align*}
{\nu^{\frac{7}{24}}\|u_{2,\not=}(t)\|_{H^2}\leq C\nu^{\frac{7}{24}}\|\Delta u_{2,\not=}(t)\|_{L^2}\leq Ce^{-c\nu^{\frac{1}{2}}t}\|u(0)\|_{H^{4}}}.
\end{align*}
Thanks to the definition of $E_3$, we have
\begin{align*}
\|\nabla u_{2,\not=}\|_{L^\infty L^2}+\|\nabla u_{2,\not=}\|_{L^2 L^2}\leq E_3\leq C  {\|u(0)\|_{H^{4}}}.
\end{align*}

Notice that $(\partial_x^2+\partial_z^2)u_{1,\not=}=\partial_z\omega_{2,\not=}-\partial_x\partial_y u_{2,\not=}, \ -(\partial_x^2+\partial_z^2)u_{3,\not=}=\partial_x\omega_{2,\not=}+\partial_z\partial_y u_{2,\not=}$, which imply that
\begin{align*}
&\|(u_{1,\not=},u_{3,\not=})(t)\|_{H^1} \\
&\leq C(\|\Lambda_{x,z}^{-1}\nabla \omega_{2,\not=}(t)\|_{L^2}+\|\Lambda_{x,z}^{-1} \Delta u_{2,\not=}(t)\|_{L^2}+\|\nabla \partial_y u_{2,\not=}(t)\|_{L^2})\\
 &\leq C(\|\Lambda_{x,z}^{-\frac{1}{2}}\nabla \omega_{2,\not=}(t)\|_{L^2}+\| \Delta u_{2,\not=}(t)\|_{L^2})\\
 &\leq C e^{-c\nu^{\frac{1}{2}}t}({ \nu^{-\frac{3}{8}}E_{3,1}}+{\nu^{-\frac{7}{24}}} E_{3,0})\leq C{\nu^{-\frac{3}{8}}  e^{-c\nu^{\frac{1}{2}}t}\|u(0)\|_{H^{4}}}
\end{align*}
and
\begin{align*}
\|\Lambda_{x,z}^{\frac{1}{2}}\nabla (u_{1,\not=},u_{3,\not=})(t)\|_{L^2} \leq & C(\|\Lambda_{x,z}^{-\frac{1}{2}}\nabla\omega_{2,\not=}(t)\|_{L^2}+\|\Lambda_{x,z}^{-\frac{1}{2}} \Delta u_{2,\not=}(t)\|_{L^2})\\
\leq & C({\nu^{-\frac{3}{8}} e^{-c\nu^{\frac{1}{2}}t}E_{3,1}+\nu^{-\frac{7}{24}} e^{-c\nu^{\frac{1}{2}}t}E_{3,0}})\\
\leq& C{\nu^{-\frac{3}{8}}  e^{-c\nu^{\frac{1}{2}}t}\|u(0)\|_{H^{4}}}.
\end{align*}
By  Lemma \ref{lemma-L-infty-interpo}, for $j=1,3$, we have
\begin{align*}
\|u_{j,\not=}(t)\|_{L^\infty} \leq& C\|(\partial_x,\partial_z)\partial_x u_{j,\not=}(t)\|_{L^2}^{\frac{1}{2}} \|\Lambda_{x,z}^{\frac{1}{2}}\nabla u_{j,\not=}(t)\|_{L^2}^{\frac{1}{2}}\\
\leq & C{\nu^{-\frac{3}{8}} e^{-c\nu^{\frac{1}{2}}t}\|u(0)\|_{H^{4}}}.
\end{align*}
Moreover, for any $\sigma>0$, by Lemma \ref{lemma-L-infty-interpo}, we obtain
\begin{align*}
\|u_{2,\not=}(t)\|_{L^\infty} \leq &C \|\Lambda_{x,z}^{\frac{1}{2}+\sigma}\partial_x u_{2,\not=}(t)\|_{L^2}^{\frac{1}{2}} \|\Lambda_{x,z}^{\frac{1}{2}}\nabla u_{2,\not=}(t)\|_{L^2}^{\frac{1}{2}}\\
 \leq &C \|\Lambda_{x,z}^{\frac{1}{2}}\partial_x u_{2,\not=}(t)\|_{L^2}^{\frac{1}{2}-\sigma} \|\Lambda_{x,z}\partial_x u_{2,\not=}(t)\|_{L^2}^{\sigma} \|\Lambda_{x,z}^{\frac{1}{2}}\nabla u_{2,\not=}(t)\|_{L^2}^{\frac{1}{2}}\\
 \leq & C \|\Lambda_{x,z}\partial_x u_{2,\not=}(t)\|_{L^2}^{\sigma} \|\Lambda_{x,z}^{\frac{1}{2}}\nabla u_{2,\not=}(t)\|_{L^2}^{1-\sigma}\\
 \leq& C\nu^{-\frac{\sigma}{4}} e^{-c\nu^{\frac{1}{2}}t} {\|u(0)\|_{H^{4}} }=:C \nu^{-s} e^{-c\nu^{\frac{1}{2}}t} {\|u(0)\|_{H^{4}}}.
\end{align*}

Summing up, \eqref{streak-1}-\eqref{streak-2} follow.

Let us turn to prove \eqref{sta-est-1}. Recall that $\overline{u}_{1}$ satisfies
\begin{align*}
\left\{
\begin{aligned}
&(\partial_t-\nu\Delta )\overline{u}_{1}+2y\overline{u}_2+\overline{u}_2\partial_y\overline{u}_{1}+\overline{u}_3
\partial_z\overline{u}_{1}+\overline{u_{\neq}\cdot\nabla u_{1,\neq}}=0,\\
&\overline{u}_{1}|_{t=0}=\overline{u}_1(0),\quad \Delta\overline{u}_{1}|_{y=\pm1}=\overline{u}_{1}|_{y=\pm1}=0.
\end{aligned}
\right.
\end{align*}
Thanks to $\Delta\overline{u}_{1}|_{y=\pm1}=0$,  the energy estimate gives
\begin{align*}
&\frac{ \mathrm{d} }{\mathrm{d}t}\| \Delta\overline{u}_{1}\|_{L^2}^2+2\nu\| \nabla\Delta\overline{u}_{1}\|_{L^2}^2+4\big\langle
\nabla(y\overline{u}_2),\nabla\Delta\overline{u}_{1}\big\rangle
\\&\quad-2\Big\langle \nabla(\overline{u}_2\partial_y\overline{u}_{1}+\overline{u}_3\partial_z\overline{u}_{1}+
\overline{u_{\neq}\cdot\nabla u_{1,\neq}}), \nabla\Delta\overline{u}_{1}\Big\rangle=0,
\end{align*}
which gives
\begin{align*}
&\frac{ \mathrm{d} }{\mathrm{d}t}\| \Delta\overline{u}_{1}\|_{L^2}^2+\nu\| \nabla\Delta\overline{u}_{1}\|_{L^2}^2\\
&\lesssim \nu^{-1}\Big(\|\overline{u}_{2}\|_{H^1}^2+\|\nabla (\overline{u}_2\partial_y\overline{u}_{1}+\overline{u}_3\partial_z\overline{u}_{1})\|_{L^2}^2+\|\nabla (\overline{u_{\neq}\cdot\nabla u_{1,\neq}})\|_{L^2}^2\Big).
\end{align*}
Thanks to $\|\nabla \Delta \overline{u}_1\|_{L^2}^2\geq (\pi/2)^2\|\Delta \overline{u}_1\|_{L^2}^2$, we can deduce that
\begin{align*}
\| e^{c\nu t}\Delta\overline{u}_{1}&\|_{L^{\infty}L^2}^2+\nu\| e^{c\nu t} \nabla\Delta\overline{u}_{1}\|_{L^2L^2}^2\lesssim \| u(0)\|_{H^2}^2+\nu^{-1}\Big(\| e^{c\nu t}\overline{u}_{2}\|_{L^2H^1}^2\\
&+\| e^{c\nu t} \nabla (\overline{u}_{2}\partial_y\overline{u}_{1}+\overline{u}_{3}\partial_z\overline{u}_{1}
)\|_{L^2L^2}^2+\| e^{c\nu t}\nabla (\overline{u_{\neq}\cdot\nabla u_{1,\neq}})\|_{L^2L^2}^2\Big).
\end{align*}

Using the fact that
\begin{align*}
\|\nabla(\overline{u}_{j}\partial_j\overline{u}_{1})\|_{L^2}^2\lesssim& \|\overline{u}_{j}\|_{H^1}^2\|\partial_j \overline{u}_{1}\|_{H^2}^2\lesssim \|\nabla\overline{u}_{j}\|_{L^2}^2 \|\nabla\Delta\overline{u}_{1}\|_{L^2}^2, \ j=2,3,
\end{align*}
we have
\begin{align*}
&\|\nabla (\overline{u}_{2}\partial_y\overline{u}_{1}+
\overline{u}_{3}\partial_z\overline{u}_{1})\|_{L^2}^2
\lesssim \big(\|\nabla\overline{u}_{2}\|_{L^2}^2+\|\nabla\overline{u}_{3}\|_{L^2}^2\big)
\|\nabla\Delta\overline{u}_{1}\|_{L^2}^2.
\end{align*}
Therefore, we obtain
\begin{align*}
&\|e^{\nu t}\nabla (\overline{u}_{2}\partial_y\overline{u}_{1}+
\overline{u}_{3}\partial_z\overline{u}_{1})\|_{L^2 L^2}^2\\
&\lesssim \big(\|e^{\nu t}\nabla\overline{u}_{2}\|_{L^\infty L^2}^2+\|e^{\nu t}\nabla\overline{u}_{3}\|_{L^\infty L^2}^2\big)
\|\nabla\Delta\overline{u}_{1}\|_{L^2 L^2}^2\\
&\lesssim ( {\|u(0)\|_{H^{4}}^2}+\nu^{-3}E_3^4) (1+\nu^{-2}E_2^2)\cdot \nu^{-1} E_1^2
\lesssim \nu^{-1}  {\|u(0)\|_{H^{4}}^2},
\end{align*}
where we have used Lemma \ref{lem:u23-H1}. By Lemma \ref{lem:u23-H1} again, we get
\begin{align*}
\|e^{\nu t}\overline{u}_2\|_{L^2 H^1}^2
\lesssim& \nu^{-1} ( {\|u(0)\|_{H^{4}}^2}+\nu^{-3}E_3^4)\lesssim \nu^{-1} {\|u(0)\|_{H^{4}}^2}.
\end{align*}
Due to $0<\nu \leq \nu_0,\ \nu_0^{\frac{1}{2}}<c$, we get by  Lemma \ref{lem:int-nn} that
\begin{align*}
\| e^{\nu t}\nabla (\overline{u_{\neq}\cdot\nabla u_{1,\neq}})\|_{L^2L^2}^2\lesssim& \| e^{c\nu^{\frac{1}{2}} t}\nabla (u_{\neq}\cdot\nabla u_{1,\neq})\|_{L^2L^2}^2\\
\lesssim&{\nu^{-\frac{11}{4}}} E_3^4\lesssim  {\nu^{\f34} \|u(0)\|_{H^{4}}^2}.
\end{align*}

Putting the above estimates together, we arrive at
\begin{align*}
\| e^{\nu t}\Delta\overline{u}_{1}\|_{L^{\infty}L^2}^2+\nu\| e^{\nu t} \nabla\Delta\overline{u}_{1}\|_{L^2L^2}^2\lesssim&  {\|u(0)\|_{H^{4}}^2}+\nu^{-2} {\|u(0)\|_{H^{4}}^2}\\
\lesssim& \nu^{-2} {\|u(0)\|_{H^{4}}^2},
\end{align*}
 which yields 
\begin{align*}
&\|\overline{u}_{1}(t)\|_{H^2}\leq C\|\Delta\overline{u}_{1}(t)\|_{L^2} \lesssim \nu^{-1} e^{- \nu t} {\|u(0)\|_{H^{4}}}.
\end{align*}
Moreover, we have
\begin{align*}
&\|\overline{u}_1(t)\|_{L^\infty}\leq C\|\overline{u}_{1}(t)\|_{H^2}\leq C \nu^{-1} e^{- \nu t} {\|u(0)\|_{H^{4}}}.
\end{align*}
Then \eqref{sta-est-1} is obtained.

\appendix

\section{Some basic inequalities}\label{ap2}

The following lemma comes from Lemma 9.3 in \cite{CLWZ}.
\begin{lemma}\label{lap-w-phi}
Suppose that $(\partial_y^2-|k|^2)\varphi=w$ with $\varphi(\pm 1)=0$ and $|k|\geq 1$. Then it holds that
\begin{equation*}
\begin{aligned}
&\|(\partial_y,|k|)\varphi\|_{L^2}\leq C|k|^{-\frac{1}{2}}\|w\|_{L^1},\\
&\|(\partial_y,|k|)\varphi\|_{L^\infty }\leq C \|w\|_{L^1},\\
& \|(\partial_y,|k|)\varphi\|_{L^\infty }\leq C|k|^{-\frac{1}{2}}\|w\|_{L^2}.
\end{aligned}
\end{equation*}
\end{lemma}

The following lemma comes from Lemma C.2 in \cite{CWZ-cmp}.

\begin{lemma}\label{lem5.3}
If $ \psi_1,\psi_2\in H^2(0,1)\cap H_0^1(0,1)$, $g\in H^2([0,1])$ satisfies $ (\partial^2_y-|k|^2)\psi_1=g(\partial^2_y-|k|^2)\psi_2$ {and $|k|\geq 1$}, then we have
\begin{align*}
&|k|^{\frac{1}{2}}\|(\partial_y,|k|)\psi_1\|_{L^2}+|\partial_y\psi_1(0)|+|\partial_y\psi_1(1)|\\
&\qquad\leq C\|g\|_{C^1}\big(|k|^{\frac{1}{2}}\|(\partial_y,|k|)\psi_2\|_{L^2}+|\partial_y\psi_2(0)|
+|\partial_y\psi_2(1)|\big),
\\&\|(\partial_y,|k|)(\psi_1-g\psi_2)\|_{L^2}\leq C\|\partial_yg\|_{H^1}\|\psi_2\|_{L^2}.
\end{align*}
\end{lemma}

\begin{lemma}\label{hardy-type}
For $f\in H^1_0(-1,1)$, it holds that
\begin{equation}\nonumber
\begin{aligned}
&\left\|\frac{f}{(1-y^2)} \right\|_{L^2}^2\leq 4 \|\partial_y f\|_{L^2}^2,\ \left\|\frac{f}{(1-y^2)^{\frac{1}{2}}}\right\|_{L^2}^2 \leq \frac{1}{2}\|\partial_y f\|_{L^2}^2,\\
& \|f\|_{L^2}^2\leq \|(1-y^2)^{\frac{1}{2}} f\|_{L^2}^{\frac{4}{3}}\|(1-y^2)^{-1} f\|_{L^2}^{\frac{2}{3}}\leq C \|(1-y^2)^{\frac{1}{2}} f\|_{L^2}^{\frac{4}{3}}\|\partial_y f\|_{L^2}^{\frac{2}{3}}.
\end{aligned}
\end{equation}
\end{lemma}

\begin{proof}
Thanks to $1-y^2\geq (1-|y|)$, the first inequality follows from the classical Hardy's inequality. For $f\in H^1_0(-1,1)$, we find that
\begin{equation}\nonumber
\begin{aligned}
0\leq &\int_{-1}^1\left|\partial_y f+\frac{2y f}{1-y^2}\right|^2\mathrm{d}y\\
=& \int_{-1}^1|\partial_y f|^2\mathrm{d}y +\int_{-1}^1\frac{2y \partial_y (|f|^2)}{1-y^2}\mathrm{d}y+\int_{-1}^1\frac{4y^2|f|^2}{(1-y^2)^2}\mathrm{d}y\\
=&\int_{-1}^1|\partial_y f|^2\mathrm{d}y -\int_{-1}^1\frac{2 |f|^2}{1-y^2}\mathrm{d}y,
\end{aligned}
\end{equation}
which gives the second inequality. Moreover, we have
\begin{equation}\nonumber
\begin{aligned}
\|f\|_{L^2}^2 \leq & \left\| \frac{f}{(1-y^2)^{\frac{1}{2}}} \right\|_{L^2} \|(1-y^2)^{\frac{1}{2}}f\|_{L^2}
\leq  \|f \|_{L^2}^{\frac{1}{2}} \left\| \frac{f}{(1-y^2)} \right\|_{L^2}^{\frac{1}{2}}  \|(1-y^2)^{\frac{1}{2}}f\|_{L^2},
\end{aligned}
\end{equation}
which along with the first inequality of this lemma yields the last inequality.
\end{proof}

\begin{lemma}\label{positive}
Suppose that $(\partial_y^2-|k|^2)\varphi=w,\varphi(\pm 1)=0, |k|\neq 0$. Then it holds that
\begin{equation}\nonumber
\begin{aligned}
&2\left(\bigg\|\varphi'+\dfrac{2y\varphi}{1-y^2}\bigg\|_{L^2}^2+|k|^2\|\varphi\|_{L^2}^2\right)\leq \langle (1-y^2)w+2\varphi,w\rangle,\\
& \|(\varphi',|k|\varphi)\|_{L^2}^2+|k|^2\|\varphi\|_{L^2}^2 \leq \big(8/|k|^2+1\big) \langle (1-y^2)w+2\varphi,w\rangle.
\end{aligned}
\end{equation}
\end{lemma}

\begin{proof}
The first inequality comes from Lemma 3.1 in \cite{DL} and we only show the second inequality.
Notice that
\begin{align*}
   & \bigg\|\varphi'+\dfrac{2y\varphi}{1-y^2}\bigg\|_{L^2}^2= \|\varphi'\|_{L^2}^2- 2\left\|\dfrac{\varphi}{(1-y^2)^{\f12}}\right\|_{L^2}^2\geq \|\varphi'\|_{L^2}^2-2\left\|\dfrac{\varphi}{1-y^2}\right\|_{L^2} \|\varphi\|_{L^2},
\end{align*}
from which and Lemma \ref{hardy-type}, we infer that
\begin{align*}
   & \bigg\|\varphi'+\dfrac{2y\varphi}{1-y^2}\bigg\|_{L^2}^2\geq \|\varphi'\|_{L^2}^2-4\|\varphi'\|_{L^2} \|\varphi\|_{L^2}\geq \dfrac{1}{2}\|\varphi'\|_{L^2}^2-8\|\varphi\|_{L^2}^2.
\end{align*}
By the first inequality of this lemma, we have
\begin{align*}
   & \|\varphi'\|_{L^2}^2+2|k|^2\|\varphi\|_{L^2}^2 \leq 16\|\varphi\|_{L^2}^2+ \langle (1-y^2)w+2\varphi,w\rangle,\\
   &2|k|^2\|\varphi\|_{L^2}^2 \leq   \langle (1-y^2)w+2\varphi,w\rangle.
\end{align*}
Thus, we conclude
\begin{align*}
   & \|(\varphi',|k|\varphi)\|_{L^2}^2+|k|^2\|\varphi\|_{L^2}^2 \leq \big(8/|k|^2+1\big) \langle (1-y^2)w+2\varphi,w\rangle.
\end{align*}
\end{proof}

The following lemma comes from Lemma 16.7 in \cite{CWZ-mem}.

\begin{lemma}\label{delta-bar{p}}If $f$ is a function in $[-1,1]\times\mathbb{T}$, i.e., $f=f(y,z)$,  then we have
\begin{align*}
\|\partial_y^2f\|_{L^2}+\|\partial_z^2f\|_{L^2}\leq C\big(\|\partial_z\partial_yf\|_{L^2}+\|\Delta f\|_{L^2}\big).
\end{align*}
\end{lemma}

\begin{lemma}\label{lemma-L-infty-interpo}
If $P_0 f=0$, then it holds that for $s_1,s_2\geq 0$ with $s_1+s_2>0$,
\begin{align*}
\|f\|_{L^\infty}\leq C\|\Lambda_{x,z}^{\frac{1+s_1}{2}} \partial_x f\|_{L^2}^{\frac{1}{2}}\|\Lambda_{x,z}^{\frac{1+s_2}{2}}\nabla f\|_{L^2}^{\frac{1}{2}}.
\end{align*}
\end{lemma}

\begin{proof}
For $P_0f=0$, let $f=\sum\limits_{k\in \mathbb{Z}^2;k_1\not=0} f_k(y)e^{ik\cdot (x,z)}$, where $f_k$ is defined as $f_k(y)=\frac{1}{2\pi}\int_{\mathbb{T}^2}f(x,y,z) e^{-ik\cdot (x,z)}\mathrm{d}x\mathrm{d}z$ with $k=(k_1,k_3)$. Then for $s_1,s_2\geq 0, s=\frac{s_1+s_2}{2}>0$, we have
\begin{align*}
\|f\|_{L^\infty}\leq & C \sum_{k\in \mathbb{Z}^2;k_1\not=0} \|f_k(y)\|_{L^\infty}\leq C \sum_{k\in \mathbb{Z}^2; k_1\not=0} \|f_k(y)\|_{L^2}^{\frac{1}{2}}\|(\partial_y,1)f_k\|_{L^2}^{\frac{1}{2}}\\
\leq&C \Big( \sum_{k\in \mathbb{Z}^2; k_1\not=0}|k_1|^2|k|^{1+s_1} \|f_k(y)\|_{L^2}^2\Big)^{\frac{1}{4}}\Big( \sum_{k\in \mathbb{Z}^2; k_1\not=0}|k|^{1+s_2} \|(\partial_y,1)f_k\|_{L^2}^2\Big)^{\frac{1}{4}} \\
&\times\Big( \sum_{k\in \mathbb{Z}^2; k_1\not=0}\frac{1}{|k_1||k|^{1+s}}\Big)^{\frac{1}{2}}\\
\leq &C \|\Lambda_{x,z}^{\frac{1+s_1}{2}}\partial_x f\|_{L^2}^{\frac{1}{2}}\|\Lambda_{x,z}^{\frac{1+s_2}{2}}\nabla f\|_{L^2}^{\frac{1}{2}}\sum_{k_1\not=0}\frac{1}{|k_1|^{1+s}}\\
\leq& C\|\Lambda_{x,z}^{\frac{1+s_1}{2}}\partial_x f\|_{L^2}^{\frac{1}{2}}\|\Lambda_{x,z}^{\frac{1+s_2}{2}}\nabla f\|_{L^2}^{\frac{1}{2}},
\end{align*}
where we have used the fact
 \begin{align*}
\sum_{k_3\in \mathbb{Z}}\frac{1}{|k|^{1+s}}\leq \frac{1}{|k_1|^{1+s}}+\int_{\mathbb{R}}\frac{\mathrm{d}z}{(|k_1|^2+z^2)^{\frac{1+s}{2}}}\leq \frac{C}{|k_1|^{s}}.
\end{align*}
Then the conclusion follows.
\end{proof}

\section{Some lemmas}\label{ap1}

Let $w$ be the solution to \eqref{OS-nopertur} and $1-y_i^2=\lambda \in [0,1]$ with $-1\leq y_1\leq 0\leq y_2\leq 1, \ y_1=-y_2$, and recall the decomposition $\varphi=\varphi_1+\varphi_2$ defined by \eqref{streamdeco1}-\eqref{streamdeco2}.\smallskip

Let us state some lemmas, which can be obtained by similar arguments in \cite{DL}.

\begin{lemma}\label{lem:outside}
It holds that for any $\delta \in (0,1]$
\begin{equation}\nonumber
\Vert w \Vert_{L^2\big((-1,1)\setminus (y_1,y_2)\big)}^2 \leq C \mathscr{E}_1(w),
\end{equation}
where
\begin{equation}\label{def:E1}
\begin{aligned}
\mathscr{E}_1(w)=&\frac{\Vert F \Vert_{L^2}\Vert w \Vert_{L^2}}{|k_1|(y_2-y_1+\delta)\delta}+\frac{\nu\Vert w'\overline{w}\Vert_{L^\infty(B(y_1,\delta)\cup B(y_2,\delta))}}{|k_1|(y_2-y_1+\delta)\delta}\\
&+\frac{\nu \Vert w' \Vert_{L^2}\Vert w \Vert_{L^\infty}}{|k_1|(y_2-y_1+\delta)\delta^{\frac{3}{2}}}+\frac{\Vert \varphi_2 \Vert_{L^\infty}^2}{(y_2-y_1+\delta)^2\delta}+\delta \Vert w \Vert_{L^\infty}^2.
\end{aligned}
\end{equation}
\end{lemma}

\begin{lemma}\label{lem:inside}
Let $\delta \in (0,\frac{y_2-y_1}{4}]$. It holds that
\begin{equation}\nonumber
\Vert w \Vert_{L^2(y_1,y_2)}^2 \leq C \mathscr{E}_2(w),
\end{equation}
where
\begin{equation}\label{def:E2}
\begin{aligned}
\mathscr{E}_2(w)=&\frac{\Vert F \Vert_{L^2}\Vert w \Vert_{L^2}}{|k_1|\delta(y_2-y_1)}+\frac{\nu \Vert w'\overline{w}\Vert_{L^\infty(B(y_1,\delta)\cup B(y_2,\delta))}}{|k_1|\delta (y_2-y_1)}\\
&+\delta \Vert w \Vert_{L^\infty}^2+\frac{\nu^2\Vert w \Vert_{L^\infty}^2}{|k_1|^2(y_2-y_1)^3\delta^4}+\frac{\nu \Vert w'\Vert_{L^2}\Vert w \Vert_{L^\infty}}{|k_1|\delta^{\frac{3}{2}}(y_2-y_1)}\\
&+\frac{\Vert \varphi_2\Vert_{L^\infty}^2}{(y_2-y_1)^2\delta}+\delta^3\Vert w'\Vert_{L^\infty(B(y_1,\delta)\cup B(y_2,\delta))}^2+\frac{\Vert F \Vert_{L^2}^2}{|k_1|^2(y_2-y_1)^3\delta}\\
&+\frac{\nu^2}{|k_1|^2}\left(\frac{\Vert w'\Vert_{L^\infty(B(y_1,\delta)\cup B(y_2,\delta))}^2}{(y_2-y_1)^3\delta^2}+\frac{\Vert w'\Vert_{L^2}^2}{(y_2-y_1)^3\delta^3}\right).
\end{aligned}
\end{equation}
\end{lemma}

\begin{lemma}\label{lem:varphi-Linfty}
For any $\nu \in (0,1],\lambda\in [0,1]$ and $\delta \in (0,1]$, there holds that
\begin{equation}\nonumber
\begin{aligned}
\frac{\Vert \varphi_2 \Vert_{L^\infty}^2}{(y_2-y_1+\delta)^2\delta} \lesssim \frac{\Vert \varphi \Vert_{L^\infty((B(y_1,\delta)\cup B(y_2,\delta))}^2}{(y_2-y_1+\delta)^2\delta}\lesssim\mathscr{F}_1(w),
\end{aligned}
\end{equation}
where
\begin{equation*}
\begin{aligned}
\mathscr{F}_1(w)=\frac{\Vert F\Vert_{L^2}^2}{|k_1|^2\delta^2(y_2-y_1+\delta)^2}+\frac{\nu\Vert F\Vert_{L^2}\Vert w \Vert_{L^2}}{|k_1|^2\delta^4(y_2-y_1+\delta)^2}+\delta \Vert w \Vert_{L^\infty}^2.
\end{aligned}
\end{equation*}
\end{lemma}
\begin{lemma}\label{lem:dw-Linfty}
For any $\nu \in (0,1],\lambda\in [0,1]$ and $\delta \in (0,1]$, there holds that
\begin{equation}\nonumber
\begin{aligned}
\delta^3\Vert w'\Vert_{L^\infty(B(y_1,\delta)\cup B(y_2,\delta))}^2\lesssim \mathscr{F}_2(w),
\end{aligned}
\end{equation}
where
\begin{equation*}
\begin{aligned}
\mathscr{F}_2(w)=\frac{\delta^6(y_2-y_1+\delta)^2|k_1|^2}{\nu^2}\mathscr{F}_1(w)+\frac{\delta^4}{\nu^2}\Vert F  \Vert_{L^2}^2+\delta\|w\|_{L^\infty}^2.
\end{aligned}
\end{equation*}
\end{lemma}

\section{Limiting  absorption principle}\label{sec:lap}

In this appendix, we consider the Rayleigh equation
\begin{equation}\label{lap-Rayleigh-general}
\begin{aligned}
(u(y)-c)(\partial_y^2-|k|^2)\psi-u''(y)\psi=g,\quad \psi(\pm 1)=0,
 \end{aligned}
\end{equation}
where $u\in\mathcal{K}:=\{u(y):u(y)\in H^3(-1,1), u'(y_c)=0,u''(y_c)\not=0, u'(\pm 1)\not=0\}$, $\mathrm{Re}\ c\in \text{Ran}\ u, 0<|\mathrm{Im}\ c| \leq \epsilon_2$ for $\epsilon_2>0$ small enough. A classical example is the Poiseuille flow $u(y)=y^2\in \mathcal{K}$.

\begin{proposition}\label{est-1}
Let $\psi\in H^1(-1,1)$ be the solution to \eqref{lap-Rayleigh-general}. Then there exists $K\geq 3$ such that for any $|k|\geq K$, we have
\begin{equation}\label{estimate-stream}
\begin{aligned}
\|\partial_y\psi\|_{L^2(-1,1)}+|k|\|\psi\|_{L^2(-1,1)}\leq C(\|\partial_y g\|_{L^2}+|k|\|g\|_{L^2}),
 \end{aligned}
\end{equation}
where $C>0$ is independent of $|k|,c,\epsilon_2$.
\end{proposition}

\begin{proposition}\label{est-2}
 Let $\psi\in H^1(-1,1)$ be the solution to \eqref{lap-Rayleigh-general} with $g=-\frac{\sinh |k|(1+y)}{\sinh 2|k|}$. Then it holds that
\begin{equation}\label{estimate-stream-2}
\begin{aligned}
\|\partial_y\psi\|_{L^2(-1,1)}+|k|\|\psi \|_{L^2(-1,1)}\leq C|k|^{-\frac{1}{2}} ,
 \end{aligned}
\end{equation}
where $C>0$ is independent of $|k|,c,\epsilon_2$.
\end{proposition}

\begin{remark}\label{rk-abp}
In fact, for any fixed $|k|\ge 1$, Proposition \ref{est-1} has been proved in \cite{WZZ-apde}. Here main new point is that the constant $C$ is uniform in $k$.
The estimate \eqref{estimate-stream-2} still holds for $g=-\frac{\sinh |k|(1-y)}{\sinh 2|k|}$. Formally, the good decay of $|k|$ for estimate \eqref{estimate-stream-2} comes from two part: the monotone part (away from the critical point $y=0$), and the fast decay part (near critical point $y=0$). It should be pointed out that the fast decay near critical point $y=0$ is due to the fast decay of $g$ away from the boundaries $\{y=\pm 1\}$.
\end{remark}

To prove Proposition \ref{est-1}-\ref{est-2}, we need the following several lemmas.

\begin{lemma}\label{hardy-type-abp}
If $f\in H^1_0(a,b), u\in C^2([a,b]), u(y)\in \mathbb{R}, |u'(y)|\geq c_0>0, |u''(y)|<C_0$ in $[a,b]$, then for $c\not\in \mathbb{R}$, it holds that
 \begin{equation}\label{hardy-type-abp-0}
\begin{aligned}
\left|\int_a^b \frac{f(y)}{u(y)-c}\mathrm{d}y \right|\leq Cc_0^{-1}  |k|^{-\frac{1}{2}} (\|\partial_y f\|_{L^2}+|k|\|f\|_{L^2}+c_0^{-1}\|f\|_{L^2}) ,
\end{aligned}
\end{equation}
where $C>0$ is independent of $|k|,c,\epsilon_2$. Moreover, if $|k|\geq c_0^{-1}$, then it holds that
\begin{equation}\label{hardy-type-abp-1}
\begin{aligned}
\left|\int_a^b \frac{f(y)}{u(y)-c}\mathrm{d}y \right|\leq C c_0^{-1} |k|^{-\frac{1}{2}}(\|\partial_y f\|_{L^2}+|k|\|f\|_{L^2}),
 \end{aligned}
\end{equation}
where $C>0$ is independent of $|k|,c,\epsilon_2$.
\end{lemma}
\begin{proof}
Without loss of generality, we suppose that $u'(y)>0$, otherwise we consider $f,-u,-c$. We only consider the case of $\mathrm{Im}\ c>0$ and the case of $\mathrm{Im}\ c<0$ can be obtained by taking conjugation.

Indeed, we have
 \begin{equation}\nonumber
\begin{aligned}
\int_a^b \frac{f(y)}{u(y)-c}\mathrm{d}y =& i\int_a^bf(y)\int_0^{+\infty} e^{-it(u(y)-c)}\mathrm{d}t\mathrm{d}y\\
=&i\int_0^{+\infty}e^{itc}\int_a^b f(y)e^{-itu(y)}\mathrm{d}y \mathrm{d}t,
 \end{aligned}
\end{equation}
which gives
 \begin{equation}\nonumber
\begin{aligned}
\left|\int_a^b \frac{f(y)}{u(y)-c}\mathrm{d}y \right|\leq \int_0^{+\infty} |e^{itc}| \left|\int_a^b f(y)e^{-itu(y)}\mathrm{d}y\right| \mathrm{d}t\leq \|g\|_{L^1(\mathbb{R})},
 \end{aligned}
\end{equation}
where $g(t)=\int_a^b f(y)e^{-itu(y)}\mathrm{d}y$. Note that for $f\in H^1_0(a,b)$,  we have
 \begin{equation}\nonumber
\begin{aligned}
&g(t)=\int_a^b f(y)e^{-itu(y)}\mathrm{d}y=\int_{u(a)}^{u(b)}e^{-itz} (f/u')\circ u^{-1}(z)\mathrm{d}z,\\
&itg(t)=\int_{u(a)}^{u(b)}e^{-itz}( (f/u')'/u') \circ u^{-1}(z)\mathrm{d}z,
 \end{aligned}
\end{equation}
which yield by Plancherel's theorem that
 \begin{equation}\nonumber
\begin{aligned}
&\|g(t)\|_{L^2(\mathbb{R})}^2=2\pi \|(f/u')\circ u^{-1}\|_{L^2(u(a),u(b))}^2=2\pi \||f|^2/u'\|_{L^1(a,b)}\leq \frac{2\pi \|f\|_{L^2}^2}{c_0},\\
&\|tg(t)\|_{L^2(\mathbb{R})}^2=2\pi \|| (f/u')'|^2 /u' \|_{L^1(a,b)}\leq Cc_0^{-3}(\|\partial_y f\|_{L^2}+c_0^{-1}\|f\|_{L^2})^2.
 \end{aligned}
\end{equation}
Therefore, we obtain
 \begin{equation}\nonumber
\begin{aligned}
&\|(|k|^2+(c_0t)^2)^{\frac{1}{2}}g(t)\|_{L^2(\mathbb{R})}^2=|k|^2\|g\|_{L^2(\mathbb{R})}^2+c_0^1 \|tg(t)\|_{L^2(\mathbb{R})}^2\\
\leq &Cc_0^{-1}(\|\partial_y f\|_{L^2}+|k|\|f\|_{L^2}+c_0^{-1}\|f\|_{L^2})^2,
 \end{aligned}
\end{equation}
which gives
\begin{equation}\nonumber
\begin{aligned}
\left|\int_a^b \frac{f(y)}{u(y)-c}\mathrm{d}y \right|\leq& \|g\|_{L^1(\mathbb{R})}\leq
\|(|k|^2+(c_0t)^2)^{\frac{1}{2}}g(t)\|_{L^2(\mathbb{R})}\|(|k|^2+(c_0t)^2)^{-\frac{1}{2}}\|_{L^2(\mathbb{R})}\\
\leq& Cc_0^{-\frac{1}{2}}(\|\partial_y f\|_{L^2}+|k|\|f\|_{L^2}+c_0^{-1}\|f\|_{L^2}) (c_0|k|)^{-\frac{1}{2}}.
 \end{aligned}
\end{equation}
Then \eqref{hardy-type-abp-0} follows. And \eqref{hardy-type-abp-1} is a direct consequence of \eqref{hardy-type-abp-0} by taking $|k|\geq c_0^{-1}$.
\end{proof}

Without loss of generality, let us assume that $y_c=0$ and there exists $\delta,c_0\in (0,1)$ so that $B(0,\delta) \subset [-1,1]$, and $|u''|>c_0$ for $y\in \Sigma_\delta$, where $\Sigma_\delta=[-1,1]\setminus B(0,\delta)$. Therefore, one has $|u'(y)|\geq c_0|y|, y\in B(0,\delta)$, and then $\inf\limits_{\Sigma_r} |u'|\geq c_0r, 0<r<\delta$.

\begin{lemma}\label{monotone}
Assume that $r\in (0,\delta), |k|>\frac{2}{c_0r},c\not\in \mathbb{R}$. Let $\Phi$ be the solution to
$$(u(y)-c)(\partial_y^2-|k|^2)\Phi=g_1, \ \Phi(\pm 1)=0.$$
Then it holds that
\begin{equation}\nonumber
\begin{aligned}
\|\partial_y\Phi\|_{L^2(\Sigma_r)}+|k|\|\Phi\|_{L^2(\Sigma_r)}\leq C(r|k|)^{-1} \left(\|\partial_y g_1\|_{L^2}+|k|\| g_1\|_{L^2}+|k|\|\Phi\|_{L^2}\right),
 \end{aligned}
\end{equation}
where $C>0$ is independent of $|k|,c,\epsilon_2$.
\end{lemma}

\begin{proof}
Let $\eta (y)\in C^\infty_0(\mathbb{R})$ be a cut-off function with $\eta (y)=1$ for $|y|\leq \frac{1}{2}$, $\eta(y)=0$ for $|y|\geq \frac{3}{4}$, and $\eta \in [0,1]$ for $y\in\mathbb{R}$. Then we define $\eta_r(y)=1-\eta(y/r)$. It is easy to see that  $|\eta_r|\in [0,1], |\eta'_r|\lesssim r^{-1},  |\eta''_r|\lesssim r^{-2}$ for $y\in [-1,1]$, $\eta_r=1$ for $y\in\Sigma_r,$ $\eta_r=0$ for $y\in [-1,1]\setminus \Sigma_{\frac{r}{2}}$.

A direct calculation yields that
\begin{equation}\label{mono-1}
\begin{aligned}
\int_{-1}^1\eta_r (|\Phi'|^2+|k|^2|\Phi|^2)&=\int_{-1}^1 \left(\eta_r''\frac{|\Phi|^2}{2}-\eta_r\mathrm{Re} [\Phi''-|k|^2\Phi ]\overline{\Phi} \right)\\
&=\int_{-1}^1 \eta_r''\frac{|\Phi|^2}{2}-\mathrm{Re} \int_{ \Sigma_{\frac{r}{2}}}\eta_r \frac{g_1\overline{\Phi}}{u(y)-c}\\
&\leq Cr^{-2} \|\Phi\|_{L^2}^2+\left| \int_{ \Sigma_{\frac{r}{2}}}\eta_r \frac{g_1\overline{\Phi}}{u(y)-c}\right|.
 \end{aligned}
\end{equation}

Note that $\inf\limits_{\Sigma_{\frac{r}{2}}} |u'|\geq \frac{1}{2}c_0r, |k|>\left(\frac{1}{2}c_0r\right)^{-1}, (g_1\Phi)|_{y=\pm 1}=0, \eta_r\left(\pm \frac{r}{2}\right)=0$. Then we get by Lemma \ref{hardy-type-abp} that
\begin{equation}\label{mono-2}
\begin{aligned}
&\left| \int_{-1}^1 \frac{\eta_r g_1\overline{\Phi}}{u(y)-c}\right|\\
&\leq C\left(\frac{1}{2}c_0r\right)^{-1}|k|^{-\frac{1}{2}}\big(\|\partial_y (\eta_rg_1\overline{\Phi})\|_{L^2 }+|k|\| (\eta_rg_1\overline{\Phi})\|_{L^2}\big).
 \end{aligned}
\end{equation}

We introduce
\begin{equation}\nonumber
\begin{aligned}
&R_1=\|\eta_r\partial_y\Phi\|_{L^2}+|k|\|\eta_r\Phi\|_{L^2}, \quad R_2=\|\partial_y (\eta_r \Phi)\|_{L^2}+ |k|\| (\eta_r \Phi)\|_{L^2},\\
& R_3=\|\partial_y g_1\|_{L^2} +|k| \| g_1\|_{L^2}.
\end{aligned}
\end{equation}
Then $R_2\leq R_1+Cr^{-1}\|\Phi\|_{L^2}.$
By the Sobolev embedding, we have
\begin{equation}\nonumber
\begin{aligned}
\|g_1\|_{L^\infty}\leq C\|g_1\|_{L^2}^{\frac{1}{2}} \|g_1\|_{H^1}^{\frac{1}{2}} \leq C|k|^{-\frac{1}{2}}(\|\partial_y g_1\|_{L^2}+|k|\|g_1\|_{L^2})\leq C|k|^{-\frac{1}{2}} R_3,
 \end{aligned}
\end{equation}
and $ \|\eta_r\Phi\|_{L^\infty}\leq C|k|^{-\frac{1}{2}} R_2$. Thus, we obtain
\begin{equation}\label{mono-3}
\begin{aligned}
&\|\partial_y (\eta_rg_1\overline{\Phi})\|_{L^2 }+|k|\| (\eta_rg_1\overline{\Phi})\|_{L^2}\\
&\leq \|\partial_y g_1\|_{L^2 }\| (\eta_r\overline{\Phi})\|_{L^\infty}+\|g_1\|_{L^\infty}\|\partial_y (\eta_r\overline{\Phi})\|_{L^2} +|k| \|g_1\|_{L^\infty} \| (\eta_r \overline{\Phi})\|_{L^2}\\
&\leq C (R_3\cdot |k|^{-\frac{1}{2}}R_2+|k|^{-\frac{1}{2}}R_3\cdot R_2+|k|\cdot |k|^{-\frac{1}{2}}R_3\cdot|k|^{-\frac{1}{2}}R_2 )\\
&\leq C |k|^{-\frac{1}{2}}R_2R_3.
 \end{aligned}
\end{equation}
It follows from \eqref{mono-1}-\eqref{mono-3} that
\begin{equation}\nonumber
\begin{aligned}
\frac{R_1^2}{2}\leq &\int_{-1}^1\eta_r (|\Phi'|^2+|k|^2|\Phi|^2)
\leq Cr^{-2} \|\Phi\|_{L^2}^2+\left| \int_{ \Sigma_{\frac{r}{2}}}\eta_r \frac{g_1\overline{\Phi}}{u(y)-c}\right|\\
\leq & Cr^{-2} \|\Phi\|_{L^2}^2+C\left(\frac{1}{2}c_0r\right)^{-1}|k|^{-1}R_2R_3\\
\leq & Cr^{-2} \|\Phi\|_{L^2}^2+C\left(c_0r\right)^{-1}|k|^{-1}(R_1+Cr^{-1}\|\Phi\|_{L^2})R_3,
 \end{aligned}
\end{equation}
which yields that
\begin{equation}\nonumber
\begin{aligned}
R_1
\leq  Cr^{-1} \|\Phi\|_{L^2}+C\left(c_0r\right)^{-1}|k|^{-1} R_3\leq C(r|k|)^{-1}(R_3+|k|\|\Phi\|_{L^2}).
 \end{aligned}
\end{equation}
Then the conclusion follows from $\eta_r=1$ in $\Sigma_r$.
\end{proof}

\begin{lemma}\label{critical-1}
Let $\psi_n,g_n \in H^1(a,b), u_n\in H^3(a,b)$ be a sequence with
\begin{equation}\nonumber
\begin{aligned}
&\psi_n \rightharpoonup \psi, g_n \rightharpoonup g \ \ in \  \ H^1(a,b),\ u_n \to u_0\ \ in \ \ H^3(a,b),\\
 &u_n (\partial_y^2 \psi_n-(\alpha_n)^2\psi_n)-u''_n\psi_n=g_n,
 \end{aligned}
\end{equation}
and $\alpha_n\to \al, \, \mathrm{Im}\ c_n<0$. In addition, $\mathrm{Im}\ u_0=0, y\in [a,b], u_0(y_0)=0, y_0\in [a,b], u_0'(y)u'_0(y_0)>0, y\in [a,b]$. Then it holds that $\psi_n\to \psi$ in $H^1(a,b)$ and for any $\phi\in H^1_0(a,b)$, we have
\begin{equation}\nonumber
\begin{aligned}
\int_{a}^b (\psi'\phi'+|k|^2\psi\phi)\mathrm{d}y +\mathrm{p.v.}\int_{a}^b \frac{(u_0''\psi+g)\phi}{u_0}\mathrm{d}y+i\pi \frac{((u_0''\psi+g)\phi)(y_0)}{|u_0'(y_0)|}=0.
 \end{aligned}
\end{equation}
\end{lemma}
\begin{proof}
See Lemma 6.2 in \cite{WZZ-apde} for the proof.
\end{proof}

\begin{lemma}\label{critical-3}
Let $\alpha_n\in\mathbb{R}, \psi_n,g_n\in H^1(a,b), u\in H^3(a,b)$ be a sequence with
\begin{equation}\nonumber
\begin{aligned}
&\|\psi_n'\|_{L^2(a,b)}+|\alpha_n|\|\psi_n\|_{L^2(a,b)} \leq 2,\quad \|g_n'\|_{L^2(a,b)}+|\alpha_n|\|g_n\|_{L^2(a,b)} \to 0,\\ &(u-c_n)(\psi_n''-\alpha_n^2\psi_n)-u''\psi_n=g_n,
 \end{aligned}
\end{equation}
and $\alpha_n\to +\infty,\mathrm{Im}\ c_n>0, \alpha_n^2(c_n-u(y_0))\to 0, u'(y_0)=g_n(y_0)=0, y_0=\frac{a+b}{2}, \delta=y_0-a \in [0,1], u''(y) u''(y_0)>0, y\in [a,b]$. Then it holds that
$$\|\psi_n'\|_{L^2(y_0-\delta/2,y_0+\delta/2)}+|\alpha_n|\|\psi_n\|_{L^2(y_0-\delta/2,y_0+\delta/2)} \to 0.$$
\end{lemma}
\begin{proof}
Without loss of generality, let us suppose that $y_0=0,u''(0)=2,u(0)=0$ and then $[a,b]=[-\delta,\delta]$. Let $c_n=r_n^2 e^{2i\theta_n}, r_n>0,\theta_n\in (0,\pi/2),$ then $\alpha_n r_n\to 0,r_n\to 0$. As in Lemma 6.3 in \cite{WZZ-apde}, we have
\begin{equation}\nonumber
\begin{aligned}
|(\psi''-\alpha^2\psi)(0)|\leq Cr_n^{-\frac{3}{2}},\ |\psi (0)|\leq C r_n^{\frac{1}{2}},\ |(u''\psi+g)(0)|\leq C r_n^{\frac{1}{2}}.
 \end{aligned}
\end{equation}
Thus, we introduce
\begin{equation}\nonumber
\begin{aligned}
\widetilde{\psi}_n=r_n^{-\frac{1}{2}}\psi_n(r_n y),\quad \widetilde{g}_n(y)=r_n^{-\frac{1}{2}} g_n(r_ny),\quad u_n(y)=r_n^{-2}(u(r_n y)-u(0)).
 \end{aligned}
\end{equation}
It holds that
\begin{equation}\nonumber
\begin{aligned}
&(u_n-e^{2i\theta_n})(\widetilde{\psi}_n''-(\alpha_n r_n)^2 \widetilde{\psi}_n) -u_n''  \widetilde{\psi}_n=\widetilde{g}_n(y),\\
&|\widetilde{\psi}_n(0)|=|r_n^{-\frac{1}{2}}\psi_n( 0)|\leq C, \ \|\widetilde{\psi}_n'\|_{L^2(a/r_n,b/r_n)}=\|{\psi}_n'\|_{L^2(a,b)}\leq C,\\
&|\widetilde{g}_n(y)(0)|=0,\ \|\widetilde{g}_n'\|_{L^2(a/r_n,b/r_n)}=\|{g}_n'\|_{L^2(a,b)}\to 0.
 \end{aligned}
\end{equation}
And $\widetilde{\psi}_n$ is bounded in $H^1_{loc}(\mathbb{R})$ and $\widetilde{g}_n(y)\to 0$ in $H^1_{loc}(\mathbb{R})$. Up to a subsequence, let us assume that $\widetilde{\psi}_n \rightharpoonup \widetilde{\psi}_0$ in $H^1_{loc}(\mathbb{R})$, $\theta_n\to \theta_0\in [0,\pi/2]$ with $\widetilde{\psi}_0' \in L^2(\mathbb{R})$.

Following similar arguments as in Lemma 6.4 in \cite{WZZ-apde}, we can show that $u_n \to y^2$ in $H^3_{loc}(\mathbb{R})$, $\widetilde{\psi}_0=0$ and $\widetilde{\psi}_n \to 0$ in $H^1_{loc}(\mathbb{R})\cap C^1_{loc}(\mathbb{R}\setminus \{\pm 1\})$. Then we have
\begin{equation}\nonumber
\begin{aligned}
&|\alpha_n| \|\psi_n\|_{L^2(-2r_n,2r_n)}+ \|\psi_n'\|_{L^2(-2r_n,2r_n)}\\
&=|\alpha_n r_n| \|\widetilde{\psi}_n\|_{L^2(-2,2)}+ \|\widetilde{\psi}_n'\|_{L^2(-2,2)}\to 0.
 \end{aligned}
\end{equation}
For $n>1$, choose $b_n\in [\delta/2,\delta]$ such that
\begin{equation}\nonumber
\begin{aligned}
&2|\alpha_n \psi_n' \overline{\psi_n}(b_n)|\leq |\alpha_n \psi_n' (b_n)|^2+|\overline{\psi_n}(b_n)|^2\\
&\leq 2\delta^{-1}(|\alpha_n|\|\psi_n\|_{L^2(a,b)}^2+ \|\psi_n'\|_{L^2(a,b)}^2) \leq 8\delta^{-1},
 \end{aligned}
\end{equation}
which gives $| \psi_n' \overline{\psi_n}(b_n)|\leq 8(|\alpha_n|\delta)^{-1}\to 0$. We get by integration by part  that
\ben\label{ap-int-1}
&\int_{2r_n}^{b_n} \left(|\psi_n'|^2+\alpha_n^2 |\psi_n|^2+\frac{u''|\psi_n|^2}{u-c_n}\right)\mathrm{d}y\nonumber\\
&=-\int_{2r_n}^{b_n}\frac{g_n\overline{\psi_n}}{u-c_n}\mathrm{d}y+\psi_n' \overline{\psi_n}|_{2r_n}^{b_n}.
\een
For $n$ large enough and $y\in [2r_n,b]$, we have $u(y)\geq u(2r_n)\geq 2r_n^2=2|c_n|$, and then
\begin{equation}\nonumber
\begin{aligned}
\mathrm{Re} \frac{1}{u(y)-c_n}\geq \frac{1}{Cy^2},\quad u''(y)\geq C^{-1}, \quad \left|\frac{1}{u(y)-c_n}\right|\leq \frac{C}{y^2}.
 \end{aligned}
\end{equation}
Taking the real part of \eqref{ap-int-1} to obtain
\begin{equation}\nonumber
\begin{aligned}
& \int_{2r_n}^{b_n} \left(|\psi_n'|^2+\alpha_n^2 |\psi_n|^2\right)\mathrm{d}y+C^{-1}\left\|\frac{\psi_n}{y}\right\|_{L^2(2r_n,b_n)}^2\\
&\leq C \left\|\frac{\psi_n}{y}\right\|_{L^2(2r_n,b_n)}\left\|\frac{g_n}{y}\right\|_{L^2(2r_n,b_n)}+\left|\psi_n' \overline{\psi_n}|_{2r_n}^{b_n}\right|,
 \end{aligned}
\end{equation}
which yields that
\begin{equation}\nonumber
\begin{aligned}
\|\psi_n'\|_{L^2(2r_n,b_n)}^2+\alpha_n^2 \|\psi_n\|_{L^2(2r_n,b_n)}^2 \leq C\left\|\frac{g_n}{y}\right\|_{L^2(2r_n,b_n)}^2+\left|\psi_n' \overline{\psi_n}|_{2r_n}^{b_n}\right|.
 \end{aligned}
\end{equation}
Due to $\psi_n' \overline{\psi_n}|_{2r_n}^{b_n}=\psi_n' \overline{\psi_n}(b_n)-\widetilde{\psi_n}' \overline{\widetilde{\psi_n}}(2) \to 0$ and
$$\left\|\frac{g_n}{y}\right\|_{L^2(2r_n,b_n)}\leq C\|g_n'\|_{L^2(a,b)}\to 0,$$
we obtain
\begin{equation}\nonumber
\begin{aligned}
&\|\psi_n'\|_{L^2(2r_n, \delta/2)}+ |\alpha_n| \|\psi_n\|_{L^2(2r_n,\delta/2)} \\
\leq& \|\psi_n'\|_{L^2(2r_n,b_n)}+|\alpha_n| \|\psi_n\|_{L^2(2r_n,b_n)} \to 0.
 \end{aligned}
\end{equation}

Similar arguments yield that $\|\psi_n'\|_{L^2(-\delta/2,-2r_n)}+ |\alpha_n| \|\psi_n\|_{L^2(-\delta/2,-2r_n)}\to 0$.
\end{proof}

\begin{lemma}\label{critical-4}
Let $\alpha_n\in \mathbb{R}, \psi_n,g_n \in H^1(a,b), u\in H^3(a,b)$ be a sequence with
\begin{equation}\nonumber
\begin{aligned}
&\|\psi_n'\|_{L^2(a,b)}+|\alpha_n|\|\psi_n\|_{L^2(a,b)} \leq 1,\quad \|g_n'\|_{L^2(a,b)}+|\alpha_n|\|g_n\|_{L^2(a,b)} \to 0,\\ &(u-c_n)(\psi_n''-\alpha_n^2\psi_n)-u''\psi_n=g_n,
 \end{aligned}
\end{equation}
and $\alpha_n\to +\infty,\mathrm{Im}\ c_n>0, u'(y_0)=0, y_0=\frac{a+b}{2}, \delta=y_0-a \in [0,1], u''(y) u''(y_0)>0, y\in [a,b]$. Then it holds that
$\widetilde{\psi_n}\to 0$ in $H^1_{loc}(\mathbb{R})$, where $\widetilde{\psi_n}(y)=|\alpha_n|^{\frac{1}{2}} \psi_n (y_0+y/|\alpha_n|).$
\end{lemma}

\begin{proof}
Without loss of generality, let us assume that $y_0=0,u''(0)=2,u(0)=0$ and then $[a,b]=[-\delta,\delta]$. Let $c_n=r_n^2 e^{2i\theta_n}, r_n>0,\theta_n\in (0,\pi/2)$ and $\widetilde{g}_n=|\alpha_n|^{\frac{1}{2}} g_n (y/|\alpha_n|), u_n(y)=|\alpha_n|^2 (u(y/|\alpha_n|)-u(0))$. Then we have
\begin{equation}\nonumber
\begin{aligned}
&(u_n-(\alpha_n r_n)^2 e^{2i\theta_n})(\widetilde{\psi}_n''- \widetilde{\psi}_n) -u_n''  \widetilde{\psi}_n=\widetilde{g}_n(y),\\
& \|\widetilde{\psi}_n'\|_{L^2 (|\alpha_n|a,|\alpha_n|b)}+ \|\widetilde{\psi}_n\|_{L^2 (|\alpha_n|a,|\alpha_n|b)}= \|{\psi}_n'\|_{L^2 (a,b)}+ |\alpha_n| \|{\psi}_n\|_{L^2 (a,b)}\leq 1,\\
& \|\widetilde{g}_n'\|_{L^2 (|\alpha_n|a,|\alpha_n|b)}+ \|\widetilde{g}_n\|_{L^2 (|\alpha_n|a,|\alpha_n|b)}= \|{g}_n'\|_{L^2 (a,b)}+ |\alpha_n| \|{g}_n\|_{L^2 (a,b)}\to 0,
 \end{aligned}
\end{equation}
which implies that $\widetilde{\psi}_n$ is bounded in $H^1_{loc}(\mathbb{R})$ and $\widetilde{g}_n \to 0$ in $H^1_{loc}(\mathbb{R})$. Up to a subsequence, let us suppose that $\widetilde{\psi}_n \rightharpoonup \widetilde{\psi}_0$ in $H^1_{loc}(\mathbb{R})$, $\theta_n\to \theta_0\in [0,\pi/2], |\alpha_n| r_n \to r_0\in [0,+\infty] $ with $\widetilde{\psi}_0 \in H^1(\mathbb{R})$. Following similar arguments as in Lemma 6.4 in \cite{WZZ-apde}, we can deduce that $u_n \to y^2$ in $H^3_{loc}(\mathbb{R})$.\smallskip

{\bf Case 1.} $r_0=+\infty.$

In this case, it holds that $(|u_n''|+1)/(u_n-(|\alpha_n|r_n)^2 e^{2i\theta_n})\to 0$ in $L^\infty_{loc}(\mathbb{R})$ and $\widetilde{\psi}_n''- \widetilde{\psi}_n\to 0$ in $L^\infty_{loc}(\mathbb{R})$, thus $\widetilde{\psi}_n \to \widetilde{\psi}_0$ in $H^1_{loc}(\mathbb{R})$ with $\widetilde{\psi}_0''- \widetilde{\psi}_0=0$. Note that $\widetilde{\psi}_0 \in H^1(\mathbb{R})$, therefore $\widetilde{\psi}_0=0$ and $\widetilde{\psi}_n \to 0$ in $H^1_{loc}(\mathbb{R})$.\smallskip

{\bf Case 2.}  $r_0\in (0,+\infty).$

Firstly, if $\theta_0\not=0$, then $\widetilde{\psi}_n''- \widetilde{\psi}_n$ is bounded in $L^\infty_{loc}(\mathbb{R})$ and $\widetilde{\psi}_n \to \widetilde{\psi}_0$ in $C^1_{loc}(\mathbb{R})$ with
\begin{equation}\nonumber
\begin{aligned}
(y^2-r_0^2  e^{2i\theta_0})(\widetilde{\psi}_0''- \widetilde{\psi}_0) -2  \widetilde{\psi}_0=0, \end{aligned}
\end{equation}
which gives
\begin{equation}\nonumber
\begin{aligned}
\int_{\mathbb{R}}( |\widetilde{\psi}_0'|^2+|\widetilde{\psi}_0|^2)\mathrm{d}y=-\int_{\mathbb{R}}\frac{2|\widetilde{\psi}_0'|^2}{y^2-r_0^2  e^{2i\theta_0}}\mathrm{d}y.
 \end{aligned}
\end{equation}
Multiplying $e^{i\theta_0}$ and taking the imaginary part of the result to yield that
\begin{equation}\nonumber
\begin{aligned}
\sin \theta_0\int_{\mathbb{R}}( |\widetilde{\psi}_0'|^2+|\widetilde{\psi}_0|^2)\mathrm{d}y=-\sin \theta_0 \int_{\mathbb{R}}\frac{2(y^2+r_0^2) |\widetilde{\psi}_0'|^2}{|y^2-r_0^2  e^{2i\theta_0}|^2}\mathrm{d}y,
 \end{aligned}
\end{equation}
which implies that $\widetilde{\psi}_0=0$ and $\widetilde{\psi}_n \to 0$ in $H^1_{loc}(\mathbb{R})$.

If $\theta_0=0$, then $\widetilde{\psi}_n''- \widetilde{\psi}_n$ is bounded in $L^\infty_{loc}(\mathbb{R} \setminus \{\pm r_0\})$ and $\widetilde{\psi}_n \to \widetilde{\psi}_0$ in $C^1_{loc}(\mathbb{R} \setminus \{\pm r_0\})$ with
\begin{equation}\nonumber
\begin{aligned}
(y^2-r_0^2  e^{2i\theta_0})(\widetilde{\psi}_0''- \widetilde{\psi}_0) -2  \widetilde{\psi}_0=0, \ y\in \mathbb{R}\setminus \{\pm r_0\}.
\end{aligned}
\end{equation}
By Lemma \ref{critical-1}, $ \widetilde{\psi}_n \to  \widetilde{\psi}_0$ in $H^1(r_0 (1-\delta),r_0(1+\delta)) \cap H^1( -r_0(1+\delta),-r_0 (1-\delta))$, then $\widetilde{\psi}_n \to \widetilde{\psi}_0$ in $H^1_{loc}(\mathbb{R})\cap C^1_{loc}(\mathbb{R} \setminus \{\pm r_0\})$. Moreover, for any $\phi \in H^1(\mathbb{R})$, it holds that
\begin{equation}\label{theta-0}
\begin{aligned}
\int_{\mathbb{R}} (\widetilde{\psi}_0' \phi'+\widetilde{\psi}_0\phi)\mathrm{d}y+\mathrm{p.v.}\int_{\mathbb{R}} \frac{2\widetilde{\psi}_0\phi}{y^2-r_0^2}\mathrm{d}y+i\pi \sum_{y=\pm r_0}\frac{\widetilde{\psi}_0 \phi}{r_0}=0.
\end{aligned}
\end{equation}
Taking $\phi=\overline{\widetilde{\psi}_0}$ and taking the real part, we get $\widetilde{\psi}_0(\pm r_0)=0$. Thus, $\widetilde{\psi}_0\in H^2 (\mathbb{R})$ with
\begin{equation}\nonumber
\begin{aligned}
\int_{\mathbb{R}} \left(\left|\widetilde{\psi}_0'+\frac{-2y \widetilde{\psi}_0}{y^2-r_0^2}\right|^2+|\widetilde{\psi}_0|^2\right) \mathrm{d}y=\int_{\mathbb{R}} \left( |\widetilde{\psi}_0'|^2+|\widetilde{\psi}_0|^2+\frac{2|\widetilde{\psi}_0|^2}{y^2-r_0^2}\right)\mathrm{d}y=0,
\end{aligned}
\end{equation}
which yields that $\widetilde{\psi}_0=0$ and $\widetilde{\psi}_n \to 0$ in $H^1_{loc}(\mathbb{R})$.\smallskip

{\bf Case 3.} $r_0=0.$

In this case, $|\alpha_n|^2(c_n-u(y_0))\to 0$. Let us introduce a cut-off function $\eta\in C^\infty_0(\mathbb{R})$ with $\eta(y)=\cosh y$ for $|y|\leq 1$, $\eta=0$ for $|y|\geq 2$ and define $\eta_1=\eta''-\eta$. Obviously, $\eta_1=0$ provided that $|y|\leq 1$ or $|y|\geq 2$. Then we define
\begin{equation}\nonumber
\begin{aligned}
&\psi_{n^*}(y)=\psi_n(y) +\frac{g_n(y)}{u''(0)}\eta(|\alpha_n|y),\\
& g_{n^*}(y)=g_n(y)-\frac{u''(y)g_n(0)}{u''(0)}\eta (|\alpha_n|y)+|\alpha_n|^2(u(y)-c_n)\frac{g_n(0)}{u''(0)} \eta_1 (|\alpha_n|y).
\end{aligned}
\end{equation}
It hods that $\psi_{n^*},g_{n^*}\in H^1(a,b), g_{n^*}(0)=0$ with
\begin{equation}\nonumber
\begin{aligned}
(u(y)-c_n)(\psi_{n^*}''-\alpha_n^2 \psi_{n^*})-u''\psi_{n^*}=g_{n^*}.
\end{aligned}
\end{equation}
Note that $u_n=|\alpha_n|^2 u(y/|\alpha_n|), u_n\to y^2$ in $H^3_{loc}(\mathbb{R})$, then $u''(y)=u''_n(y/|\alpha_n|)$, $\alpha_n^2 (u(y)-c_n)=u_n (|\alpha_n| y)-\alpha_n^2 c_n$, $|\alpha_n^2 c_n|\to r_0^2=0,$ and for $n$ large enough, it holds that $\|u_n''\|_{H^1(-2,2)}+\|u_n-\alpha_n^2 c_n\|_{H^1(-2,2)}\leq C$.
Recalling the scaling $z=|\alpha_n|y$, we obtain
\begin{equation}\nonumber
\begin{aligned}
& |\alpha_n| \|\eta(|\alpha_n| y )\|_{L^2(a,b)}+  \|\partial_y \eta(|\alpha_n| y )\|_{L^2(a,b)}+|\alpha_n| \|u''(y)\eta(|\alpha_n| y )\|_{L^2(a,b)}\\
&\quad+\|\partial_y (u''(y)\eta(|\alpha_n| y ))\|_{L^2(a,b)}+ |\alpha_n| \| \alpha_n^2(u(y)-c_n)\eta_1(|\alpha_n| y )\|_{L^2(a,b)}\\&\quad+  \|\partial_y \left( \alpha_n^2(u(y)-c_n)\eta_1(|\alpha_n| y )\right)\|_{L^2(a,b)}\\
&\leq |\alpha_n|^{\frac{1}{2}} \Big(\|\eta\|_{L^2 (a|\alpha_n|,b|\alpha_n|)}+\|\eta'\|_{L^2 (a|\alpha_n|,b|\alpha_n|)}+\|u''_n\eta\|_{L^2 (a|\alpha_n|,b|\alpha_n|)}\\
&\quad+\|(u''_n\eta)'\|_{L^2 (a|\alpha_n|,b|\alpha_n|)}+\| (u_n- \alpha_n^2c_n)\eta_1\|_{L^2(a|\alpha_n|,b|\alpha_n|)}\\
&\quad +\|( (u_n- \alpha_n^2c_n)\eta_1)'\|_{L^2(a|\alpha_n|,b|\alpha_n|)}\Big)\\
&\leq C|\alpha_n|^{\frac{1}{2}} \Big((1+\|u''_n\|_{H^1(-2,2)})\|\eta\|_{H^1(\mathbb{R})}+\|u_n-\alpha_n^2 c_n\|_{H^1(-2,2)}\|\eta_1\|_{H^1(\mathbb{R})}\Big)\\
&\leq C|\alpha_n|^{\frac{1}{2}}.
\end{aligned}
\end{equation}
By Gagliardo-Nirenberg inequality, we have
$$|g_n(0)|\leq |\alpha_n|^{-\frac{1}{2}} (\|g_n'\|_{L^2 (a,b)}+|\alpha_n|\|g_n\|_{L^2(a,b)}),$$
and then
\begin{equation}\nonumber
\begin{aligned}
&|\alpha_n| \|g_{n^*}\|_{L^2 (a,b)}+ \|g_{n^*}'\|_{L^2 (a,b)}\leq \|g_n'\|_{L^2 (a,b)}+|\alpha_n|\|g_n\|_{L^2(a,b)}+C|\alpha_n|^{\frac{1}{2}} |g_n(0)|\\
&\quad\leq C  (\|g_n'\|_{L^2 (a,b)}+|\alpha_n|\|g_n\|_{L^2(a,b)})\to 0,\\
& |\alpha_n| \|\psi_{n^*}-\psi_n \|_{L^2(a,b)}+ \|\partial_y (\psi_{n^*}-\psi_n )\|_{L^2(a,b)}\to 0.
\end{aligned}
\end{equation}
Note that $\|\psi_n'\|_{L^2(a,b)}+|\alpha_n|\|\psi_n\|_{L^2(a,b)} \leq 1$, then for large $n$, it holds that $ |\alpha_n| \|\psi_{n^*} \|_{L^2(a,b)}+ \|\partial_y \psi_{n^*}\|_{L^2(a,b)}\leq 2$. Applying Lemma \ref{critical-3} to show that $ |\alpha_n| \|\psi_{n^*} \|_{L^2(-\delta/2,\delta/2)}+ \|\partial_y \psi_{n^*}\|_{L^2(-\delta/2,\delta/2)}\to 0$, which gives
\begin{equation}\nonumber
\begin{aligned}
& |\alpha_n| \|\psi_n \|_{L^2(-\delta/2,\delta/2)}+ \|\partial_y \psi_n \|_{L^2(-\delta/2,\delta/2)}\leq  |\alpha_n| \|\psi_{n^*}-\psi_n \|_{L^2(a,b)}\\
&+ \|\partial_y (\psi_{n^*}-\psi_n )\|_{L^2(a,b)}+ |\alpha_n| \|\psi_{n^*} \|_{L^2(-\delta/2,\delta/2)}+ \|\partial_y \psi_{n^*}\|_{L^2(-\delta/2,\delta/2)} \to 0.
\end{aligned}
\end{equation}
Thanks to $\alpha_n \to \infty$ and the definition of $\widetilde{\psi_n}$, we have
\begin{equation}\nonumber
\begin{aligned}
& \|\widetilde{\psi_n} \|_{L^2(-( |\alpha_n|\delta)/2,( |\alpha_n| \delta)/2)}+ \|\widetilde{ \psi_n}' \|_{L^2(-( |\alpha_n|\delta)/2,( |\alpha_n|\delta)/2)}\\
&= |\alpha_n| \|\psi_n \|_{L^2(-\delta/2,\delta/2)}+ \|\partial_y \psi_n \|_{L^2(-\delta/2,\delta/2)}\to 0,
\end{aligned}
\end{equation}
and $-( |\alpha_n|\delta)/2\to -\infty, ( |\alpha_n|\delta)/2\to +\infty$, then $\widetilde{\psi_n} \to 0$ in $H^1_{loc}(\mathbb{R})$.
\end{proof}

Now we are in a position to prove Proposition \ref{est-1}.

\begin{proof}[Proof of Proposition \ref{est-1}]
We only consider the case of $\mathrm{Im}\ c>0$ and the case of $\mathrm{Im}\ c<0$ can be obtained via taking conjugation. The proof is based on the contradiction argument.

If the conclusion is not true, then there exists $\psi_n \in H^1_0(-1,1),g_n \in H^1(-1,1)$ and $c_n,\alpha_n$ with $\mathrm{Im}\ c_n >0, \alpha_n>0$ such that $\|\psi_n'\|_{L^2}+|\alpha_n|\|\psi_n\|_{L^2}=1,\|g_n'\|_{L^2}+|\alpha_n|\|g_n \|_{L^2} \to 0, \alpha_n \to +\infty$ and
$$(u-c_n)(\psi_n''-\alpha_n^2 \psi_n)-u''\psi_n=g_n.$$
Note that $y_0=0$ is the critical point of $u$, then there exists $\delta,c_0\in (0,1)$ so that $ B(0,\delta) \subset [-1,1]$, and $|u''|>c_0$ for $y\in \Sigma_\delta$, where $\Sigma_\delta=[-1,1]\setminus B(0,\delta)$. For $0<r<\delta, \alpha_n\geq \frac{2}{c_0r}$, by Lemma \ref{monotone}, we have
\begin{equation}\nonumber
\begin{aligned}
&\|\partial_y\psi_n\|_{L^2(\Sigma_r)}+|\alpha_n|\|\psi_n\|_{L^2(\Sigma_r)}\\
&\leq C(r|\alpha_n|)^{-1} \left(\|\partial_y (g_n+u''\psi_n)\|_{L^2}+|\alpha_n|\| g_n+u''\psi_n \|_{L^2}+|\alpha_n|\|\psi\|_{L^2}\right)\\
&\leq C(r|\alpha_n|)^{-1} \left(\|\partial_y\psi_n\|_{L^2}+|\alpha_n|\|\psi_n\|_{L^2}+\|\partial_yg_n\|_{L^2}+|\alpha_n|\|g_n\|_{L^2}\right)\leq C(r|\alpha_n|)^{-1},
 \end{aligned}
\end{equation}
where $C>0$ is independent of $n,r$.

Set $r_n=R/|\alpha_n|$. For any fixed $R>\frac{2}{c_0}$ and  $\alpha_n>R/\delta$(for large  $n$),  it holds that $0<r<\delta,\alpha_n =R/r_n >\frac{2}{c_0r_n}$ and
\begin{equation}\nonumber
\begin{aligned}
&\|\partial_y\psi_n\|_{L^2(\Sigma_{R/|\alpha_n|})}+|\alpha_n|\|\psi_n\|_{L^2(\Sigma_{R/|\alpha_n|})}\leq C(r_n |\alpha_n|)^{-1}=CR^{-1},
 \end{aligned}
\end{equation}
which gives
\begin{equation}\label{lim-sup}
\begin{aligned}
\limsup_{n\to +\infty}\big(\|\partial_y\psi_n\|_{L^2(\Sigma_{R/|\alpha_n|})}+|\alpha_n|\|\psi_n\|_{L^2(\Sigma_{R/|\alpha_n|})}\big)\leq CR^{-1}.
 \end{aligned}
\end{equation}

It remains to deal with the interval near the critical point $y_0=0$. Indeed, for any fixed $R>\frac{2}{c_0}$, let $\widetilde{\psi_n}(y)=|\alpha_n|^{\frac{1}{2}} \psi_n (y_0+y/|\alpha_n|)$, Lemma \ref{critical-4} implies that $\widetilde{\psi_n}\to 0$ in $H^1_{loc}(\mathbb{R})$. Notice that
\begin{equation}\nonumber
\begin{aligned}
\|\partial_y\psi_n\|_{L^2(-R/|\alpha_n|,R/|\alpha_n|)}+|\alpha_n|\|\psi_n\|_{L^2(-R/|\alpha_n|,R/|\alpha_n|)}\\
=\|\partial_y\widetilde{\psi_n}\|_{L^2(-R,R)}+\|\widetilde{\psi_n}\|_{L^2(-R,R)} ,
 \end{aligned}
\end{equation}
which together with $\widetilde{\psi_n}\to 0$ in $H^1_{loc}(\mathbb{R})$ implies that
\begin{equation}\label{lim-2}
\begin{aligned}
\lim_{n\to +\infty} (\|\partial_y\psi_n\|_{L^2(-R/|\alpha_n|,R/|\alpha_n|)}+|\alpha_n|\|\psi_n\|_{L^2(-R/|\alpha_n|,R/|\alpha_n|)})=0.
 \end{aligned}
\end{equation}
Therefore, we obtain
\begin{equation}\nonumber
\begin{aligned}
1=&\|\partial_y\psi_n\|_{L^2}+|\alpha_n|\|\psi_n\|_{L^2}
\leq \|\partial_y\psi_n\|_{L^2(\Sigma_{R/|\alpha_n|})}+|\alpha_n|\|\psi_n\|_{L^2(\Sigma_{R/|\alpha_n|})}\\
&+(\|\partial_y\psi_n\|_{L^2(-R/|\alpha_n|,R/|\alpha_n|)}+|\alpha_n|\|\psi_n\|_{L^2(-R/|\alpha_n|,R/|\alpha_n|)}).
 \end{aligned}
\end{equation}
This along with \eqref{lim-sup} and \eqref{lim-2} shows that $1\leq CR^{-1}$ for any $R>\frac{2}{c_0}$ with $C>0$ independent of $R$, but this is a contradiction for $R$ large enough.
\end{proof}

Finally, we  prove Proposition \ref{est-2}.

\begin{proof}[Proof of Proposition \ref{est-2}]
First of all, let us assume that $|k|\geq M$ for some large $M$ determined later.

We introduce $\psi=\Psi_1+\Psi_2$ with
\begin{equation*}
\begin{aligned}
&(y^2-c)(\partial_y^2-|k|^2)\Psi_1-u''\Psi_1=g\chi_{(-\delta,\delta)},\\
&(y^2-c)(\partial_y^2-|k|^2)\Psi_2-u''\Psi_2=g\chi_{(-1,1)\setminus (-\delta,\delta)},
 \end{aligned}
\end{equation*}
where $\delta=1-(100|k|)^{-1+\beta}\geq \frac{1}{2}$ for any small enough $0<\beta\ll 1$.

For $y\in (-\delta,\delta)$, it holds that $|g\chi_{(-\delta,\delta)}|\leq e^{-|k|(1+y)}\lesssim e^{-c|k|^\beta}$. Applying Proposition \ref{est-1} to $\Psi_1$, we have
\begin{equation}\label{est-psi-1}
\begin{aligned}
\|\partial_y \Psi_1\|_{L^2}+|k|\|\Psi_1\|_{L^2}\leq C (\|\partial_y (g\chi_{(-\delta,\delta)})\|_{L^2}+|k|\| g\chi_{(-\delta,\delta)}\|_{L^2})\leq C|k|^{-n},
 \end{aligned}
\end{equation}
where $n$ is any positive constant.

Recall that $(y^2-c)(\partial_y^2-|k|^2)\Psi_2=g\chi_{(-1,1)\setminus (-\delta,\delta)}+2\Psi_2$. We get by Lemma \ref{monotone} that
\begin{equation*}
\begin{aligned}
& \|\partial_y \Psi_2\|_{L^2\big((-1,1)\setminus (-\delta,\delta)\big)}+|k| \| \Psi_2\|_{L^2\big((-1,1)\setminus (-\delta,\delta)\big)}\\
&\leq C(\delta |k|)^{-1} \Big(\|\partial_y g\chi_{(-1,1)\setminus (-\delta,\delta)}\|_{L^2}+|k|\|g\chi_{(-1,1)\setminus (-\delta,\delta)}\|_{L^2}+|k|\|\Psi_2\|_{L^2}\Big)\\
&\leq 2C |k|^{-1}\big(\|\partial_y g\chi_{(-1,1)\setminus (-\delta,\delta)}\|_{L^2}+|k|\|g\chi_{(-1,1)\setminus (-\delta,\delta)}\|_{L^2}\big)\\
&\quad+2C |k|^{-1}\big(\|\partial_y \psi \|_{L^2}+|k| \| \psi \|_{L^2})+2C|k|^{-1}(\|\partial_y \Psi_1\|_{L^2}+|k|\|\Psi_1\|_{L^2}\big)\\
&\leq C |k|^{-\frac{1}{2}}+C |k|^{-1}\big(\|\partial_y \psi \|_{L^2}+|k| \| \psi \|_{L^2}\big)+C|k|^{-1-n}.
 \end{aligned}
\end{equation*}
By the Sobolev inequality, we have
\begin{equation*}
\begin{aligned}
|\Psi_2(\pm \delta)|\leq & \|\Psi_2\|_{L^\infty \big((-1,1)\setminus (-\delta,\delta)\big)}\leq  \|\Psi_2\|_{L^2 \big((-1,1)\setminus (-\delta,\delta)\big)}^{\frac{1}{2}} \|\Psi_2\|_{H^1 \big((-1,1)\setminus (-\delta,\delta)\big)}^{\frac{1}{2}}\\
\leq &C|k|^{-\frac{1}{2}} \left( \|\partial_y \Psi_2\|_{L^2\big((-1,1)\setminus (-\delta,\delta)\big)}+|k| \| \Psi_2\|_{L^2\big((-1,1)\setminus (-\delta,\delta)\big)}\right).
 \end{aligned}
\end{equation*}

Next we give the estimate of $\Psi_2$ on $(-\delta,\delta)$. To this end, let us define
$$\Psi_{2,1}=\frac{\sinh |k|(\delta+y)}{\sinh 2\delta |k|}\Psi_2(\delta),\Psi_{2,2}=\frac{\sinh |k|(\delta-y)}{\sinh 2\delta |k|}\Psi_2(-\delta),$$
$$\Psi_{2,0}=\Psi_2-\Psi_{2,1}-\Psi_{2,2}.$$
Then we have
\begin{equation*}
\left\{
\begin{aligned}
&(y^2-c)(\Psi_{2,i}''-|k|^2\Psi_{2,i})-2\Psi_{2,i}=-2\Psi_{2,i},\ i=1,2,\\
&\Psi_{2,1}(\delta)=\Psi_2(\delta),\ \Psi_{2,1}(-\delta)=0,\ \Psi_{2,2}(\delta)=0,\ \Psi_{2,2}(-\delta)=\Psi_2(-\delta)
 \end{aligned}\right.
\end{equation*}
and
\begin{equation*}
\left\{
\begin{aligned}
& (y^2-c)(\Psi_{2,0}''-|k|^2\Psi_{2,0})-2\Psi_{2,0}=2(\Psi_{2,1}+\Psi_{2,2})+g\chi_{(-1,1)\setminus (-\delta,\delta)},\\
& \Psi_{2,0}(\pm \delta)=0.
 \end{aligned}\right.
\end{equation*}
Applying Proposition \ref{est-1} to $\Psi_{2,0}$ on $(-\delta,\delta)$, we infer that
\begin{equation*}
\begin{aligned}
&\|\partial_y \Psi_{2,0} \|_{L^2 (-\delta,\delta)} + |k| \| \Psi_{2,0} \|_{L^2 (-\delta,\delta)}\leq C|k|^{\frac{1}{2}} |\Psi_2(\pm \delta)|\\
&\leq C\left( \|\partial_y \Psi_2\|_{L^2\big((-1,1)\setminus (-\delta,\delta)\big)}+|k| \| \Psi_2\|_{L^2\big((-1,1)\setminus (-\delta,\delta)\big)}\right),
 \end{aligned}
\end{equation*}
which gives
\begin{equation*}
\begin{aligned}
&\|\partial_y \Psi_{2} \|_{L^2 (-\delta,\delta)} + |k| \| \Psi_{2} \|_{L^2 (-\delta,\delta)}\\
&\leq \sum_{i=0}^2 ( \|\partial_y \Psi_{2,i} \|_{L^2 (-\delta,\delta)} + |k| \| \Psi_{2,i} \|_{L^2 (-\delta,\delta)}) \leq C|k|^{\frac{1}{2}} |\Psi_2(\pm \delta)|\\
&\leq C\left( \|\partial_y \Psi_2\|_{L^2\big((-1,1)\setminus (-\delta,\delta)\big)}+|k| \| \Psi_2\|_{L^2\big((-1,1)\setminus (-\delta,\delta)\big)}\right)\\
&\leq C |k|^{-\frac{1}{2}}+C |k|^{-1}  (\|\partial_y \psi \|_{L^2}+|k| \| \psi \|_{L^2}).
 \end{aligned}
\end{equation*}
Moreover, we have
\begin{equation*}
\begin{aligned}
\|\partial_y \Psi_{2} \|_{L^2 } + |k| \| \Psi_{2} \|_{L^2}
\leq C |k|^{-\frac{1}{2}}+C |k|^{-1}  (\|\partial_y \psi \|_{L^2}+|k| \| \psi \|_{L^2}).
 \end{aligned}
\end{equation*}
Thus, we obtain
\begin{equation*}
\begin{aligned}
\|\partial_y \psi \|_{L^2} + |k| \| \psi \|_{L^2 }
\leq &\|\partial_y \Psi_1\|_{L^2}+|k|\|\Psi_1\|_{L^2}+\|\partial_y \Psi_2\|_{L^2}+|k|\|\Psi_2\|_{L^2}\\
\leq &C |k|^{-\frac{1}{2}}+C |k|^{-1}  (\|\partial_y \psi \|_{L^2}+|k| \| \psi \|_{L^2}).
 \end{aligned}
\end{equation*}
Choosing $|k|\geq 2C=:M$,  we have
\begin{equation*}
\begin{aligned}
\|\partial_y \psi \|_{L^2} + |k| \| \psi \|_{L^2 }
\leq C |k|^{-\frac{1}{2}}.
 \end{aligned}
\end{equation*}

If $1\leq |k|\leq M$, by Proposition 6.1 in \cite{WZZ-apde}, we have
\begin{equation*}
\begin{aligned}
\|\partial_y \psi \|_{L^2} + |k| \| \psi \|_{L^2 }\leq& (1+M) \| \psi \|_{H^1} \leq C(1+M)\|g\|_{H^1}\\
\leq& C(1+M)|k|^{\frac{1}{2}}\leq C(1+M) M |k|^{-\frac{1}{2}}\leq C|k|^{-\frac{1}{2}}.
 \end{aligned}
\end{equation*}

The proof is completed.
\end{proof}

\section{Space-time estimates for linearized Euler}

We need the following space-time estimates for the linearized Euler equation around the Poiseuille flow
\beno
 \partial_t{\om}+ik_1(1-y^2)\omega+2ik_1\psi=f
\eeno
for $t\in[0,T]$ and $y\in[-1,1]$, and $\psi=(\partial^2_y-|k|^2)^{-1}\om$.

\begin{proposition}\label{lem:Rayleigh-000}
It holds that
\begin{equation}\label{est:omT-000}
  \begin{aligned}
 & \|\om(T)\|_{L^2}^2+|k_1|\int_0^T\big(\|\partial_y\psi(t)\|_{L^2}^2+|k|^2\|\psi(t)\|_{L^2}^2\big)dt\\
&\leq C \|(\partial_y,|k|)\om(0)\|_{L^2}^2+C|k_1|^{-1}\int_0^T\big(\|\partial_yf(t)\|_{L^2}^2+|k|^2\|f(t)\|_{L^2}^2\big)dt.
  \end{aligned}
\end{equation}
If $f(t,-1)=f(t,1)=0$, then we also have
\begin{equation}\label{est:parpsi0-000}
  \begin{aligned}
&\int_0^T\big(|\partial_y\psi(t,0)|^2+|\partial_y\psi(t,1)|^2\big)dt\\
&\leq C \|(\partial_y,|k|)\om(0)\|_{L^2}^2+C|k_1|^{-1}\int_0^T\big(\|\partial_yf(t)\|_{L^2}^2+|k|^2\|f(t)\|_{L^2}^2\big)dt.
\end{aligned}
\end{equation}
When $f=0$, we have
\begin{align}\label{est:kpsi-Linf-000}
    &\|k\psi(t)\|_{L^2}\leq C\min\Big\{\dfrac{|k|^{\f12}}{\langle t\rangle^2}\|\omega(0)\|_{H^2},\dfrac{\|\omega(0)\|_{H^1}}{|k|^{\f12}\langle t\rangle},\dfrac{\|\omega(0)\|_{L^2}}{|k|}\Big\},
 \end{align}
 which in particular gives
 \begin{align}\label{est:tkpsi-L2-000}
     &\int_{0}^{T}\langle t\rangle^2\|k\psi(t)\|_{L^2}^2dt \leq C|k|^{-1}(\|\omega(0)\|_{H^2}^2+|k|^2\|\omega(0)\|_{H^1}^2).
  \end{align}
\end{proposition}

\begin{proof}
 \textbf{Step 1.} We consider the case when $\omega(0)=0$, $f(t,\pm1)=0$.\smallskip

 We can take Laplace transform in $t$. For $\text{Re}(\lambda)>0$, let
\begin{align*}
   &\Phi(\lambda,y)=\int_{0}^{+\infty}\psi(t,y)e^{-\lambda t}dt, \ F(\lambda,y)=\int_{0}^{+\infty}f(t,y)e^{-\lambda t}dt.
\end{align*}
Then $\Phi(\lambda,\cdot)\in H^2(0,1),\ F(\lambda,\cdot)\in H^1(0,1)$ for $\text{Re}(\lambda)>0$. Using Plancherel's formula, we know that for $\varepsilon>0,\ j=0,1$,
\begin{align}
    &\int_{\mathbb{R}}\|\partial_y^j\Phi(\varepsilon+is)\|_{L^2}^2ds=2\pi \int_{0}^{+\infty}e^{-2\varepsilon t}\|\partial_y^j\psi(t)\|_{L^2}^2dt,\label{est:parPhi-Psi-001}\\
    &\int_{\mathbb{R}}|\partial_y\Phi(\varepsilon+is,\pm1)|^2ds=2\pi \int_{0}^{+\infty}e^{-2\varepsilon t}|\partial_y\psi(t,\pm1)|^2dt,\label{est:parPhi-Psi-002}\\
    &\int_{\mathbb{R}}\|\partial_y^jF(\varepsilon+is)\|_{L^2}^2ds=2\pi \int_{0}^{+\infty}e^{-2\varepsilon t}\|\partial_y^jf(t)\|_{L^2}^2dt.\label{est:parPhi-Psi-003}
 \end{align}
 Furthermore, $\Phi$ satisfies
 \begin{align}\label{est:Phieq-000}
     &(1-y^2-i\lambda/k_1)(\partial_y^2\Phi-|k|^2\Phi)+2\Phi=W,\quad \Phi(\lambda,0)=\Phi(\lambda,1)=0
  \end{align}
  with $W=(i/k_1)F$.

  If $|\text{Re}(\lambda)|\in (0,|k_1|\epsilon_2)$($\epsilon_2>0$ is the constant in Proposition \ref{est-1}), then $|\text{Im}(i\lambda/k_1)|\in (0,\epsilon_2)$, and by   Proposition \ref{est-1},
  \begin{align}\label{est:paryPhiL2-F-000}
    &\|\partial_y\Phi\|_{L^2}+|k|\|\Phi\|_{L^2}\leq C(\|\partial_y F\|_{L^2}+|k|\|F\|_{L^2}).
  \end{align}
 Integrating this over $\text{Re}(\lambda)=\varepsilon\in (0,|k_1|\epsilon_2)$ and using \eqref{est:parPhi-Psi-001} and \eqref{est:parPhi-Psi-003}, we deduce that
  \begin{align*}
      & \int_{0}^{+\infty}e^{-2\varepsilon t}\|(\partial_y,|k|)\psi(t)\|_{L^2}^2dt\leq C|k_1|^{-2}\int_{0}^{+\infty}e^{-2\varepsilon t}\|(\partial_y,|k|)f(t)\|_{L^2}^2dt.
   \end{align*}
   Letting $\varepsilon\rightarrow 0+$, we obtain
   \begin{align}\label{est:intparpsi-000}
      \int_{0}^{T}\|(\partial_y,|k|)\psi(t)\|_{L^2}^2dt\leq C|k_1|^{-2}\int_{0}^{T}\|(\partial_y,|k|)f(t)\|_{L^2}^2dt.
   \end{align}

  Let $\gamma_1=-\frac{\sinh(|k|(1+y))}{\sinh(2|k|)}$, and $\gamma_1^{(0)}$ solve
   \begin{align*}
      & (1-y^2+i\bar{\lambda}/k_1)(\partial_y^2-|k|^2)\gamma_1^{(0)}
      +2\gamma_1^{(0)}=\gamma_1, \quad \gamma_1^{(0)}(\pm1)=0.
   \end{align*}
   By Proposition \ref{est-2},  we have
\begin{align}\label{est:gamm10-000}
      & \|(\partial_y,|k|)\gamma_1^{(0)}\|_{L^2}\leq C|k|^{-\f12}.
   \end{align}
   Then by a dual argument, we have (thanks to $f(\pm1)=0$, we have $W(\pm1)=0$)
  \begin{align*}
     &\langle (\partial_y^2-|k|^2)\Phi,\gamma_1\rangle=\big \langle (1-y^2-i\lambda/k_1)(\partial_y^2\Phi-|k|^2\Phi)+2\Phi,(\partial_y^2-|k|^2)\gamma_1^{(0)}\big\rangle\\
     &=\langle W,(\partial_y^2-|k|^2)\gamma_1^{(0)}\rangle=-\langle \partial_yW,\partial_y\gamma_1^{(0)} \rangle-|k|^2\langle W,\gamma_1^{(0)}\rangle.
  \end{align*}
  Then we conclude by \eqref{est:gamm10-000} and  $W=(i/k_1)F$ that
  \begin{align*}
     &|\partial_y\Phi(t,1)|= |\langle (\partial_y^2-|k|^2)\Phi,\gamma_1\rangle|\leq C|k_1|^{-1}|k|^{-\f12}\|(\partial_y,|k|)F\|_{L^2},
  \end{align*}
  which along with \eqref{est:parPhi-Psi-002} and \eqref{est:parPhi-Psi-003} gives
  \begin{align*}
     &  \int_{0}^{+\infty}e^{-2\varepsilon t}|\partial_y\psi(t,1)|^2dt \leq C |k_1|^{-2}|k|^{-1} \int_{0}^{+\infty}e^{-2\varepsilon t}\|(\partial_y,|k|)f(t)\|_{L^2}^2dt.
  \end{align*}
Letting $\varepsilon\rightarrow 0+$, we obtain
   \begin{align*}
      \int_{0}^{T}|\partial_y\psi(t,1)|^2dt \leq C |k_1|^{-2}|k|^{-1} \int_{0}^{T}\|(\partial_y,|k|)f(t)\|_{L^2}^2dt.
   \end{align*}
  Similarly, we can get the same estimate for $|\partial_y\psi(t,-1)|$, then we coclude
  \begin{align}\label{est:intparpsi(pm1)-000}
      \int_{0}^{T}|\partial_y\psi(t,\pm1)|^2dt \leq C |k_1|^{-2}|k|^{-1} \int_{0}^{T}\|(\partial_y,|k|)f(t)\|_{L^2}^2dt.
   \end{align}

     \begin{remark}\label{Rem:intpar-00}
       In the proof, we can find that \eqref{est:intparpsi-000} still holds for the case when $\omega(0)=0, \ f(t,\pm1)\neq 0$.
     \end{remark}

  \textbf{Step 2.} We first consider the case of $f=0$.\smallskip

 From \cite{WZZ-apde}, we know that
 \begin{align}\label{est:parpsi-wH1-00}
     & \|\partial_y\psi\|_{L^2}+|k|\|\psi\|_{L^2}\leq C\min\Big\{\dfrac{\|\omega(0)\|_{H^1}}{|k|^{\f12}\langle t\rangle},|k|^{-1}\|\omega(0)\|_{L^2}\Big\},
  \end{align}and
   \begin{align}\label{est:kpsi-wH1-00}
     &\|k\psi(t)\|_{L^2}\leq C\min\Big\{\dfrac{|k|^{\f12}}{\langle t\rangle^2}\|\omega(0)\|_{H^2},\dfrac{\|\omega(0)\|_{H^1}}{|k|^{\f12}\langle t\rangle},\dfrac{\|\omega(0)\|_{L^2}}{|k|}\Big\}.
  \end{align}

 Taking the inner product with $w$ and taking the real part, we have
 \begin{align}\label{est:w-cons-000}
     &\dfrac{1}{2}\dfrac{d}{dt}\|\omega\|_{L^2}^2= 0\Rightarrow \|\omega\|_{L^2}\leq \|\omega(0)\|_{L^2}.
  \end{align}
  By \eqref{est:parpsi-wH1-00}, we have
  \begin{align}\label{est:parpsiL2-wH1-00}
     & \int_{0}^{T}(\|\partial_y\psi(t)\|_{L^2}^2+|k|^2\|\psi(t)\|_{L^2}^2)dt\leq C|k|^{-\f52}\|(\partial_y,|k|)\omega(0)\|_{L^2}^2.
  \end{align}

 Let $\gamma_1=-\frac{\sinh(|k|(1+y))}{\sinh(2|k|)}$, $\chi \geq 0$ be a cut-off function so that $\chi(y)=0$ if $|y|\leq 1/4$ and $\chi(y)=1$ if $1/2\leq |y|\leq 1$ .
Let $\omega_1=\omega\chi$ and $\omega_2=\omega(1-\chi)$, $\psi_j=(\partial_y^2-|k|^2)^{-1}\omega_j,\ j=1,2$. Then we have
\begin{equation*}
  \left\{\begin{aligned}
&\partial_t{\om_1}+ik_1(1-y^2)\omega_1=-2ik_1\psi\chi,\\
    &\omega_1|_{t=0}=\omega(0)\chi.
  \end{aligned}\right.
\end{equation*}
Let $c_0=\inf_{1/4\leq |y|\leq 1}|\partial_y(1-y^2)|=1/2$. Thanks to the fact that on the support of $\omega_1$, $|\partial_y(1-y^2)|\geq c_0>0$,
we deduce from Proposition 2.1 in \cite{WZZ-cmp} that
\begin{equation*}
\begin{aligned}
    &\int_{0}^{T}|\langle \omega_1(t,y),\gamma_1\rangle|^2dt=\int_{0}^{T}|\partial_y\psi_1(t,1)|^2dt\\
    &\leq  Cc_0^{-1}|k|^{-1}\|\omega_1(0)\|^2_{L^2}+ Cc_0^{-2} \int_{0}^{T}\|(\partial_y,|k|)(k_1\psi\chi)\|_{L^2}^2dt\\
    &\leq  C|k|^{-1}\|\omega(0)\|^2_{L^2}+ C |k_1|^2\int_{0}^{T}\|(\partial_y,|k|)\psi\|_{L^2}^2dt,
 \end{aligned}
\end{equation*}
which along with \eqref{est:parpsiL2-wH1-00} gives
\begin{align}
     \int_{0}^{T}|\langle \omega_1(t,y),\gamma_1\rangle|^2dt \leq & C|k|^{-1}\|\omega(0)\|^2_{L^2}+ C |k_1|^2|k|^{-\f52}\|(\partial_y,|k|)\omega(0)\|_{L^2}^2\nonumber\\
     \leq &C\|(\partial_y,|k|)\omega(0)\|_{L^2}^2.\label{est:om1-ga1-000}
 \end{align}
 For $\omega_2$, we have $\omega_2|_{y=\pm1}=0$, then by integration by parts and \eqref{est:parpsiL2-wH1-00}, we have
 \begin{align*}
    \int_{0}^{T}|\langle \omega_2(t,y),\gamma_1\rangle|^2dt&=\int_{0}^{T}|\langle \omega(t,y),(1-\chi)\gamma_1\rangle|^2dt\\
    &\leq \int_{0}^{T}\|(\partial_y,|k|)\psi(t)\|_{L^2}^2\|(\partial_y,|k|)((1-\chi)\gamma_1)
    \|^2_{L^2}dt\\
    &\leq Ce^{-|k|/2}\int_{0}^{T}\|(\partial_y,|k|)\psi(t)\|_{L^2}^2dt\leq C\|(\partial_y,|k|)\omega(0)\|_{L^2}^2,
 \end{align*}
 which along with \eqref{est:om1-ga1-000} gives
 \begin{align*}
    \int_{0}^{T}|\partial_y\psi(t,1)|^2dt=\int_{0}^{T}|\langle \omega(t,y),\gamma_1\rangle|^2dt\leq C\|(\partial_y,|k|)\omega(0)\|_{L^2}^2.
  \end{align*}
We can obtain the same estimate for $\partial_y\psi(t,-1)$, thus,
 \begin{align}\label{est:om-ga1-000}
    \int_{0}^{T}|\partial_y\psi(t,\pm1)|^2dt\leq  C\|(\partial_y,|k|)\omega(0)\|_{L^2}^2.
  \end{align}

    \textbf{Step 3.} We consider the case of $\omega(0)=0$.\smallskip

As in Case 2,  we still denote $\omega_1=\omega\chi$ and $\omega_2=\omega(1-\chi)$, $\psi_j=(\partial_y^2-|k|^2)^{-1}\omega_j,\ j=1,2$, then
\begin{equation*}
  \left\{\begin{aligned}
&\partial_t{\om_1}+ik_1(1-y^2)\omega_1=-2ik_1\psi\chi-f\chi,\\
    &\omega_1|_{t=0}=0.
  \end{aligned}\right.
\end{equation*}

Let $c_0=\inf_{1/4\leq |y|\leq 1}|\partial_y(1-y^2)|=1/2$. Thanks to the fact that on the support of $\omega_1$, $|\partial_y(1-y^2)|\geq c_0>0$, we deduce from Proposition 2.1 in \cite{WZZ-cmp} that
\begin{align*}
   \|\omega_1(T)\|_{L^2}^2
   \leq & Cc_0^{-1}|k_1|^{-2} \int_{0}^{T}\|(\partial_y,|k|)(2ik_1\psi\chi+f\chi)\|_{L^2}^2dt\\
    \leq &C \int_{0}^{T}\big(\|(\partial_y,|k|)\psi\|_{L^2}^2+|k_1|^{-2}\|(\partial_y,|k|)f\|_{L^2}^2\big)dt\\
    \leq &C|k_1|^{-2}\int_{0}^{T}\|(\partial_y,|k|)f\|_{L^2}^2dt.
\end{align*}
Here we have used \eqref{est:intparpsi-000} (recall Remark \ref{Rem:intpar-00}).

 For $\omega_2$, we have
 \begin{equation*}
  \left\{\begin{aligned}
&\partial_t{\om_2}+ik_1(1-y^2)\omega_2=-2ik_1\psi(1-\chi)-f(1-\chi),\\
    &\omega_2|_{t=0}=0.
  \end{aligned}\right.
\end{equation*}
Notice that $\big(2ik_1\psi(1-\chi)+f(1-\chi)\big)|_{y=\pm1}=0$, then
   \begin{align*}
       \dfrac{1}{2}\dfrac{d}{dt}\|\omega_2\|_{L^2}^2& =-\mathbf{Re}\langle \omega_2,2ik_1\psi(1-\chi)+f(1-\chi)\rangle \\
       &\leq C\|(\partial_y,|k|)\psi\|_{L^2}\|(\partial_y,|k|)(2ik_1\psi(1-\chi)+f(1-\chi))\|_{L^2}\\
       &\leq C\|(\partial_y,|k|)\psi\|_{L^2}\big(|k_1|\|(\partial_y,|k|)\psi\|_{L^2}+\|(\partial_y,|k|)f\|_{L^2})
    \end{align*} In summary, we obtain (using \eqref{est:intparpsi-000})
     \begin{align*}
        \|\omega_2(T)\|_{L^2}^2&\leq C |k_1| \int_{0}^{T}\big(\|(\partial_y,|k|)\psi\|_{L^2}^2+|k_1|^{-2}\|(\partial_y,|k|)f\|_{L^2}^2\big)dt\\
        &\leq C|k_1|^{-1}\int_{0}^{T}\|(\partial_y,|k|)f(t)\|_{L^2}^2dt.
     \end{align*}
 Thus, we arrive at
     \begin{equation}\label{est:w-intparpsi-000}
     \begin{aligned}
         & \|\omega(T)\|_{L^2}\leq C|k_1|^{-1}\int_{0}^{T}(\|\partial_yf(t)\|_{L^2}^2+ |k|^2\|f(t)\|_{L^2}^2)dt.
      \end{aligned}
     \end{equation}

  In summary,  \eqref{est:omT-000} follows from \eqref{est:intparpsi-000} (recall Remark \ref{Rem:intpar-00}), \eqref{est:parpsiL2-wH1-00}, \eqref{est:w-cons-000} and \eqref{est:w-intparpsi-000}. \eqref{est:parpsi0-000}  follows from \eqref{est:intparpsi(pm1)-000} and \eqref{est:om-ga1-000}.
 \eqref{est:kpsi-Linf-000} and \eqref{est:tkpsi-L2-000}  follow from \eqref{est:parpsi-wH1-00} and \eqref{est:kpsi-wH1-00}.
\end{proof}

\section*{Acknowledgment}
Appendix C comes from an unpublished note related to a joint work \cite{WZZ-apde} of Zhifei Zhang. We thank Dongyi Wei and Weiren Zhao for allowing us to share this interesting note in this paper. We also thank Weiren Zhao for pointing out an error about the proof of the space-time estimates in the previous version. S. Ding is supported by the Key Project of National Natural Science Foundation of China  under Grant 12131010, National Natural Science Foundation of China  under Grant 12271032 and  Natural Science Foundation of Guangdong Province under Grant 2021A1515010249 and 2021A1515010303. Z. Lin and Z. Zhang are partially supported by National Natural Science Foundation of China under Grant  11971009, 12171010 and 12288101.

\end{document}